\documentclass{amsart}%
\usepackage{amsfonts, amsmath, amssymb}%
\usepackage{graphicx, epic, epsf, pstricks}

\input xy
\xyoption{all}

\newtheorem{theorem}{Theorem}
\theoremstyle{definition}

\newtheorem{corollary}{Corollary}
\newtheorem{definition}{Definition}
\newtheorem{example}{Example}

\newtheorem{proposition}{Proposition}
\newtheorem{remark}{Remark}
\numberwithin{equation}{section}
\newcommand{\brak}[1]{\langle #1\rangle}

\def\co{\colon\thinspace} 
\def\bbZ{{\mathbb Z}}
\def\bbR{{\mathbb R}}

\begin{document}

\title{The Kauffman polynomial and trivalent graphs }

\author{Carmen Caprau}
\address{Department of Mathematics, California State University, Fresno\\ 
5245 North Backer Avenue M/S PB 108, Fresno, CA 93740-8001, USA}
\email{ccaprau@csufresno.edu}
\urladdr{}
\author{James Tipton}
\address{Department of Mathematics, The University of Iowa \\
14 MacLean Hall, Iowa City, IA 52242-1419, USA}

\date{}
\thanks{The first author was supported in part by NSF grant DMS 0906401}

\begin{abstract}
We construct a state model for the two-variable Kauffman polynomial using planar trivalent graphs. We also use this model to obtain a polynomial invariant for a certain type of trivalent graphs embedded in $\bbR^3$. 
\end{abstract}
\maketitle

\section{Introduction}

In~\cite{K2}, Kauffman constructed a two-variable Laurent polynomial invariant of regular isotopy for classical unoriented knots and links. The invariant of a link $L$ is denoted by $D_L : = D_L(z, a)$ and is uniquely determined by the axioms: 
\begin{enumerate}
\item[1.] $D_{L_1} = D_{L_2}$ whenever $L_1$ and $L_2$ are regular isotopic links.
\item[2.] $D_{\,\raisebox{-3pt}{\includegraphics[height=0.15in]{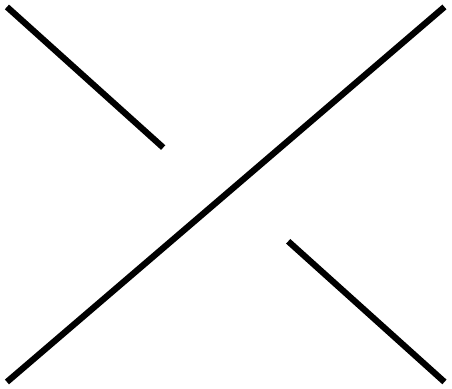}}}\, - D_{\,\raisebox{-3pt}{\includegraphics[height=0.15in]{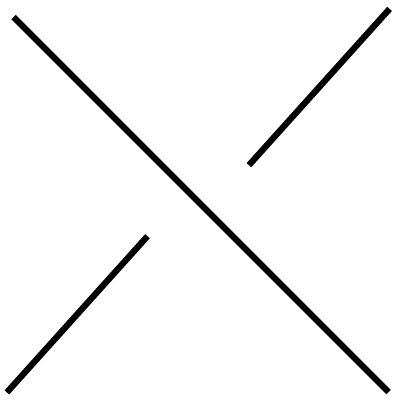}}}\,= z \left[D_{\,\raisebox{-3pt}{\includegraphics[height=0.15in]{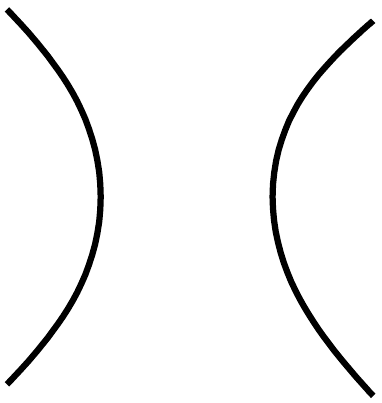}}}\, - D_{\,\raisebox{-3pt}{\includegraphics[height=0.15in]{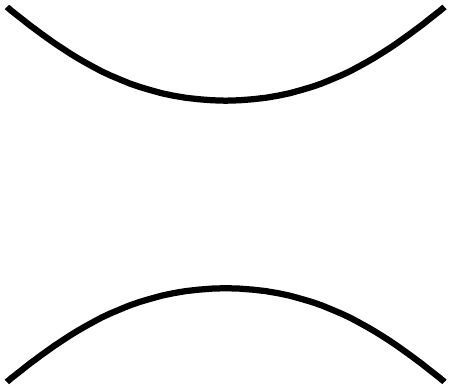}}}\,\right]$.
\item[3.] $D_{\,\bigcirc} = 1$.
\item[4.] $D_{\,\raisebox{-3pt}{\includegraphics[height=0.15in]{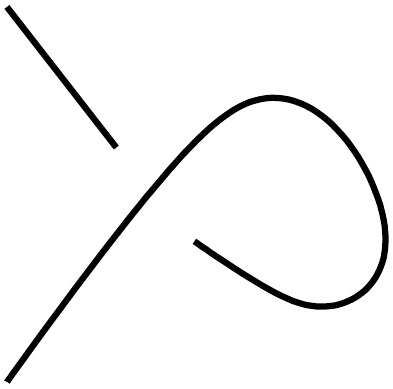}}}\, = a D_{\,\raisebox{-3pt}{\includegraphics[height=0.15in]{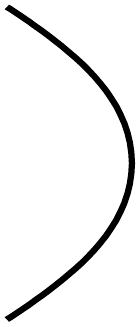}}}\,, \quad D_{\,\raisebox{-3pt}{\includegraphics[height=0.15in]{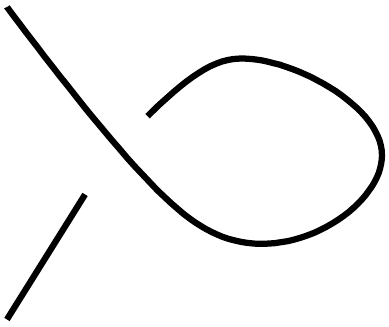}}}\, = a^{-1} D_{\,\raisebox{-3pt}{\includegraphics[height=0.15in]{arc}}}$.
\end{enumerate}
The diagrams in both sides of the second or fourth equation represent parts of larger link diagrams that are identical except near a point where they look as indicated.

We remark that this polynomial is sometimes called the Dubrovnik version of the two-variable \textit{Kauffman polynomial} of unoriented links. For more details about these polynomials we refer the reader to~\cite{K2, K3}. Throughout the paper we call $D_L$ the Kauffman polynomial of the link $L$.

The corresponding normalized invariant, $Y_L$, of ambient isotopy of oriented links is then given by 
\[ Y_L = a^{-w(L)}D_L \]
where $w(L)$ is the writhe of the oriented link $L$ and $D_L$ is the Kauffman polynomial of the associated unoriented link (corresponding to the oriented link $L$).

Kauffman and Vogel~\cite{KV} extended the Kauffman polynomial to a three-variable rational function for knotted 4-valent graphs (4-valent graphs embedded in $\bbR^3$) with rigid vertices by defining 
\[ D_{\,\raisebox{-3pt}{\includegraphics[height=0.15in]{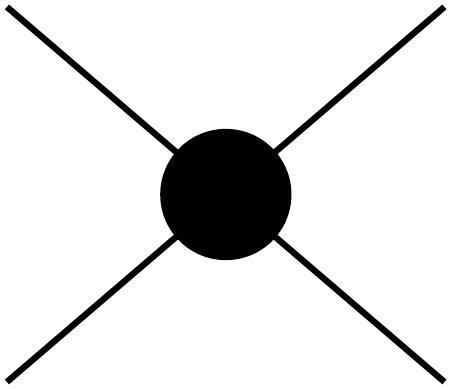}}}\, = D_{\,\raisebox{-3pt}{\includegraphics[height=0.15in]{poscrossing}}}\, - AD_{\,\raisebox{-3pt}{\includegraphics[height=0.15in]{A-smoothing}}}\, -B D_{\,\raisebox{-3pt}{\includegraphics[height=0.15in]{B-smoothing}}}\, =  D_{\,\raisebox{-3pt}{\includegraphics[height=0.15in]{negcrossing}}}\, - AD_{\,\raisebox{-3pt}{\includegraphics[height=0.15in]{B-smoothing}}}\, -B D_{\,\raisebox{-3pt}{\includegraphics[height=0.15in]{A-smoothing}}}\, \]
where $A$ and $B$ are commuting variables and $A - B = z$. In particular, the invariant for knotted 4-valent graphs is defined in terms of the Kauffman polynomial. In~\cite{KV}, it was also shown that the resulting polynomial $D_G$ of a knotted 4-valent graph $G$ satisfies certain graphical relations, which determine values for each planar unoriented 4-valent graph by recursive formulas defined entirely in the category of planar graphs.

Kauffman's and Vogel's results imply that there is a model for the Kauffman polynomial of an unoriented link via planar 4-valent graphs (these graphs have ordinary vertices; there is no reason to consider rigid vertices when working with planar graphs). This model can be also deduced from Carpentier's work~\cite{C} on the Kauffman and Vogel's polynomial by changing one's perspective, as the focus of that paper is not on the Kauffman polynomial invariant of links but on invariants of graphs. Thus we believe it is worthwhile presenting this model explicitly. We give in Section~\ref{sec:comments} a brief description of how this state model for the Kauffman polynomial can be obtained, but throughout the rest of the paper we work with trivalent graphs. The main reason for this is that trivalent graphs are generic graphs, in the sense that any graph, in particular a 4-valent graph, can be perturbed into a trivalent graph. In this paper, a trivalent graph may contain a circle component. We regard a circle as a graph without vertices.

The purpose of this paper is twofold. One goal is to use trivalent graphs to construct a rational function $\textbf{P}_L = \textbf{P}_L(A, B, a) \in \bbZ[A^{\pm1}, B^{\pm 1}, a^{\pm 1}, (A-B)^{\pm 1}]$ which is an invariant of regular isotopy for unoriented links, and we do this using two equivalent constructions. The first construction is via a state summation and a recurrence method to evaluate planar trivalent graphs - via skein theory. The inspiration for this construction came from the well-known model for the Homflypt polynomial via trivalent graphs constructed by Murakami, Ohtsuki and Yamada~\cite{MOY} (and, as already mentioned, the results of~\cite{KV}). The second method makes use of a representation of the braid group $B_n$ into some algebra $\mathcal{A}_n$ whose trace function recovers the invariant $\textbf{P}_L$. This algebra is given by generators and relations; the relations come from the skein theory used in our first definition. We remark that the algebra $\mathcal{A}_n$ is isomorphic to the Birman-Murakami-Wenzl algebra introduced in~\cite{BW, M}. Our construction and results easily imply that $D_L(A - B, a) = \textbf{P}_L(A, B, a)$. Although the first method may be implicitly obtained from Carpentier's work~\cite{C} by replacing 4-valent vertices with trivalent vertices, and by shifting the approach to the problem,  it might be beneficial for less experienced researchers to give all of the details here.  On the other hand, the second method giving rise to the Kauffman polynomial of a link is completely new.

 Another purpose of this paper is to construct a three-variable rational function $[G]$ (or equivalently, a four-variable Laurent polynomial) which is an invariant for a certain type of knotted unoriented trivalent graphs $G$, and we do this without relying on the existence of the Kauffman polynomial for links. Instead of extending the Kauffman polynomial to trivalent graphs embedded in $\bbR^3$, we give sufficient conditions that the graph polynomial $[G]$ must satisfy in order to be an invariant of regular isotopy of knotted graphs. These conditions imply that the restriction of $[G]$ to links gives the Laurent polynomial $\textbf{P}_L$ (thus a version of the Kauffman polynomial). We also take a look at the one-variable specialization of the Kauffman polynomial, the $SO(N)$ Kauffman polynomial, where we focus on the case $N = 2$. We show that in this particular case, the evaluation $[G]$ of a planar trivalent graph diagram depends only on the number of its connected components and the number of vertices. This implies that the polynomial $P_L$, as well as the corresponding polynomial $[G]$ of knotted trivalent graphs, are easily computed in this case.
 
The trivalent graphs under consideration in this paper have two types of edges, namely ``standard" edges and ``wide" edges, such that there is exactly one wide edge incident to a vertex: \raisebox{-5pt}{\includegraphics[height=0.2in]{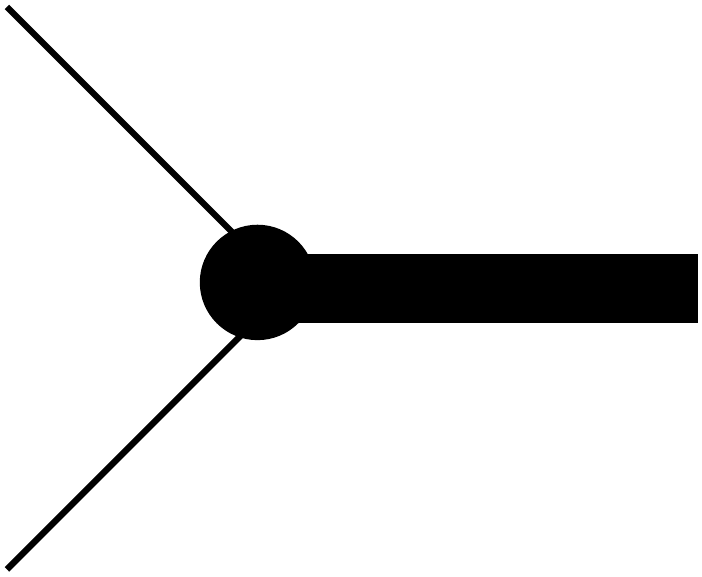}}\,. Unless otherwise stated, all trivalent graphs in this paper are of this type. When a trivalent graph is knotted, its wide edges  are regarded as being rigid, in the sense that there is a cyclic ordering of the four standard edges ``attached'' to a wide edge, given by the rigidity of the wide edge. Moreover, a wide edge is not allowed to cross itself. Such a knotted graph will be called \textit{knotted rigid-edge trivalent graph} (or shortly knotted RE 3-graph).
 
The paper is organized as follows: In Section~\ref{sec:model} we give the definition of the polynomial $\textbf{P}_L$ using a state sum formula and skein relations for trivalent graphs, and prove its invariance under the Reidemeister II and III moves, as well as its behavior under the Reidemeister I move. In Subsection~\ref{ssec:algebra} we define the algebra $\mathcal{A}_n$, while in Subsection~\ref{ssec:representation} we construct a representation of the braid group on $n$ strands, $B_n$, into $\mathcal{A}_n$. There is a trace function defined on the algebra $\mathcal{A}_n$, and we employ it in Subsection~\ref{ssec:bracket-braids} to define a \textit{bracket} polynomial of a braid and to recover the link invariant $\textbf{P}_L$. Section~\ref{sec: knotted graphs} is dedicated to knotted RE 3-graphs, and especially to the graph polynomial $[G]$. We give the model for the $SO(N)$ Kauffman polynomial in Subsection~\ref{ssec:one-variable poly}, and we treat the case $N =2$ for the corresponding polynomial for planar trivalent graph diagrams in Subsection~\ref{ssec:N=2}. A few concluding remarks are given in Section~\ref{sec:comments}. (Finally, an appendix is used to keep some technicalities out of the way.)

\section{A model for the Kauffman polynomial} \label{sec:model}

Let $D$ be a plane diagram of an unoriented link $L$ in $\bbR^3$. Associate to each crossing of $D$ a formal linear combination of planar trivalent graphs as shown in Figure~\ref{fig:replacing crossings} (where $A$ and $B$ are commuting variables).

\begin{figure}[ht]
\raisebox{-3pt}{\includegraphics[height=0.15in]{poscrossing}} = $A$ \raisebox{-3pt}{\includegraphics[height=0.15in]{A-smoothing}} + $B$ \raisebox{-3pt}{\includegraphics[height=0.15in]{B-smoothing}} +  \raisebox{-6pt}{\includegraphics[height=0.25in]{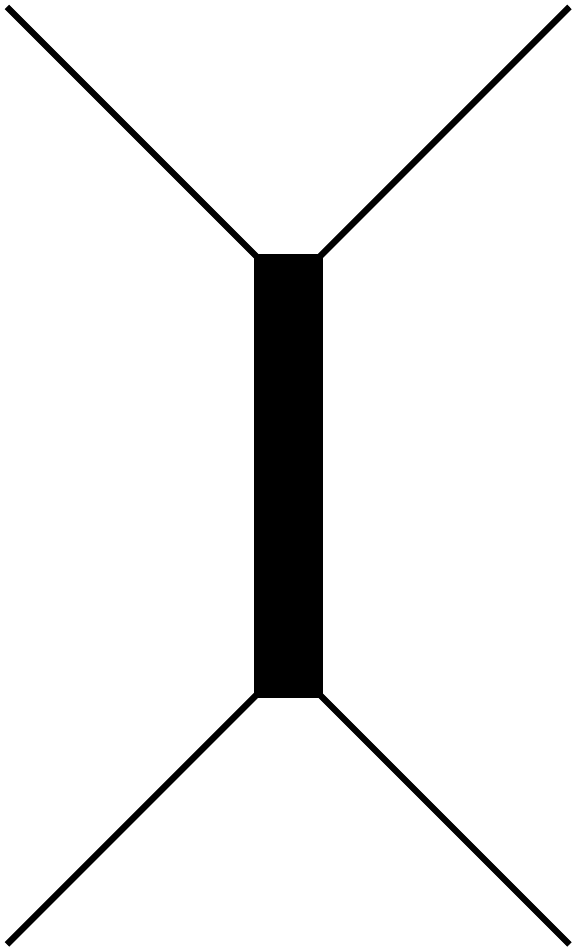}} \\
\vspace{0.5cm}
\raisebox{-3pt}{\includegraphics[height=0.15in]{negcrossing}} = $A$  \raisebox{-3pt}{\includegraphics[height=0.15in]{B-smoothing}} +$B$  \raisebox{-3pt}{\includegraphics[height=0.15in]{A-smoothing}} +  \raisebox{-4pt}{\includegraphics[height=0.17in]{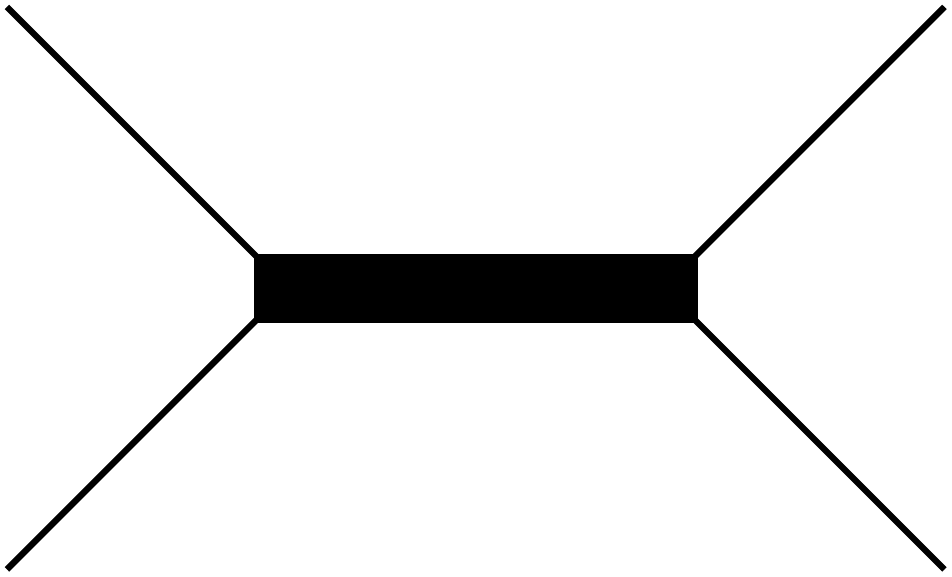}}
\caption{From crossings to planar trivalent graphs}\label{fig:replacing crossings}
\end{figure}

Here and throughout this paper, small diagrams represent parts of larger diagrams, and the collection of small diagrams appearing in a single equation all share the same larger diagram. 

Replacing each of the crossings in $D$ with one of the local diagrams \raisebox{-3pt}{\includegraphics[height=0.15in]{A-smoothing}}\,,\, \raisebox{-3pt} {\includegraphics[height=0.15in]{B-smoothing}}\,, \raisebox{-6pt}{\includegraphics[height=0.25in]{resol}}, or \raisebox{-4pt}{\includegraphics[height=0.17in]{resol-ho}}, we arrive, after finitely many steps, at planar trivalent graphs, called the \textit{states} of the link diagram $D$. As mentioned in the introduction, these graphs have both standard and wide edges, such that there is exactly one wide edge incident to a vertex. 

There is a unique way to assign a polynomial $P(\Gamma) \in \bbZ[A^{\pm1}, B^{\pm 1}, a^{\pm 1}, (A-B)^{\pm 1}]$ to a state $\Gamma$, so that it takes value $1$ for the unknot and satisfies the following \textit{graph skein relations} (see Theorem~\ref{thm:unique poly}, whose proof is given in appendix~\ref{appendix:trivalent}):

\begin{eqnarray}
P\left( \Gamma \cup \bigcirc \right) &=& \alpha \, P( \Gamma) \label{gi2}\\
P\left( \,\raisebox{-7pt}{\includegraphics[height=0.25in]{resol}}\,\right) &= &P \left(\, \raisebox{-3pt}{\includegraphics[height=0.17in]{resol-ho}}\,\right) \label{gi3} \\
P \left( \,\raisebox{-7pt}{\includegraphics[height=0.25in]{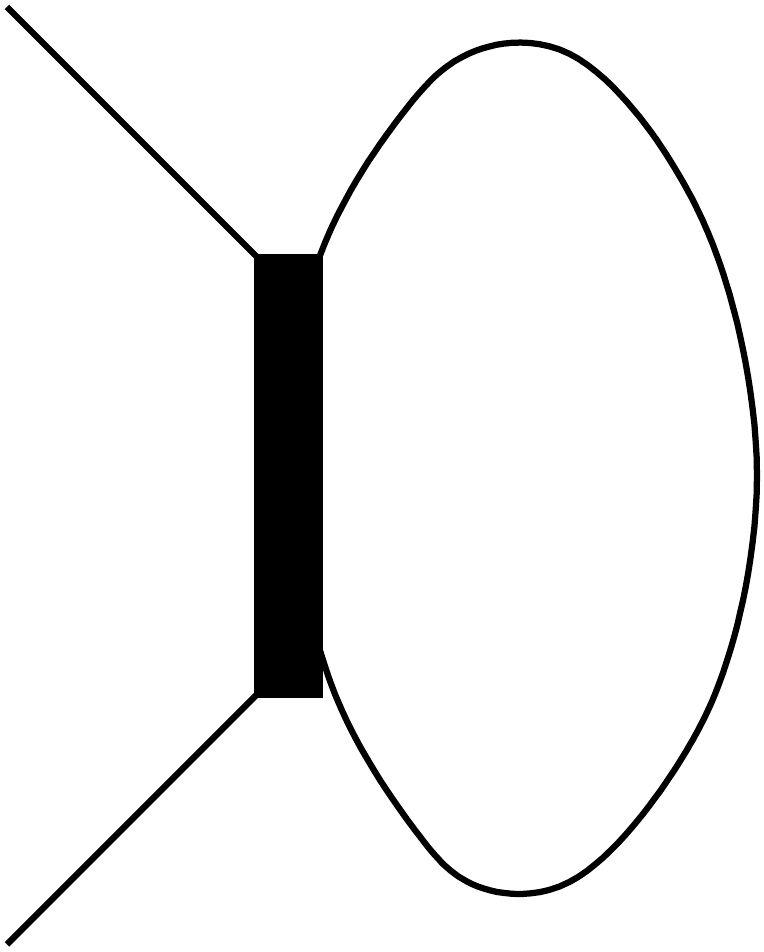}}\, \right) &=& \beta \,P \left( \,\raisebox{-7pt}{\includegraphics[height=0.25in]{arc}}\, \right)  \label{gi4}\\
P \left( \,\raisebox{-13pt}{\includegraphics[height=0.45in]{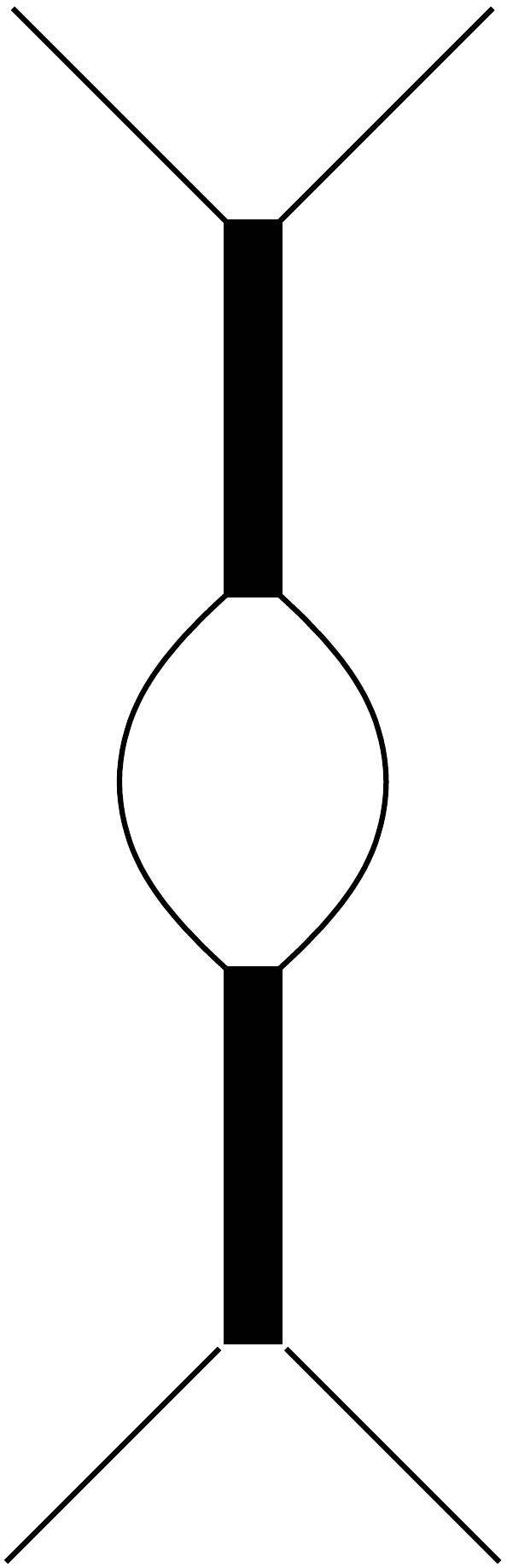}}\, \right) &=& (1-AB)\, P \left( \, \raisebox{-4pt}{\includegraphics[height=0.18in]{A-smoothing}} \, \right) + \gamma \,P \left( \, \raisebox{-4pt}{\includegraphics[height=0.18in]{B-smoothing}} \, \right) - (A+B)\, P \left( \, \raisebox{-7pt}{\includegraphics[height=0.25in]{resol}}\, \right) \label{gi5} 
\end{eqnarray}
\begin{eqnarray}
P \left( \,\raisebox{-13pt}{\includegraphics[height=0.45in, width=0.25in]{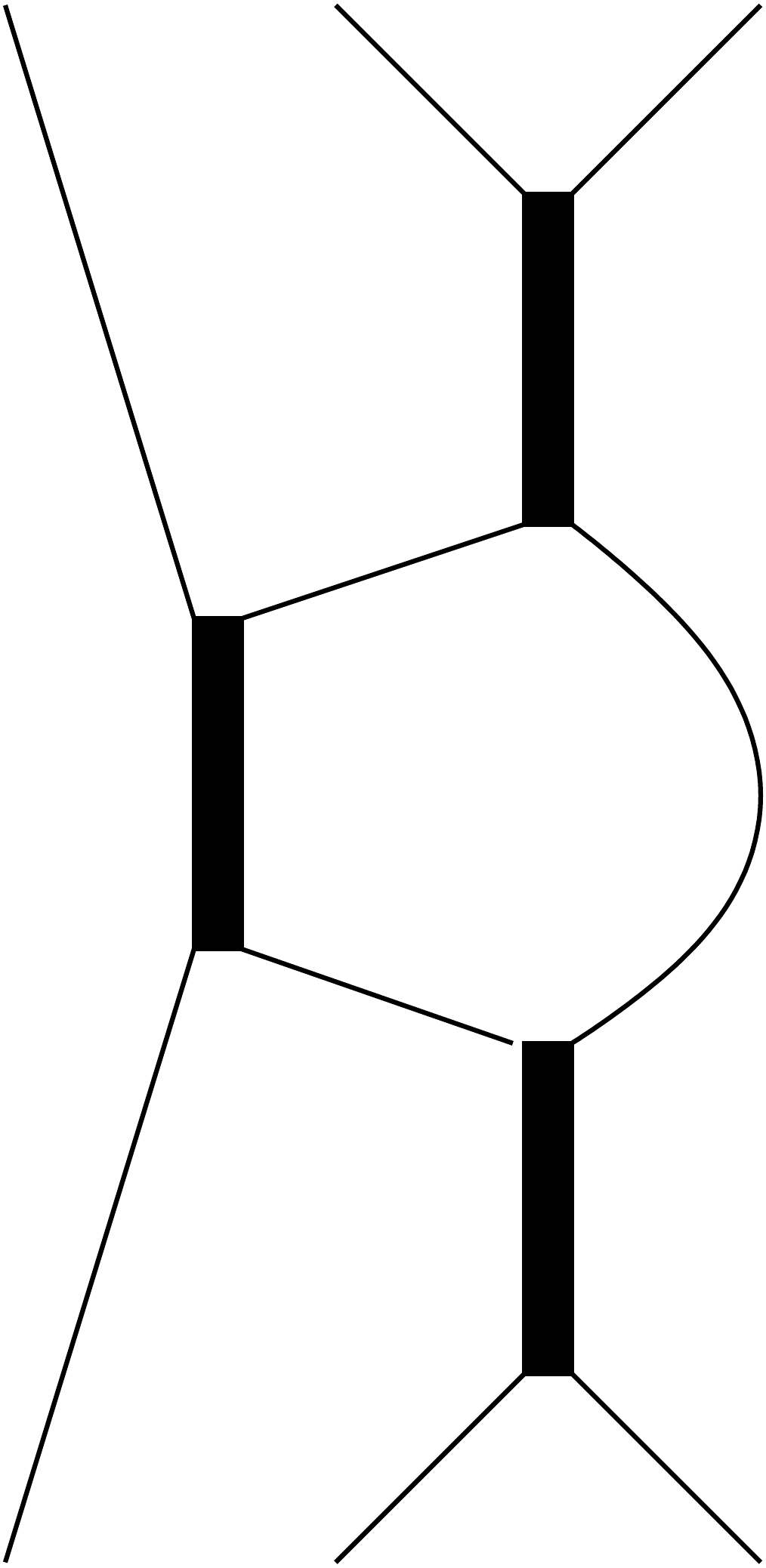}}\,\right) - P \left( \,\raisebox{-13pt}{\includegraphics[height=0.45in, width=0.25in]{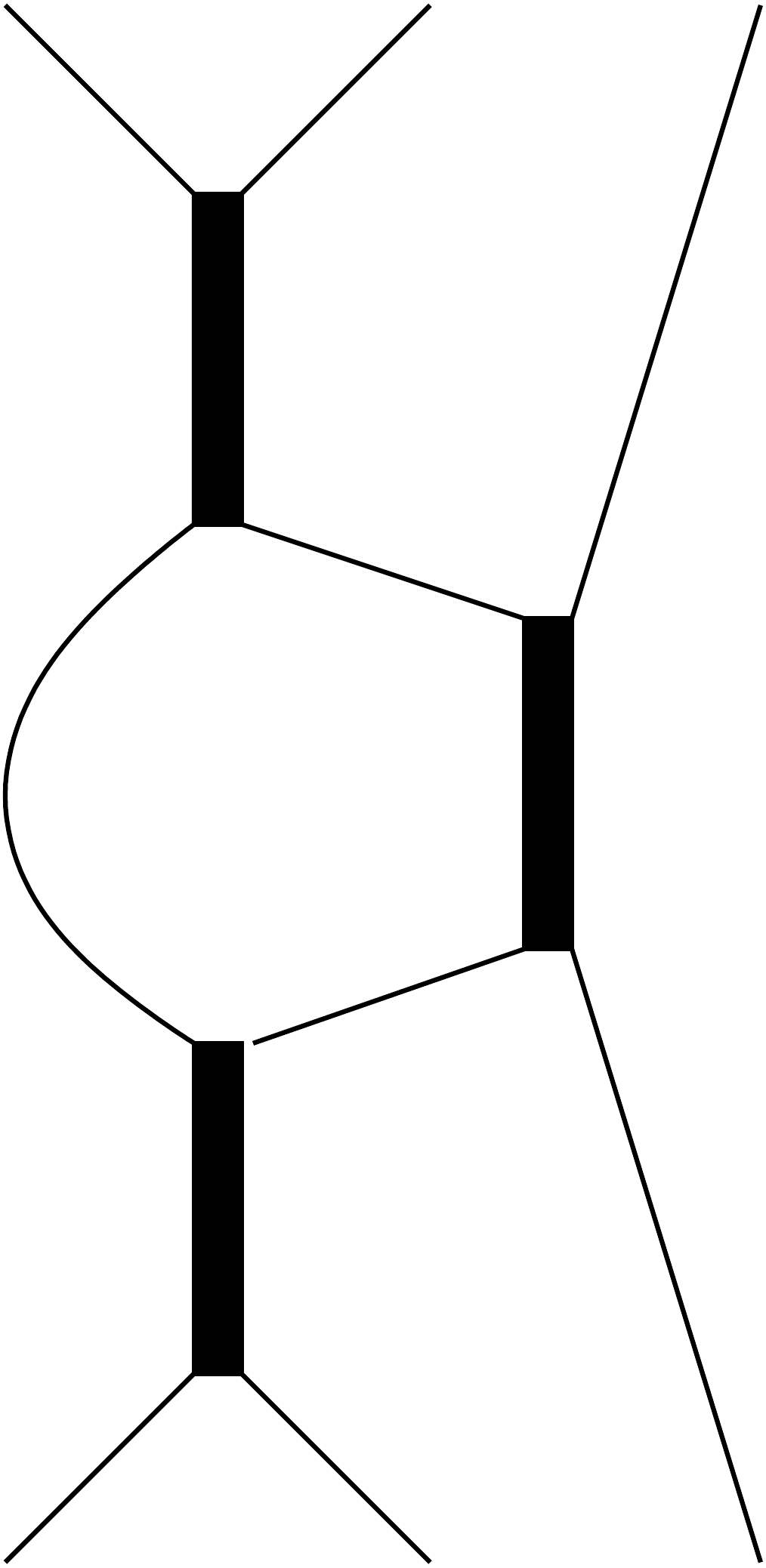}}\, \right) =  \delta \, \left [P \left(\, \raisebox{-7pt}{\includegraphics[height=0.25in]{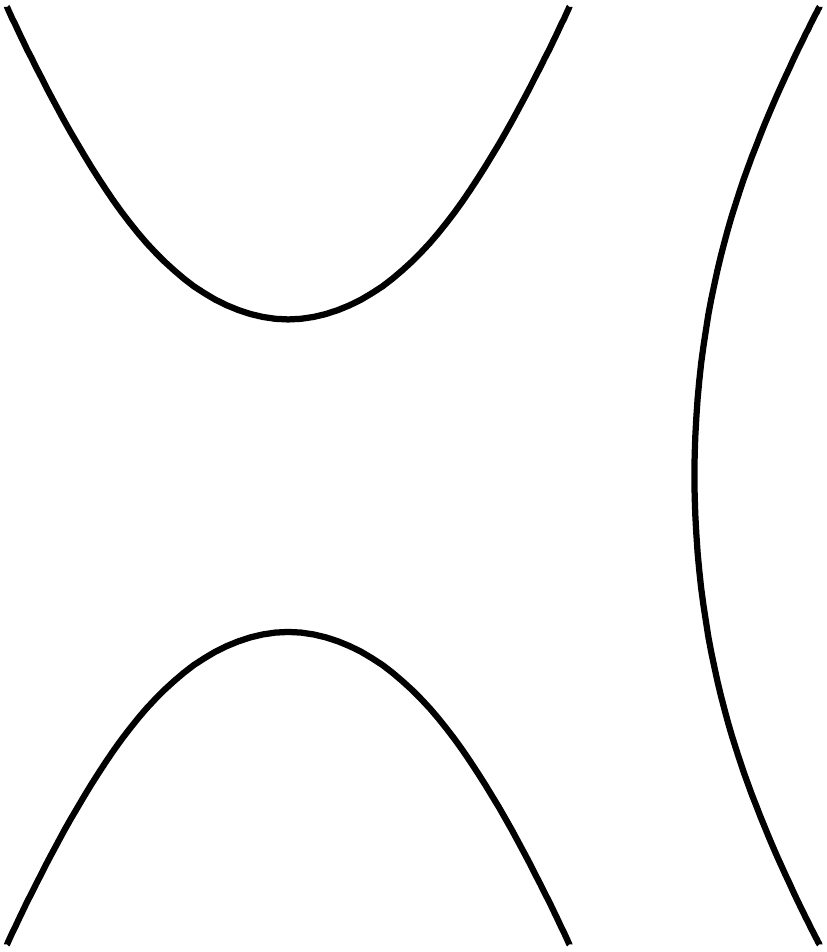}}\, \right) - P \left( \,\, \raisebox{-7pt}{\includegraphics[height=0.25in]{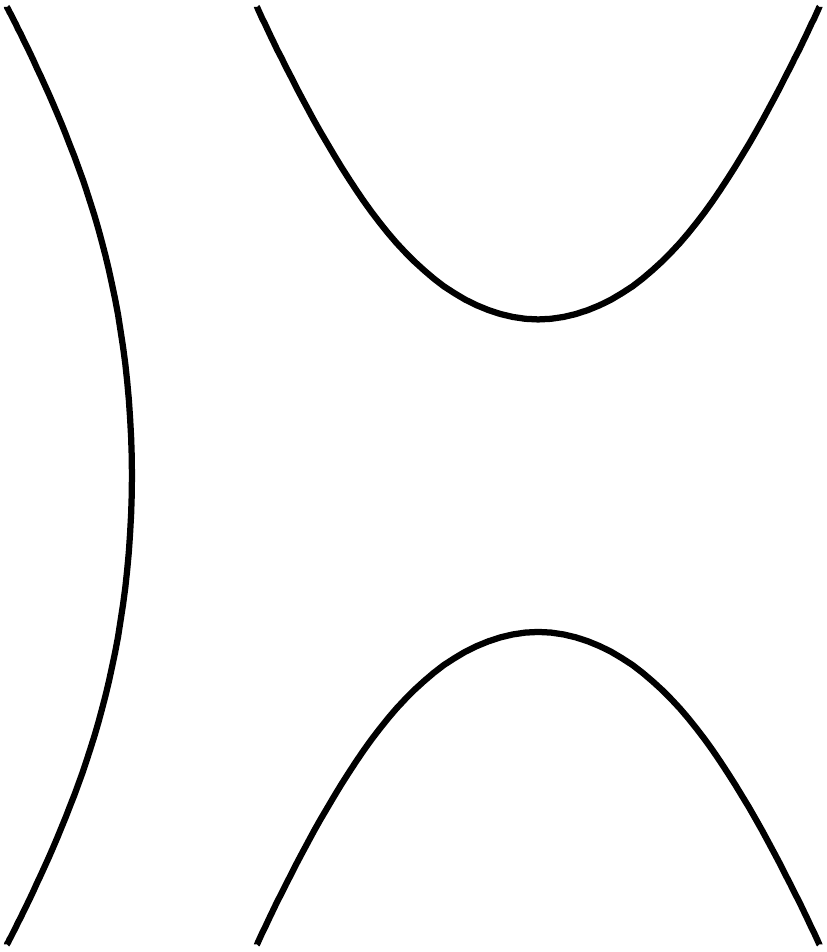}}\, \right)\,\right ] + \label{gi6}
\end{eqnarray}
\begin{eqnarray*}
 + AB \left [P\left( \,\raisebox{-7pt}{\includegraphics[height=0.25in]{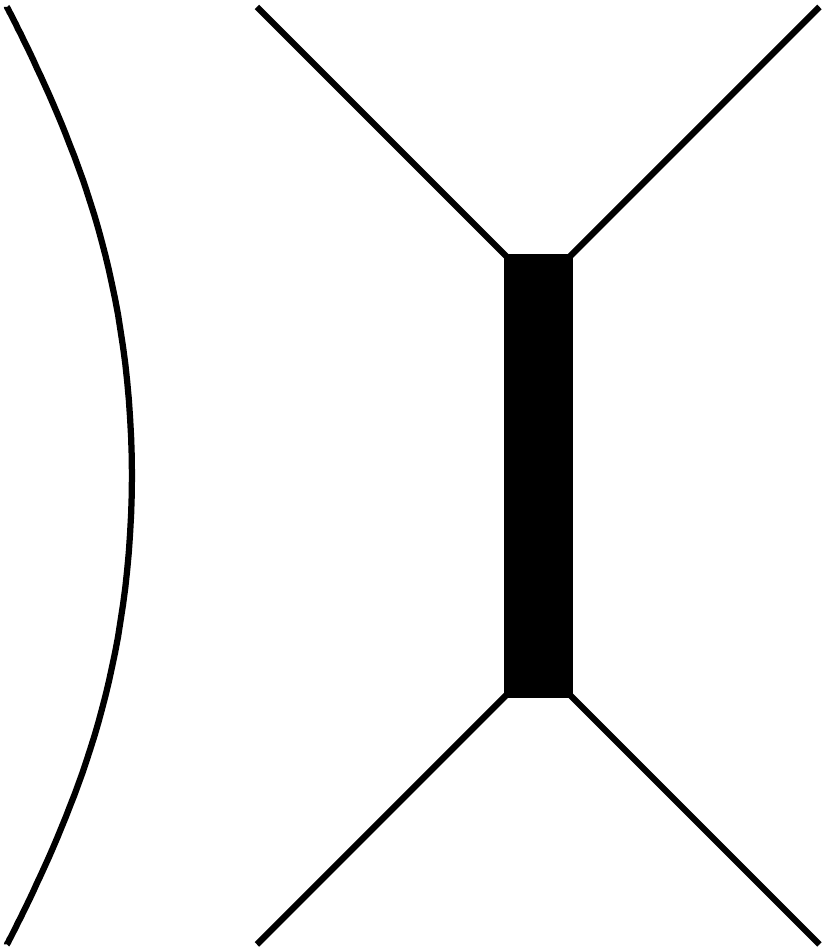}}\, \right) - P \left( \,\raisebox{-7pt}{\includegraphics[height=0.25in]{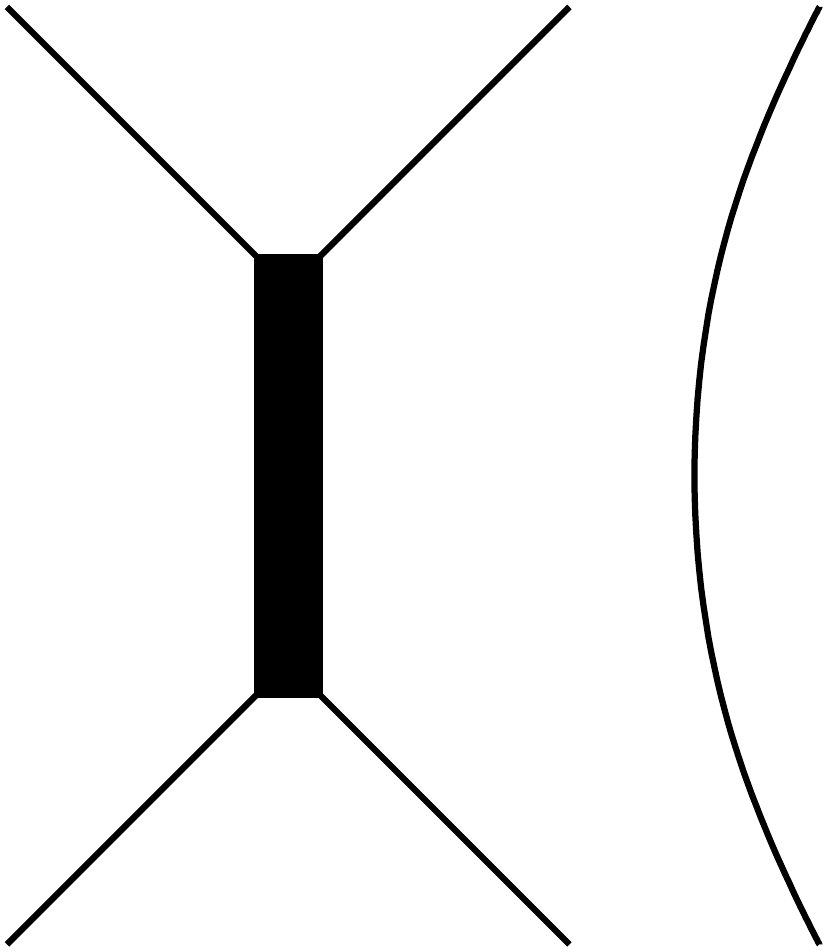}} \,\, \right)+  P \left( \, \raisebox{-7pt}{\includegraphics[height=0.25in, width=0.25in]{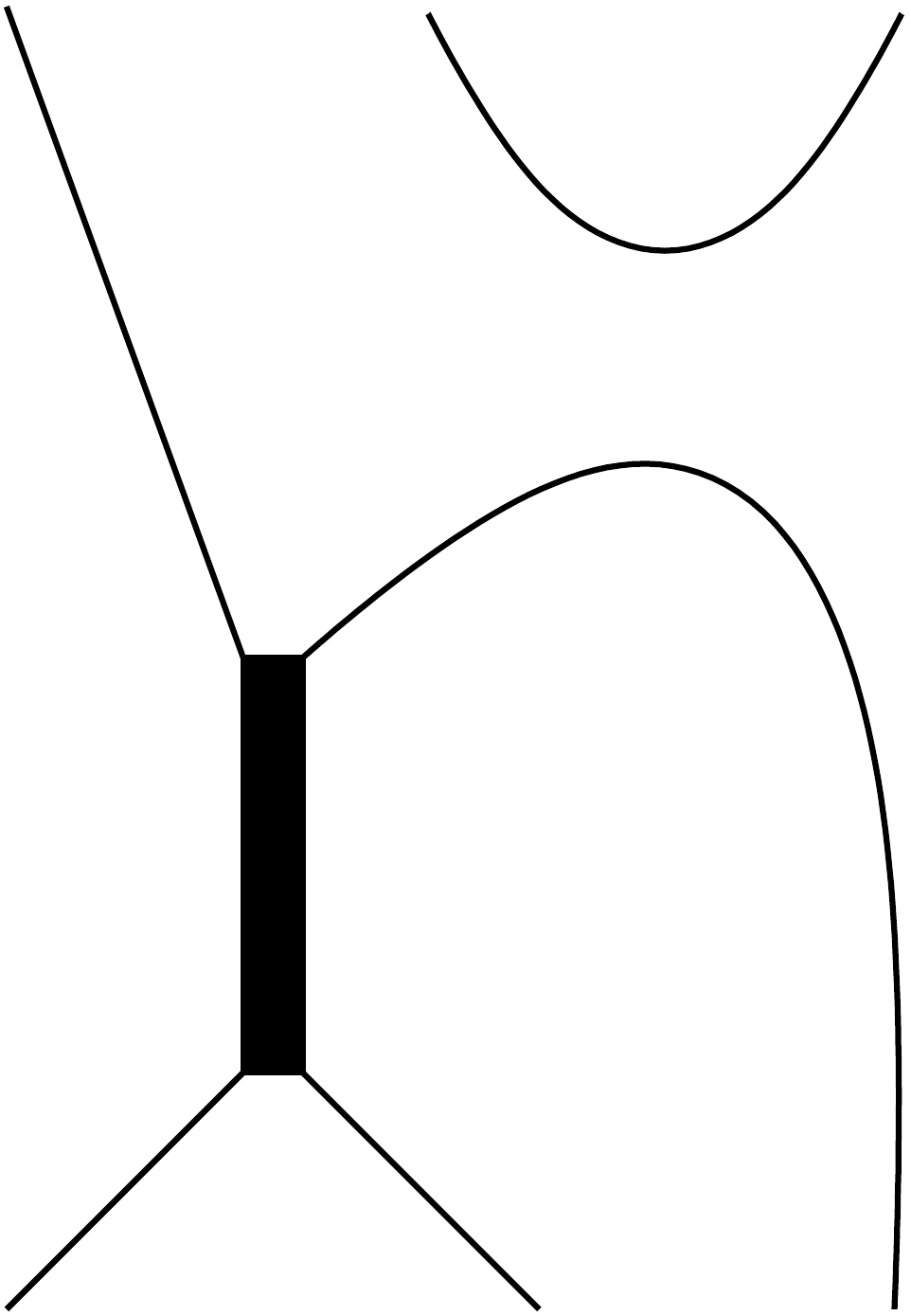}}\, \right) - P \left( \,\raisebox{-7pt}{\includegraphics[height=0.25in , width=0.25in]{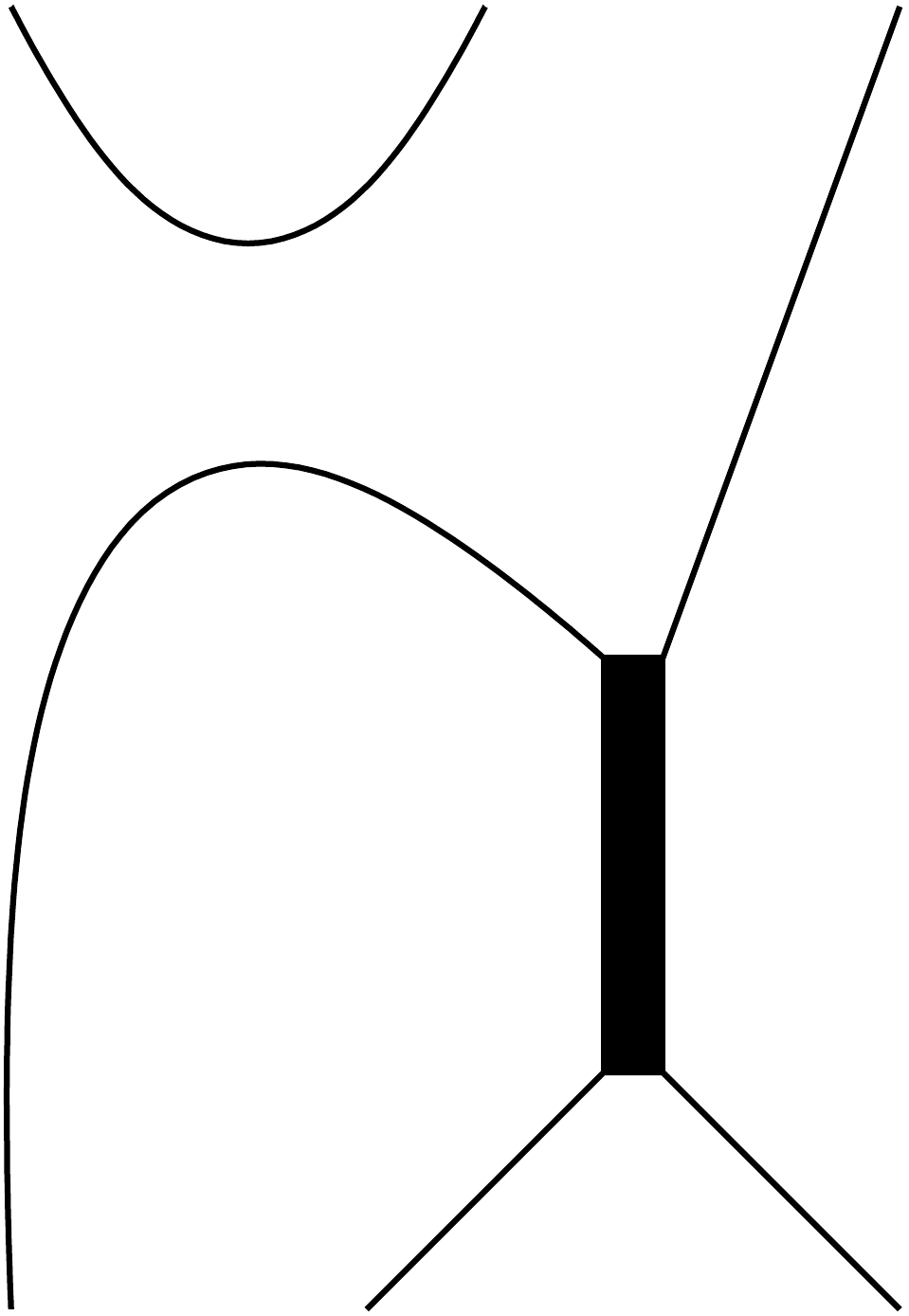}}\,\right)  + P\left(\, \raisebox{-7pt}{\includegraphics[height=0.25in, width=0.25in]{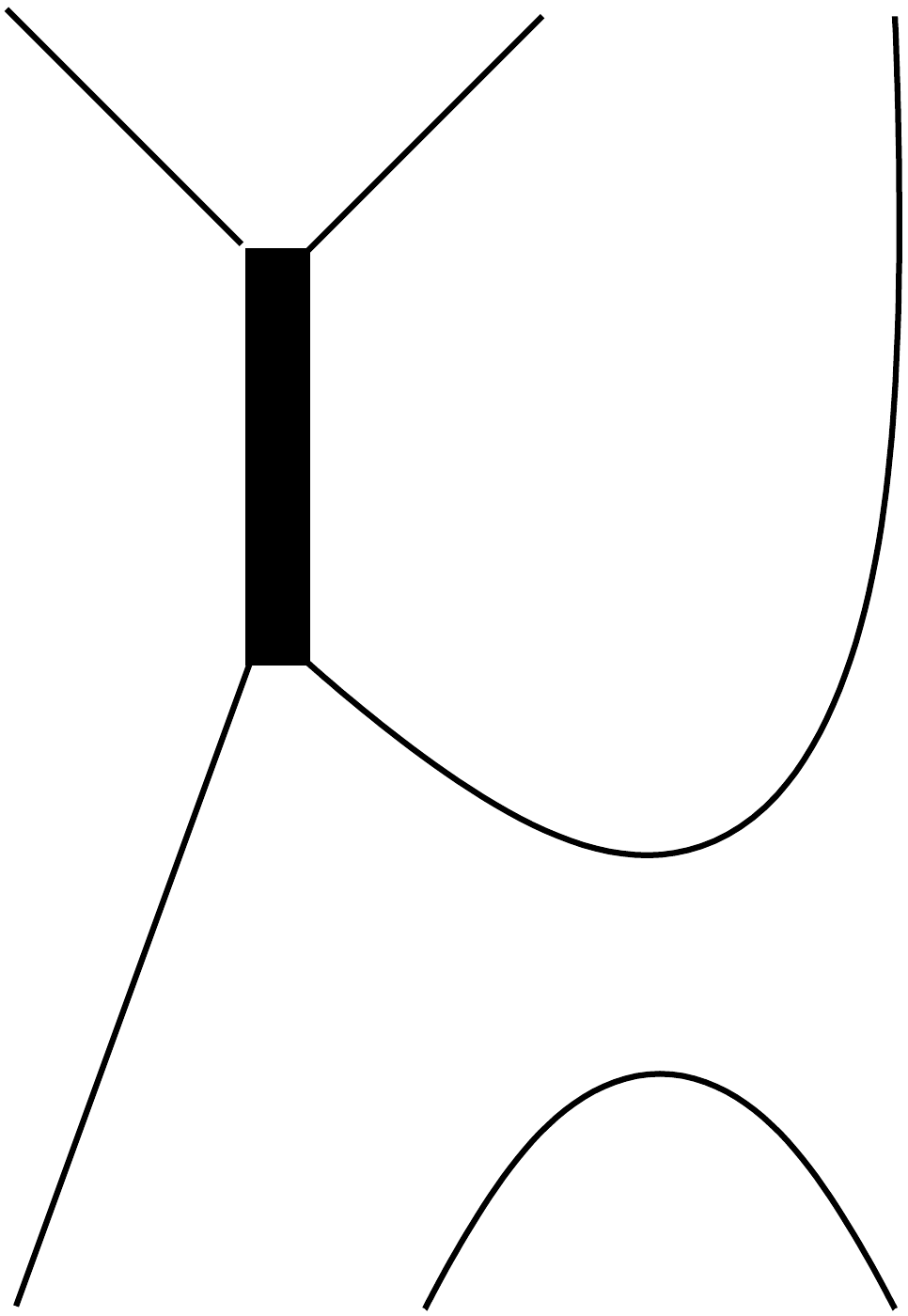}} \,\right) - P \left(\, \raisebox{-7pt}{\includegraphics[height=0.25in, width=0.25in]{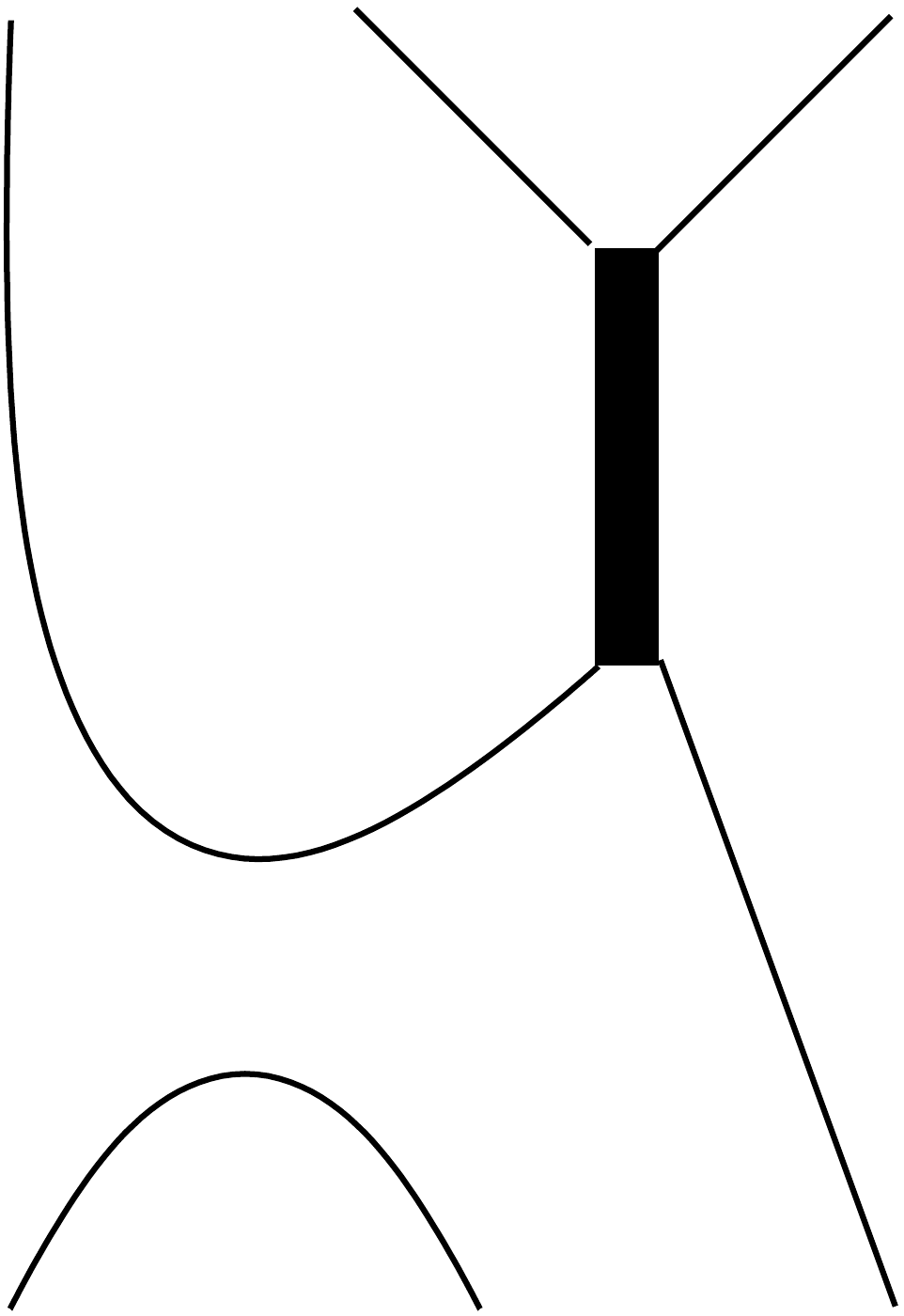}}\,\right)\right], \\
 \end{eqnarray*}
 where
 \[\alpha =  \frac{a-a^{-1}}{A-B} + 1,\qquad \beta = \frac{Aa^{-1}-Ba}{A-B}-A-B,\]
 \[\gamma = \frac{B^2a-A^2a^{-1}}{A-B} + AB, \qquad\delta =  \frac{B^3a - A^3a^{-1}}{A - B}.\]

\begin{remark}
The identity~\eqref{gi3} implies that when replacing a crossing by a formal linear combination of its resolutions (as shown in Figure~\ref{fig:replacing crossings}) we can either use \raisebox{-7pt}{\includegraphics[height=0.25in]{resol}} or \raisebox{-5pt}{\includegraphics[height=0.17in]{resol-ho}}, regardless of the type of the crossing.
\end{remark}

\begin{proposition}
The following identities for the graph polynomial $P$ hold:
\begin{eqnarray}
P \left( \,\raisebox{-4pt}{\includegraphics[height=0.18in]{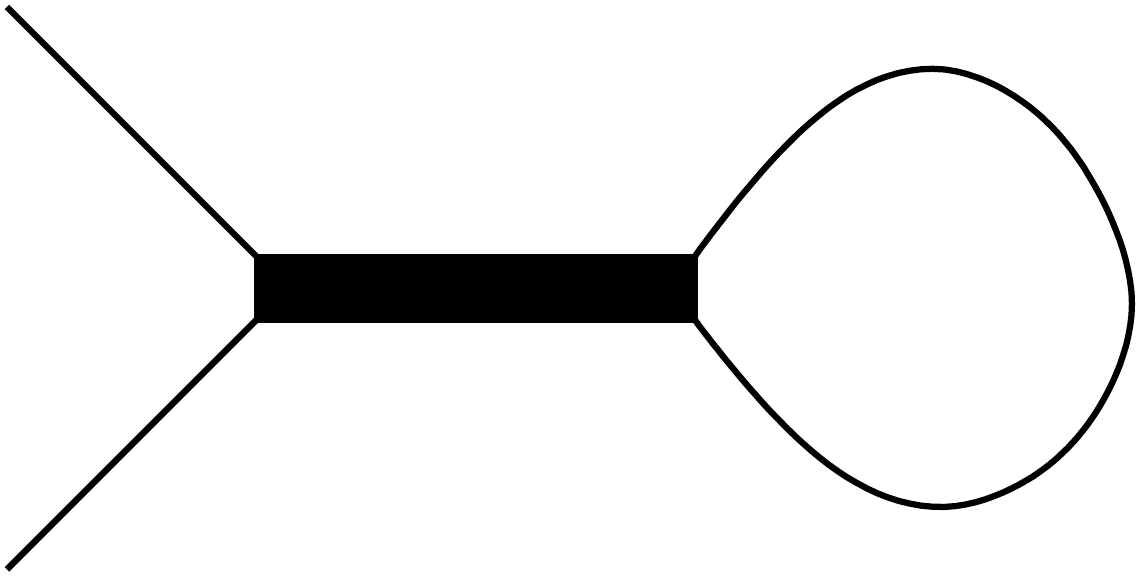}}\, \right) &=& \beta \, P \left( \,\raisebox{-6pt}{\includegraphics[height=0.25in]{arc}}\, \right) \label{gi7}\\
P \left( \,\raisebox{-6pt}{\includegraphics[height=0.25in]{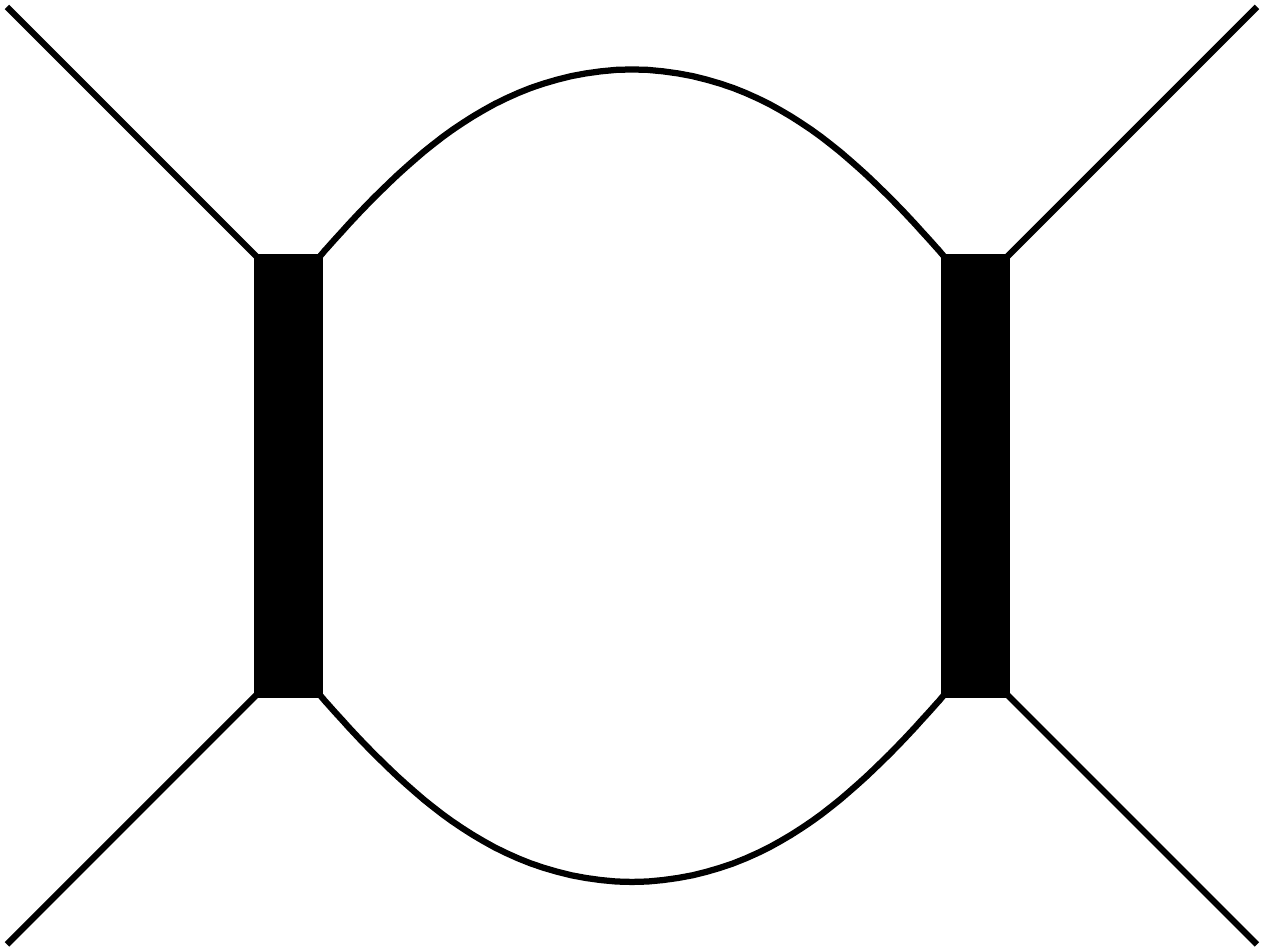}}\, \right) &=& (1-AB)\, P \left( \, \raisebox{-4pt}{\includegraphics[height=0.18in]{B-smoothing}}\, \right) + \gamma \, P \left( \,\raisebox{-4pt}{\includegraphics[height=0.18in]{A-smoothing}}\, \right) - (A+B) \,P \left( \, \raisebox{-6pt}{\includegraphics[height=0.25in]{resol}}\, \right) \label{gi8}\\
P \left( \,\raisebox{-6pt}{\includegraphics[height=0.25in]{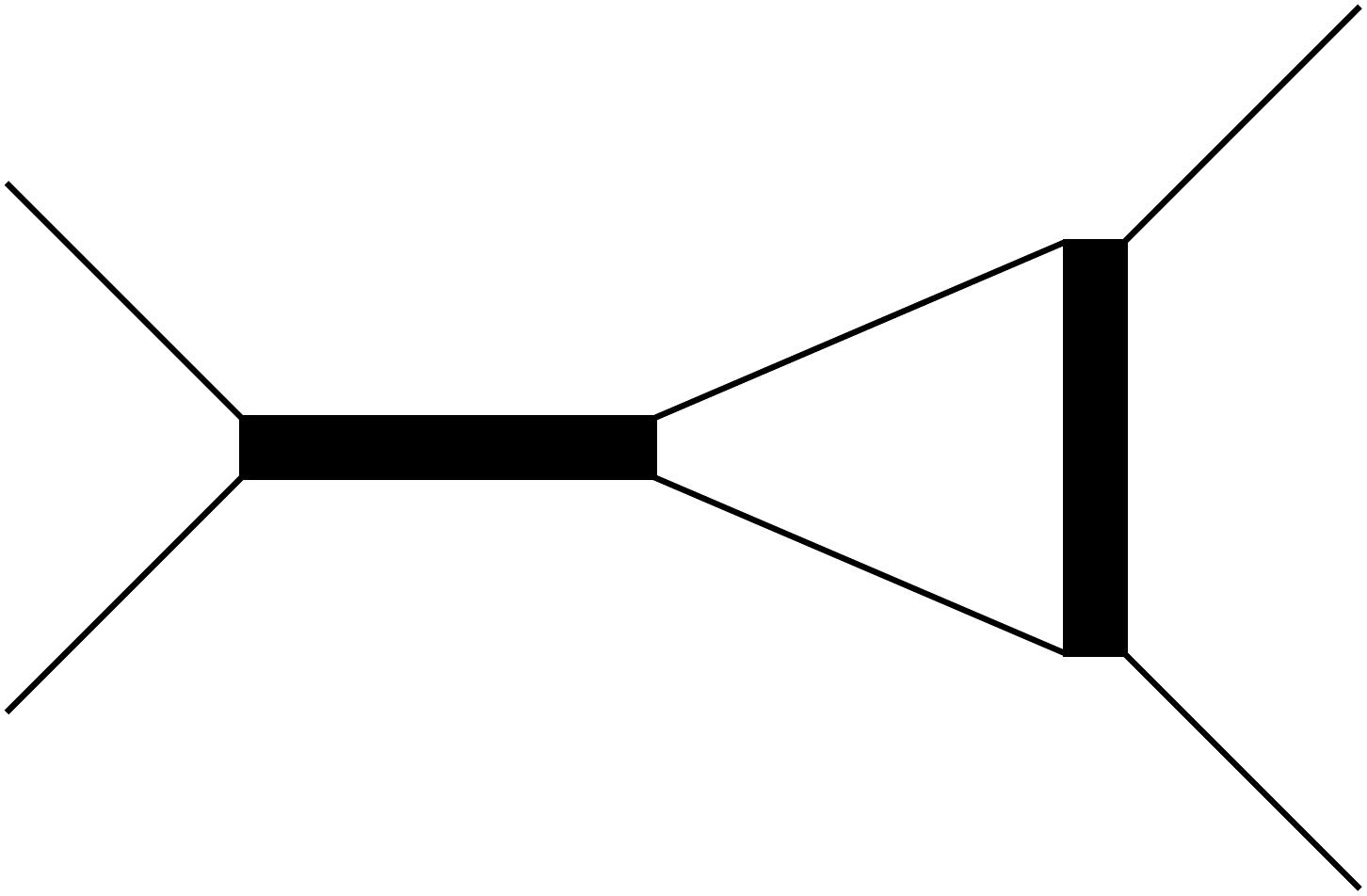}}\, \right)&=& (1-AB)\,P \left( \, \raisebox{-4pt}{\includegraphics[height=0.18in]{B-smoothing}}\, \right) + \gamma \,P \left(\, \raisebox{-4pt}{\includegraphics[height=0.18in]{A-smoothing}}\,\right) - (A+B) \,P \left(\, \raisebox{-6pt}{\includegraphics[height=0.25in]{resol}}\, \right). \label{gi9}
\end{eqnarray}
\end{proposition}
\begin{proof}
These follow at once from the graph skein relations~(\ref{gi3}) - (\ref{gi5}).
\end{proof}

\begin{theorem}\label{thm:unique poly}
There is a unique polynomial for planar trivalent graph diagrams that takes value $1$ for the unknot and satisfies the skein relations~(\ref{gi2}) - (\ref{gi6}). 
\end{theorem}

Applying the rules in Figure~\ref{fig:replacing crossings} to all of the crossings in $D$, we express the link diagram $D$ as a formal linear combination of its associated states, whose coefficients are monomials in $A$ and $B$. Using the polynomial $P$ from Theorem~\ref{thm:unique poly}, we define a three-variable rational function $\textbf{P}_D = \textbf{P}_D(A, B, a) \in \bbZ[A^{\pm1}, B^{\pm 1}, a^{\pm 1}, (A-B)^{\pm 1}]$ associated to $D$, obtained by summing up the graph polynomials $P(\Gamma)$ weighted by powers of $A$ and $B,$ over all states $\Gamma$ of $D$. That is,

\[ \textbf{P}_D = \sum_{\text{states} \, \Gamma}A^{\alpha(\Gamma)}B^{\beta(\Gamma)}P(\Gamma),\]
where the integers $\alpha(\Gamma)$ and $\beta(\Gamma)$ are determined by the rules in Figure~\ref{fig:replacing crossings}. In particular, $\alpha(\Gamma)$ is the number of ``A-smoothings" used to form the state $\Gamma$, while $\beta(\Gamma)$ is the number of ``B-smoothings'' in $\Gamma$.

\begin{theorem}\label{thm:invariance}
The Laurent polynomial $\textbf{P}_D$ satisfies the following:
\begin{enumerate}

\item $\textbf{P}_{\,\raisebox{-3pt}{\includegraphics[height=0.15in]{poscrossing}}}\, - \textbf{P}_{\,\raisebox{-3pt}{\includegraphics[height=0.15in]{negcrossing}}}\,= (A - B) \left[\textbf{P}_{\,\raisebox{-3pt}{\includegraphics[height=0.15in]{A-smoothing}}}\, - \textbf{P}_{\,\raisebox{-3pt}{\includegraphics[height=0.15in]{B-smoothing}}}\,\right]$.
\item $\textbf{P}_{\,\raisebox{-3pt}{\includegraphics[height=0.15in]{poskink}}}\, = a \textbf{P}_{\,\raisebox{-3pt}{\includegraphics[height=0.15in]{arc}}}\,, \quad \textbf{P}_{\,\raisebox{-3pt}{\includegraphics[height=0.15in]{negkink}}}\, = a^{-1} \textbf{P}_{\,\raisebox{-3pt}{\includegraphics[height=0.15in]{arc}}} \,$.
\item If $D_1$ and $D_2$ are link diagrams related by a Reidemeister II or III move, then $\textbf{P}_{D_1} = \textbf{P}_{D_2}$.
\end{enumerate}
\end{theorem}

\begin{proof}
(1) To show this statement, one only needs to subtract the second skein relation from the first skein relation given in Figure~\ref{fig:replacing crossings}.

(2) We will show that the first equality holds (the second equality follows similarly and we leave it to the reader).

\begin{eqnarray*}
\textbf{P}_{\,\raisebox{-3pt}{\includegraphics[height=0.25in]{poskink}}}\,  &=& A \textbf{P}_{\,\raisebox{-3pt}{\includegraphics[height=0.25in]{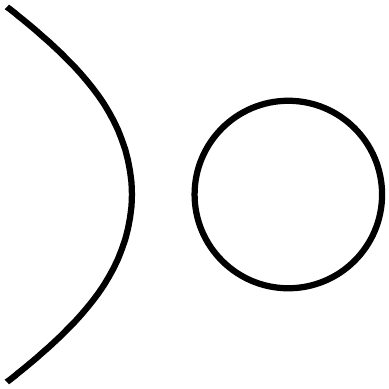}}}\, + B \textbf{P}_{\,\raisebox{-3pt}{\includegraphics[height=0.25in]{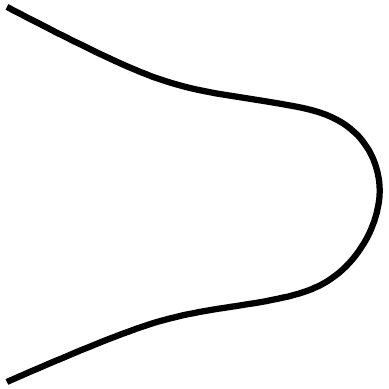}}}\, +  \textbf{P}_{\,\raisebox{-3pt}{\includegraphics[height=0.25in, width = 0.2in]{rem-edge}}}\, \\
&=& A \alpha \textbf{P}_{\,\raisebox{-3pt}{\includegraphics[height=0.25in]{arc}}}\, + B \textbf{P}_{\,\raisebox{-3pt}{\includegraphics[height=0.25in]{arc}}}\, + \beta \textbf{P}_{\,\raisebox{-3pt}{\includegraphics[height=0.25in]{arc}}}\,\\
&=& a \textbf{P}_{\,\raisebox{-3pt}{\includegraphics[height=0.25in]{arc}}}\,,
\end{eqnarray*}
since $A\alpha + B + \beta = a$.

(3) Using the skein relations in Figure~\ref{fig:replacing crossings} and the graph skein relations, we have

\begin{eqnarray*}
\textbf{P}_{\,\raisebox{-3pt}{\includegraphics[height=0.25in]{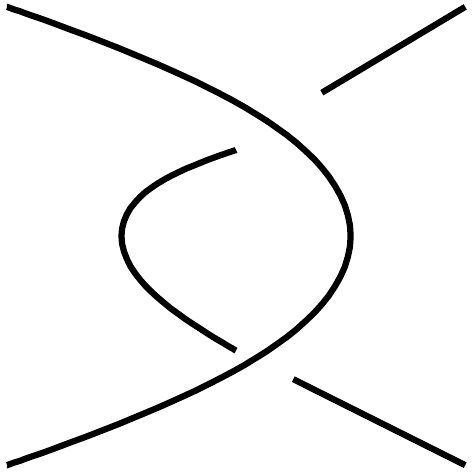}}}\,  &=& A \textbf{P}_{\,\raisebox{-3pt}{\includegraphics[height=0.25in]{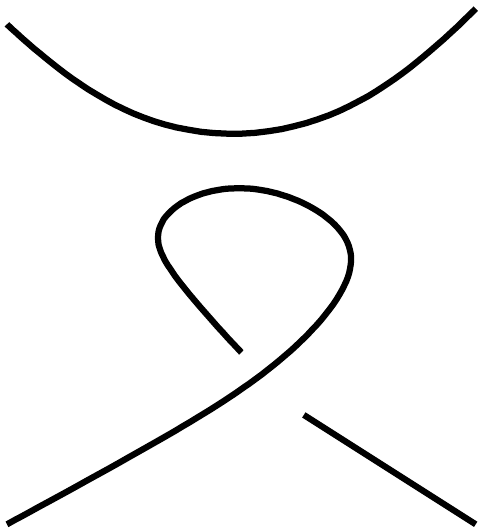}}}\, + B \textbf{P}_{\,\raisebox{-3pt}{\includegraphics[height=0.25in]{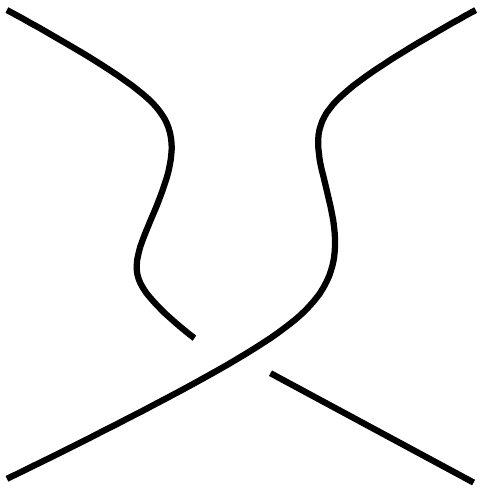}}}\, +  \textbf{P}_{\,\raisebox{-3pt}{\includegraphics[height=0.25in, width = 0.2in]{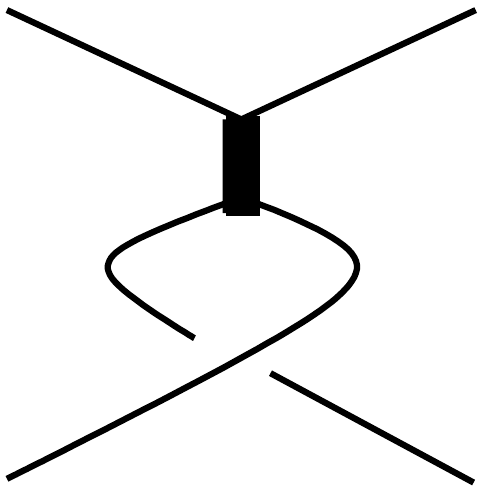}}}\, \\
&=& Aa^{-1}\textbf{P}_{\,\raisebox{-3pt}{\includegraphics[height=0.25in]{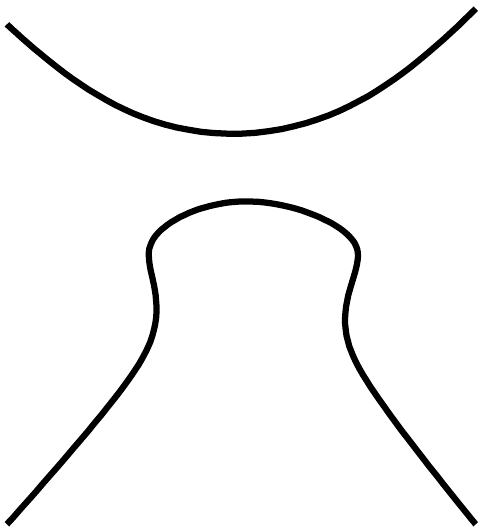}}}\, + B\left ( A \textbf{P}_{\,\raisebox{-3pt}{\includegraphics[height=0.25in]{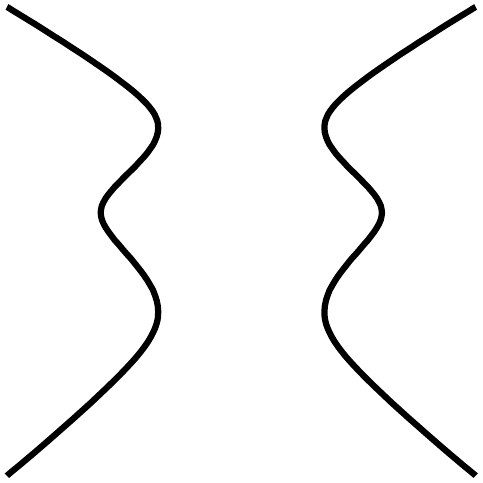}}}\, + B \textbf{P}_{\,\raisebox{-3pt}{\includegraphics[height=0.25in]{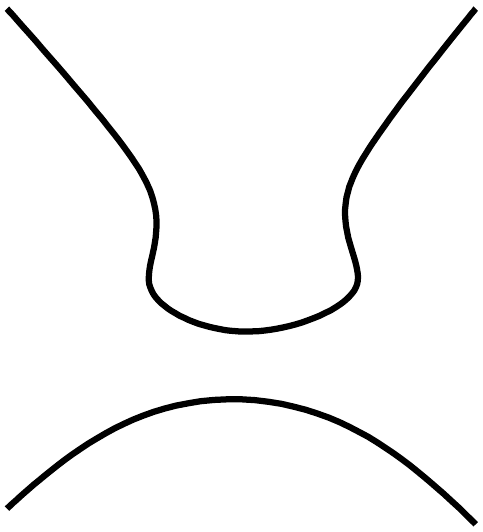}}}\, +  \textbf{P}_{\,\raisebox{-3pt}{\includegraphics[height=0.25in, width = 0.2in]{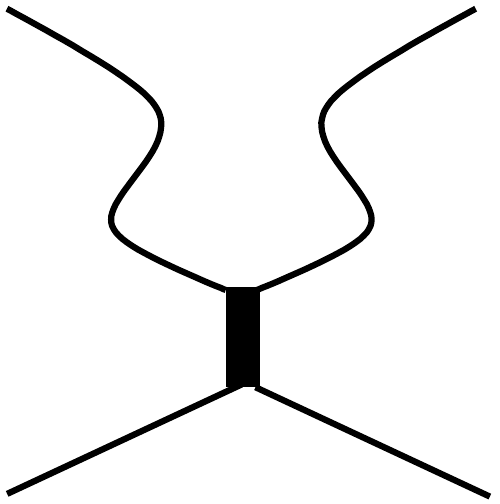}}}\, \right) \\
&& + A\textbf{P}_{\,\raisebox{-3pt}{\includegraphics[height=0.25in, width = 0.2in]{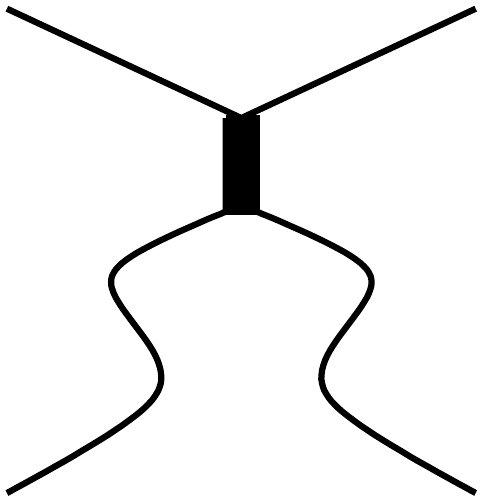}}}\, + B \textbf{P}_{\,\raisebox{-3pt}{\includegraphics[height=0.25in, width = 0.25in]{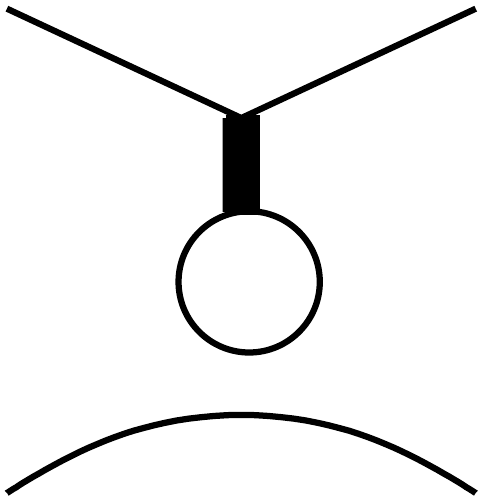}}}\, + \textbf{P}_{\,\raisebox{-3pt}{\includegraphics[height=0.25in, width = 0.15in]{rem-digon}}}\, \\
&=& Aa^{-1}  \textbf{P}_{\,\raisebox{-3pt}{\includegraphics[height=0.25in]{B-smoothing}}}\, + AB\textbf{P}_{\,\raisebox{-3pt}{\includegraphics[height=0.25in]{A-smoothing}}}\, + B^2\textbf{P}_{\,\raisebox{-3pt}{\includegraphics[height=0.25in]{B-smoothing}}}\, + B\textbf{P}_{\,\raisebox{-3pt}{\includegraphics[height=0.25in, width = 0.2in]{resol}}}\, \\
&& + A\textbf{P}_{\,\raisebox{-3pt}{\includegraphics[height=0.25in, width = 0.2in]{resol}}}\, + B \beta \textbf{P}_{\,\raisebox{-3pt}{\includegraphics[height=0.2in]{B-smoothing}}}\, \\
&&+ \left( 1-AB \right) \textbf{P}_{\,\raisebox{-3pt}{\includegraphics[height=0.25in]{A-smoothing}}}\, + \gamma  \textbf{P}_{\,\raisebox{-3pt}{\includegraphics[height=0.25in]{B-smoothing}}}\, -(A + B) \textbf{P}_{\,\raisebox{-3pt}{\includegraphics[height=0.25in, width = 0.2in]{resol}}} \\
&=& (Aa^{-1} + B^2 + B \beta + \gamma) \textbf{P}_{\,\raisebox{-3pt}{\includegraphics[height=0.25in]{B-smoothing}}}\, + \textbf{P}_{\,\raisebox{-3pt}{\includegraphics[height=0.25in]{A-smoothing}}}\\ 
&=& \textbf{P}_{\,\raisebox{-3pt}{\includegraphics[height=0.25in]{A-smoothing}}}\,,
\end{eqnarray*}
since $Aa^{-1} + B^2 + B \beta + \gamma = 0$. The invariance of the polynomial $\textbf{P}$ under Reidemeister II move follows.

It remains to prove the invariance under the Reidemeister III move. We first show that $\textbf{P}_{\,\raisebox{-3pt}{\includegraphics[height=0.35in, width = 0.25in, width = 0.25in]{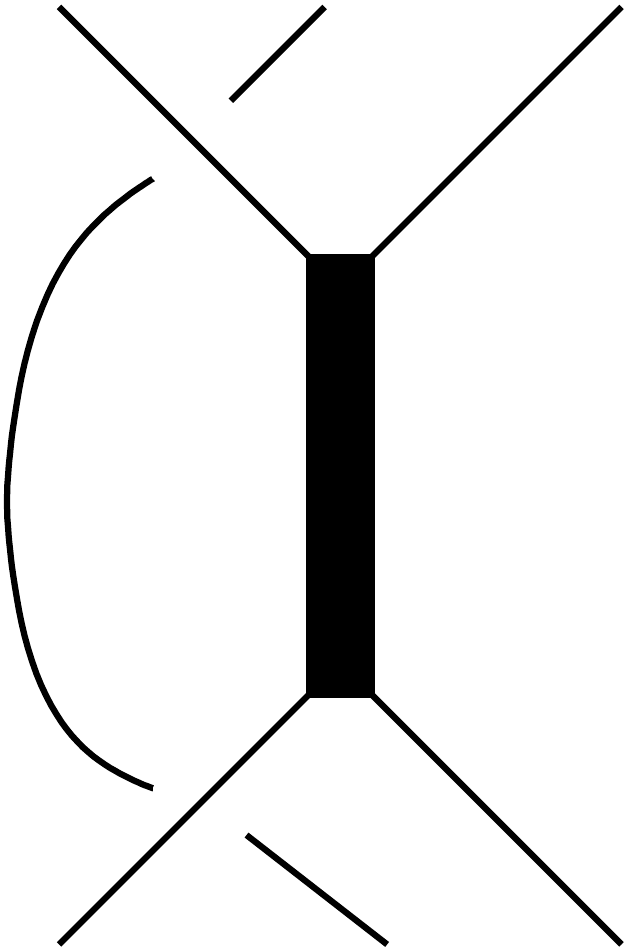}}} = \textbf{P}_{\,\raisebox{-3pt}{\includegraphics[height=0.35in, width = 0.25in]{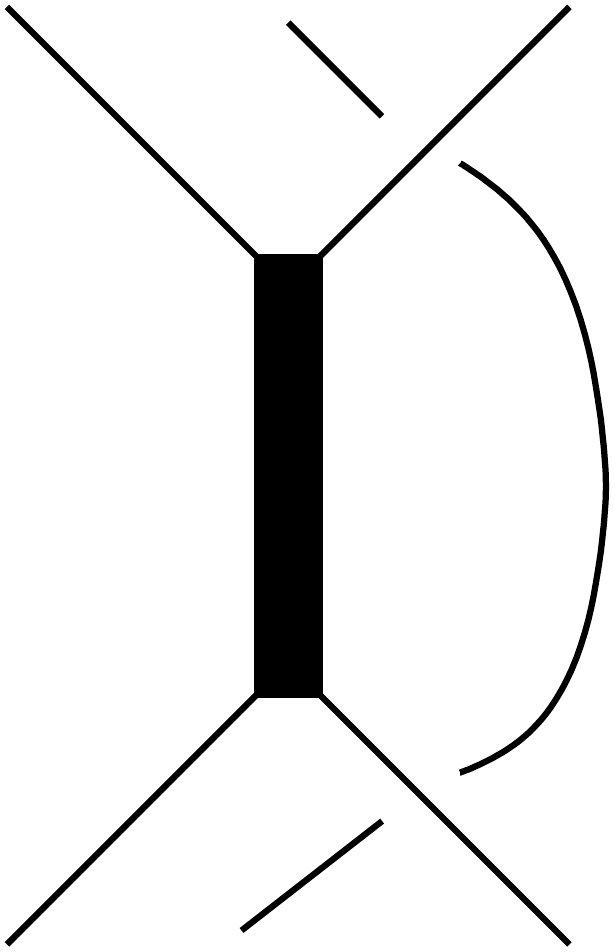}}}$\,.

\begin{eqnarray*}
\textbf{P}_{\,\raisebox{-3pt}{\includegraphics[height=0.35in]{thm1L}}}&=& A\textbf{P}_{\,\raisebox{-3pt}{\includegraphics[height=0.35in]{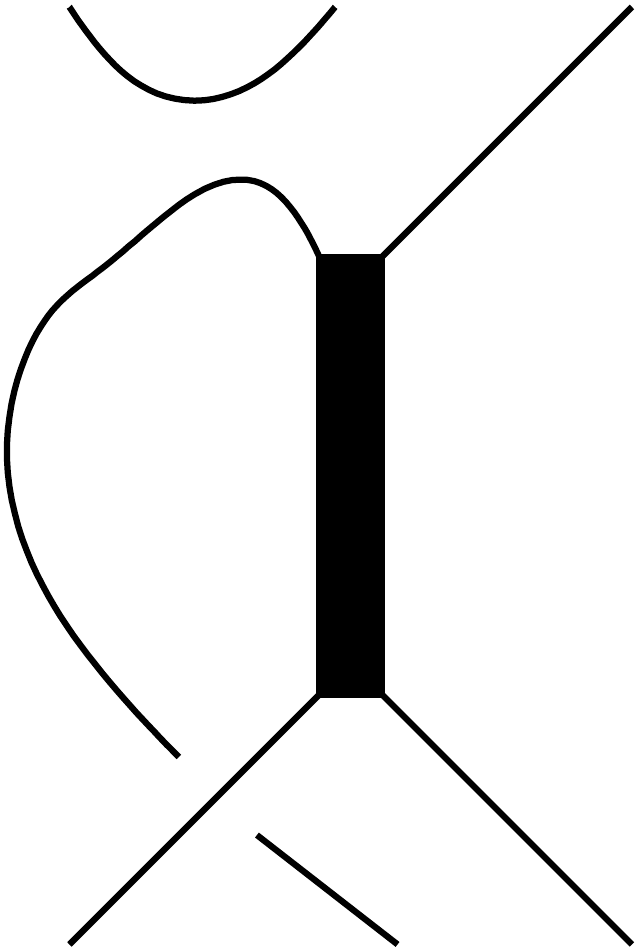}}}\, + B\textbf{P}_{\,\raisebox{-3pt}{\includegraphics[height=0.35in]{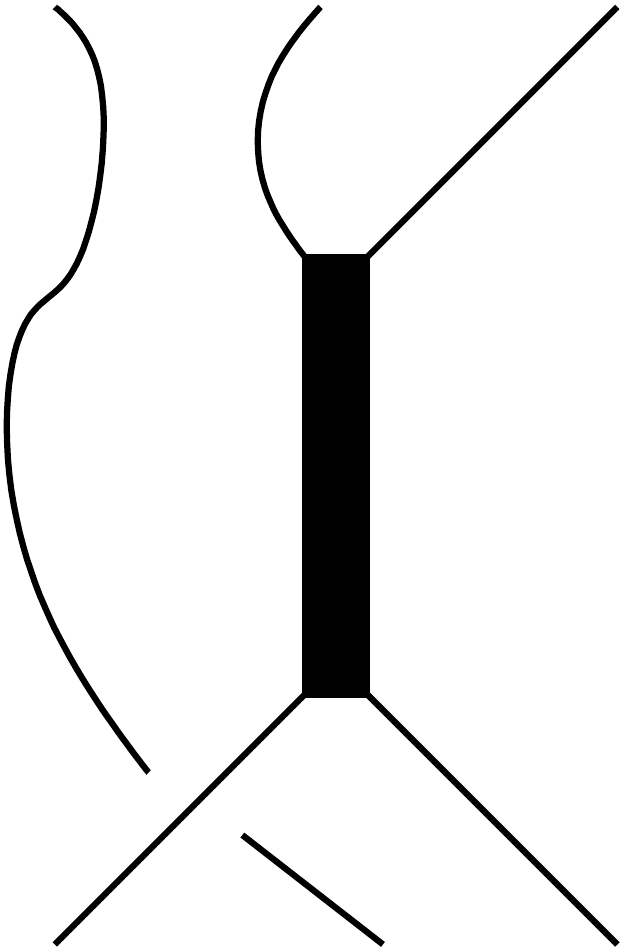}}}\, + \textbf{P}_{\,\raisebox{-3pt}{\includegraphics[height=0.35in]{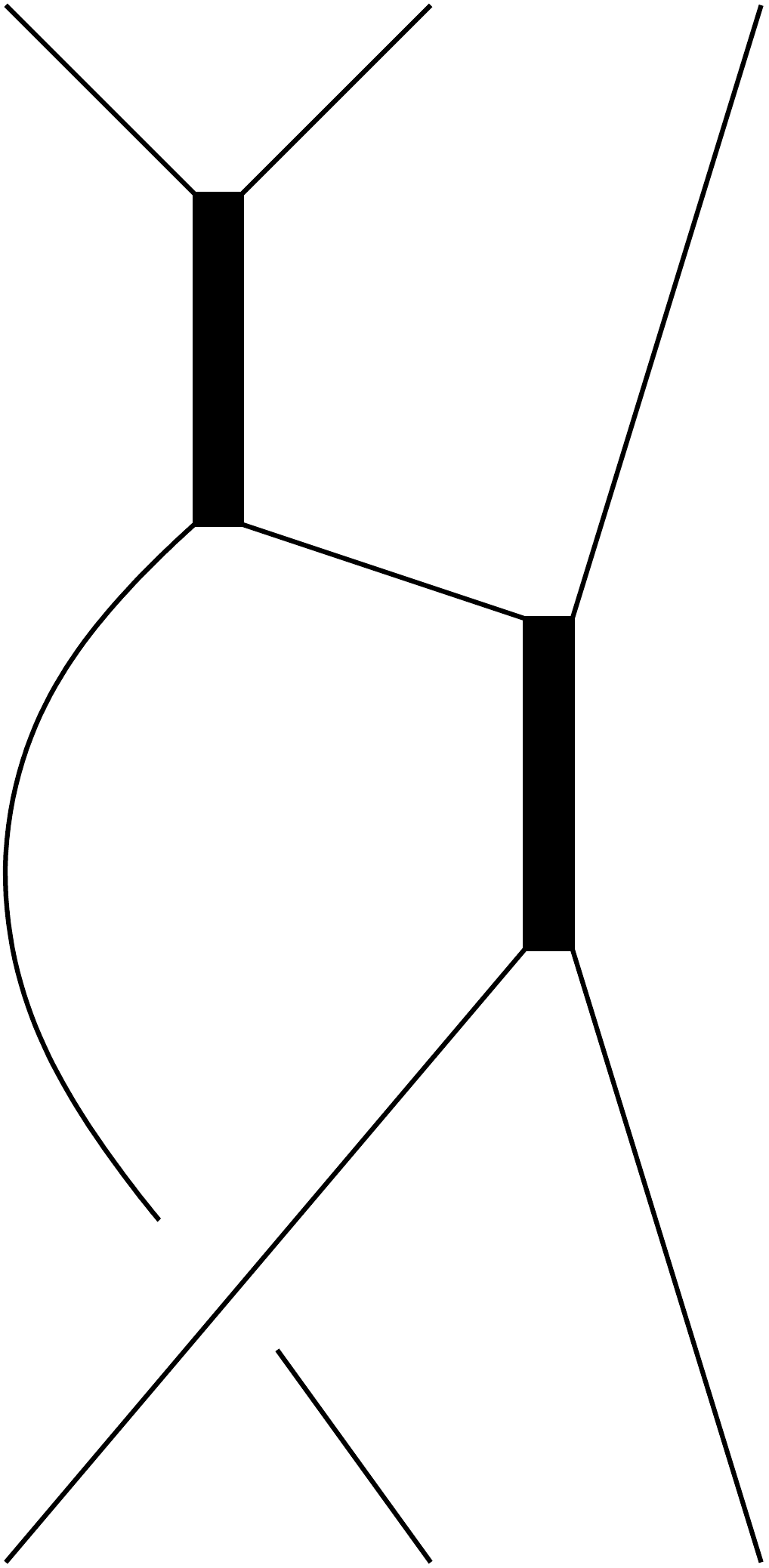}}}\\
&=& A\left( A\textbf{P}_{\,\raisebox{-3pt}{\includegraphics[height=0.35in]{long-6}}} + B\textbf{P}_{\,\raisebox{-3pt}{\includegraphics[height=0.35in, width = 0.2in]{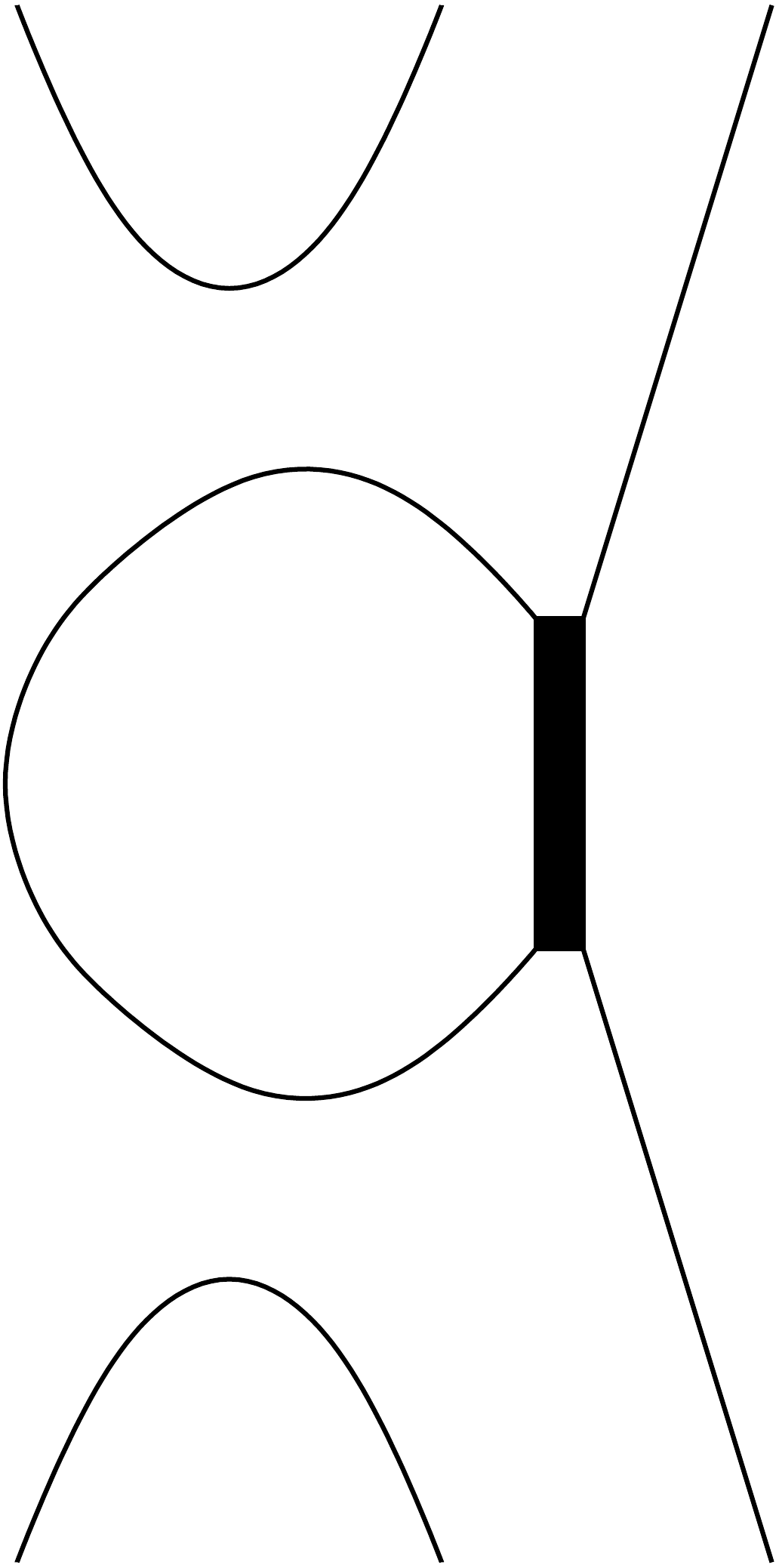}}} + \textbf{P}_{\,\reflectbox{\raisebox{3pt}{\includegraphics[height=0.35in, width = 0.2in, angle = 180]{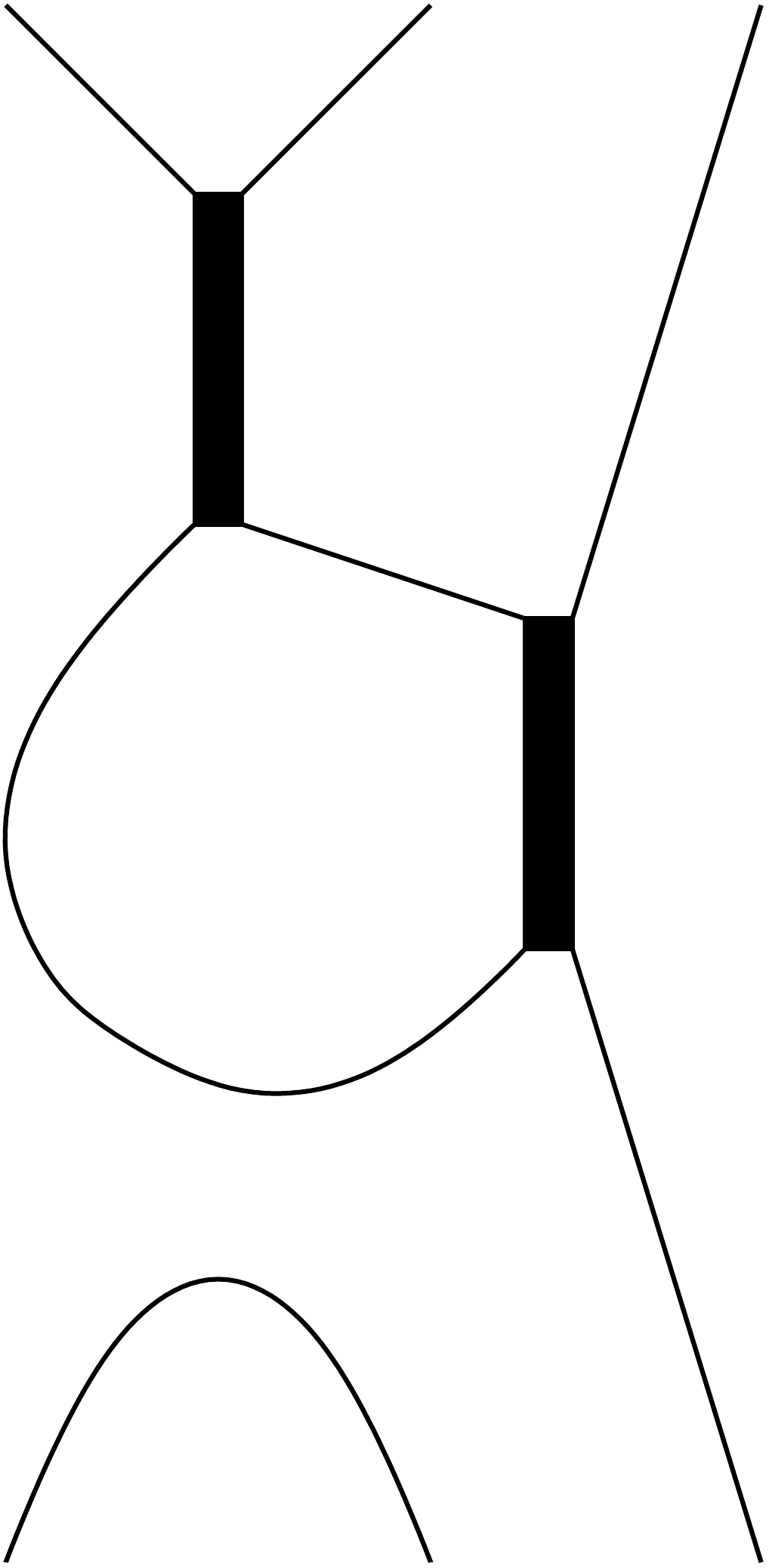}}}}\right) + B\left( A\textbf{P}_{\,\raisebox{-3pt}{\includegraphics[height=0.35in, width = 0.25in]{long-3}}} + B\textbf{P}_{\,\raisebox{-3pt}{\includegraphics[height=0.35in]{long-8}}} + \textbf{P}_{\,\raisebox{-3pt}{\includegraphics[ height=0.35in, width = 0.25in]{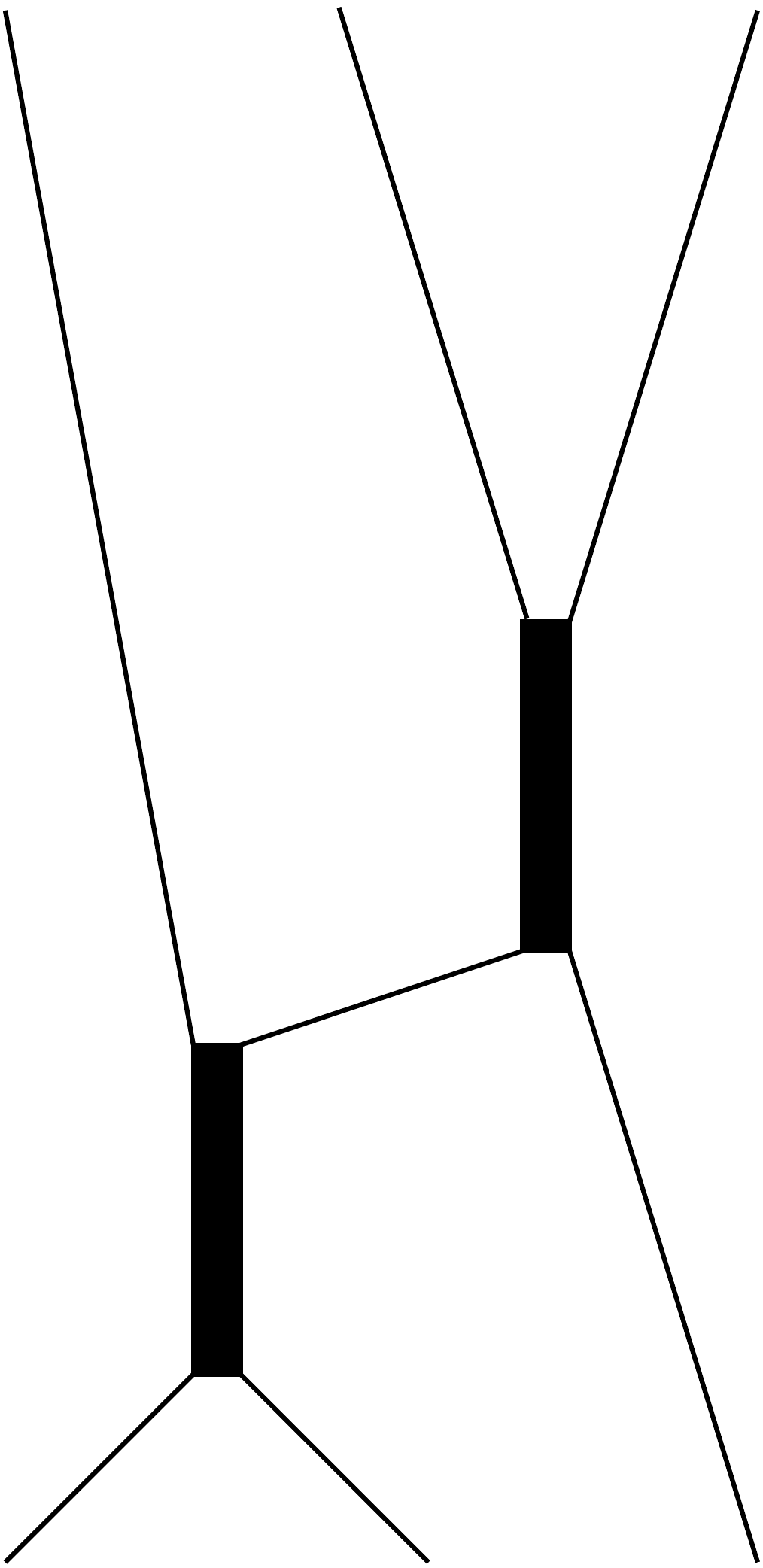}}}\right) \\
&&+ A\textbf{P}_{\,\reflectbox{\raisebox{3pt}{\includegraphics[height=0.35in,  width = 0.25in, angle = 180]{thm1_9}}}} + B\textbf{P}_{\,\raisebox{-3pt}{\includegraphics[height=0.35in, width = 0.2in]{thm1_11}}} + \textbf{P}_{\,\raisebox{-3pt}{\includegraphics[height=0.35in, width = 0.2in]{long-2}}}\\
&=& A^2\textbf{P}_{\,\raisebox{-3pt}{\includegraphics[height=0.35in]{long-6}}} + AB\beta\textbf{P}_{\,\raisebox{-3pt}{\includegraphics[height=0.35in, width = 0.25in]{long-10}}} + A\textbf{P}_{\,\reflectbox{\raisebox{3pt}{\includegraphics[height=0.35in, width = 0.2in, angle = 180]{thm1_11}}}} + AB\textbf{P}_{\,\raisebox{-3pt}{\includegraphics[height=0.35in, width = 0.25in]{long-3}}} + B^2\textbf{P}_{\,\raisebox{-3pt}{\includegraphics[height=0.35in]{long-8}}} + B\textbf{P}_{\,\raisebox{-3pt}{\includegraphics[height=0.35in,  width = 0.25in]{thm1_9}}}\\
&& + A\textbf{P}_{\,\reflectbox{\raisebox{3pt}{\includegraphics[height=0.35in,  width = 0.25in, angle = 180]{thm1_9}}}} + B\textbf{P}_{\,\raisebox{-3pt}{\includegraphics[height=0.35in, width = 0.2in]{thm1_11}}} + \textbf{P}_{\,\raisebox{-3pt}{\includegraphics[height=0.35in, width = 0.2in]{long-2}}}\,.
\end{eqnarray*}
We can use relation~(\ref{gi9}) to rewrite 
\begin{eqnarray*}
\textbf{P}_{\,\reflectbox{\raisebox{3pt}{\includegraphics[height=0.35in, width = 0.2in, angle = 180]{thm1_11}}}} &=& (1-AB) \textbf{P}_{\,\reflectbox{\raisebox{-3pt}{\includegraphics[height=0.35in, width = 0.25in]{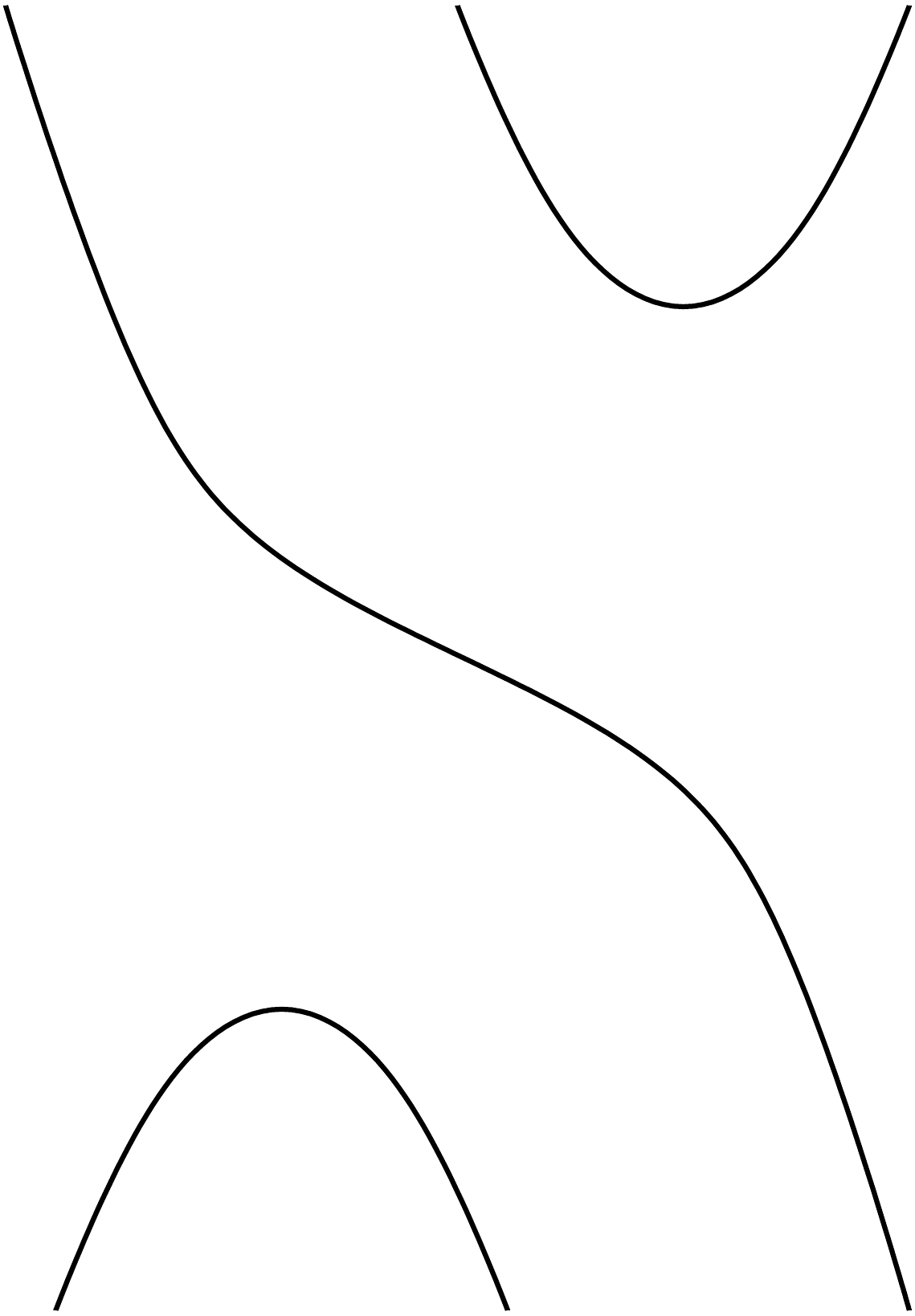}}}} + \gamma \textbf{P}_{\,\raisebox{-3pt}{\includegraphics[height=0.35in, width = 0.25in]{long-10}}} - (A + B) \textbf{P}_{\,\raisebox{-3pt}{\includegraphics[height=0.35in]{long-6}}} \\ 
\textbf{P}_{\,\raisebox{-3pt}{\includegraphics[height=0.35in, width = 0.2in]{thm1_11}}} &=& (1-AB) \textbf{P}_{\,\raisebox{-3pt}{\includegraphics[height=0.35in, width = 0.25in]{thm1_12}}} + \gamma \textbf{P}_{\,\raisebox{-3pt}{\includegraphics[height=0.35in, width = 0.25in]{long-10}}} - (A + B) \textbf{P}_{\,\raisebox{-3pt}{\includegraphics[height=0.35in]{long-8}}}
\end{eqnarray*}
and thus we obtain
\begin{eqnarray*}
\textbf{P}_{\,\raisebox{-3pt}{\includegraphics[height=0.35in]{thm1L}}}&=& (AB \beta + A\gamma + B\gamma)\textbf{P}_{\,\raisebox{-3pt}{\includegraphics[height=0.35in, width = 0.25in]{long-10}}} + A(1-AB) \textbf{P}_{\,\reflectbox{\raisebox{-3pt}{\includegraphics[height=0.35in, width = 0.25in]{thm1_12}}}} + B(1-AB)\textbf{P}_{\,\raisebox{-3pt}{\includegraphics[height=0.35in, width = 0.25in]{thm1_12}}}\\
&& -AB \left(\textbf{P}_{\,\raisebox{-3pt}{\includegraphics[height=0.35in]{long-6}}} + \textbf{P}_{\,\raisebox{-3pt}{\includegraphics[height=0.35in]{long-8}}} \ \right) + AB\textbf{P}_{\,\raisebox{-3pt}{\includegraphics[height=0.35in, width = 0.25in]{long-3}}} + B\textbf{P}_{\,\raisebox{-3pt}{\includegraphics[height=0.35in,  width = 0.25in]{thm1_9}}} + A\textbf{P}_{\,\reflectbox{\raisebox{3pt}{\includegraphics[height=0.35in,  width = 0.25in, angle = 180]{thm1_9}}}} + \textbf{P}_{\,\raisebox{-3pt}{\includegraphics[height=0.35in, width = 0.2in]{long-2}}}\,.
\end{eqnarray*}
Similarly we get
\begin{eqnarray*}
\textbf{P}_{\,\raisebox{-3pt}{\includegraphics[height=0.35in, width = 0.25in]{thm1R}}} &=& (AB \beta + A\gamma + B\gamma)\textbf{P}_{\,\raisebox{-3pt}{\includegraphics[height=0.35in, width = 0.25in]{long-9}}} + A(1-AB) \textbf{P}_{\,\reflectbox{\raisebox{-3pt}{\includegraphics[height=0.35in, width = 0.25in]{thm1_12}}}} + B(1-AB) \textbf{P}_{\,\raisebox{-3pt}{\includegraphics[height=0.35in, width = 0.25in]{thm1_12}}}\\
&& -AB \left(\textbf{P}_{\,\raisebox{-3pt}{\includegraphics[height=0.35in]{long-7}}} + \textbf{P}_{\,\raisebox{-3pt}{\includegraphics[height=0.35in]{long-5}}}  \right)+ AB\textbf{P}_{\,\raisebox{-3pt}{\includegraphics[height=0.35in, width = 0.25in]{long-4}}} + B\textbf{P}_{\,\raisebox{3pt}{\includegraphics[height=0.35in,  width = 0.25in, angle = 180]{thm1_9}}} + A\textbf{P}_{\,\reflectbox{\raisebox{-3pt}{\includegraphics[height=0.35in,  width = 0.25in]{thm1_9}}}}  + \textbf{P}_{\,\raisebox{-3pt}{\includegraphics[height=0.35in, width = 0.2in]{long-1}}}\,.
\end{eqnarray*}
Computing the following difference we have:

\begin{eqnarray*}
\textbf{P}_{\,\raisebox{-3pt}{\includegraphics[height=0.35in, width = 0.25in, width = 0.25in]{thm1L}}} - \textbf{P}_{\,\raisebox{-3pt}{\includegraphics[height=0.35in, width = 0.25in]{thm1R}}} &=& (AB \beta + A\gamma + B\gamma)\left(\textbf{P}_{\,\raisebox{-3pt}{\includegraphics[height=0.35in, width = 0.25in]{long-10}}} - \textbf{P}_{\,\raisebox{-3pt}{\includegraphics[height=0.35in, width = 0.25in]{long-9}}} \right)\\
&& + AB \left(\textbf{P}_{\,\raisebox{-3pt}{\includegraphics[height=0.35in]{long-7}}} + \textbf{P}_{\,\raisebox{-3pt}{\includegraphics[height=0.35in]{long-5}}} - \textbf{P}_{\,\raisebox{-3pt}{\includegraphics[height=0.35in]{long-6}}} - \textbf{P}_{\,\raisebox{-3pt}{\includegraphics[height=0.35in]{long-8}}} + \textbf{P}_{\,\raisebox{-3pt}{\includegraphics[height=0.35in, width = 0.25in]{long-3}}} - \textbf{P}_{\,\raisebox{-3pt}{\includegraphics[height=0.35in, width = 0.25in]{long-4}}} \right)\\
&& + \textbf{P}_{\,\raisebox{-3pt}{\includegraphics[height=0.35in, width = 0.2in]{long-2}}} - \textbf{P}_{\,\raisebox{-3pt}{\includegraphics[height=0.35in, width = 0.2in]{long-1}}}\,,
\end{eqnarray*}
and making use of relation~(\ref{gi6}) we arrive at

\begin{eqnarray*}
\textbf{P}_{\,\raisebox{-3pt}{\includegraphics[height=0.35in, width = 0.25in, width = 0.25in]{thm1L}}} - \textbf{P}_{\,\raisebox{-3pt}{\includegraphics[height=0.35in, width = 0.25in]{thm1R}}}&=& (AB \beta + A\gamma + B\gamma)\left(\textbf{P}_{\,\raisebox{-3pt}{\includegraphics[height=0.35in, width = 0.25in]{long-10}}} - \textbf{P}_{\,\raisebox{-3pt}{\includegraphics[height=0.35in, width = 0.25in]{long-9}}} \right)\\
&& + AB \left(\textbf{P}_{\,\raisebox{-3pt}{\includegraphics[height=0.35in]{long-7}}} + \textbf{P}_{\,\raisebox{-3pt}{\includegraphics[height=0.35in]{long-5}}} - \textbf{P}_{\,\raisebox{-3pt}{\includegraphics[height=0.35in]{long-6}}} - \textbf{P}_{\,\raisebox{-3pt}{\includegraphics[height=0.35in]{long-8}}} + \textbf{P}_{\,\raisebox{-3pt}{\includegraphics[height=0.35in, width = 0.25in]{long-3}}} - \textbf{P}_{\,\raisebox{-3pt}{\includegraphics[height=0.35in, width = 0.25in]{long-4}}} \right)\\
&& - AB\left(\textbf{P}_{\,\raisebox{-3pt}{\includegraphics[height=0.35in, width = 0.25in]{long-3}}} - \textbf{P}_{\,\raisebox{-3pt}{\includegraphics[height=0.35in, width = 0.25in]{long-4}}} + \textbf{P}_{\,\raisebox{-3pt}{\includegraphics[height=0.35in]{long-5}}} - \textbf{P}_{\,\raisebox{-3pt}{\includegraphics[height=0.35in]{long-6}}} + \textbf{P}_{\,\raisebox{-3pt}{\includegraphics[height=0.35in]{long-7}}} - \textbf{P}_{\,\raisebox{-3pt}{\includegraphics[height=0.35in]{long-8}}}\right)\\
&& - \delta\left(\textbf{P}_{\,\raisebox{-3pt}{\includegraphics[height=0.35in, width = 0.25in]{long-10}}} - \textbf{P}_{\,\raisebox{-3pt}{\includegraphics[height=0.35in, width = 0.25in]{long-9}}}\right)\\
&=&(AB\beta + A\gamma + B\gamma - \delta)\left(\textbf{P}_{\,\raisebox{-3pt}{\includegraphics[height=0.35in, width = 0.25in]{long-10}}} - \textbf{P}_{\,\raisebox{-3pt}{\includegraphics[height=0.35in, width = 0.25in]{long-9}}}\right)  =  0,
\end{eqnarray*}
since $AB\beta + A\gamma + B\gamma - \delta = 0$.

Using the previous result along with the invariance under the Reidemeister II move, we get

\begin{eqnarray*}
\textbf{P}_{\,\raisebox{-3pt}{\includegraphics[height=0.27in, angle = 90]{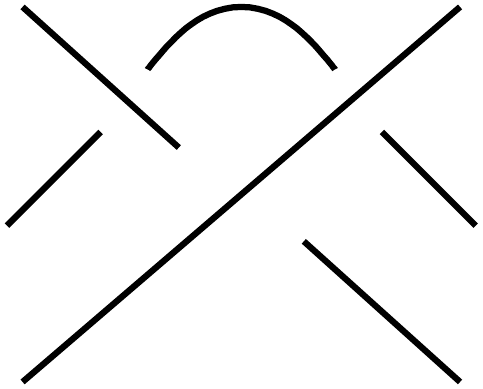}}} &=& A\textbf{P}_{\,\raisebox{-3pt}{\includegraphics[height=0.3in, angle = 90]{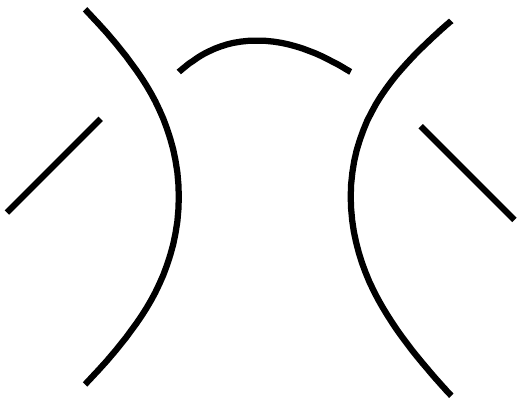}}} + B\textbf{P}_{\,\raisebox{-3pt}{\includegraphics[height=0.32in]{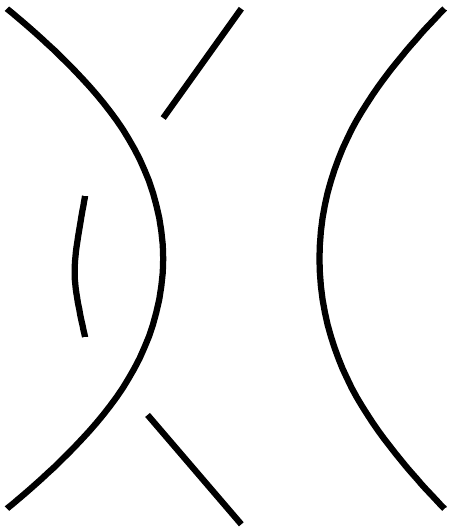}}} + \textbf{P}_{\,\raisebox{-3pt}{\includegraphics[height=0.35in,width = 0.25in]{thm1L}}}\\
&=& A\textbf{P}_{\,\reflectbox{\raisebox{-3pt}{\includegraphics[height=0.3in, angle = 90]{thm1_a}}}} + B\textbf{P}_{\,\raisebox{3pt}{\includegraphics[height=0.32in, angle=180]{slide-1}}} + \textbf{P}_{\,\raisebox{3pt}{\includegraphics[height=0.35in, width = 0.25in, angle = 180]{thm1L}}} = 
\textbf{P}_{\,\raisebox{3pt}{\includegraphics[height=0.27in, angle = 270]{reid3-1}}}
\end{eqnarray*}
and therefore the invariance under the Reidemeister move III holds.\end{proof}

Let $L$ be a link. Defining $\textbf{P}_L\co = \textbf{P}_D$, where $D$ is any plane diagram of $L$, we have proved that $\textbf{P}_L$ is a regular isotopy invariant of $L$.

\begin{corollary}
$\textbf{P}_L$ is equivalent to the two-variable Kauffman polynomial of $L$, for $z = A - B$. Namely,
\[ \textbf{P}_L(A, B, a) = D_L(A - B, a). \]
\end{corollary}

\begin{example}
We consider the Hopf link $L =  \raisebox{-7pt}{\includegraphics[height=0.3in]{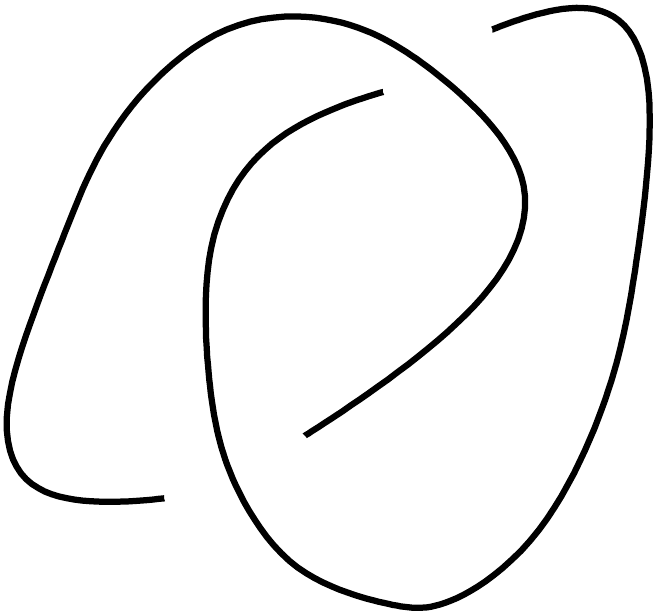}}$ and compute $P_L$. Then we have

\begin{eqnarray*}
\textbf{P}_{\,\raisebox{-3pt}{\includegraphics[height=0.27in]{hopf}}} &=& A\, \textbf{P}_{\,\raisebox{-3pt}{\includegraphics[height=0.27in]{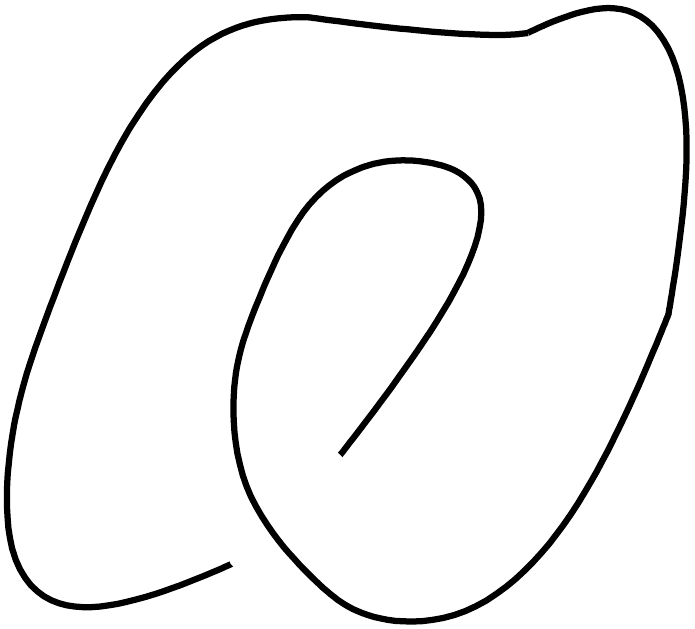}}} + B\, \textbf{P}_{\,\raisebox{-3pt}{\includegraphics[height=0.27in]{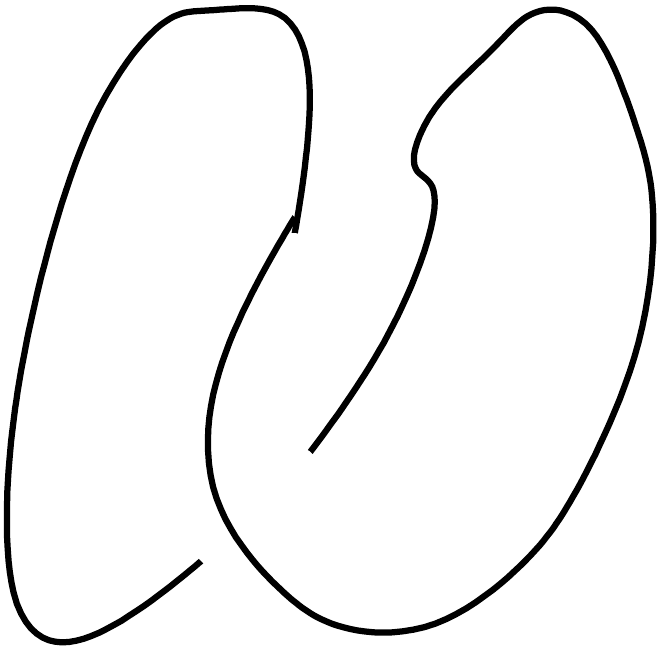}}} +  \textbf{P}_{\,\raisebox{-3pt}{\includegraphics[height=0.27in]{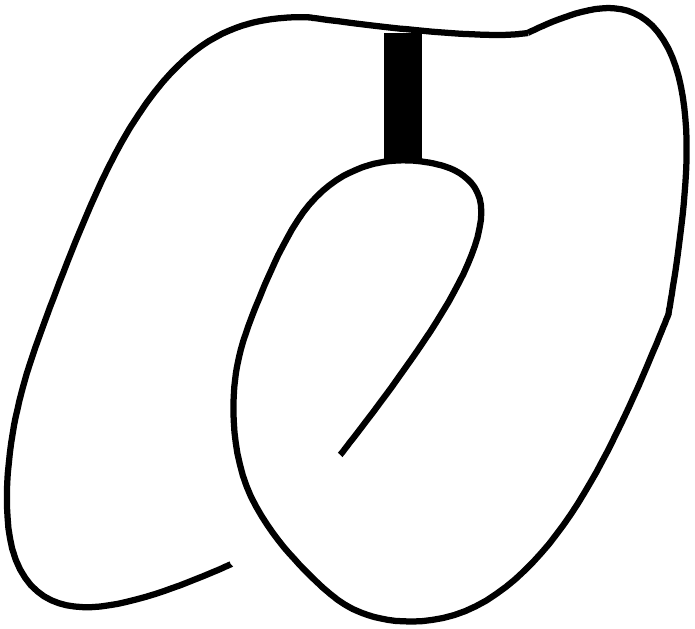}}}\\
&=& Aa + Ba^{-1} + A\, \textbf{P}_{\,\raisebox{-3pt}{\includegraphics[height=0.27in]{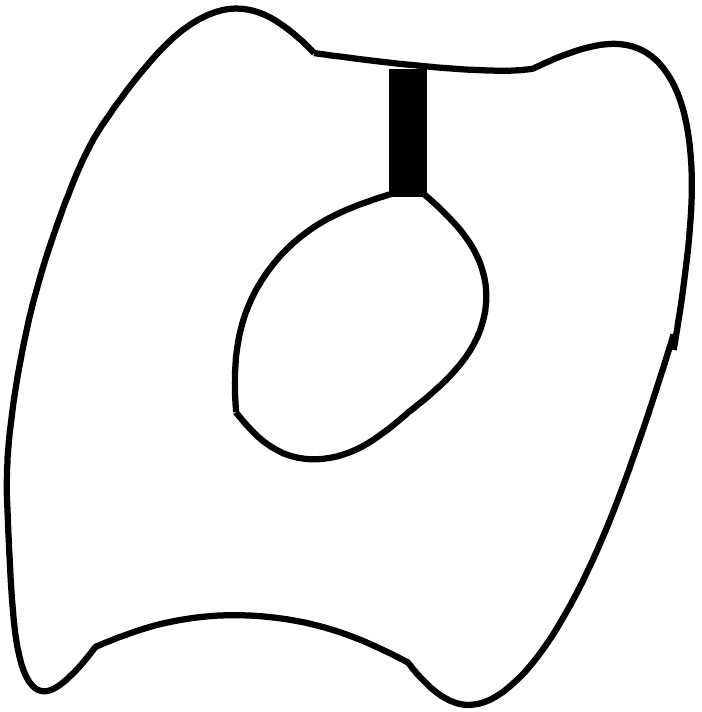}}} + B\, \textbf{P}_{\,\raisebox{-3pt}{\includegraphics[height=0.27in]{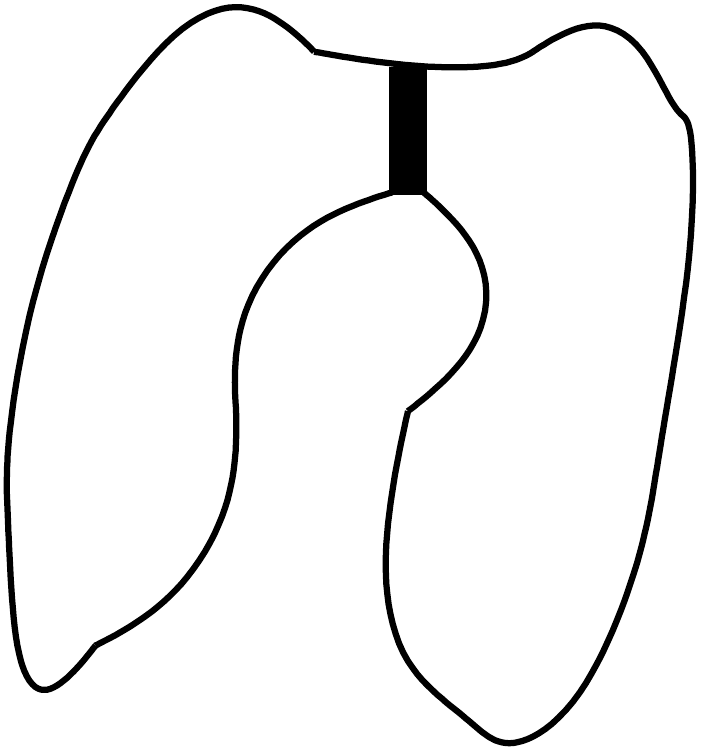}}} + \, \textbf{P}_{\,\raisebox{-3pt}{\includegraphics[height=0.27in]{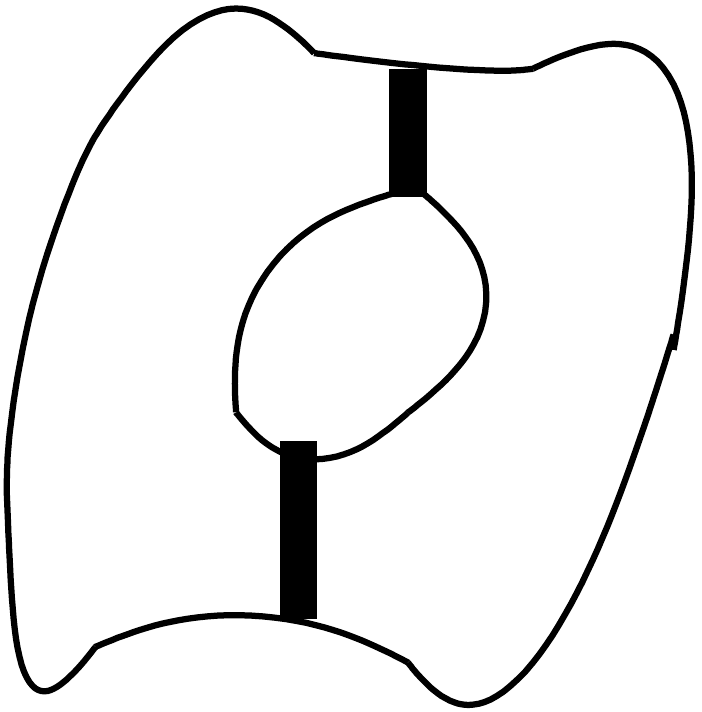}}}\,,
\end{eqnarray*}
where we used the behavior of the polynomial under the Reidemeister move I and the fact that $\textbf{P}_{
\bigcirc} = 1$. We further use relations~(\ref{gi4}) and~(\ref{gi5}), and obtain
\begin{eqnarray*}
\textbf{P}_{\,\raisebox{-3pt}{\includegraphics[height=0.27in]{hopf}}} &=& Aa + Ba^{-1} + A \beta + B \beta + (1-AB)\alpha + \gamma - (A+ B) \beta\\
& = & Aa + B a^{-1} + (1 -AB)\alpha + \gamma = (a-a^{-1})(A-B) + \frac{a-a^{-1}}{A-B}+1.
\end{eqnarray*}
\end{example}
\section{Yet another definition for the Kauffman polynomial} \label{sec:second definition}

The purpose of this section is to give another definition of the link invariant $\textbf{P}_L$, or equivalently, to the Kauffman polynomial of a link $L$, with the aid of a trace function on a certain algebra and well-known connections between links and braids. 

The braid group on $n$ strands, $B_n$, is the group generated by elements $\sigma_i$ (with inverses $\sigma_i^{-1}$): 
\[\sigma_i = \raisebox{-19pt}{\includegraphics[height=0.7in]{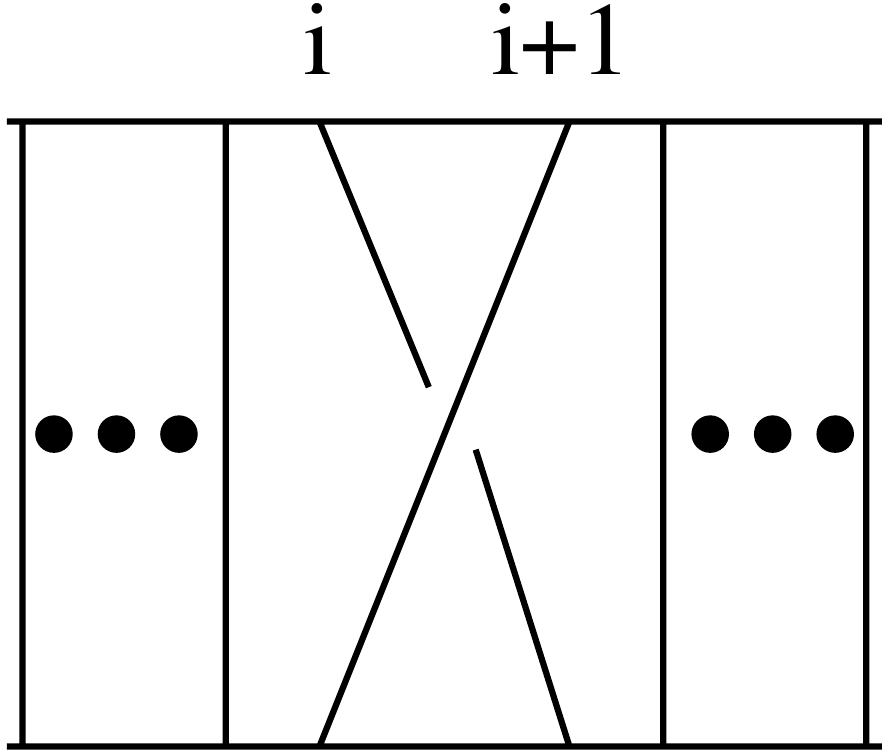}}\,, \qquad \sigma_i^{-1} = \raisebox{-19pt}{\includegraphics[height=0.7in]{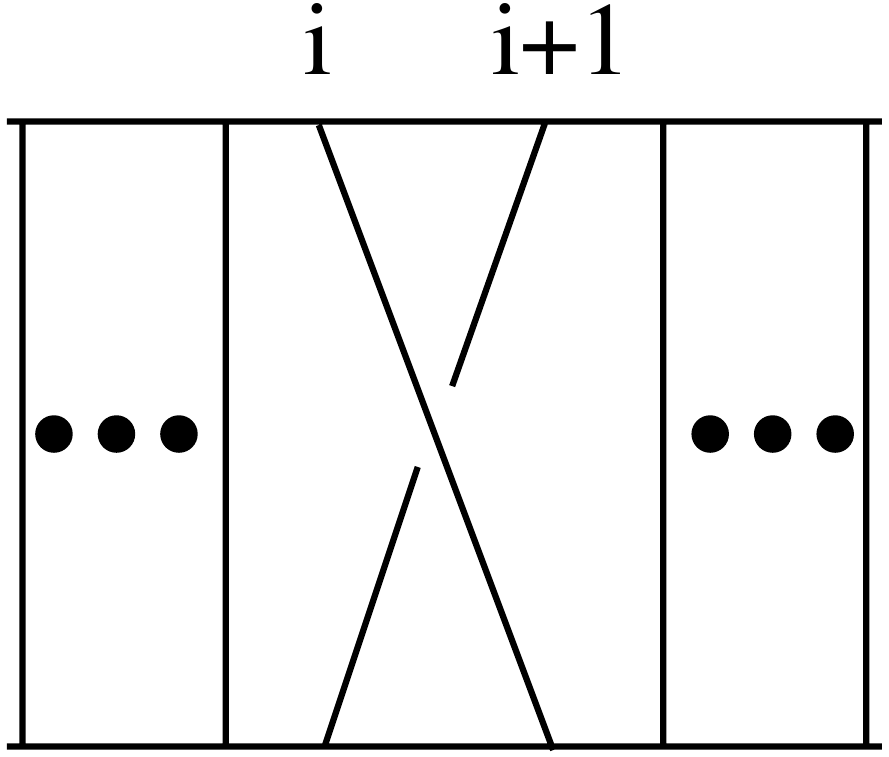}} \qquad  \text{for} \quad i = 1, 2, \dots, n-1,\] and satisfying the following relations:
\begin{eqnarray*}
 \sigma_i \sigma_j &=& \sigma _j \sigma _i, \,\,\text{for} \,\, |i -j| >1\\
 \sigma_i \sigma_{i+1} \sigma _i &=& \sigma_{i+1} \sigma_i \sigma _{i+1},\,\, \text{for} \,\,1\leq i \leq n-2.
 \end{eqnarray*}
 
 The \textit{closure} $\overline{\beta}$ of a braid $\beta \in B_n$ is obtained by connecting the initial points to the endpoints by a collection of parallel, non-weaving strands. Obviously, $\overline{\beta}$ is a knot or a link.

\textit{Alexander's Theorem}~\cite{A} states that any knot or link can be represented as the closure of a  braid, via ambient isotopy. But there are many different ways to represent a link as a closed braid, and the fundamental theorem that relates the theory of knots and the theory of braids is \textit{Markov's Theorem}, which states that if $\beta_n \in B_n$ and $\beta'_m \in B_m$ are two braids, then the links $L = \overline{\beta_n}$ and $L' = \overline{\beta'_m}$, the closures of the braids $\beta$ and $\beta'$, respectively, are ambient isotopic if and only if $\beta'_m$ can be obtained from $\beta_n$ by a series of equivalences and conjugations in a given braid group, and \textit{Markov moves} (such a move replaces $\beta \in B_n$ with $\beta \sigma_n^{\pm 1} \in B_{n+1}$ or replaces $\beta \sigma_n^{\pm 1} \in B_{n+1}$ with $\beta \in B_n$). The interested reader can find a proof of Markov's theorem in Birman's work~\cite{B}.

\subsection{The algebra $\mathcal{A}_n$}\label{ssec:algebra}

For each natural number $n$ we form a free (unital) additive algebra $\mathcal{A}_n$ over the ring $\bbZ[A^{\pm1}, B^{\pm1}, a^{\pm1}, (A-B)^{\pm 1}]$ generated by elements $c_i$ and $t_i$, for $i = 1, 2, \dots, n-1$, 
and subject to the following  relations:
\begin{enumerate}
\item[(1)] For $|i-j| >1$: $c_ic_j = c_jc_i,\, t_it_j = t_jt_i,\, c_it_j = t_jc_i$
\item[(2)] For $1\leq i\leq n-2$: $t_it_{i\pm1}t_i = t_i$
\item[(3)] For $1\leq i\leq n-2$:
\begin{enumerate}
\item  $c_it_{i+1} = c_{i+1}t_it_{i+1}$
\item $t_ic_{i+1} = t_it_{i+1}c_i$
\item $c_{i+1}c_ic_{i+1} - c_ic_{i+1}c_i = AB(c_{i+1} - c_i +  t_{i+1}c_i - t_ic_{i+1} + c_it_{i+1} - c_{i+1}t_i) + \delta \, (t_{i}- t_{i+1})$
 \end{enumerate}
 \item[(4)] For $1\leq i\leq n-1$: 
\begin{enumerate}
\item $t_i^2 = \alpha t_i$
\item $c_it_i = t_ic_i = \beta t_i$
\item $c_i^2 = (1 - AB) 1_n + \gamma \, t_i -(A+B)c_i$.
 \end{enumerate}
 \end{enumerate}
 where, as before, $\alpha = \frac{a-a^{-1}}{A-B} + 1, \, \, \beta = \frac{Aa^{-1}-Ba}{A-B}-A-B, \,\, \gamma = \frac{B^2a-A^2a^{-1}}{A-B} + AB$ and $\delta =  \frac{B^3a - A^3a^{-1}}{A - B}.$

The generators of $\mathcal{A}_n$ can be diagrammatically presented as shown below:  
\[c_i = \raisebox{-19pt}{\includegraphics[height=.7in]{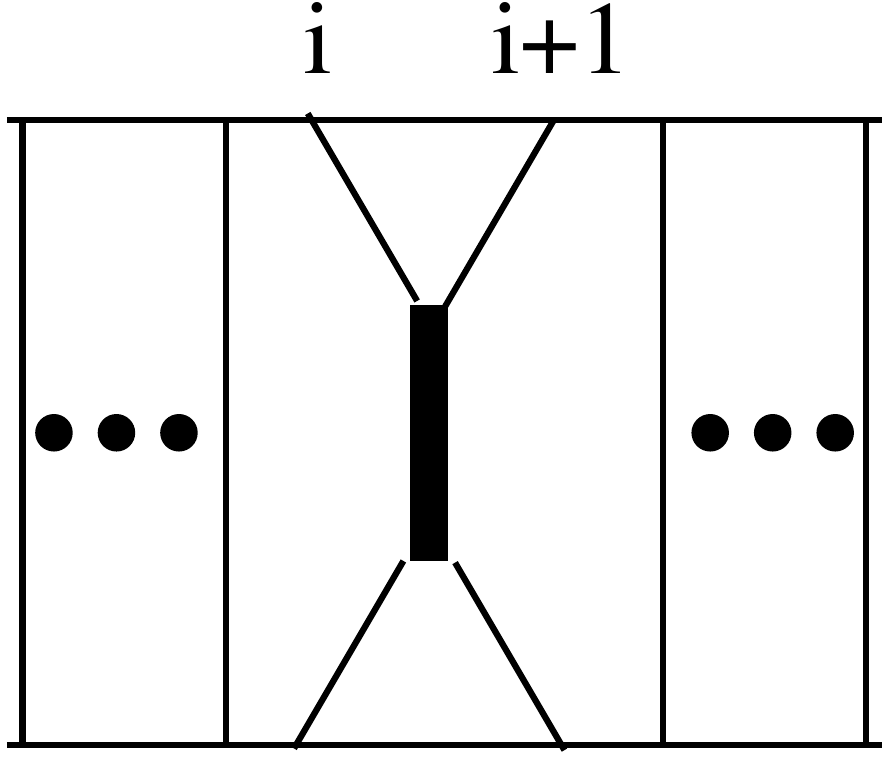}} \quad \text{and} \quad t_i = \raisebox{-19pt}{\includegraphics[height=.7in]{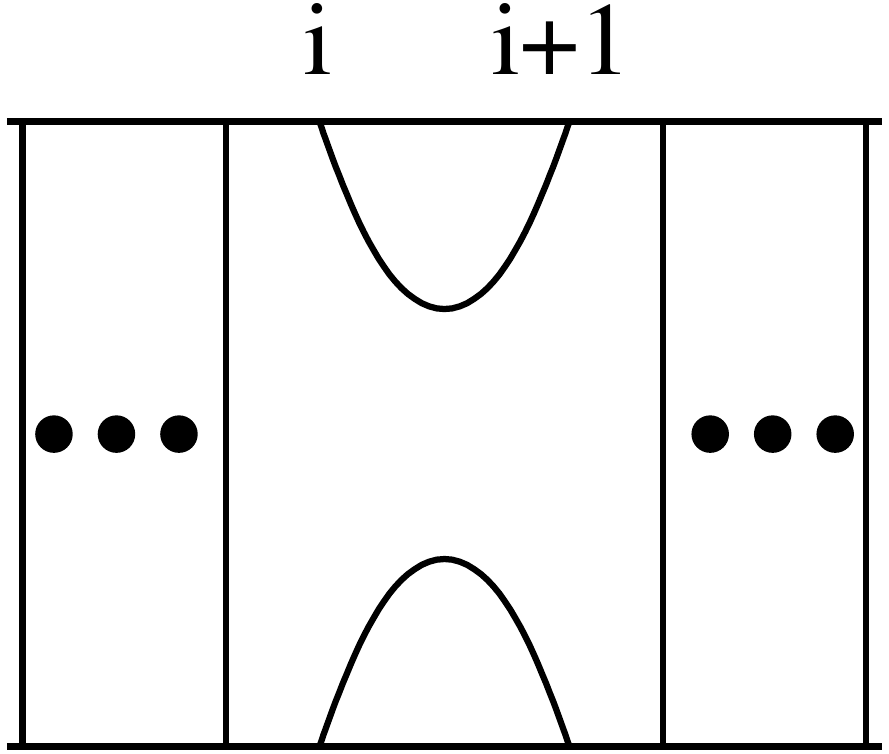}}.\]  
An element of $\mathcal{A}_n$ is a finite formal linear combination of planar trivalent graphs embedded in a rectangle with $n$ endpoints at the top and $n$ endpoints at the bottom of the rectangle. The multiplication is by concatenation and rescaling, and the identity element, $1_n$, is the diagram consisting of $n$ parallel strands in a rectangle.
 
 Relations (1) and (2) for the algebra $\mathcal{A}_n$ correspond to plane isotopies of the strands, while relations (3) and (4) are suggested by the graph skein relations. Specifically, relations (3a) and (3b) correspond to the graphical identity (\ref{gi3}) combined with plane isotopies of the strands, relation (3c) corresponds to the graphical identity (\ref{gi6}), and relations (4a) - (4c) correspond to the graphical identities~(\ref{gi2}), (\ref{gi7}) and (\ref{gi5}), respectively. We leave to the reader the enjoyable exercise of drawing all the pictures corresponding to the relations in $\mathcal{A}_n$. 
 
 \begin{remark}
 As noted in the introduction, the algebra $\mathcal{A}_n$ is isomorphic to the Birman-Murakami-Wenzl algebra $BMW_n(m, l)$ introduced in~\cite{BW, M} (corresponding to the Dubrovnik version of the two-variable Kauffman polynomial) via the map defined by the equation
 
 \[\raisebox{-6pt}{\includegraphics[height=0.25in]{resol}} = \raisebox{-3pt}{\includegraphics[height=0.15in]{poscrossing}} - A\,\, \raisebox{-3pt}{\includegraphics[height=0.15in]{A-smoothing}} - B\,\, \raisebox{-3pt}{\includegraphics[height=0.15in]{B-smoothing}}=
\raisebox{-3pt}{\includegraphics[height=0.15in]{negcrossing}} - A\,\,  \raisebox{-3pt}{\includegraphics[height=0.15in]{B-smoothing}} - B \, \, \raisebox{-3pt}{\includegraphics[height=0.15in]{A-smoothing}}\,, \]
where $m = A - B $ and $l = a$.
 \end{remark}
  
 \subsection{A representation of the braid groups}\label{ssec:representation}
 
 Inspired by the skein relations used in our state model for the Kauffman polynomial, we define a map $\rho : B_n \rightarrow \mathcal{A}_n$ given on generators by

\[\rho(\sigma_i) = A1_n  + Bt_i + c_i, \qquad \rho(\sigma_i^{-1}) = At_i + B1_n  + c_i, \quad 1 \leq i \leq n-1,\]
and extended to arbitrary braids by linearity.

We can think of $\rho$ as a function that resolves the crossings of the braid, since each $\sigma_i \in B_n$ or $\sigma_i^{-1} \in B_n$  represents a crossing of the strands in the braid:

\[ \raisebox{-19pt}{\includegraphics[height=0.7in]{sigma_i}} \stackrel{\rho}{\longrightarrow} A\, \raisebox{-19pt}{\includegraphics[height=0.7in]{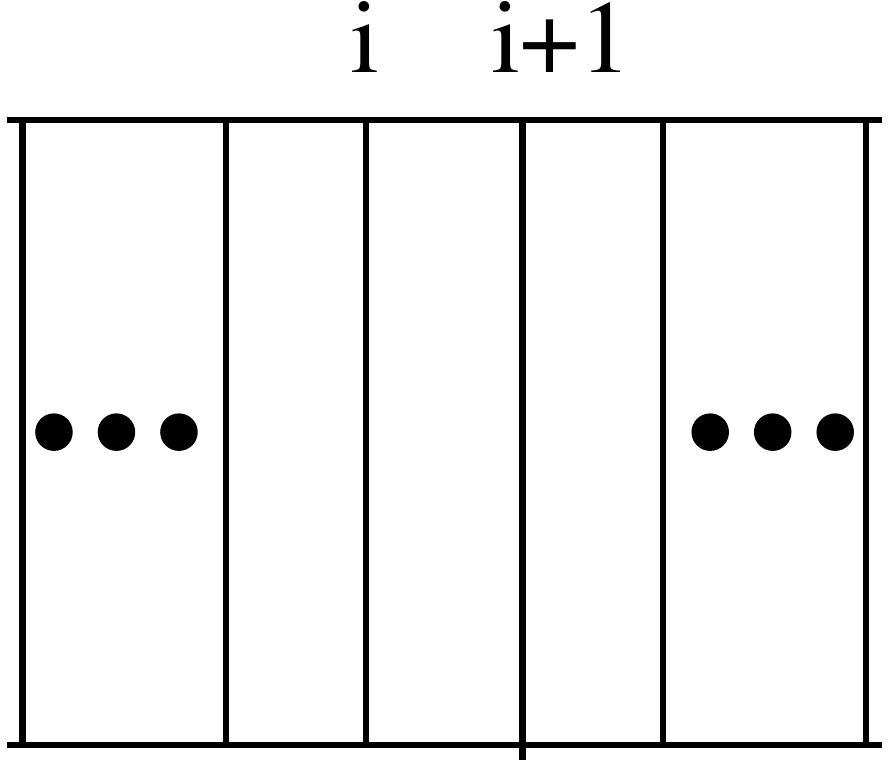}} + B\, \raisebox{-19pt}{\includegraphics[height=0.7in]{t_i}} + \raisebox{-19pt}{\includegraphics[height=0.7in]{c_i}}\,  \]

\[ \raisebox{-19pt}{\includegraphics[height=0.7in]{sigma_i_inv}}  \stackrel{\rho}{\longrightarrow}  A\, \raisebox{-19pt}{\includegraphics[height=0.7in]{t_i}} + B\, \raisebox{-19pt}{\includegraphics[height=0.7in]{identity}} + \raisebox{-19pt}{\includegraphics[height=0.7in]{c_i}}\,.\]

\begin{proposition}
The map $\rho$ is a representation of the braid group $B_n$ into the algebra $\mathcal{A}_n$.
\end{proposition}

\begin{proof} We need to show that $\rho$ preserves the braid group relations. Considering the connection between the definitions of $\rho$ and of $\textbf{P}_L$, along with the fact that $\textbf{P}_L$ is invariant under the second and third Reidemeister moves, the reader might be already convinced that $\rho$ does preserve these relations. However, we will check two of the braid group relations and leave the third one as an exercise.

First we show that  $\rho(\sigma_i \sigma_i^{-1}) = 1_n,$ for all $1 \leq i\leq n-1$. We have
\begin{eqnarray*}
\rho(\sigma_i \sigma_i^{-1}) &=& \rho(\sigma_i)\rho(\sigma_i^{-1}) = (A\, 1_n + B\,t_i + c_i) ( A\,t_i + B\, 1_n + c_i)\\
&=& A^2 \,t_i + AB\, 1_n + A\,c_i + AB \,t_i^2 + B^2\, t_i + B\, t_i\,c_i + A\, c_i\,t_i + B\,c_i+ c_i^2\\
&=& (A^2 + B^2)t_i + (A + B)c_i + (A + B)t_i\,c_i + AB\,t_i^2 + c_i^2  + AB\, 1_n\\
&=& (A^2 + B^2)t_i + (A + B)c_i + (A + B)\beta\, t_i + AB\,\alpha\, t_i + (1 - AB)1_n\\
&& + \gamma\, t_i - (A + B)c_i + AB\, 1_n\\
&=& [A^2 + B^2 + (A + B)\beta + AB\alpha + \gamma]t_i + 1_n = 1_n,
\end{eqnarray*}
where we used relations (4a) - (4c) and that $A^2 + B^2 + (A + B)\beta + AB\alpha + \gamma = 0$.

Now we verify that  $\rho(\sigma_i \sigma_j) = \rho(\sigma_j \sigma_i)$, for all $1 \leq i, j\leq n-1$ with $|i-j| > 1$. Using relations (1) of the algebra $\mathcal{A}_n$ we have
\begin{eqnarray*}
\rho(\sigma_i \sigma_j) &=& \rho(\sigma_i) \rho(\sigma_j) = (A\, 1_n + B\,t_i + c_i)(A\, 1_n + B\,t_j + c_j) \\
&=& A^2\, 1_n+ AB\,t_j + A\,c_j + AB\, t_i + B^2\,t_i\,t_j + B\,t_i\,c_j + A\, c_i + B\,c_i\,t_j + c_i\,c_j \\
&=& A^2\, 1_n + AB \,t_j + A\,c_j + A\,B t_i + B^2\,t_j\,t_i + B\,c_j\,t_i + A \,c_i + B\,t_j\,c_i + c_j\,c_i \\
&=& (A\, 1_n + B\,t_j + c_j)(A\, 1_n + B\,t_i + c_i) = \rho(\sigma_j) \rho(\sigma_i) = \rho(\sigma_j \sigma_i).
\end{eqnarray*}

We are left with checking that $\rho(\sigma_i \sigma_{i+1} \sigma_i) =  \rho(\sigma_{i+1} \sigma_i \sigma_{i+1})$, for all $1 \leq i\leq n-2$. We leave the proof of this one to the enthusiast.  One will need to use all of the relations in $\mathcal{A}_n$ except for the first one.
\end{proof}

\subsection{The bracket polynomial of a braid}\label{ssec:bracket-braids}

We shall show now that we can derive the Kauffman polynomial from the representation $\rho$.

Let $\beta = \sigma_{i_1}^{l_1}\sigma_{i_2}^{l_2}\dots \sigma_{i_s}^{l_s}$ (where $\, l_1, l_2, \dots, l_s \in \bbZ$) be an arbitrary braid element in $B_n$. Then
\[ \rho(\beta) = (A+ Bt_{i_1} + c_{i_1})^{l_1} (A+ Bt_{i_2} + c_{i_2})^{l_2} \dots  (A+ Bt_{i_s} + c_{i_s})^{l_s}.   \]
Denote by $\Gamma_k$ a generic resolved state of the braid $\beta$ and by $\beta_k$ the coefficient of $\Gamma_k$ in the expanded expression of $\rho(\beta)$. Note that $\Gamma_k$ is an element of $\mathcal{A}_n$. We have:
\[ \rho(\beta) = \sum_k \beta_k \Gamma_k,\] 
where $k$ is indexed over all states of the braid $\beta$. Denote by $\overline{\Gamma}_k$ the closure of $\Gamma_k$, obtained in the same manner as the closure of a braid.

\begin{definition}
Define a \textit{trace} function $\text{tr}\co \mathcal{A}_n \to \bbZ[A^{\pm 1}, B^{\pm 1}, a^{\pm 1}, (A-B)^{\pm 1}]$ by $ \text{tr}(z)= P(\overline{z})$, for all $z \in \mathcal{A}_n$, where $P$ is the graph polynomial defined in Section~\ref{sec:model}. 
\end{definition}
It is not hard to see that the function tr satisfies $\text{tr}(xy) = \text{tr}(yx)$, for all $x, y \in \mathcal{A}_n$.

We define the \textit{writhe} of the braid $\beta = \sigma_{i_1}^{l_1}\sigma_{i_2}^{l_2}\dots \sigma_{i_s}^{l_s}$, denoted by $w(\beta)$, to be the sum of the exponents of the generators in the braid expression. That is,
\[ w(\beta) = \sum_{i=1}^s l_i. \]
We notice that $w(\beta) = w(\overline{\beta})$, where $\overline{\beta}$ is the oriented link obtained by closing the braid $\beta$, with downward oriented strands.

To recover the Kauffman polynomial of a link, it suffices to employ Alexander's theorem and the functions $\rho$ and tr. To be specific, we want to construct a map from $B_n$ to $\bbZ[A^{\pm 1}, B^{\pm 1}, a^{\pm 1}, (A-B)^{\pm 1}]$ such that it satisfies the braid relations (i.e. it is well-defined on braids) and is invariant under conjugation in $B_n$ and under Markov moves.

\begin{definition}
Define the function $\brak{\,\,} \co B_n \to \bbZ[A^{\pm 1}, B^{\pm 1}, a^{\pm 1}, (A-B)^{\pm 1}]$ given by 
\[\brak{\beta} = a^{-w(\beta)} \text{tr}(\rho(\beta)), \quad \text{for \ all} \quad \beta \in B_n.\]
We call $\brak{\beta}$ the \textit{bracket} polynomial of the braid $\beta$.
\end{definition}

\begin{proposition}
The bracket polynomial of a braid is well-defined on braids and is invariant under conjugation in $B_n$ and under Markov moves. Moreover, if $L$ is a link that has a plane diagram which is the closure of a braid $\beta \in B_n$, then
\[ \brak{\beta} = a^{-w(\beta)}\textbf{P}_{\overline{\beta}} =  a^{-w(\beta)}\textbf{P}_L. \]
\end{proposition}

\begin{proof} It is easy to notice that for $\beta \in B_n$, we have $\text{tr}(\rho(\beta)) = \textbf{P}_{\overline{\beta}}$. Then the fact that $\textbf{P}_L$ is a regular invariant for links implies that the function $\brak{\,\,}$ is invariant under equivalences and conjugations in $B_n$. Moreover, the coefficient $a^{-w(\beta)}$ in the expression of $\brak{\,\,}$ cancels the effect of a Markov move, therefore $\brak{\,\,}$ is invariant under Markov moves as well. Finally, Markov's theorem implies the second part of the proposition.
\end{proof}

\section{Knotted trivalent graphs} \label{sec: knotted graphs}

Inspired by the work of Kauffman and Vogel~\cite{KV}, we consider now knotted trivalent graphs and create an invariant of these objects, closely related to invariants in knot theory. We are interested in using our state model for the Kauffman polynomial, that is, the polynomial $\textbf{P}_L$ introduced in Section~\ref{sec:model}, which forces us to work with knotted rigid-edge trivalent graphs (knotted RE 3-graphs). The reason for working with these type of graphs is similar to that for considering 4-valent graphs with rigid vertices in~\cite{KV}.

We remark that any trivalent graph has an even number of vertices. Thus if $G$ is a diagram of a knotted RE 3-graph, then it contains an even number of vertices, and for each pair of adjacent vertices there is exactly one wide (rigid) edge incident with both of them.

\subsection{Isotopies for RE 3-graphs} 

Two knotted graphs are called equivalent if there is an isotopy of $\bbR^3$ taking one onto the other. Kauffman~\cite{K1} showed that any two diagrams of equivalent knotted trivalent graphs (topological trivalent graphs not RE 3-graphs) are related by a finite sequence of the classical knot theoretic Reidemeider moves and the local moves depicted in~(\ref{graph moves}) (together with their mirror image). 

\begin{eqnarray}\label{graph moves}
\raisebox{-8pt}{\includegraphics[height=0.27in]{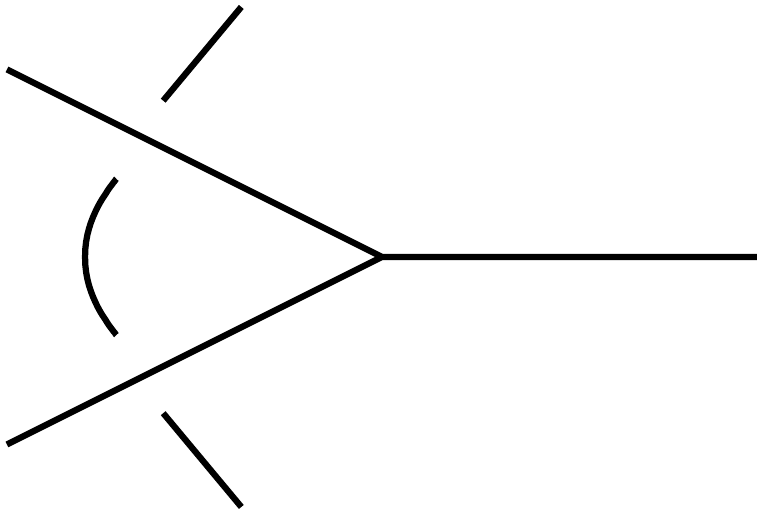}} \longleftrightarrow \raisebox{-8pt}{\includegraphics[height=0.27in]{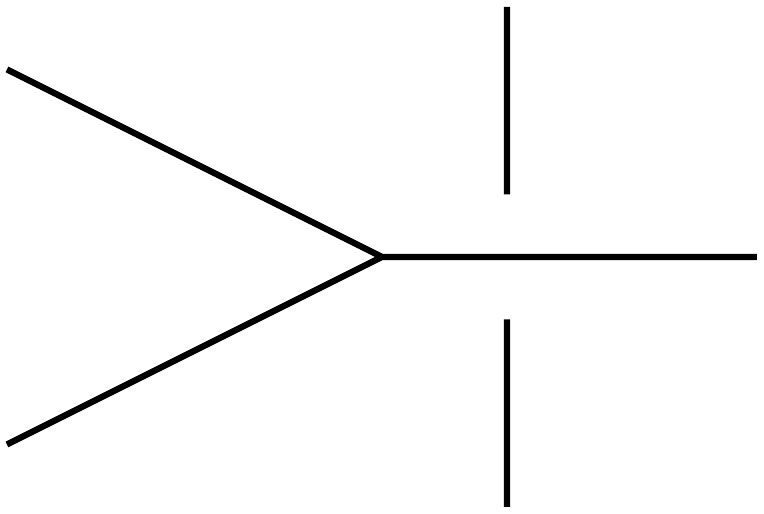}} \qquad \raisebox{-8pt}{\includegraphics[height=0.27in]{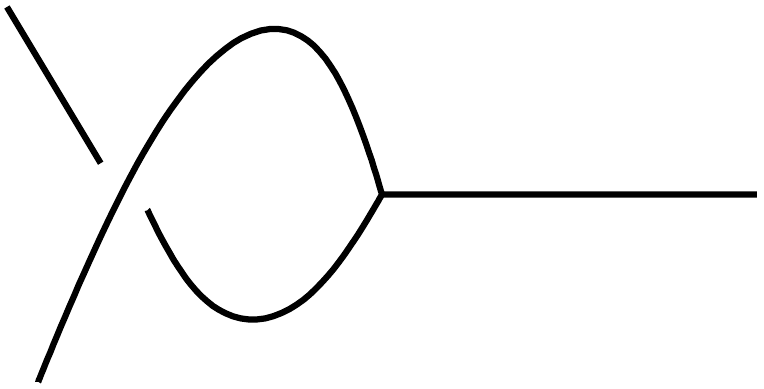}} \longleftrightarrow \raisebox{-8pt}{\includegraphics[height=0.27in]{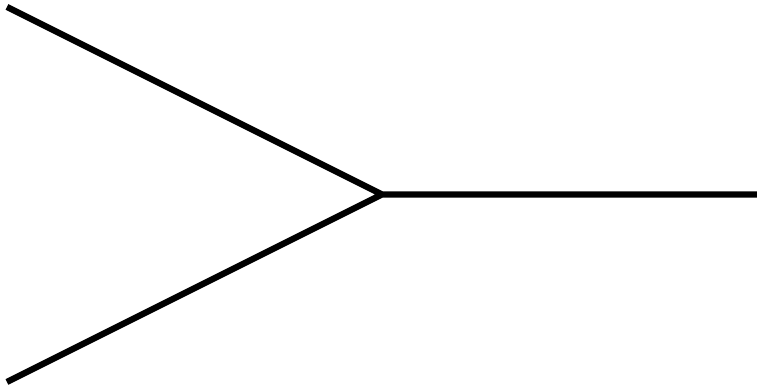}}
\end{eqnarray} 

Then it is not hard to see that two diagrams represent equivalent knotted RE 3-graphs if one can be transformed into the other by a finite sequence of, what we call, \textit{rigid-edge ambient isotopies}. The collection of moves (up to mirror image) that generate rigid-edge ambient isotopies is given below:

\begin{eqnarray*}
(I) && \raisebox{-8pt}{\includegraphics[height=0.3in]{poskink}}  \longleftrightarrow  \raisebox{-8pt}{\includegraphics[height=0.3in]{arc}}\\ 
(II) && \raisebox{-8pt}{\includegraphics[height=0.3in]{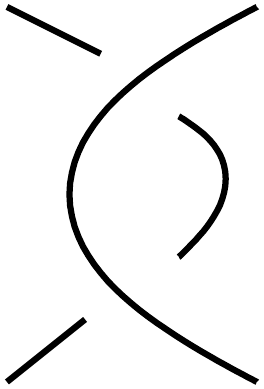}}  \longleftrightarrow \raisebox{-8pt}{\includegraphics[height=0.3in]{A-smoothing}}\\ \ 
(III) && \raisebox{-8pt}{\includegraphics[height=0.3in]{reid3-1}} \longleftrightarrow  \raisebox{-8pt}{\includegraphics[height=0.3in]{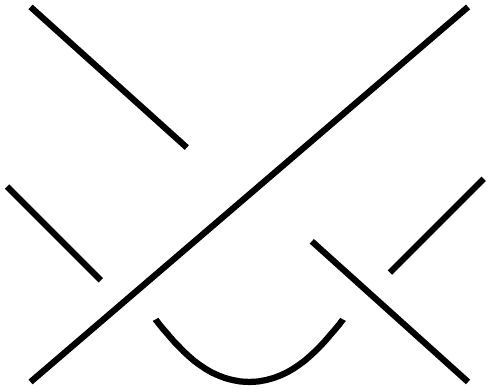}} \\(IV) && \raisebox{-8pt}{\includegraphics[height=0.3in]{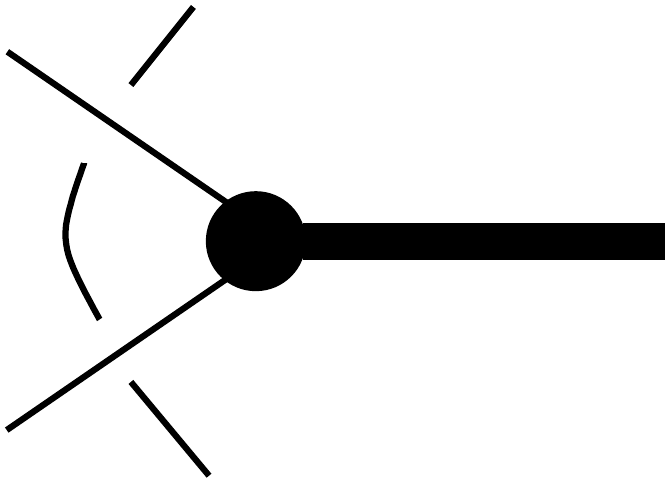}} \longleftrightarrow  \raisebox{-8pt}{\includegraphics[height=0.3in]{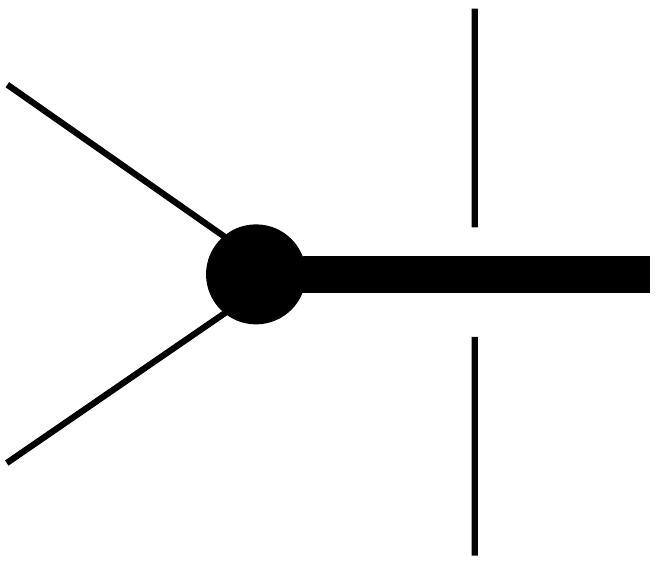}} \hspace{1cm}\raisebox{-8pt}{\includegraphics[height=0.3in]{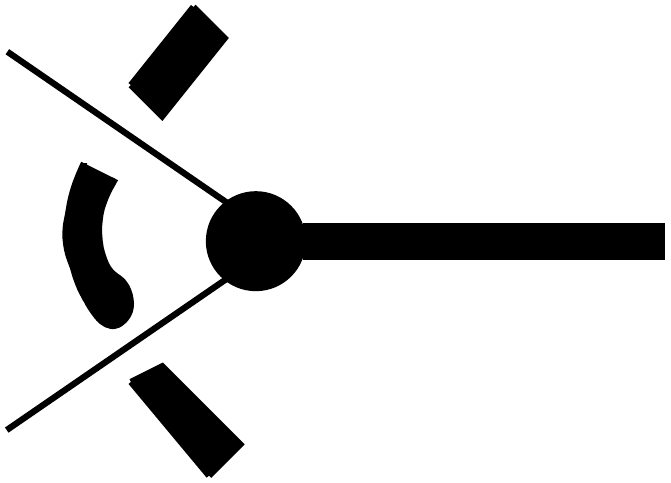}} \longleftrightarrow \raisebox{-8pt}{\includegraphics[height=0.3in]{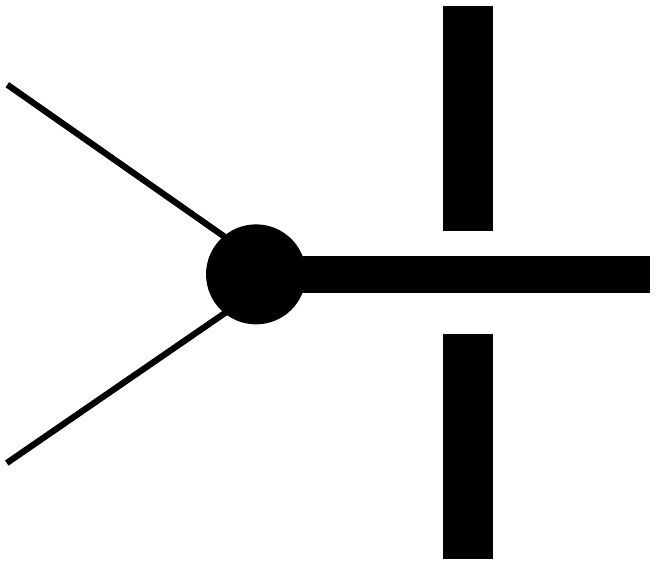}}\\ 
(V) && \raisebox{-7pt}{\includegraphics[height=0.25in]{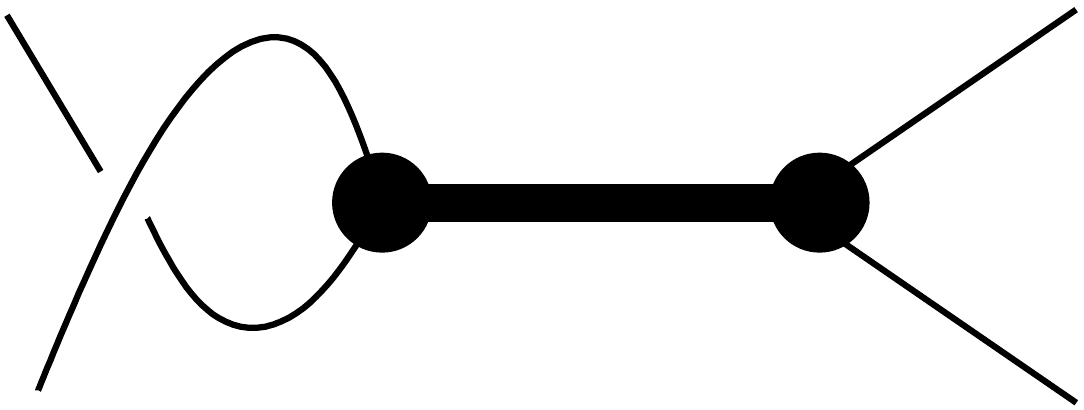}} \longleftrightarrow  \raisebox{-7pt}{\includegraphics[height=0.26in]{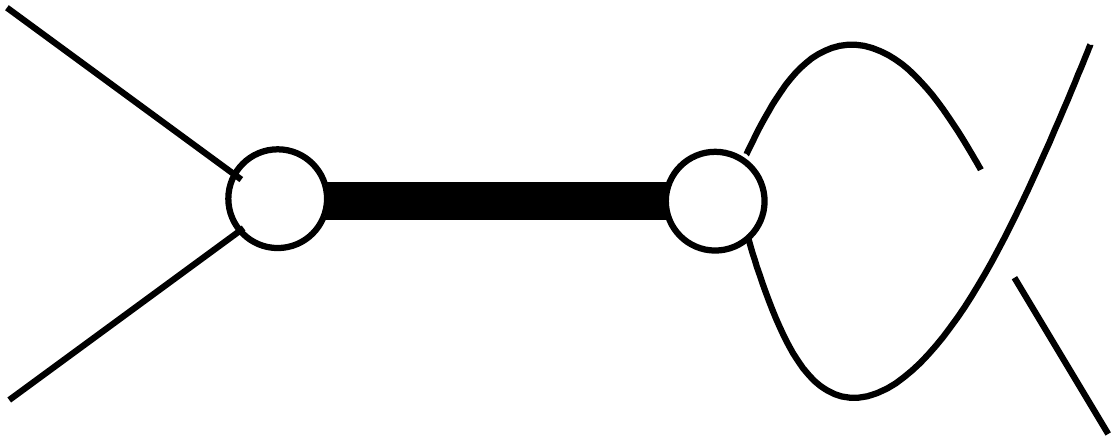}}
\end{eqnarray*}

A few words are needed here. A wide (rigid) edge is not allowed to cross itself; that is, there is no Reidemester I move involving a wide edge. However, there are other versions of the type II and type III Reidemeister moves for RE 3-graphs than those listed here, namely those containing at least one wide edge, but it is easy to show that these moves follow from type IV moves combined with types II and III Reidemester moves involving standard edges. We exemplify below the Reidemeister move of type II involving one wide edge:

\[ \raisebox{-9pt}{\includegraphics[height=0.3in]{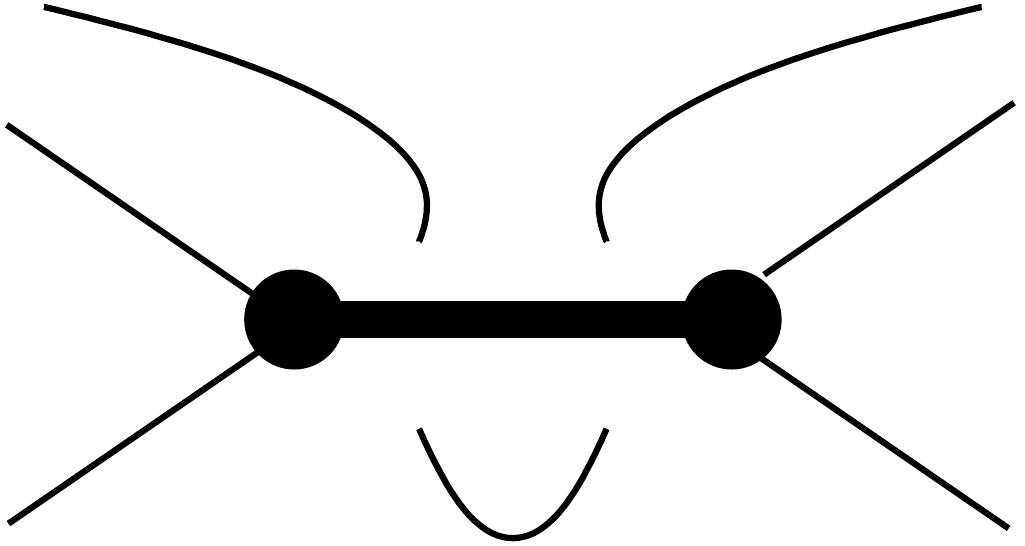}} \longleftrightarrow \raisebox{-9pt}{\includegraphics[height=0.3in]{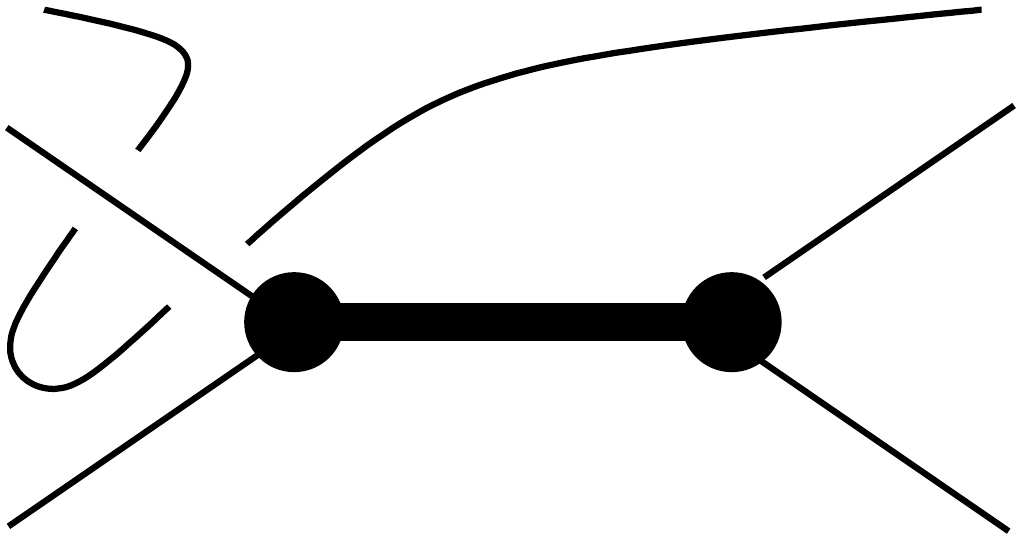}} \longleftrightarrow \raisebox{-9pt}{\includegraphics[height=0.3in]{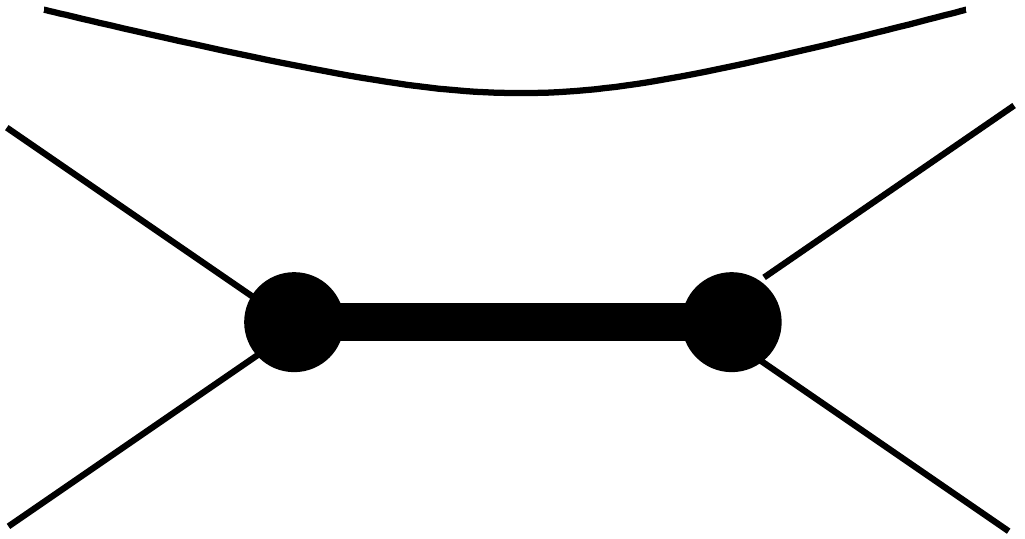}}   \] 

Moreover, there are other two variations (up to mirror image) of the type IV moves, corresponding to the other choices of the wide edge adjacent to the vertex; these moves are consequences of the type IV moves that are listed, followed by Reidemeister II moves, as we explain below:

\[ \raisebox{-9pt}{\includegraphics[height=0.3in]{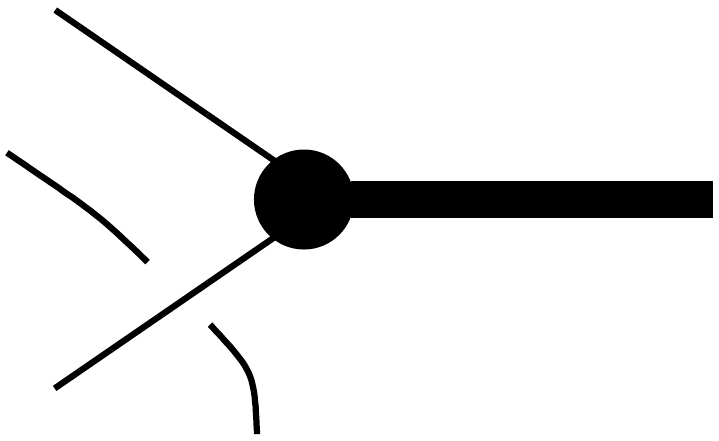}} \longleftrightarrow \raisebox{-9pt}{\includegraphics[height=0.3in]{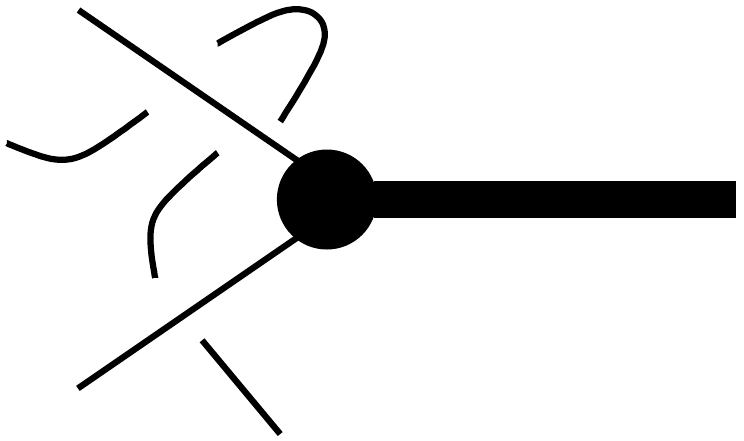}} \longleftrightarrow \raisebox{-9pt}{\includegraphics[height=0.3in]{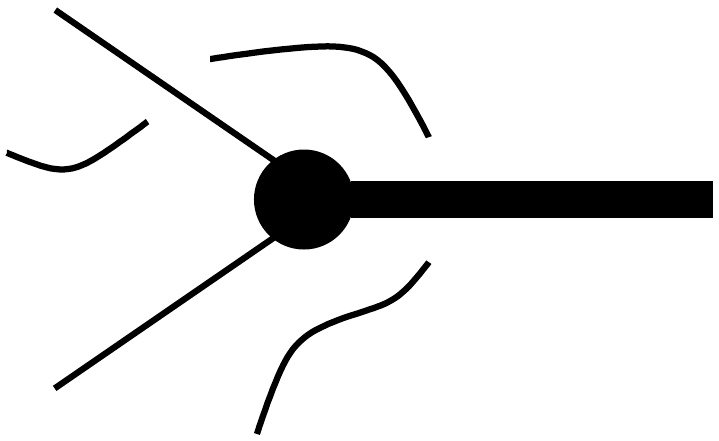}}   \] 

We say that some quantity is an \textit{RE ambient isotopy invariant} of knotted RE 3-graphs if it invariant under rigid-edge ambient isotopy (that is, it is unchanged under the moves of type I through type V), and we say that it is an \textit{ RE regular isotopy invariant} if it is unchanged under the moves of type II  through type V (in analogy with ambient and regular isotopy for link diagrams).

The move of type V tells us that, indeed, we regard the wide edges as being rigid. Specifically, the move consists of flipping the disk containing the local diagram (there is exactly one wide edge in this diagram) while keeping the endpoints fixed; we drew the vertices on the right hand side of the move as little white disks, to represent the flipping behavior under the move. The reader might wonder now what is the reason for considering rigid wide edges and, correspondingly, the move of type V in this form, since a more appropriate move would have been the following (as suggested by~(\ref{graph moves})):
 \[(VI) \qquad \raisebox{-7pt}{\includegraphics[height=0.25in]{move5-1}} \longleftrightarrow  \raisebox{-7pt}{\includegraphics[height=0.25in]{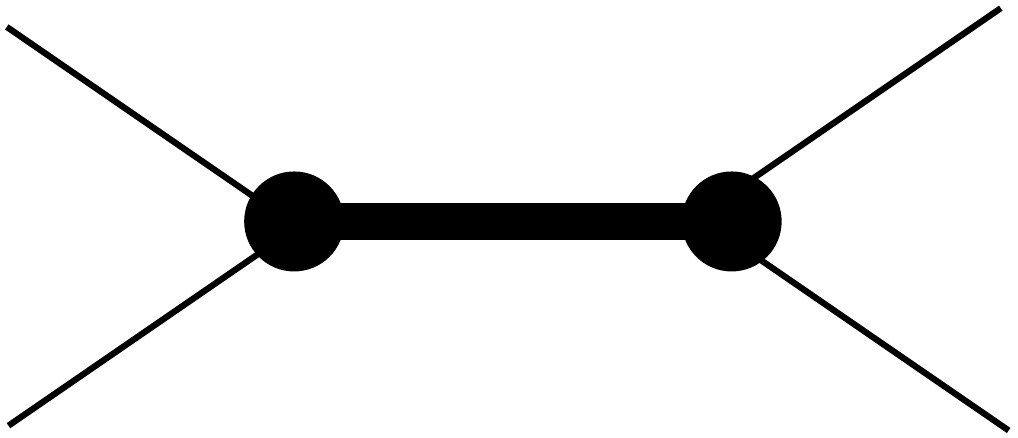}}\]
 
The reason for this choice is that the invariant for trivalent graphs that we construct in this section is not invariant under the type VI move, but it is invariant under the type V move.

\subsection{An invariant for knotted RE 3-graphs} The goal of this subsection is to construct a polynomial invariant for knotted RE 3-graphs. For this, we rely on the graphical identities for planar trivalent graphs given in Section~\ref{sec:model}.

\begin{theorem}\label{thm:graph poly}
If $[G]$ is a polynomial for RE 3-graph diagrams $G$ such that it is invariant under the moves of type IV, it takes value $1$ for the unknot, and it satisfies the graph skein relations~(\ref{gi2}) - (\ref{gi6}), as well as

 \[\left[ \, \raisebox{-5pt}{\includegraphics[height=0.2in]{move5-3}}\, \right ]= \left [ \,\raisebox{-5pt}{\includegraphics[height=0.2in]{poscrossing}}\, \right] - A \left[ \, \raisebox{-5pt}{\includegraphics[height=0.2in]{A-smoothing}}\, \right ] - B \left[\, \raisebox{-5pt}{\includegraphics[height=0.2in]{B-smoothing}}\,\right ] = \left [ \,\raisebox{-5pt}{\includegraphics[height=0.2in]{negcrossing}}\, \right] - A \left[ \, \raisebox{-5pt}{\includegraphics[height=0.2in]{B-smoothing}}\, \right ] - B \left[\, \raisebox{-5pt}{\includegraphics[height=0.2in]{A-smoothing}}\,\right ] ,\] 
 then it is an RE regular isotopy invariant of knotted RE 3-graphs, and satisfies
\[ \left[ \,\raisebox{-5pt}{\includegraphics[height=0.2in]{poskink}}\,\right] = a \left[ \,\raisebox{-5pt}{\includegraphics[height=0.2in]{arc}}\,\right], \quad \left[\,\raisebox{-5pt}{\includegraphics[height=0.2in]{negkink}}\,\right] = a^{-1} \left[ \,\raisebox{-5pt}{\includegraphics[height=0.2in]{arc}} \,\right].\]
\end{theorem}

\begin{proof} The hypothesis of the theorem together with the proof of Theorem~\ref{thm:invariance} imply that $\left[ \,\raisebox{-5pt}{\includegraphics[height=0.2in]{poskink}}\,\right] = a \left[ \,\raisebox{-5pt}{\includegraphics[height=0.2in]{arc}}\,\right], \, \left[\,\raisebox{-5pt}{\includegraphics[height=0.2in]{negkink}}\,\right] = a^{-1} \left[ \,\raisebox{-5pt}{\includegraphics[height=0.2in]{arc}} \,\right]$, and that $[G]$ is invariant under the type II and III moves for knotted RE 3-graphs. 
We demonstrate the invariance of $[G]$ under the move of type V: 
\begin{eqnarray*}
\left[\,\raisebox{-7pt}{\includegraphics[height=0.25in]{move5-1}}\,\right] &=& \left[\, \raisebox{-7pt}{\includegraphics[height=0.23in]{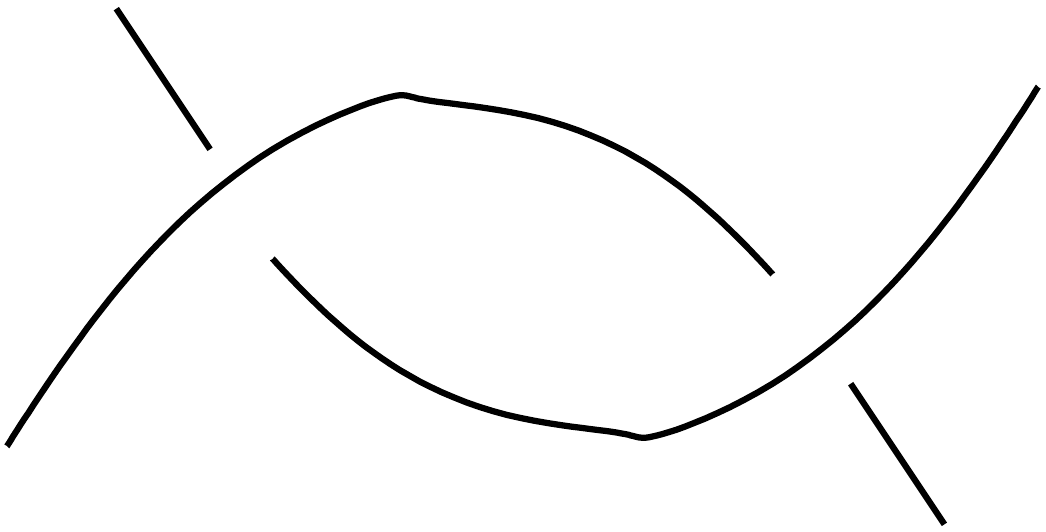}}\,\right] - A \left[\, \raisebox{-7pt}{\includegraphics[height=0.25in]{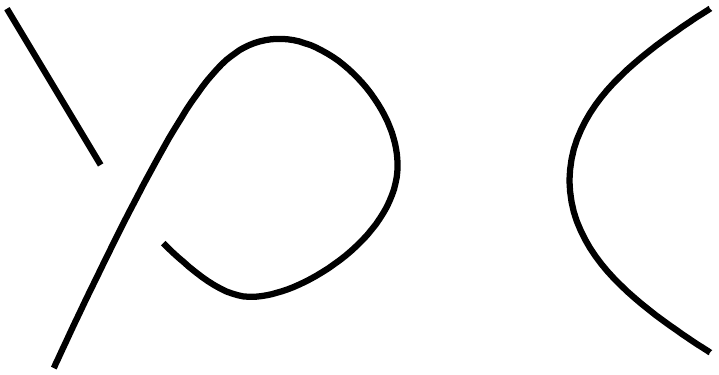}}\,\,\right] - B \left[\, \raisebox{-7pt}{\includegraphics[height=0.25in]{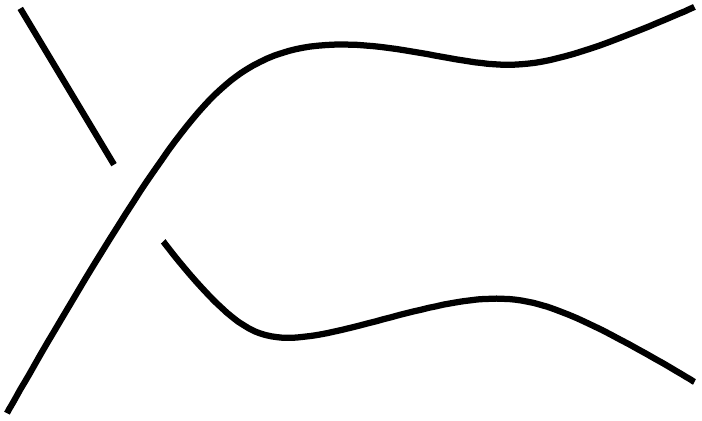}}\,\, \right] \\
&=&  \left[\, \raisebox{-7pt}{\includegraphics[height=0.23in]{thm4_1}}\,\right] - A a \left[ \,\raisebox{-7pt}{\includegraphics[height=0.25in]{A-smoothing}}\,\right] - B \left[ \,\raisebox{-7pt}{\includegraphics[height=0.25in]{poscrossing}}\, \right] \\
&=& \left[\, \raisebox{-7pt}{\includegraphics[height=0.23in]{thm4_1}}\,\right] -A \left[\, \raisebox{12pt}{\includegraphics[height=0.25in, angle = 180]{thm4_4}}\,\right] - B \left[\, \raisebox{12pt}{\includegraphics[height=0.25in, angle = 180]{thm4_6}}\, \right] \\
&=&\left[\,\raisebox{-7pt}{\includegraphics[height=0.25in]{move5-2}}\,\right].
\end{eqnarray*}
\end{proof}

 \begin{remark}
The reader may have noticed that according to Theorem~\ref{thm:graph poly}, the polynomial $[G] \in  \bbZ[A^{\pm1}, B^{\pm 1}, a^{\pm 1}, (A-B)^{\pm 1}]$ can be computed in the same way as the polynomial $\textbf{P}_L$ introduced in Section~\ref{sec:model}. Specifically, if $G$ is a knotted RE 3-graph, we slide edges under/over crossings so that the new isotopic version of the graph, call it $G$ as well, has only crossings involving standard edges. Then we resolve each crossing according to the rules in Figure~\ref{fig:replacing crossings}. After this operation, $G$ is written as a finite formal linear combination of its associated states (where each state $\Gamma$ is a planar trivalent graph) whose coefficients are monomials in $A$ and $B$. Then the graph polynomial $[G]$ is computed by replacing each state $\Gamma$ in this summation with its polynomial $P(\Gamma)$, evaluated using the graph skein relations~(\ref{gi2}) - (\ref{gi6}). That is,
\begin{eqnarray}\label{formula for [G]} 
 [G] = \sum_{\text{states} \, \Gamma}A^{\alpha(\Gamma)}B^{\beta(\Gamma)}P(\Gamma),
 \end{eqnarray}
where the integers $\alpha(\Gamma)$ and $\beta(\Gamma)$ are determined by the rules in Figure~\ref{fig:replacing crossings}. In particular, when restricted to a link $L$, the polynomial $[L]$ is the Kauffman polynomial $D_L$ with variables $A - B$ and $a$.
 \end{remark}
 
 If $G$ is knotted RE 3-graph, its \textit{mirror image} $G^*$ is a knotted RE 3-graph obtained by replacing each over-crossing with an under-crossing and vice versa. $G$ is called \textit{achiral} if it is RE ambient isotopic to its mirror image. Otherwise, $G$ is called \textit{chiral}. 
 
 The integers $\alpha(\Gamma)$ and $\beta(\Gamma)$ in equation~(\ref{formula for [G]}) represent the numbers of ``A-smoothings" and  ``B-smoothings'', respectively, in the state $\Gamma$ of the graph $G$. Then the skein relations
 
\begin{eqnarray}
 \left[ \,\raisebox{-5pt}{\includegraphics[height=0.2in]{poscrossing}}\,\right] = A \left[ \, \raisebox{-5pt}{\includegraphics[height=0.2in]{A-smoothing}}\, \right] + B \left[ \, \raisebox{-5pt}{\includegraphics[height=0.2in]{B-smoothing}}\, \right] +  \left[ \, \raisebox{-5pt}{\includegraphics[height=0.2in]{move5-3}} \, \right] \label{pos crossing}\\
\left[ \,\raisebox{-5pt}{\includegraphics[height=0.2in]{negcrossing}}\, \right] = A \left[ \, \raisebox{-5pt}{\includegraphics[height=0.2in]{B-smoothing}}\, \right] + B \left[\, \raisebox{-5pt}{\includegraphics[height=0.2in]{A-smoothing}}\, \right] + \left[\, \raisebox{-5pt}{\includegraphics[height=0.2in]{move5-3}}\, \right] \label{neg crossing},
\end{eqnarray}
together with the fact that the coefficients $\alpha, \beta, \gamma$ and $\delta$, appearing in the polynomial evaluation $P(\Gamma)$ of the state $\Gamma$, are invariant under the substitutions $A \longleftrightarrow B$ and $a \longleftrightarrow a^{-1}$, imply that the following result holds.

\begin{proposition}
The invariant $[G^*]$ of the mirror image $G^*$ is obtained from $[G]$ by interchanging $A$ and $B$ and replacing $a$ with $a^{-1}$.
\end{proposition}

\begin{remark} If $[G]$ changes by interchanging $A$ and $B$ and replacing $a$ with $a^{-1}$, then $G$ is a chiral RE 3-graph. Therefore the invariant can detect chiral RE 3-graphs (however, it does not detect all chiral RE 3-graphs).
\end{remark}
 
 A graph diagram is a \textit{connected sum} if it is displayed as two disjoint graph diagrams connected by two parallel embedded arcs (up to planar isotopy). We write $G'\, \sharp\, G''$ for the connected sum of graph diagrams $G'$ and $G''$. Moreover, we write $G' \cup G''$ for the disjoint union of $G'$ and $G''$.
 
 \begin{proposition}
 The following formulas hold for the graph polynomial with respect to disjoint union and connected sum:
 \[ [G'\, \sharp\, G''] = [G'] [G''], \qquad [G' \cup G''] =\alpha [G'] [G''], \]
 where $\alpha = \displaystyle\frac{a-a^{-1}}{A-B} + 1$.
 \end{proposition}

\begin{proof}
The proof is done by induction on the number of crossings in $G''$. Since we can slide edges under/over a crossing, we may assume that the crossings in $G''$ involve only standard edges. 

If $G''$ has no crossings then it is either the unknot or a planar diagram of an RE 3-graph. If $G'' = \bigcirc$ the statement clearly holds, since $G' \, \sharp \, \bigcirc = G', \, [\bigcirc] = 1$ and $[G' \cup \bigcirc] = \alpha [G']$. If $G''$ is a planar diagram of an RE 3-graph, then the statement follows from the graph skein relations.

Assume the statement is true for all graph diagrams $G''$ with less than $n$ crossings. If $G''$ contains $n$ crossings, then choose any of the crossings in $G''$ and use the skein relation~(\ref{pos crossing}) or ~(\ref{neg crossing}).  It is easy to see that the conclusion follows from the induction hypothesis. 
\end{proof}

  \subsection{The $SO(N)$ Kauffman polynomial}\label{ssec:one-variable poly}

The $SO(N)$ Kauffman polynomial is a one-variable specialization of the two-variable Kauffman polynomial, and is related to Chern-Simons gauge theory for  $SO(N)$, in the sense that the expectation value of Wilson loop operators in 3-dimensional $SO(N)$ Chern-Simons gauge theory gives the $SO(N)$ Kauffman polynomial.

The $SO(N)$ Kauffman polynomial, $D_L = D_L(q) \in \bbZ[q^{\pm 1}]$, of an unoriented link $L$ is a regular isotopy invariant of $L$ that takes value 1 for the unknot and satisfies 
\[D_{\,\raisebox{-3pt}{\includegraphics[height=0.15in]{poscrossing}}}\, - D_{\,\raisebox{-3pt}{\includegraphics[height=0.15in]{negcrossing}}}\,= (q-q^{-1}) \left[D_{\,\raisebox{-3pt}{\includegraphics[height=0.15in]{A-smoothing}}}\, - D_{\,\raisebox{-3pt}{\includegraphics[height=0.15in]{B-smoothing}}}\,\right], \quad D_{\,\raisebox{-3pt}{\includegraphics[height=0.15in]{poskink}}}\, = q^{N-1} D_{\,\raisebox{-3pt}{\includegraphics[height=0.15in]{arc}}}\,, \quad D_{\,\raisebox{-3pt}{\includegraphics[height=0.15in]{negkink}}}\, = q^{1-N} D_{\,\raisebox{-3pt}{\includegraphics[height=0.15in]{arc}}}.\]

Thus the one-variable specialization of the Kauffman polynomial corresponds to  $z = q-q^{-1}$ and $a = q^{N-1}$. Then the state model for $D_L(q) \in \bbZ[q^{\pm 1}]$ via planar trivalent graphs, that is, the Laurent polynomial $\textbf{P}_L(q)$, corresponds to the graph skein relations ~(\ref{gi2}) - (\ref{gi6}) with $A = q,\, B = q^{-1},\, a = q^{N-1}$. In particular 
\[\alpha = \frac{q^{N-1} - q^{1-N}}{q-q^{-1}} +1,\qquad \beta = \frac{q^{2-N} - q^{N-2}}{q-q^{-1}} - q-q^{-1},\]
\[ \gamma = \frac{q^{N-3} - q^{3-N}}{q- q^{-1}} + 1, \qquad \delta = \frac{q^{N-4} - q^{4-N}}{q - q^{-1}}.\]

\subsection{The case $N = 2$}\label{ssec:N=2}
It is worth considering the case when $N = 2$, as in this situation, the parameters $\alpha, \beta, \gamma$ and $\delta$ have a very simple form, and consequently, the polynomial of a planar trivalent graph can be easily evaluated.

Setting $N = 2$, we have $A = q, B = q^{-1}$ and $a = q$. In particular,
\[ \alpha = 2,\,\, \beta = - q - q^{-1}, \,\,\gamma = 0,\,\, \delta = - q - q^{-1}, \]
and the graph skein relations take the form shown in Figure~\ref{fig:N=2}.

\begin{figure}[ht]
\[P \left( \Gamma \cup \bigcirc \right) = 2\,P \left( \Gamma \right), \quad
 P \left( \,\raisebox{-6pt}{\includegraphics[height=0.25in]{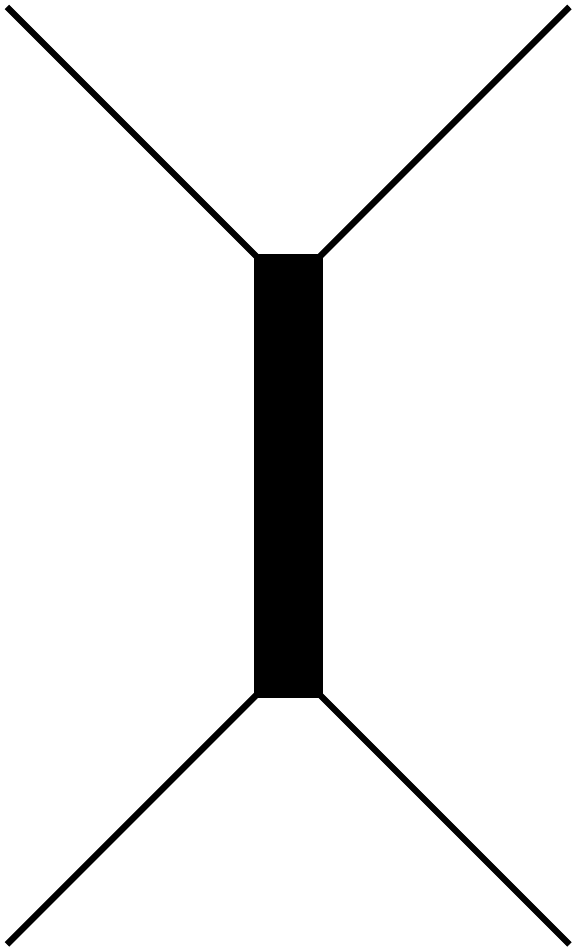}} \, \right) = P \left( \,  \raisebox{-3pt}{\includegraphics[height=0.17in]{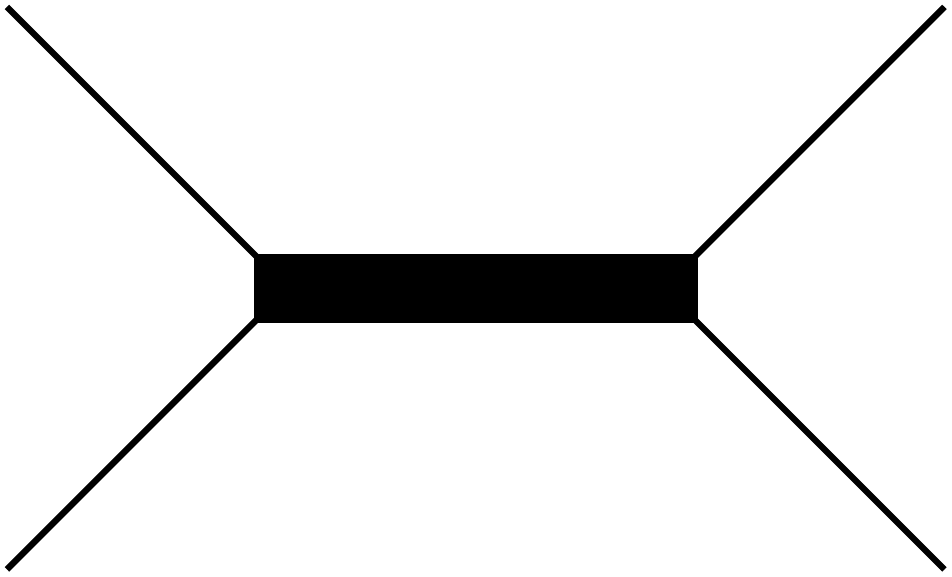}}\,\right), \]
\[ P\left( \,\raisebox{-6pt}{\includegraphics[height=0.25in]{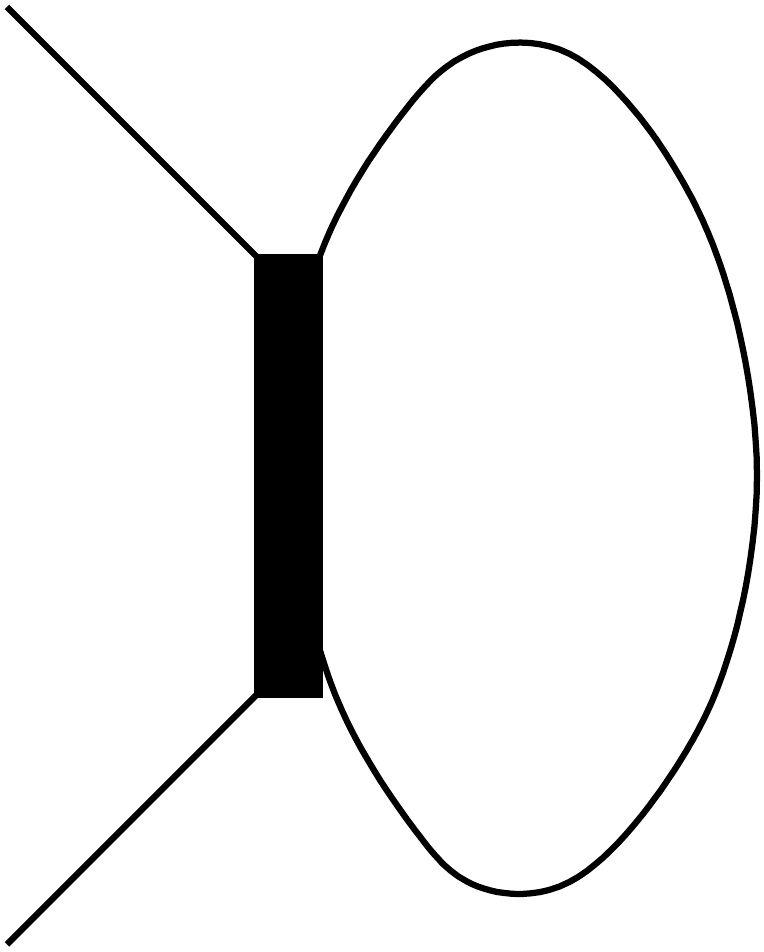}} \, \right) = -(q+q^{-1}) \,P \left(\,\raisebox{-6pt}{\includegraphics[height=0.25in]{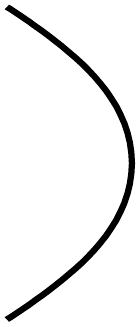}}\,\right), \quad
P \left( \, \raisebox{-13pt}{\includegraphics[height=0.45in]{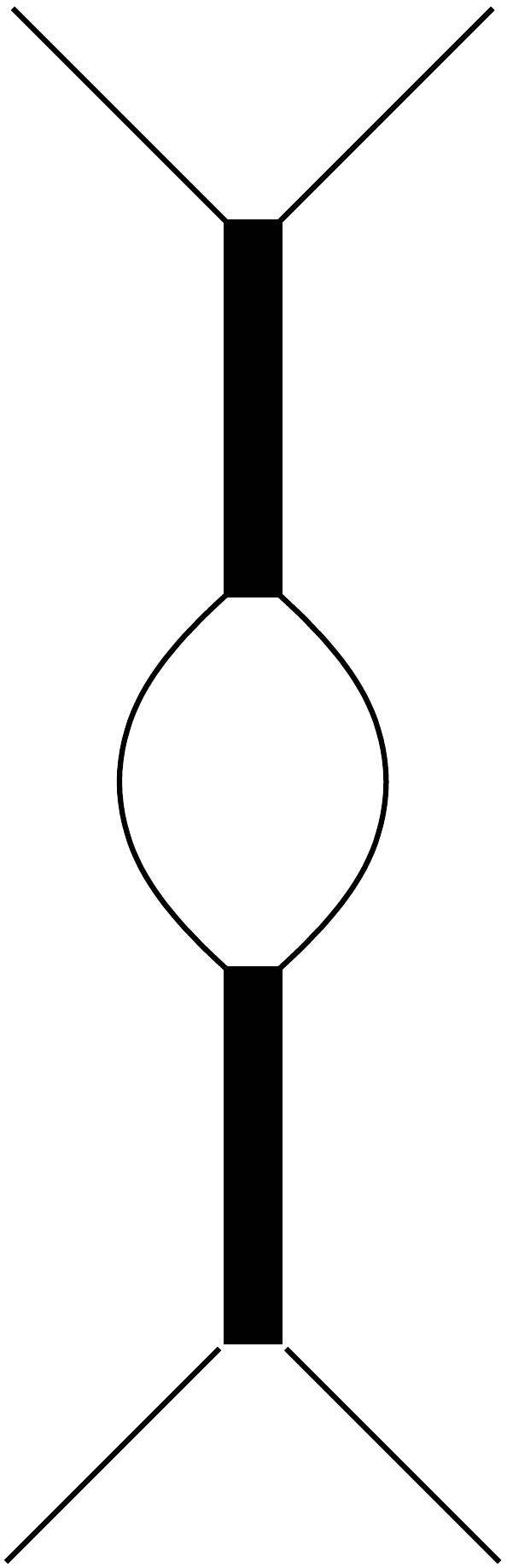}}\, \right) = - (q+q^{-1}) \,P \left(\, \raisebox{-6pt}{\includegraphics[height=0.25in]{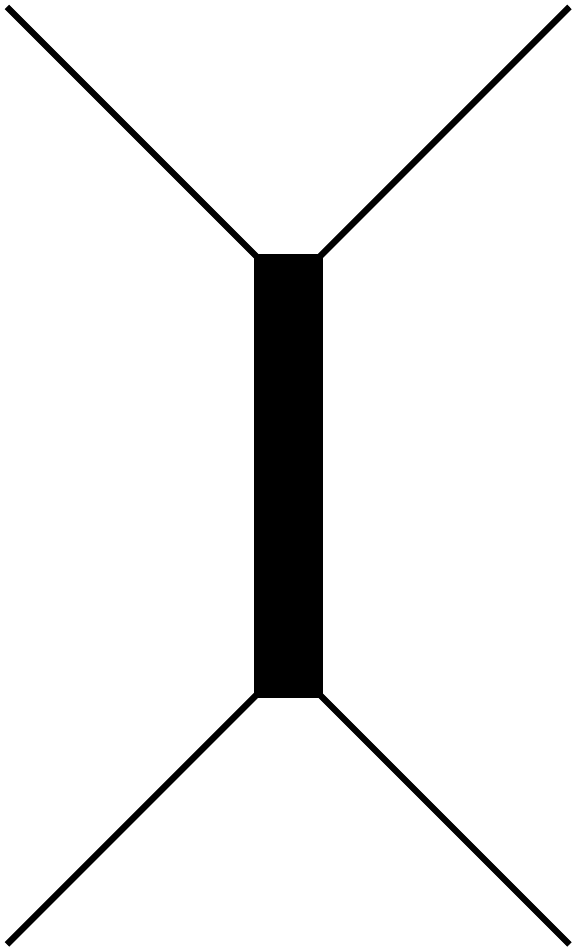}}\,\right), \]
\begin{eqnarray*}
P \left( \,\raisebox{-13pt}{\includegraphics[height=0.45in, width=0.25in]{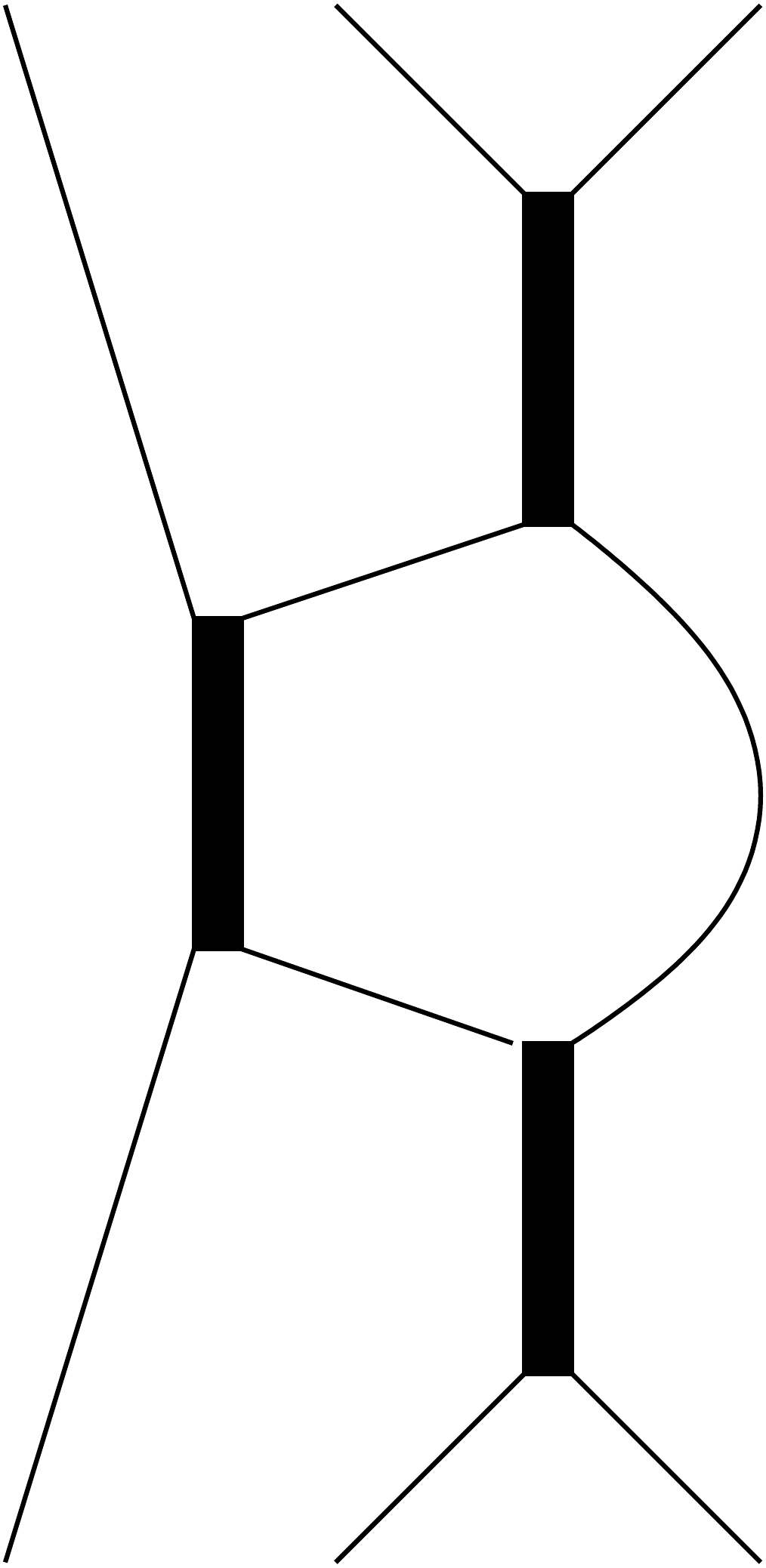}} \,\right) - P \left(\, \raisebox{-13pt}{\includegraphics[height=0.45in, width=0.25in]{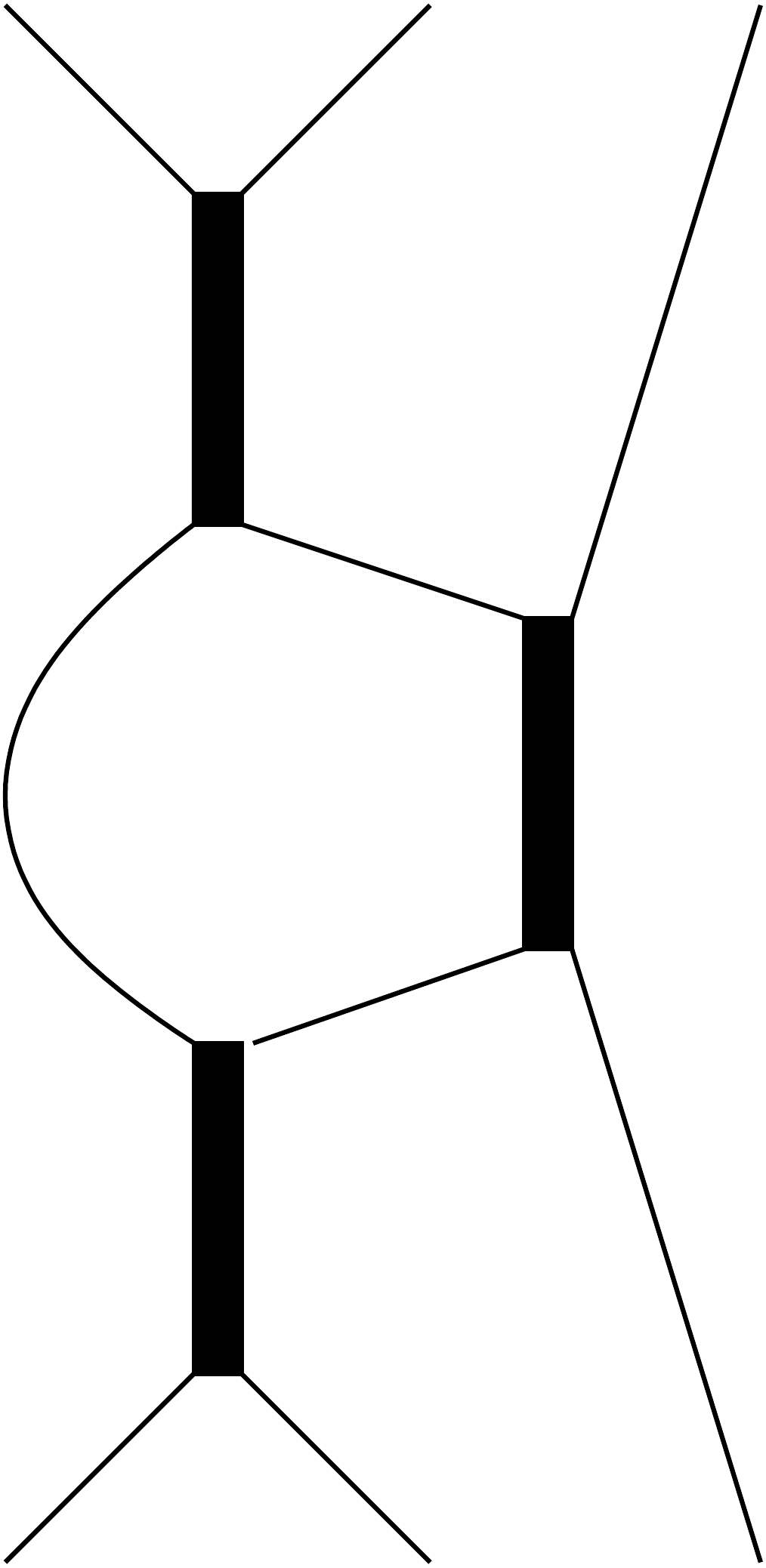}}\, \right) =  P \left( \,\raisebox{-6pt}{\includegraphics[height=0.25in]{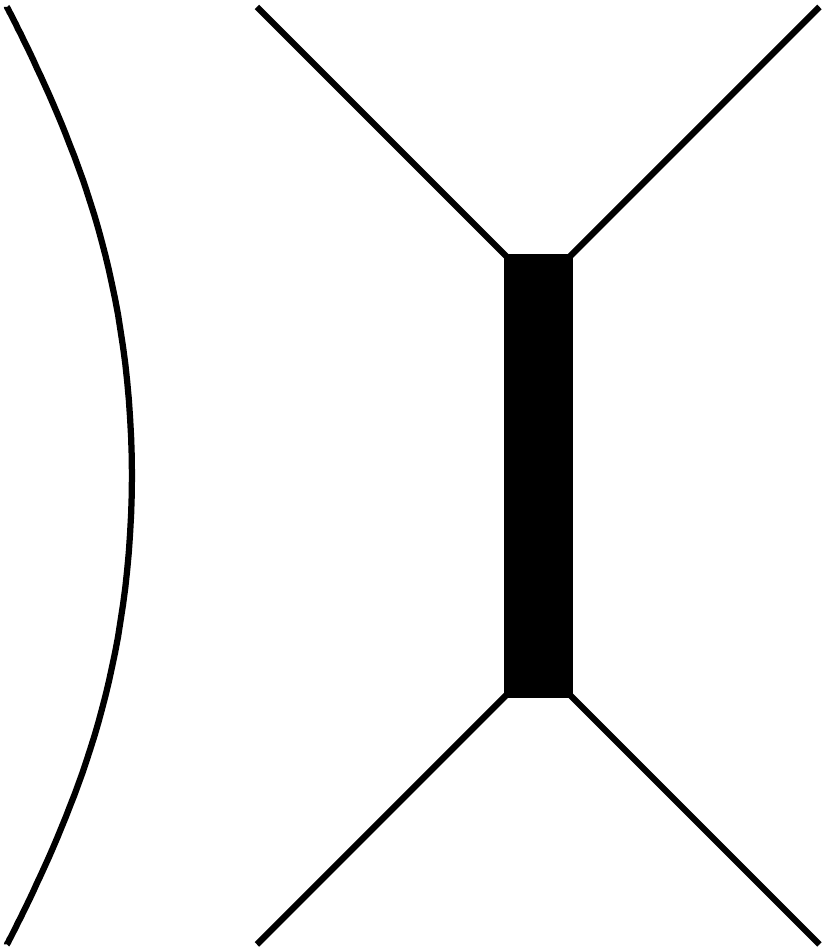}}\,\right) - P \left(\, \raisebox{-6pt}{\includegraphics[height=0.25in]{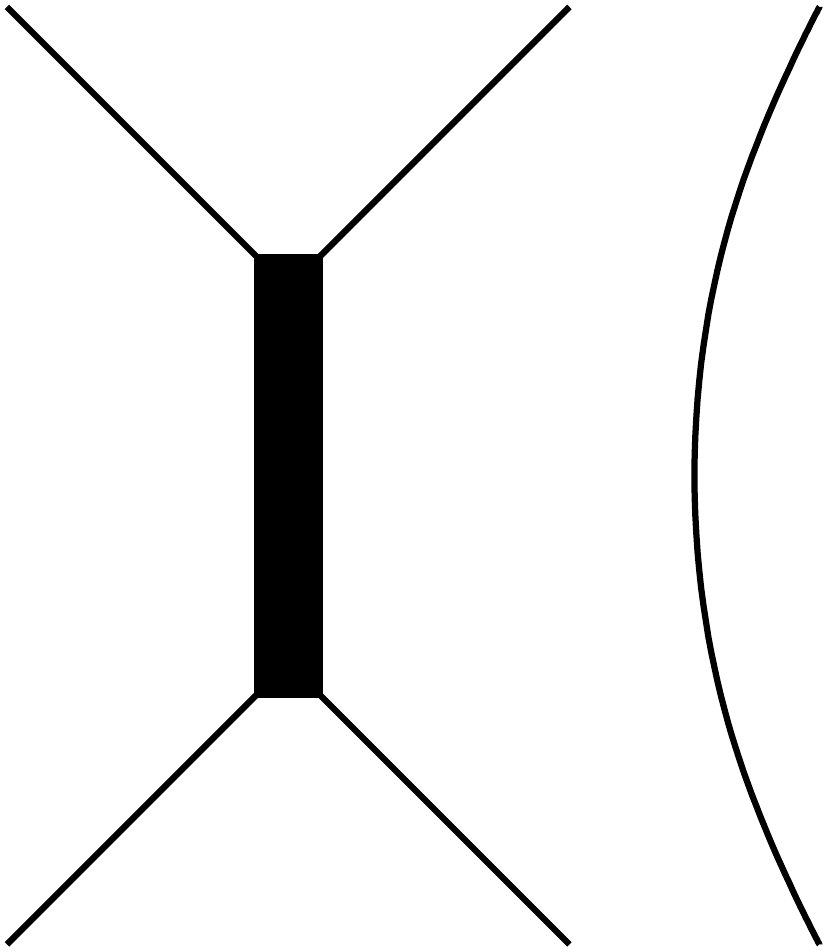}}\,\right)
 + P \left( \, \raisebox{-6pt}{\includegraphics[height=0.25in, width=0.25in]{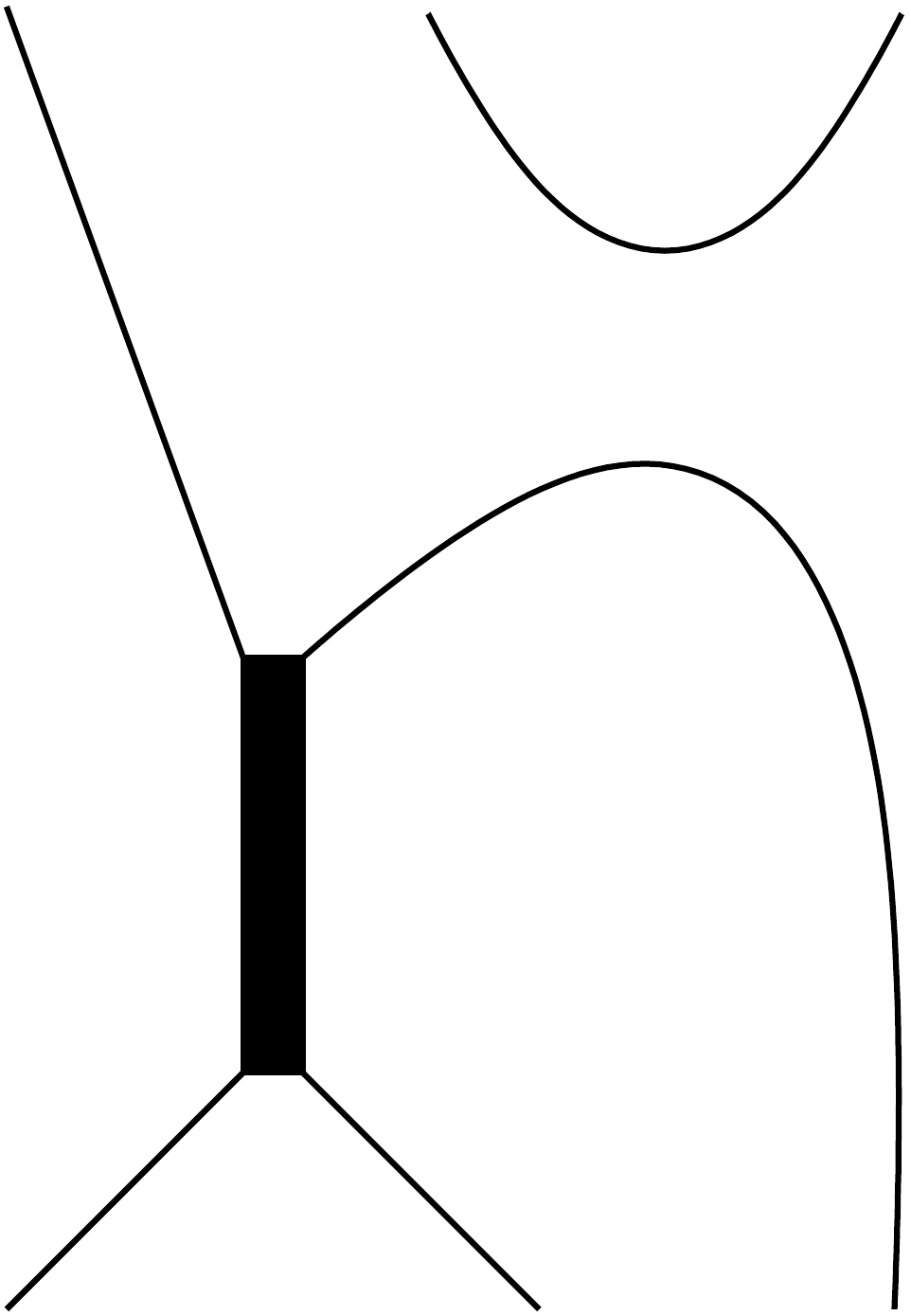}} \,\right) - P \left( \, \raisebox{-6pt}{\includegraphics[height=0.25in , width=0.25in]{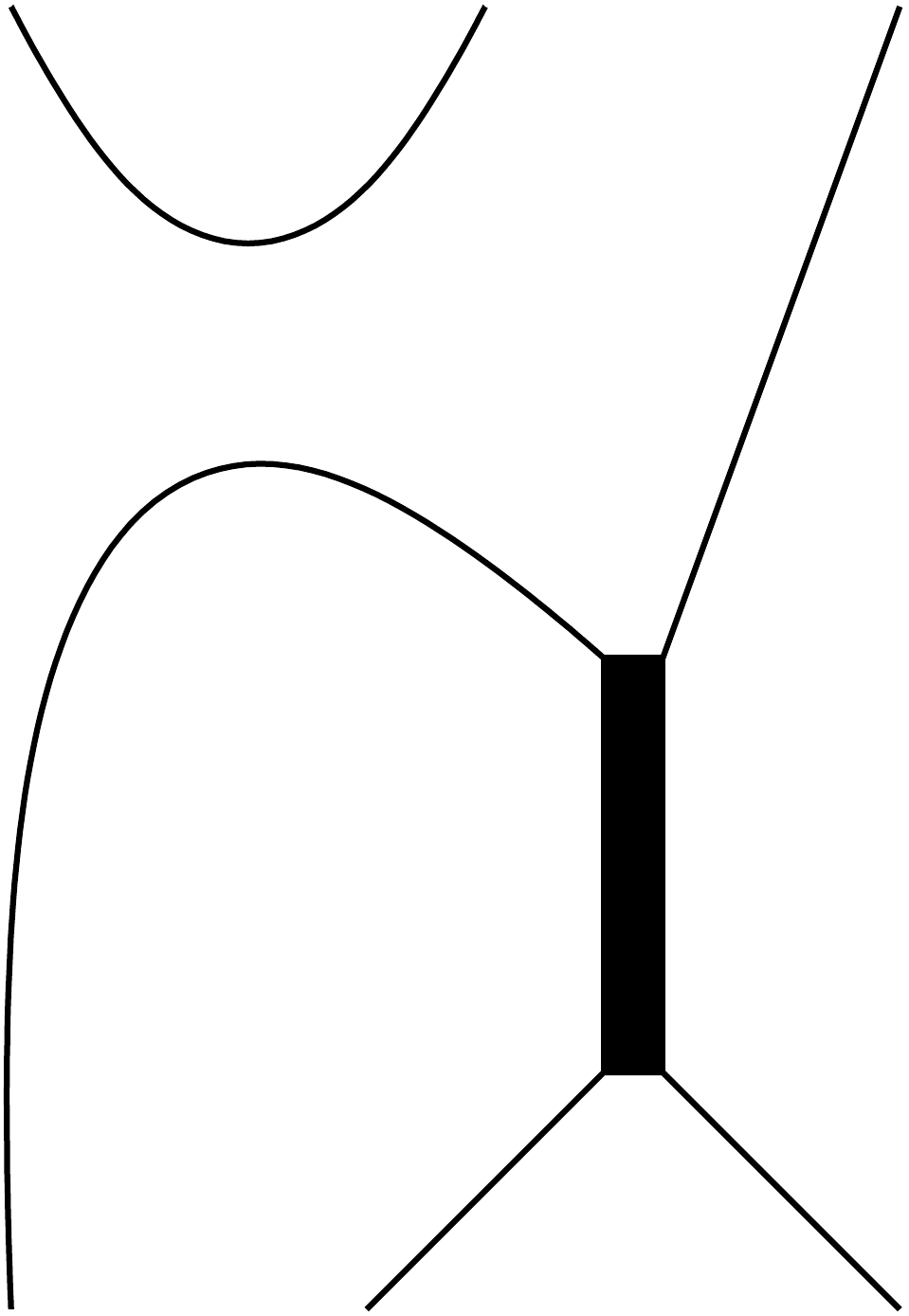}}\,\right)\\ + P \left(\, \raisebox{-6pt}{\includegraphics[height=0.25in, width=0.25in]{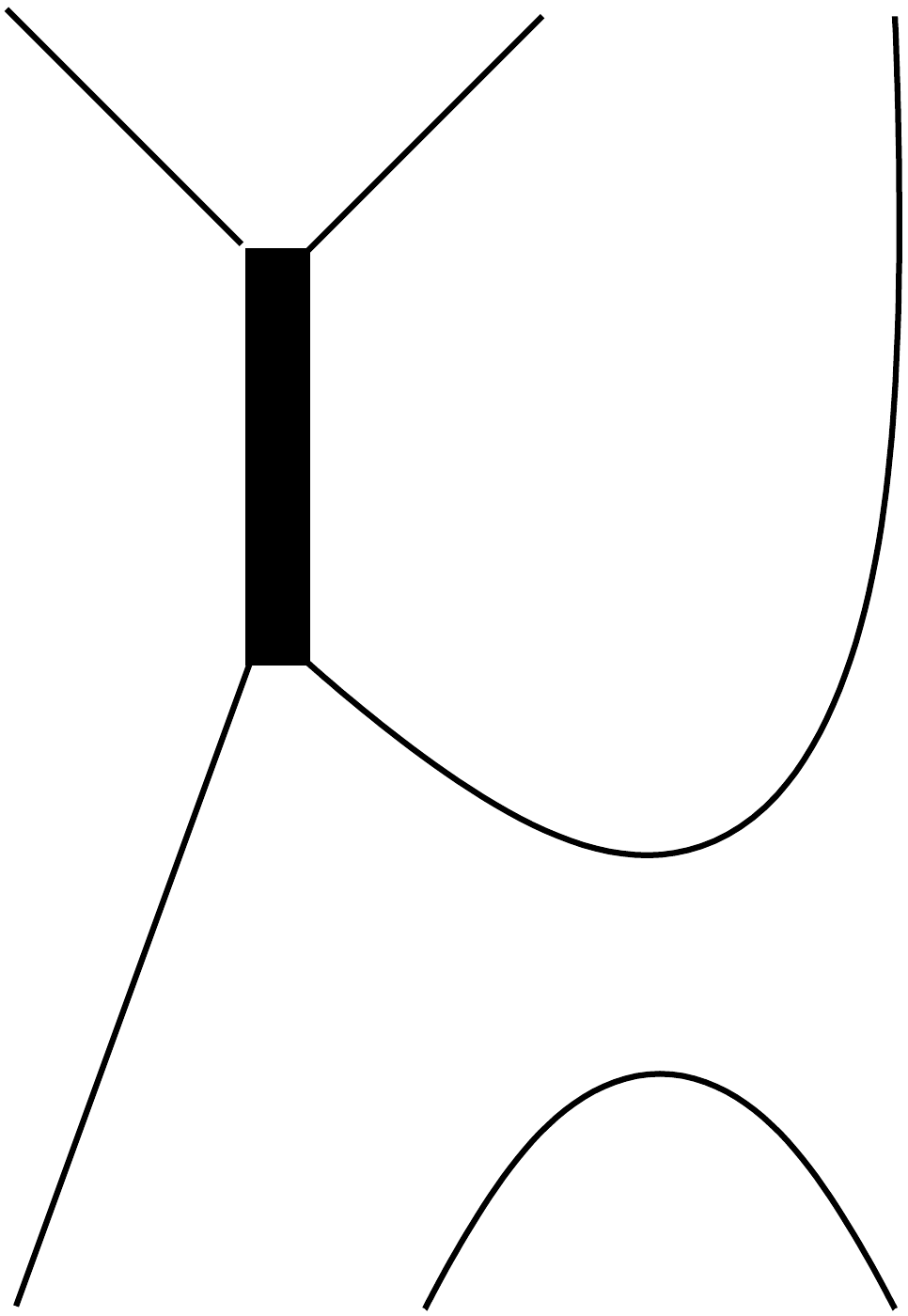}} \,\right) - P \left(\, \raisebox{-6pt}{\includegraphics[height=0.25in, width=0.25in]{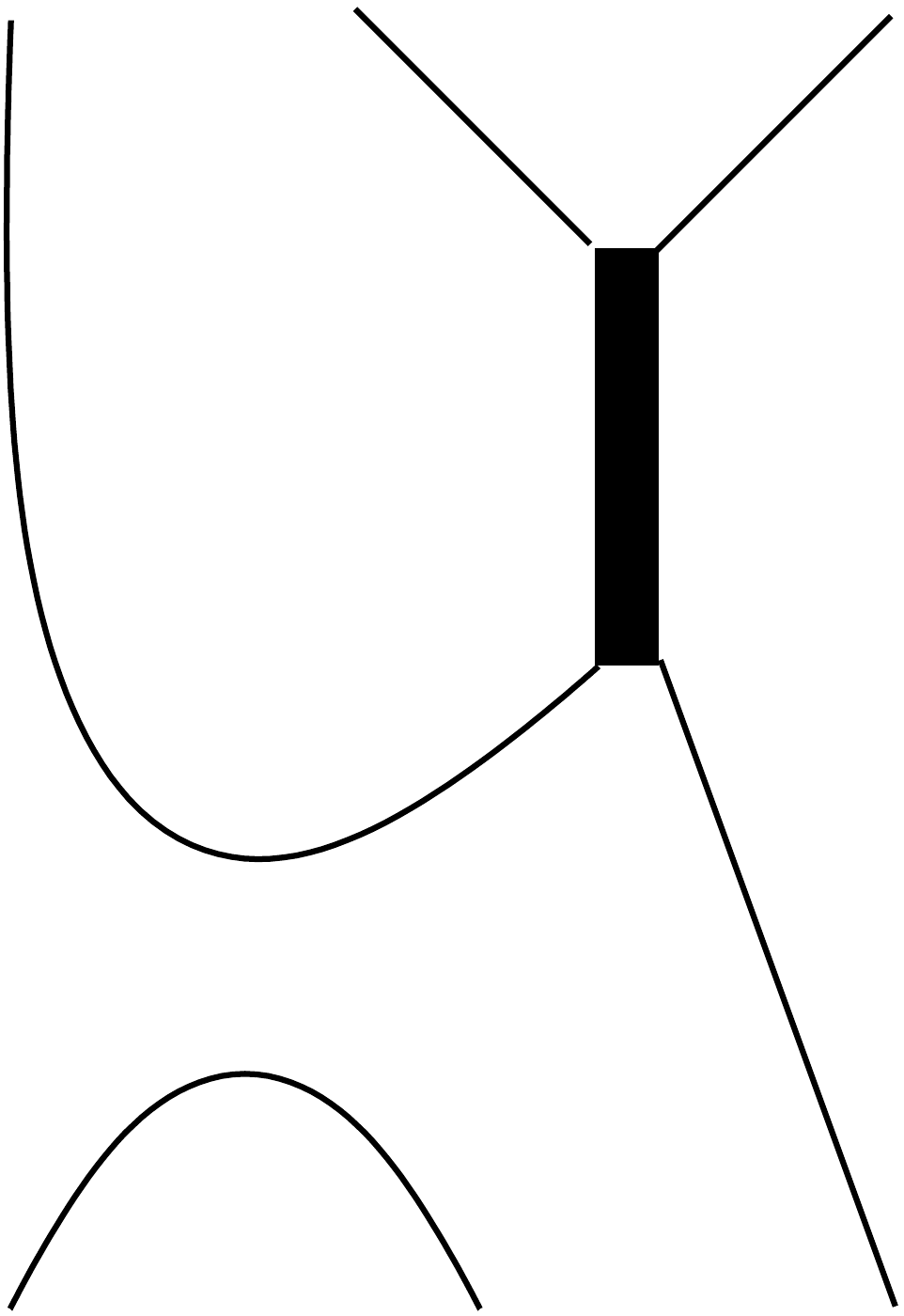}}\,\right) - ( q+q^{-1} ) \, \left [P \left( \,\raisebox{-6pt}{\includegraphics[height=0.25in]{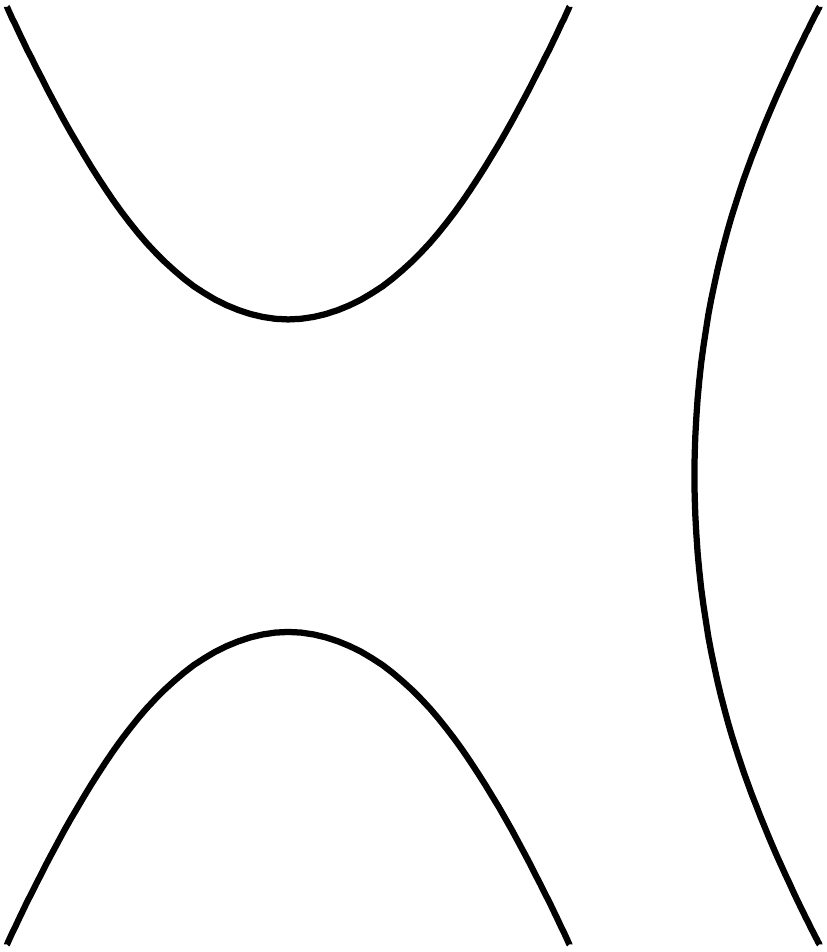}} \, \right) - P \left( \,  \raisebox{-6pt}{\includegraphics[height=0.25in]{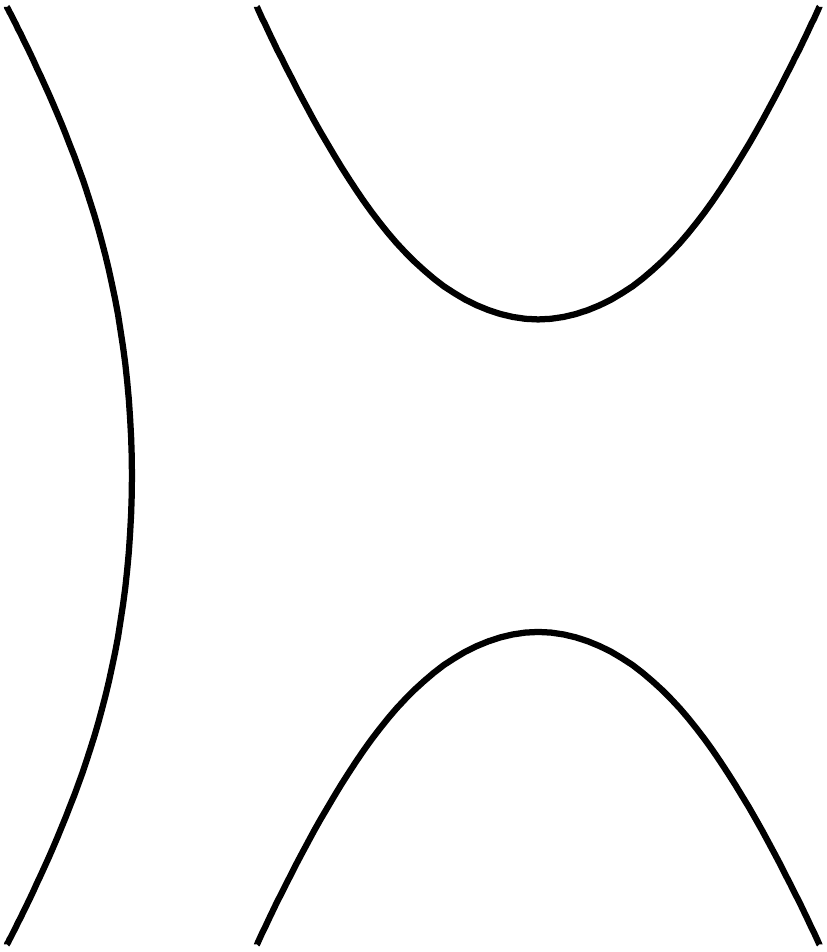}}\,\right)\right ] . 
 \end{eqnarray*}
\caption{Graph skein relations with $A = q, B = q^{-1}$ and $a = q$}\label{fig:N=2}
\end{figure}

\begin{theorem}\label{thm:N=2}
In the case $N = 2$, for any planar trivalent diagram $G$ we have \[P(G) = [G] = 2^{c-1}(-q -q^{-1})^{n/2},\] where $c$ is the number of connected components and $n$ the number of vertices of $G$.
\end{theorem}

\begin{proof} By Theorem~\ref{thm:unique poly}, it suffices to prove that the polynomial $f = 2^{c-1}(-q -q^{-1})^{n/2}$ takes value $1$ for the unknot and satisfies the graph skein relations in Figure~\ref{fig:N=2}. All of these are obviously satisfied, except for the last identity which requires some explanations. 

Since the diagrams in a single graph skein relation are identical outside the region shown, to check that the polynomial satisfies the last skein relation, we need to consider all possible ways to connect the six boundary points outside the region. Label these points 1 to 6 starting with the rightmost point on the top and going counterclockwise. The six points may live in one, two or three distinct connect components, and since the skein relation is invariant under reflections, we only need to consider the following cases: 

\textit{Case 1.} All points $1, 2, 3, 4, 5, 6$ are in the same connected component: $\raisebox{-15pt}{\includegraphics[height=0.5in]{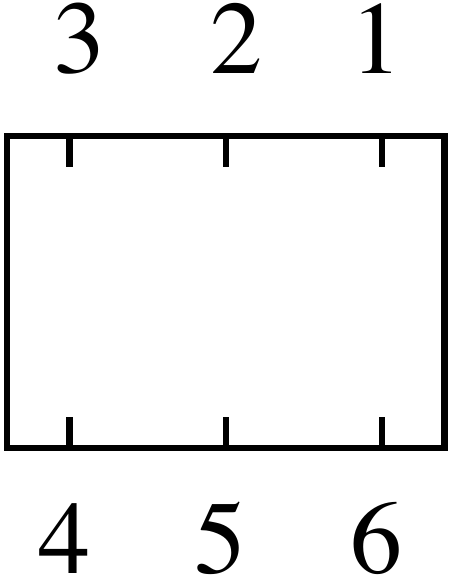}}$\,. The polynomial $f$ depends only on the number of connected components and the number of trivalent vertices of the graph, and therefore we have $f\left(\,\raisebox{-10pt}{\includegraphics[height=0.35in]{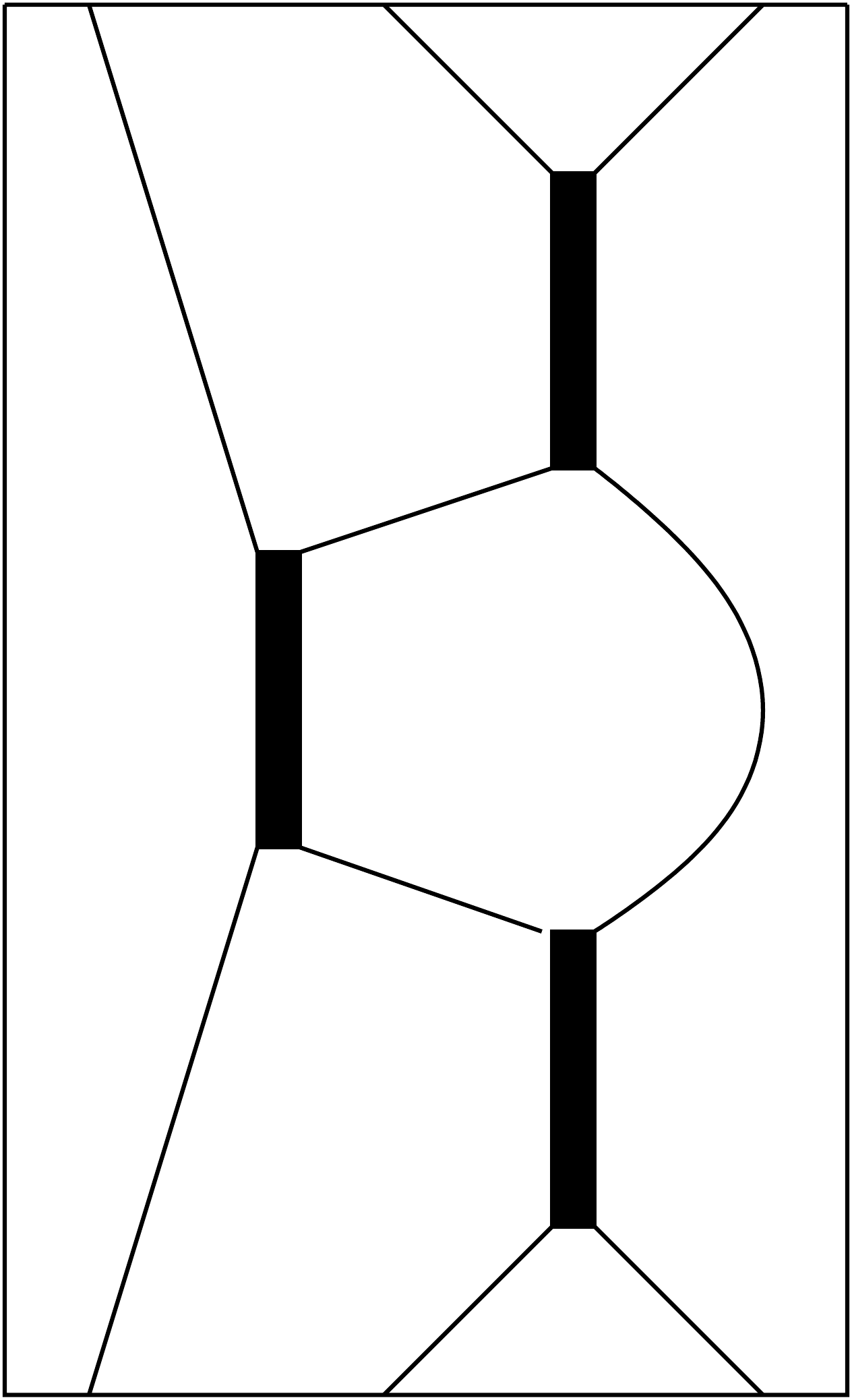}}\,\right) = f\left(\,\raisebox{-10pt}{\includegraphics[height=0.35in]{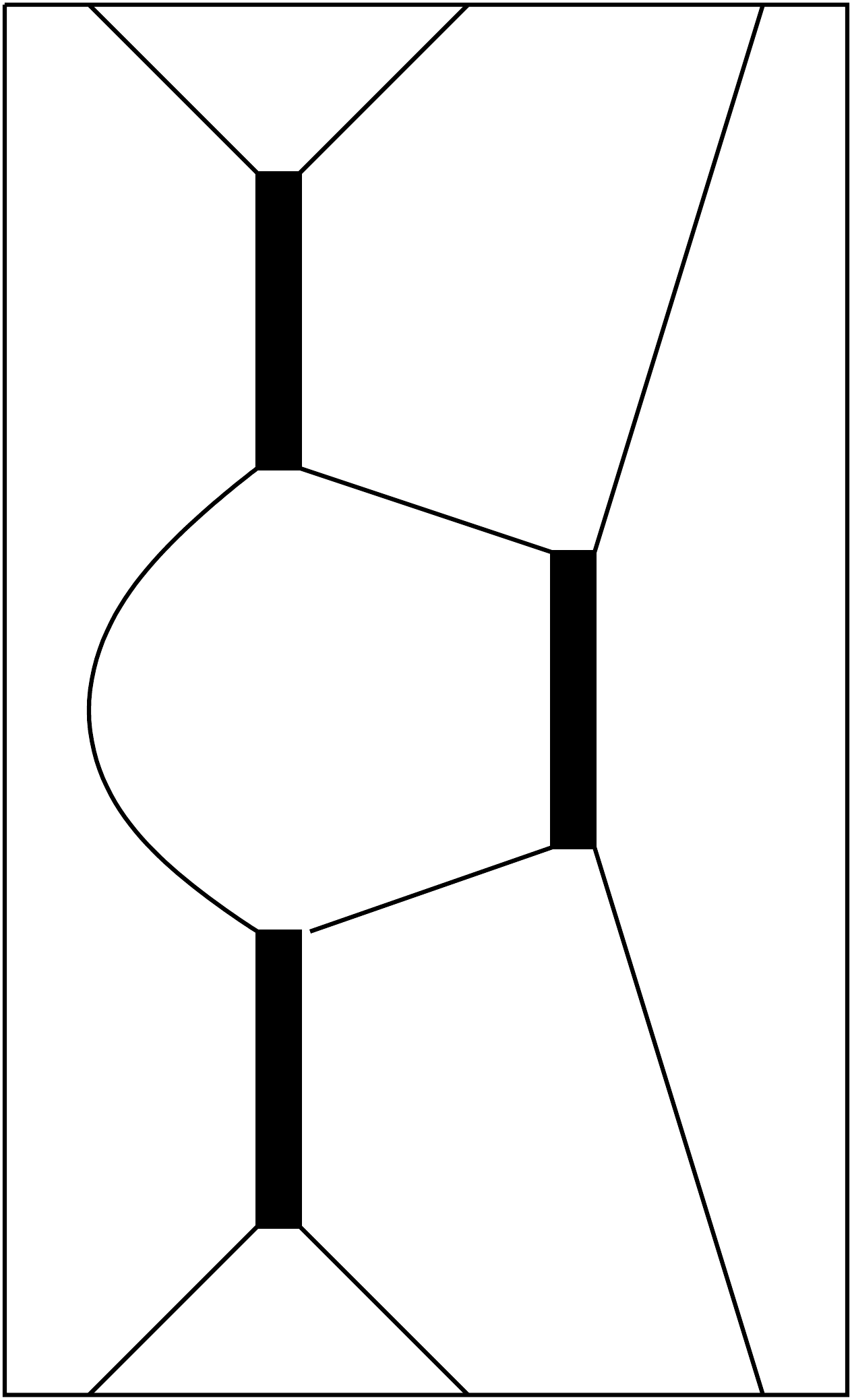}}\, \right), f\left(\,\raisebox{-7pt}{\includegraphics[height=0.25in]{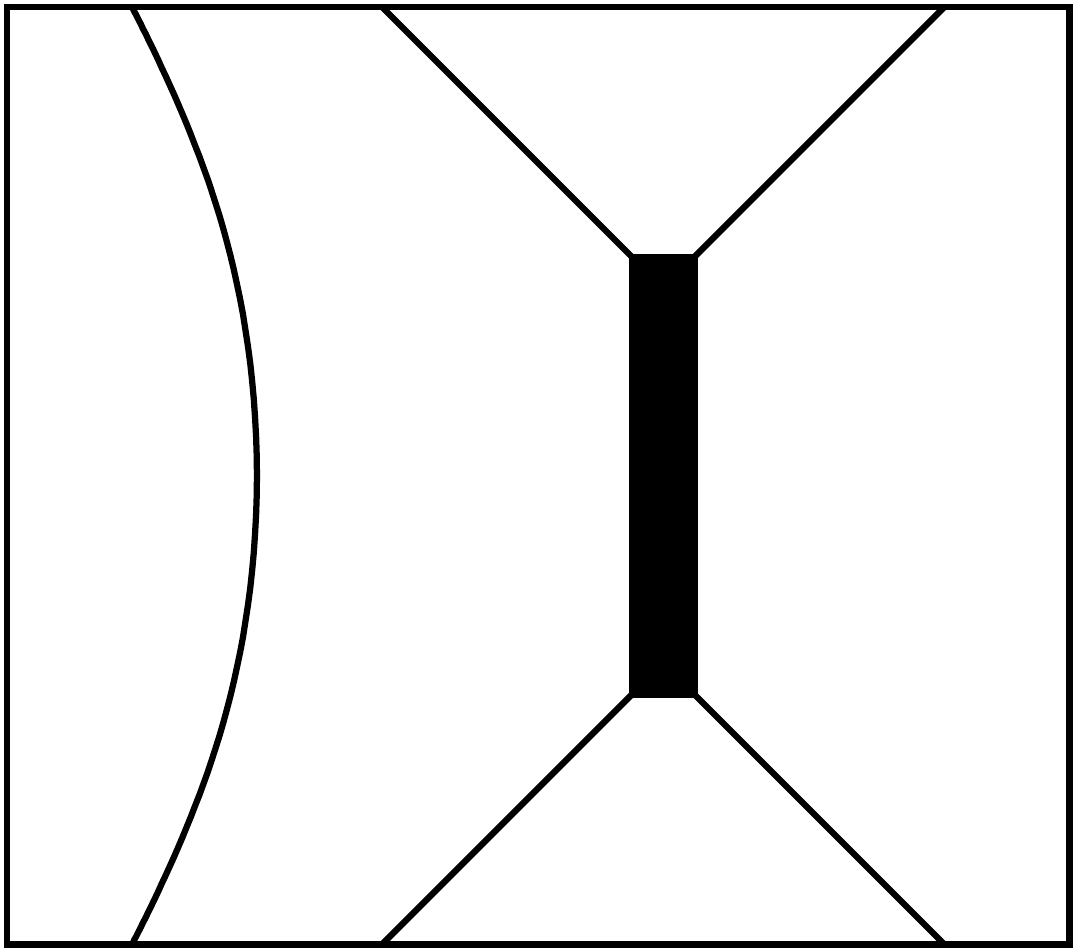}}\,\right) = f\left(\,\raisebox{-7pt}{\includegraphics[height=0.25in]{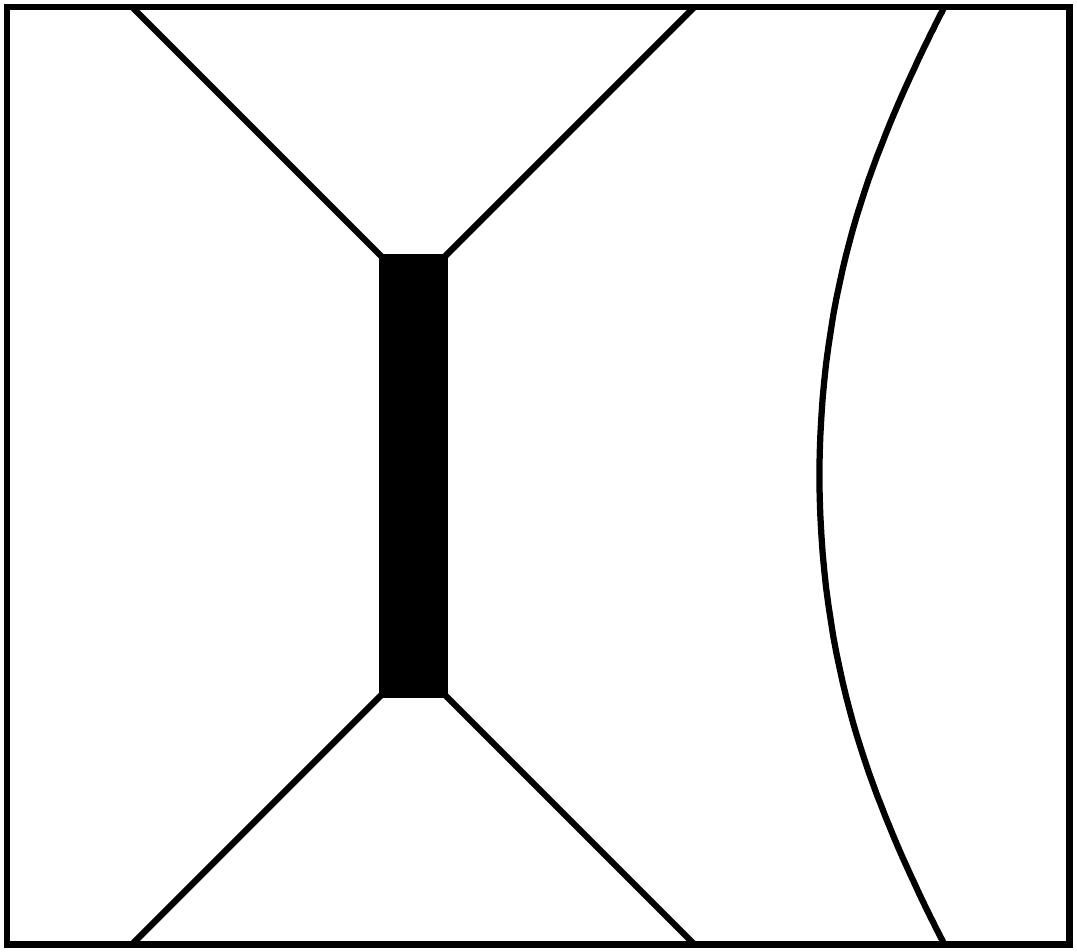}}\, \right), f\left(\,\raisebox{-7pt}{\includegraphics[height=0.25in]{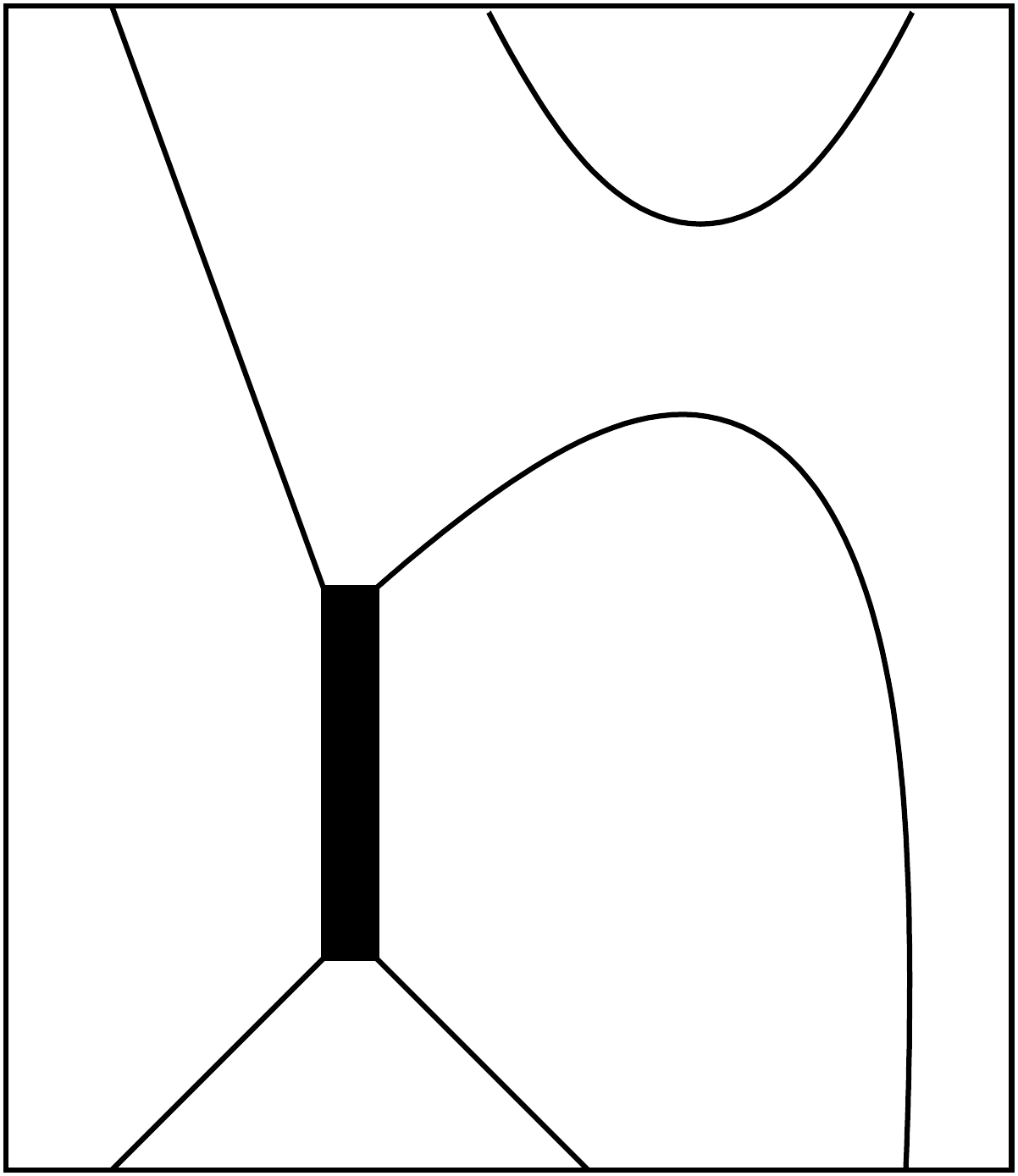}}\,\right) = f\left(\,\raisebox{-7pt}{\includegraphics[height=0.25in]{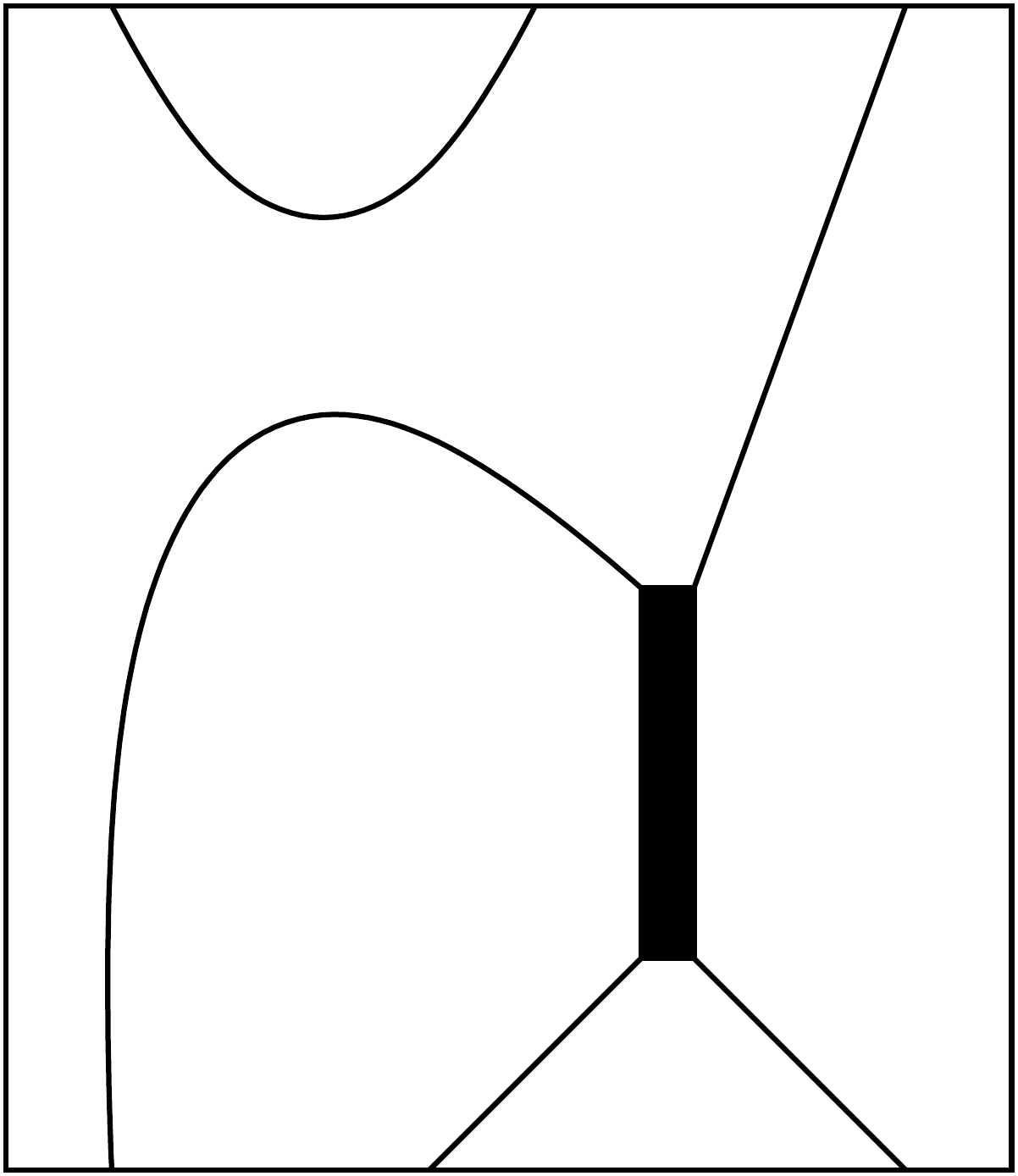}}\, \right), f\left(\,\raisebox{-7pt}{\includegraphics[height=0.25in]{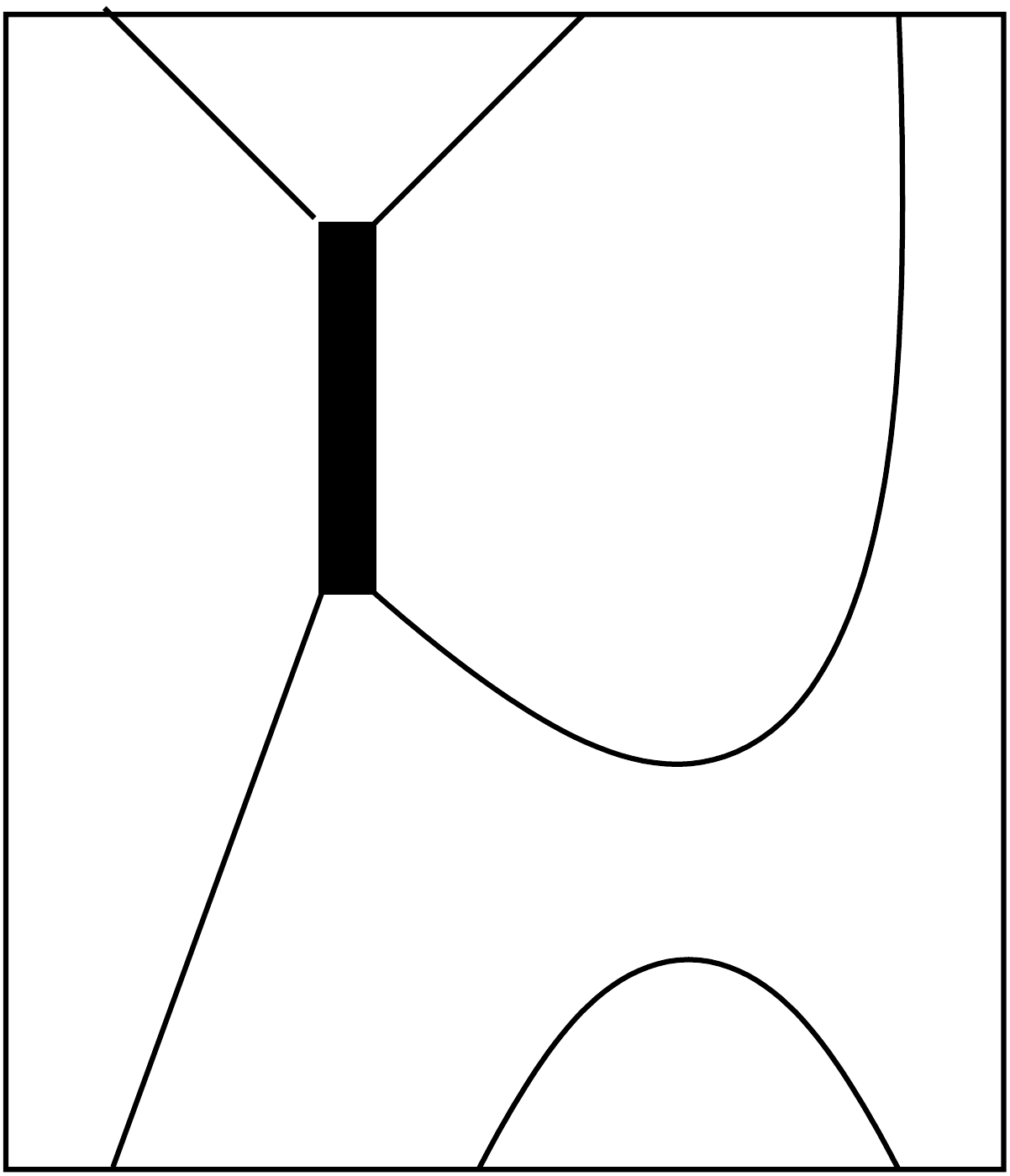}}\,\right) = f\left(\,\raisebox{-7pt}{\includegraphics[height=0.25in]{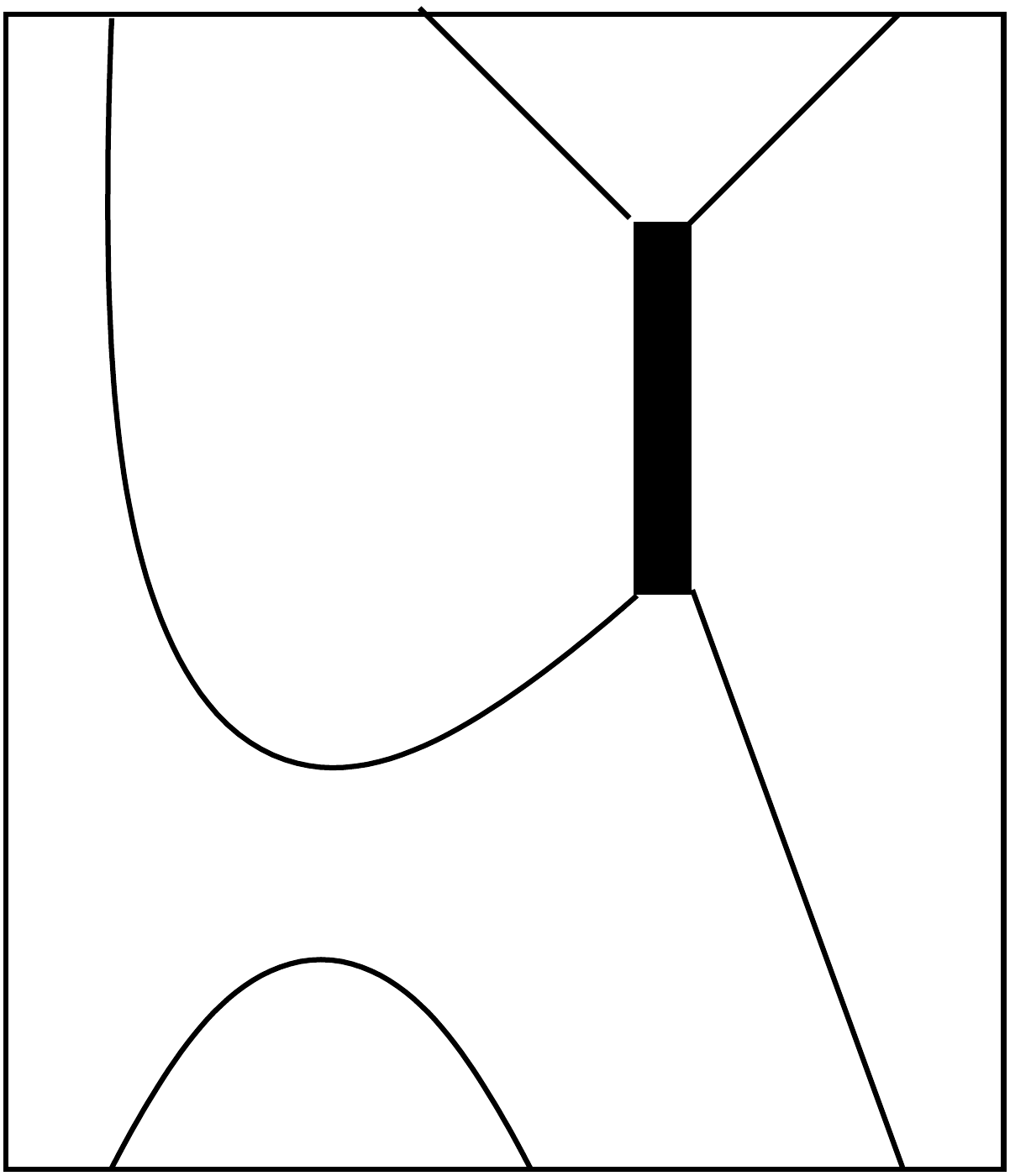}}\, \right)$ and $f\left(\,\raisebox{-7pt}{\includegraphics[height=0.25in]{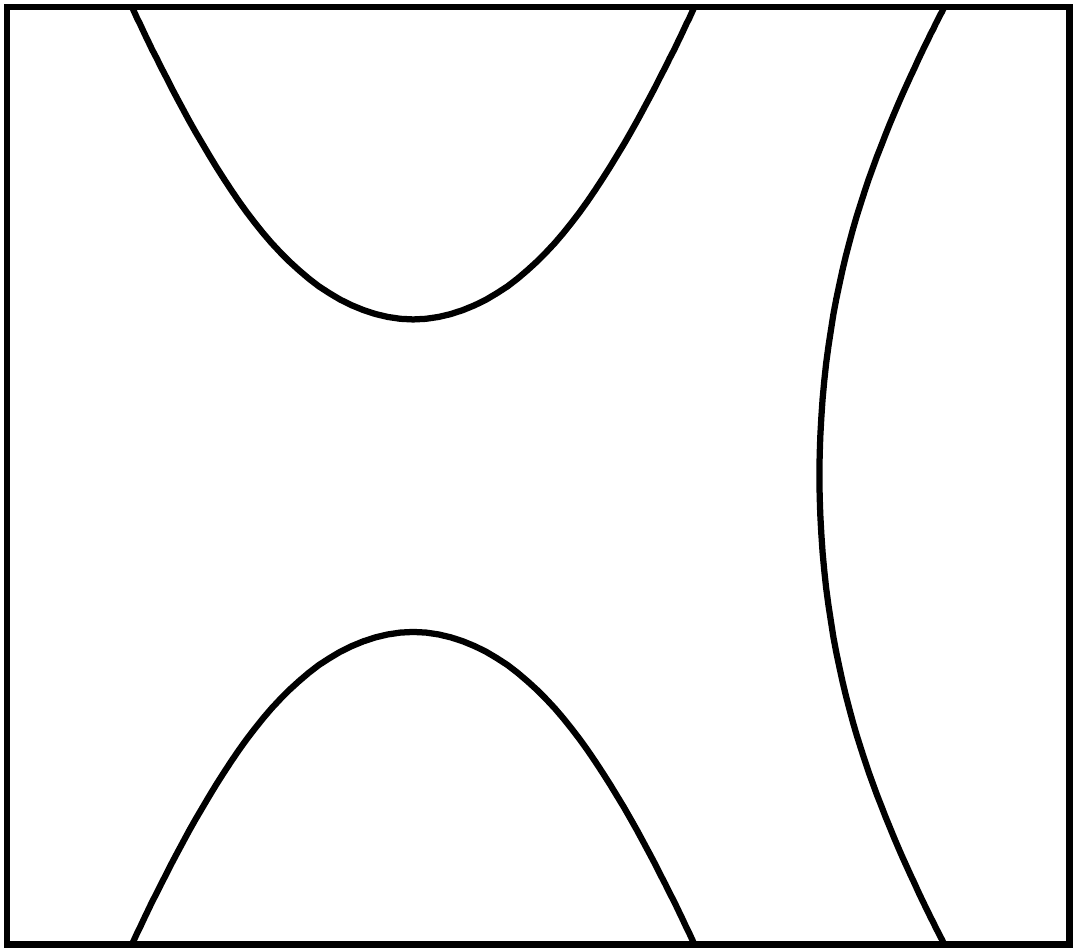}}\,\right) = f\left(\,\raisebox{-7pt}{\includegraphics[height=0.25in]{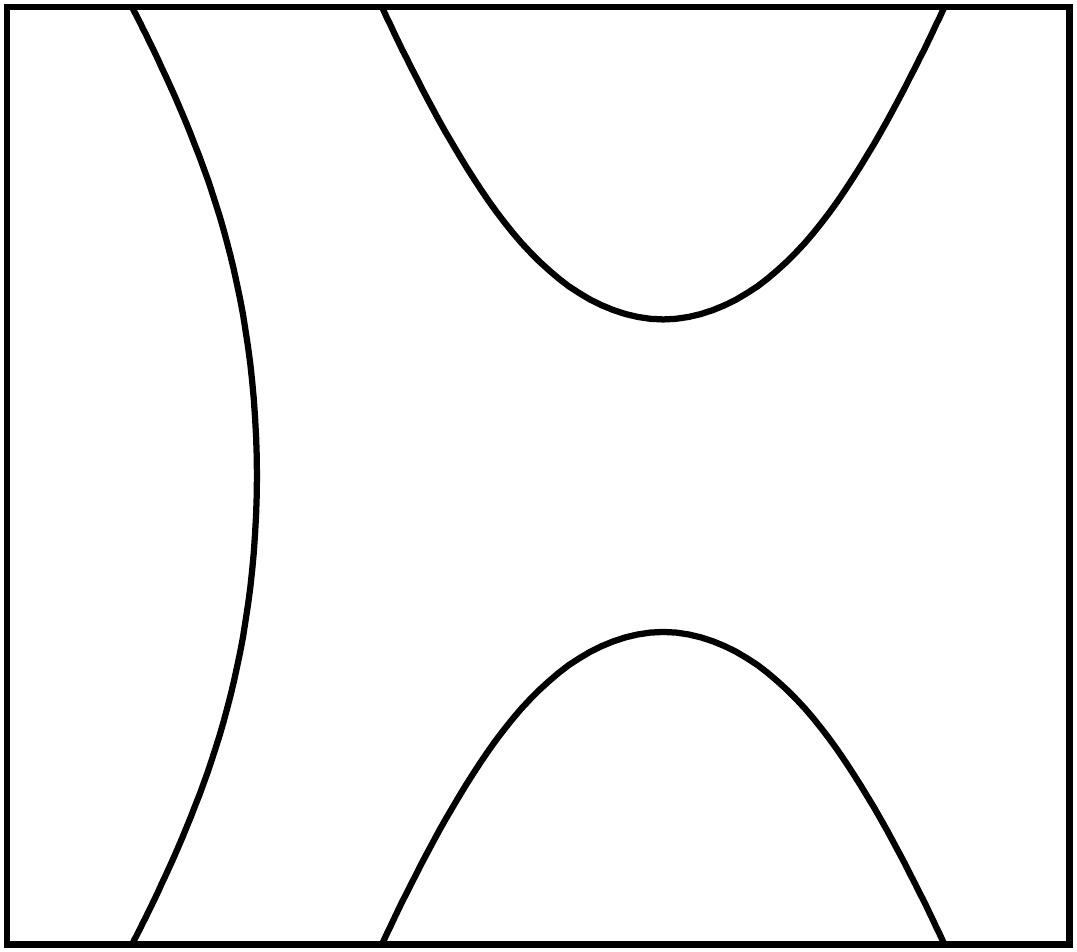}}\, \right)$. Thus the identity is trivially satisfied.

\textit{Case 2.} The points are in two connected components, say: $\raisebox{-15pt}{\includegraphics[height=0.5in]{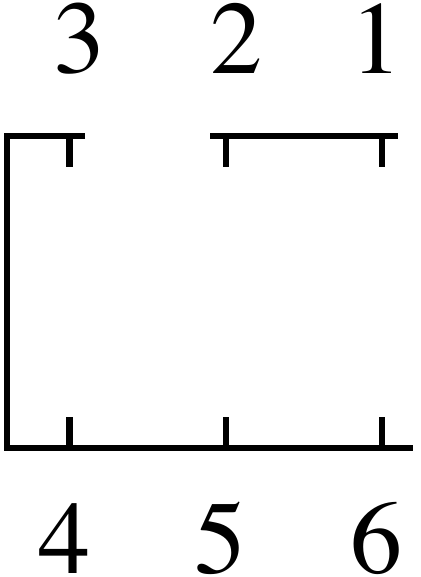}}$\,.  In this case we have $f\left(\,\raisebox{-10pt}{\includegraphics[height=0.35in]{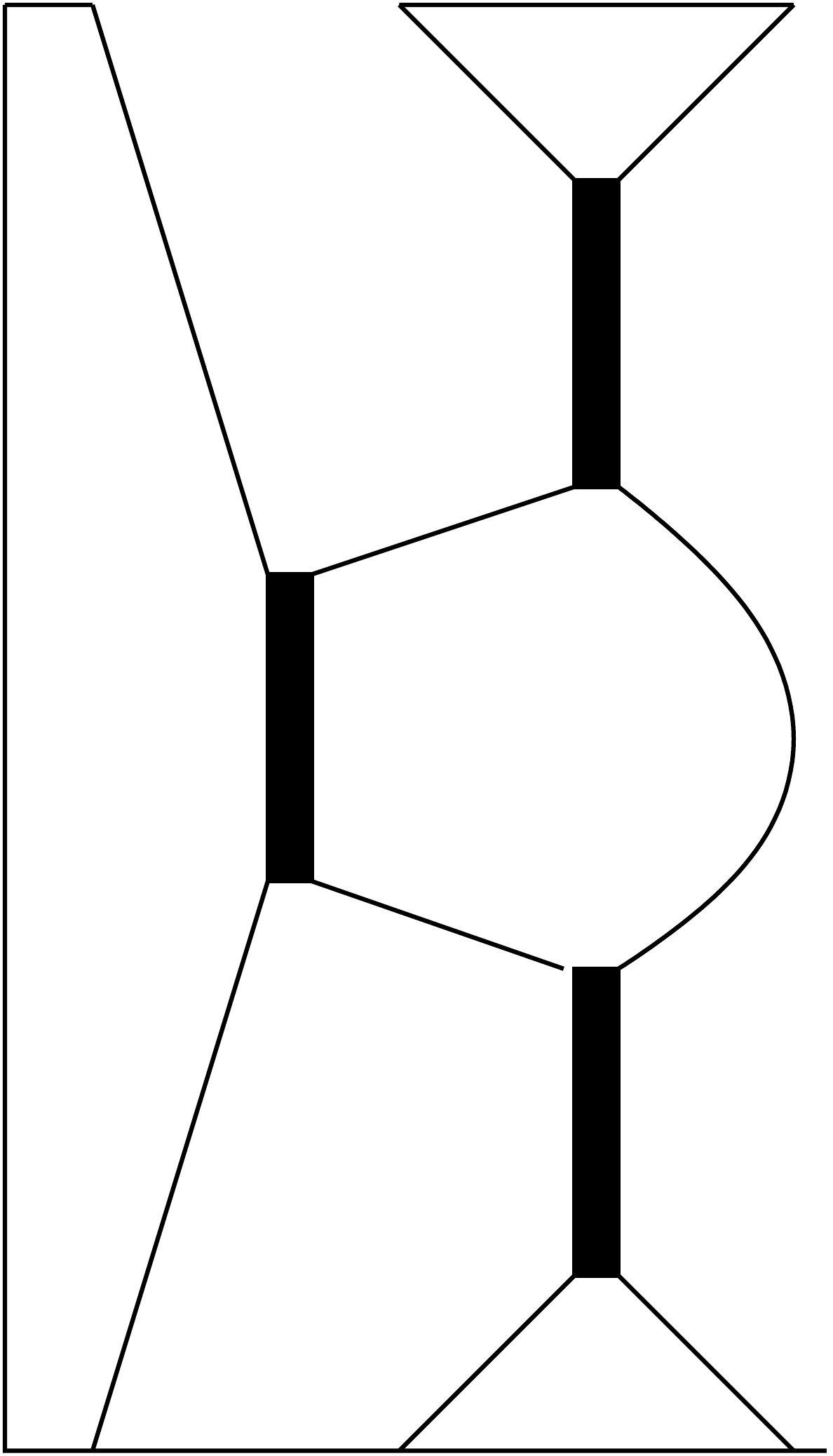}}\,\right) = f\left(\,\raisebox{-10pt}{\includegraphics[height=0.35in]{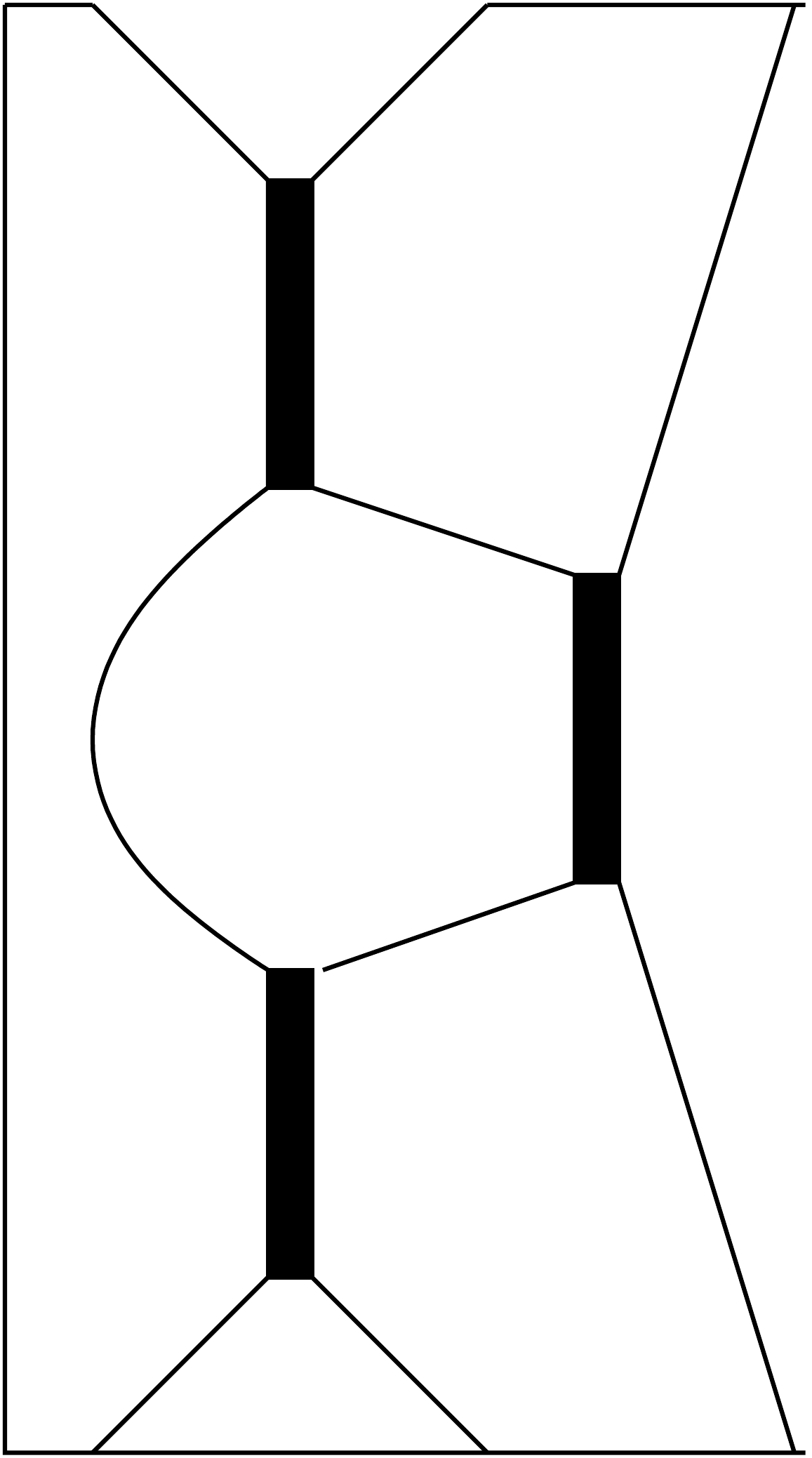}}\, \right), f\left(\,\raisebox{-7pt}{\includegraphics[height=0.25in]{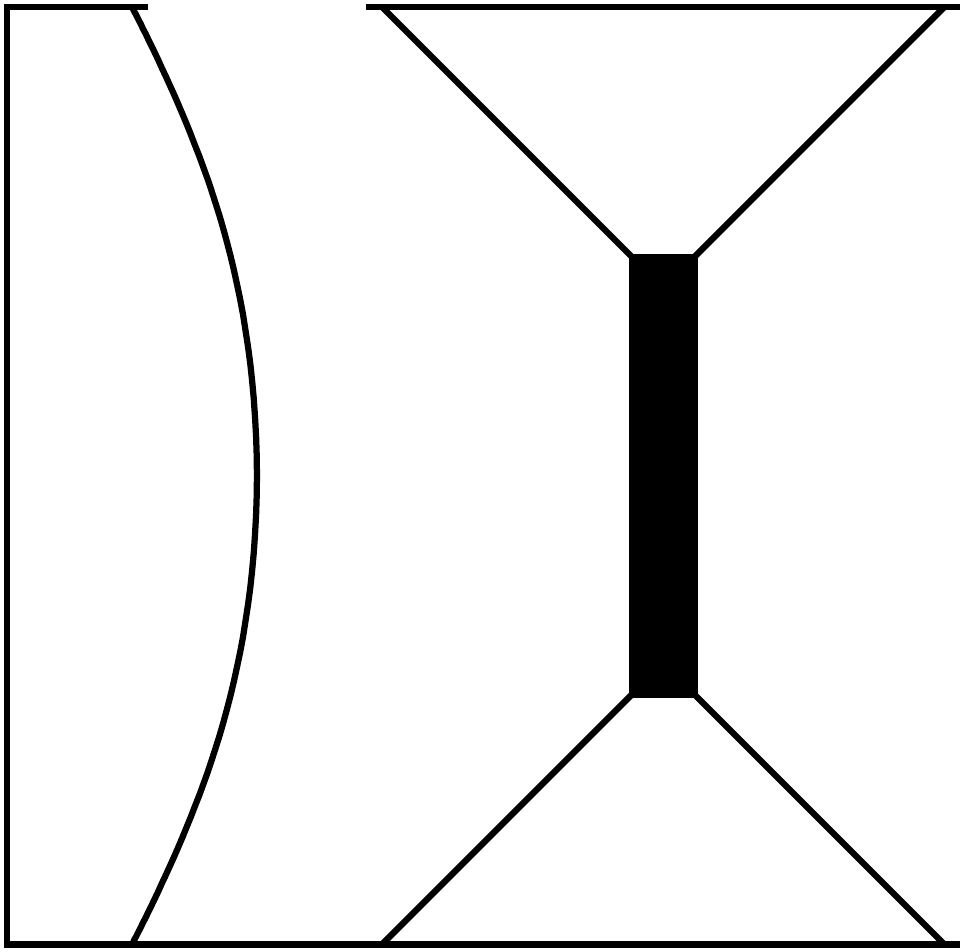}}\,\right) = f\left(\,\raisebox{-7pt}{\includegraphics[height=0.25in]{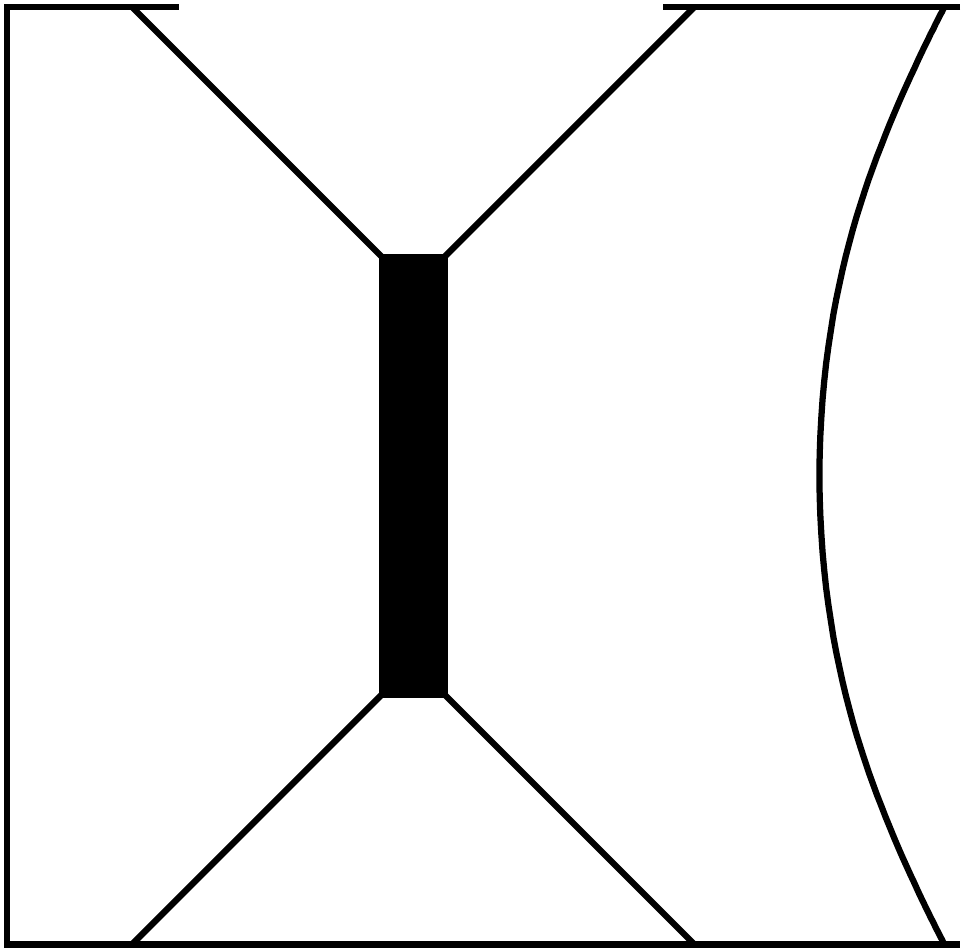}}\, \right), f\left(\,\raisebox{-7pt}{\includegraphics[height=0.25in]{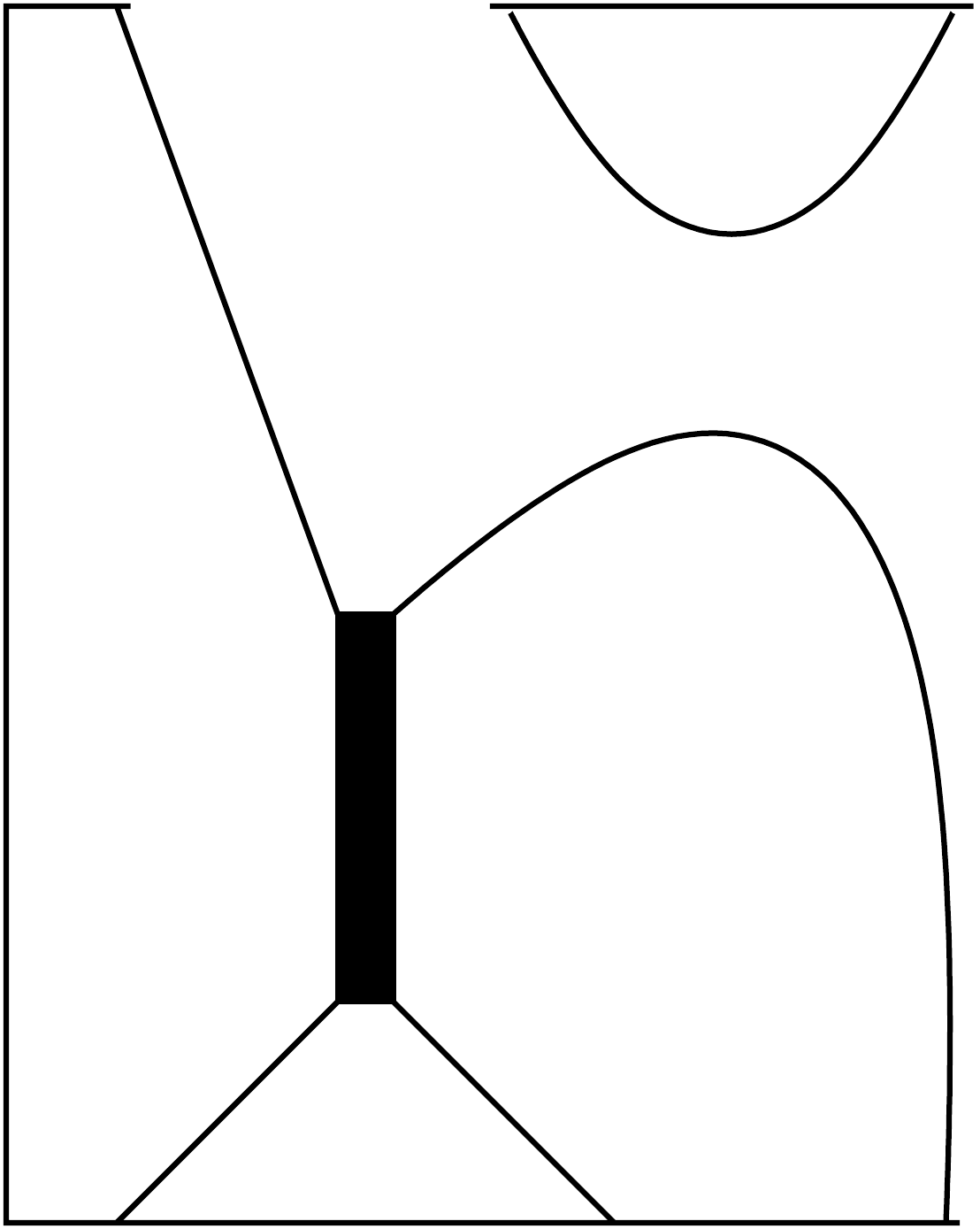}}\,\right) = 2 f\left(\,\raisebox{-7pt}{\includegraphics[height=0.25in]{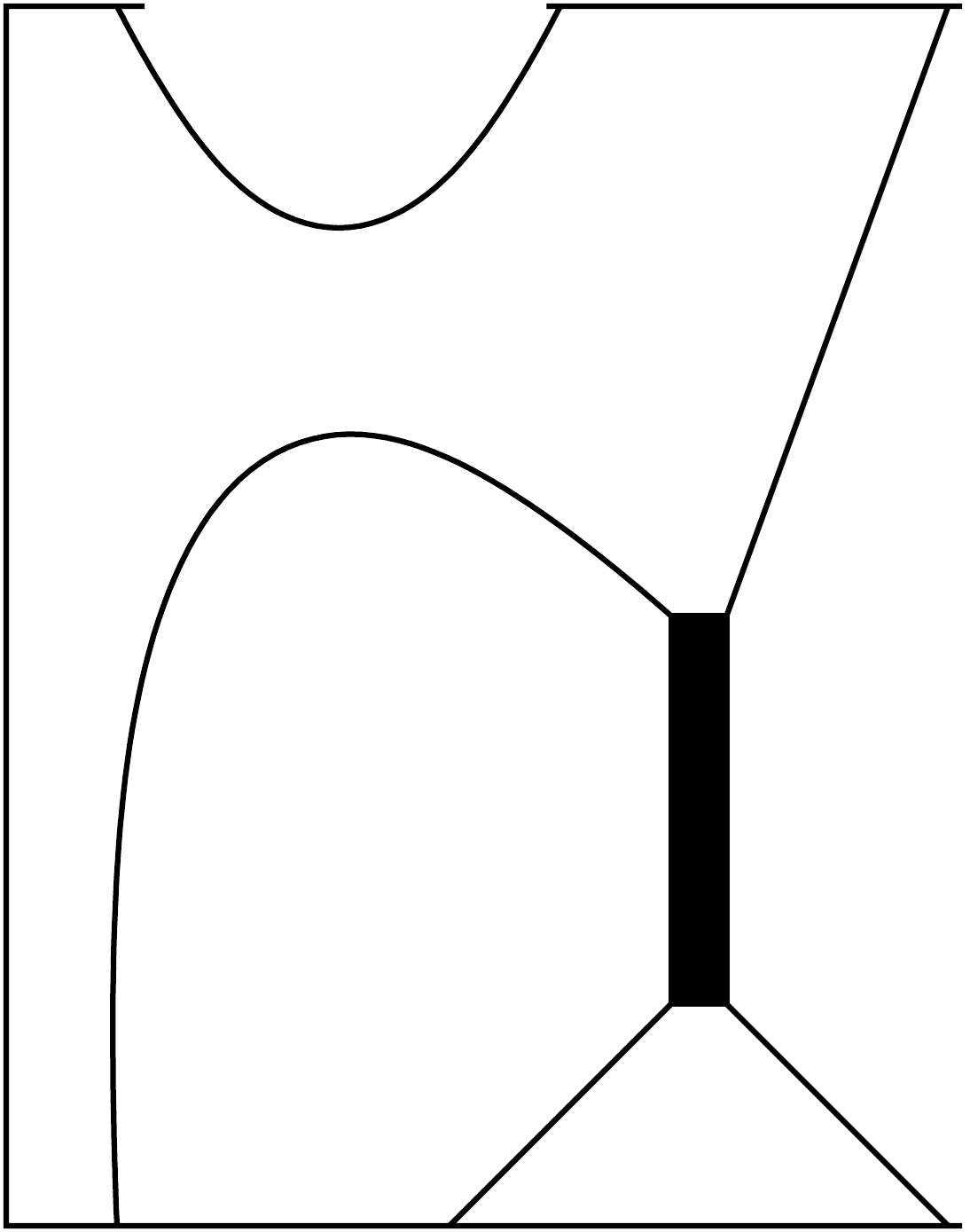}}\, \right)$, $f\left(\,\raisebox{-7pt}{\includegraphics[height=0.25in]{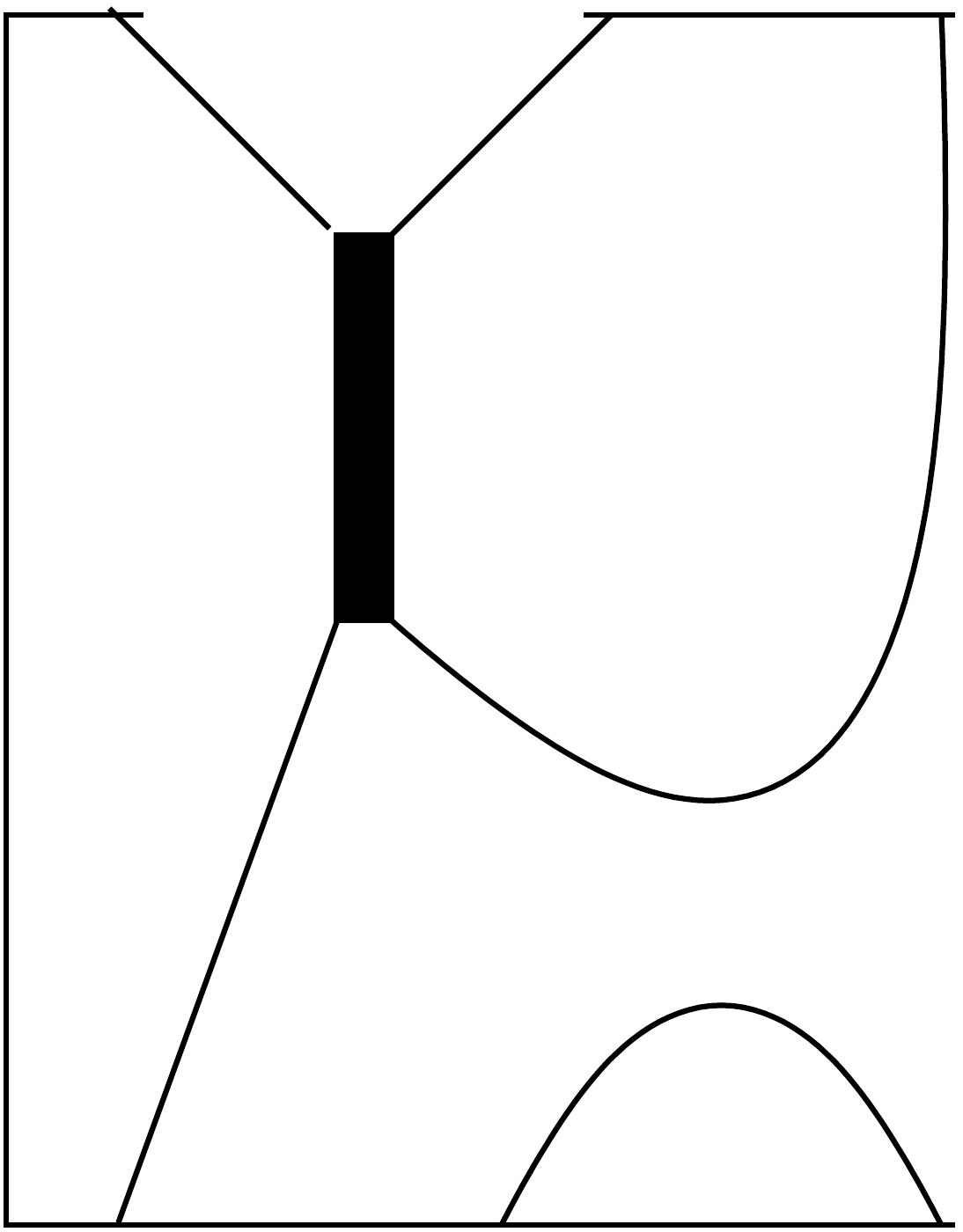}}\,\right) = f\left(\,\raisebox{-7pt}{\includegraphics[height=0.25in]{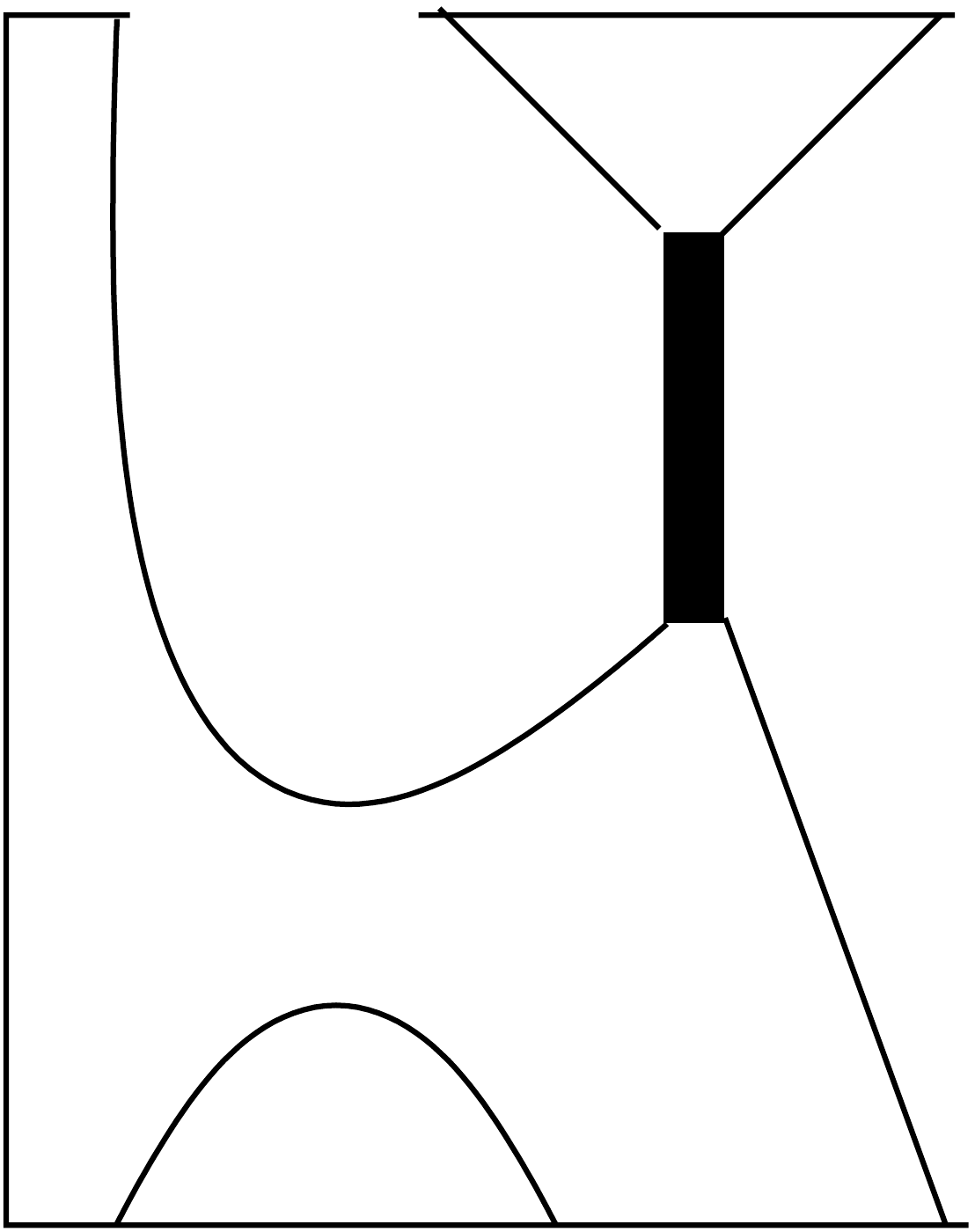}}\, \right)$ and $2f\left(\,\raisebox{-7pt}{\includegraphics[height=0.25in]{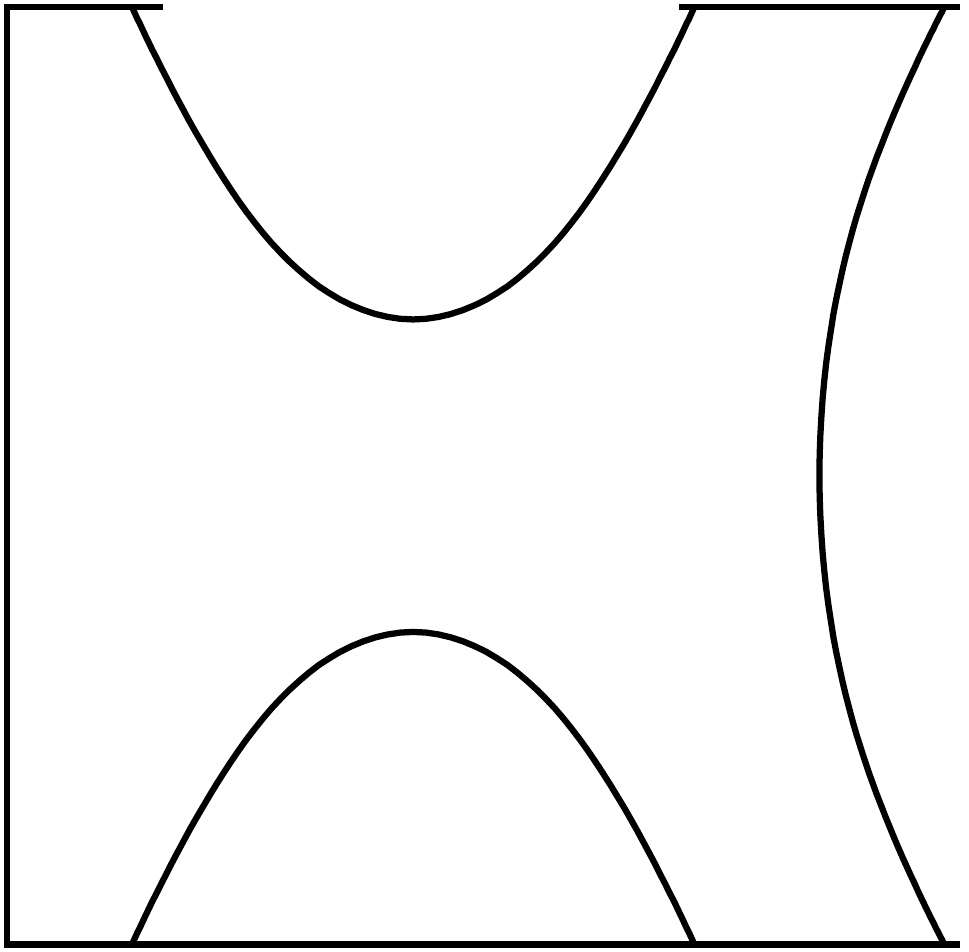}}\,\right) = f\left(\,\raisebox{-7pt}{\includegraphics[height=0.25in]{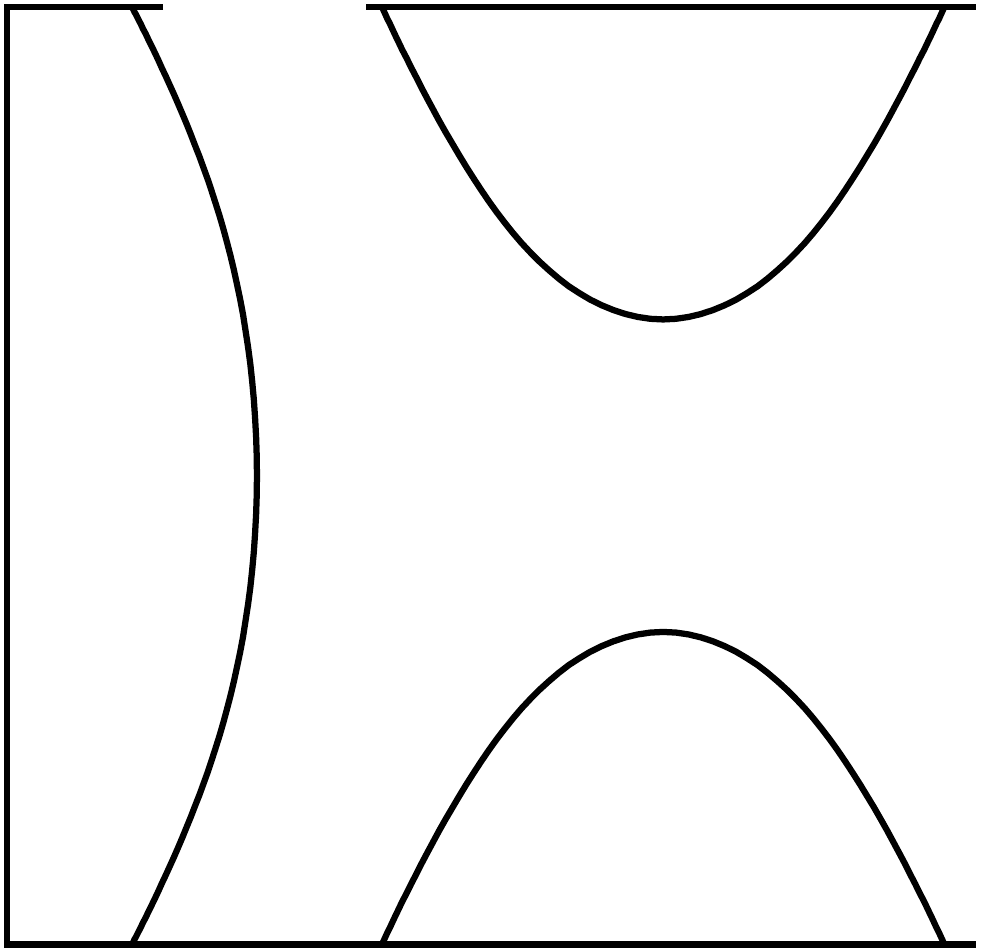}}\, \right)$. The identity becomes 
\[0 =  f\left(\,\raisebox{-7pt}{\includegraphics[height=0.25in]{long-6-case2}}\, \right) + (q + q^{-1})  f\left(\,\raisebox{-7pt}{\includegraphics[height=0.25in]{long-10-case2}}\, \right) =  (-q - q^{-1}) f\left(\,\raisebox{-7pt}{\includegraphics[height=0.25in]{long-10-case2}}\, \right) + (q + q^{-1})  f\left(\,\raisebox{-7pt}{\includegraphics[height=0.25in]{long-10-case2}}\, \right) = 0,\]
and hence it is satisfied in this case, as well.

\textit{Case 3.} The points are in three connected components, say $\raisebox{-15pt}{\includegraphics[height=0.5in]{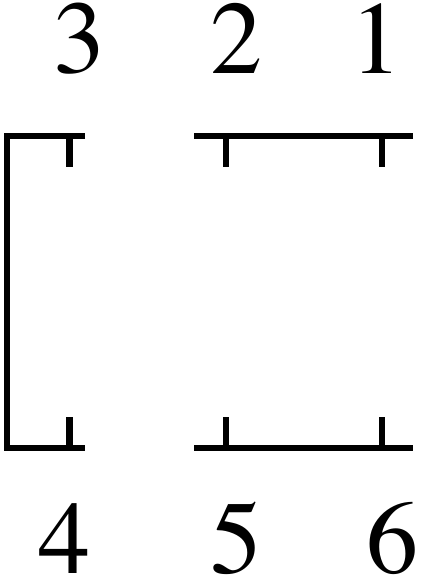}}\,,\raisebox{-15pt}{\includegraphics[height=0.5in]{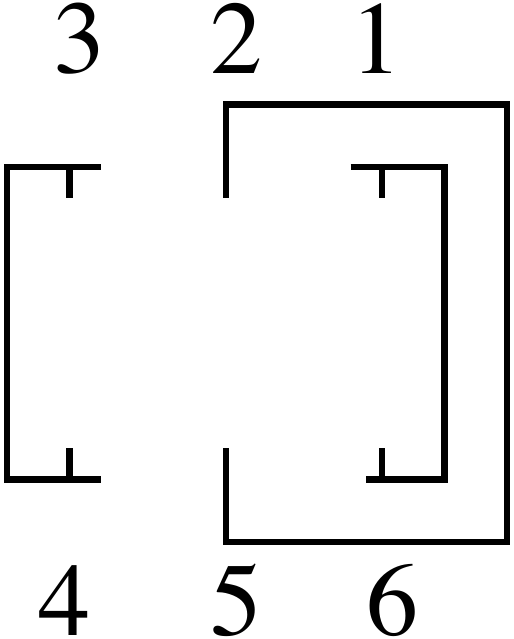}}\,$, or $\raisebox{-15pt}{\includegraphics[height=0.5in]{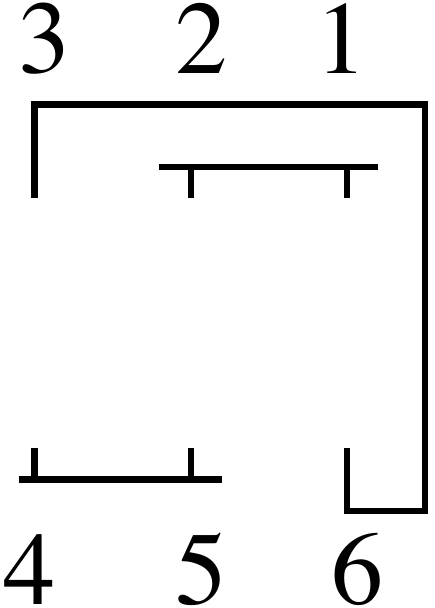}}\,$. In the first situation we have $f\left(\,\raisebox{-10pt}{\includegraphics[height=0.35in]{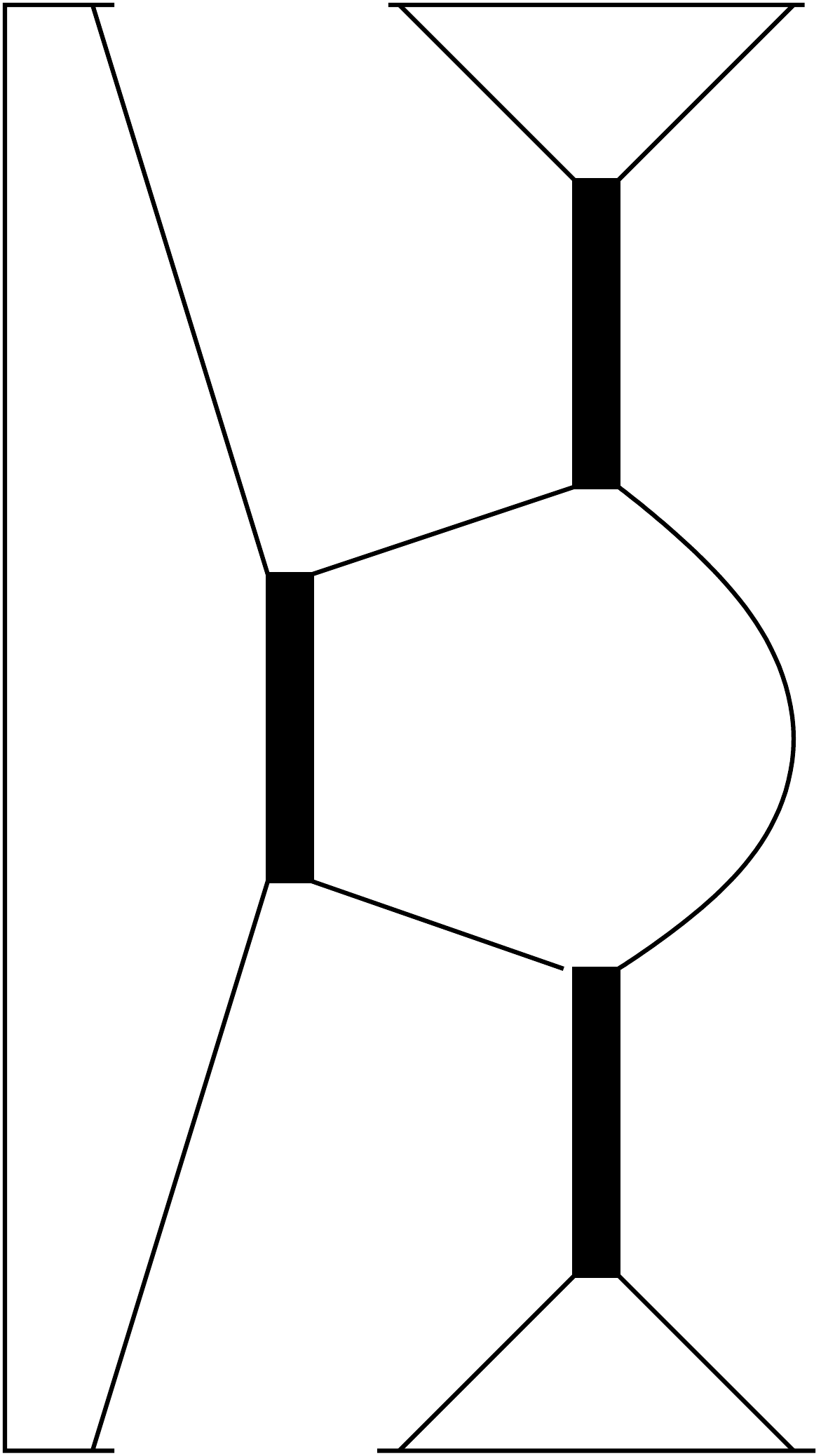}}\,\right) = f\left(\,\raisebox{-10pt}{\includegraphics[height=0.35in]{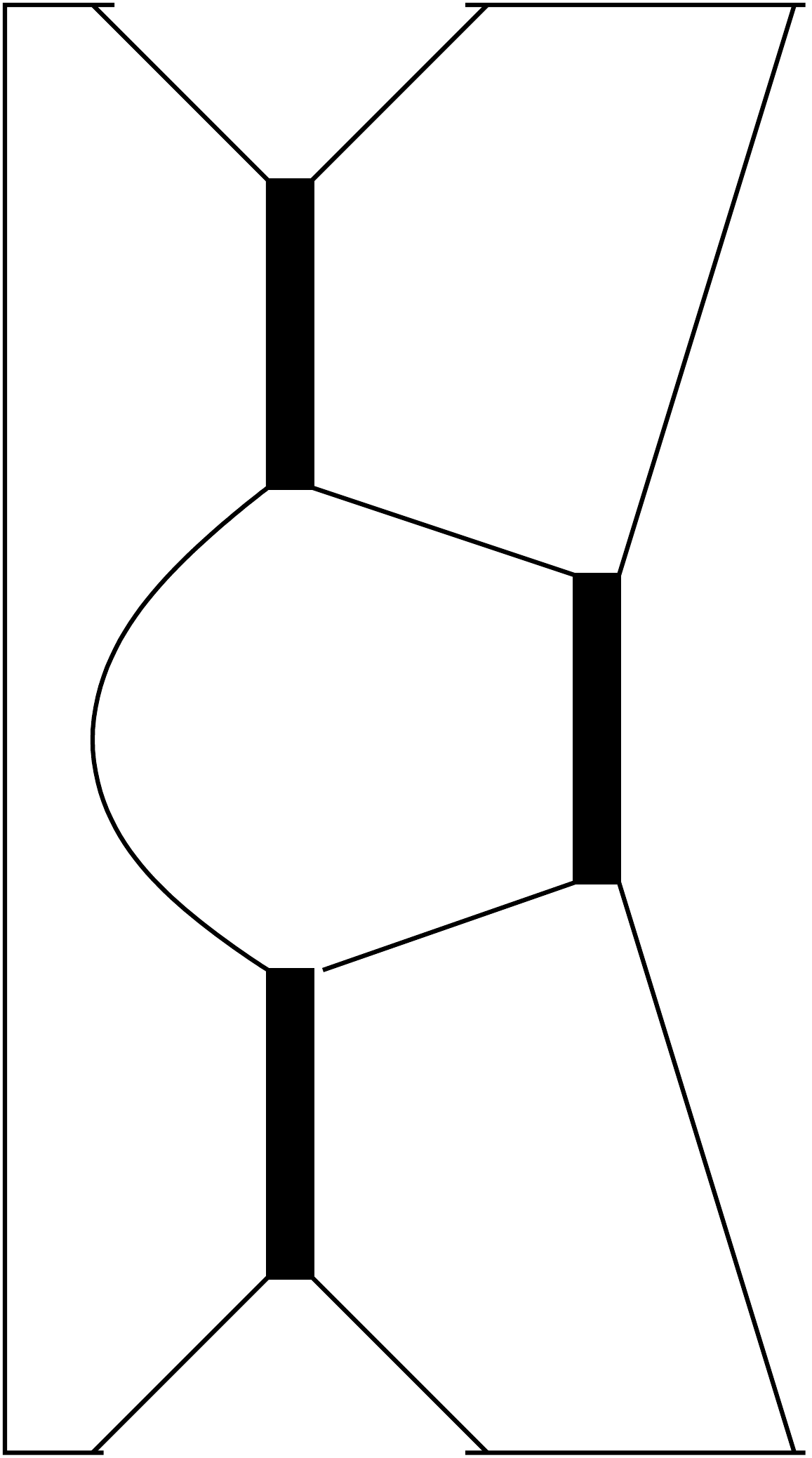}}\, \right), f\left(\,\raisebox{-7pt}{\includegraphics[height=0.25in]{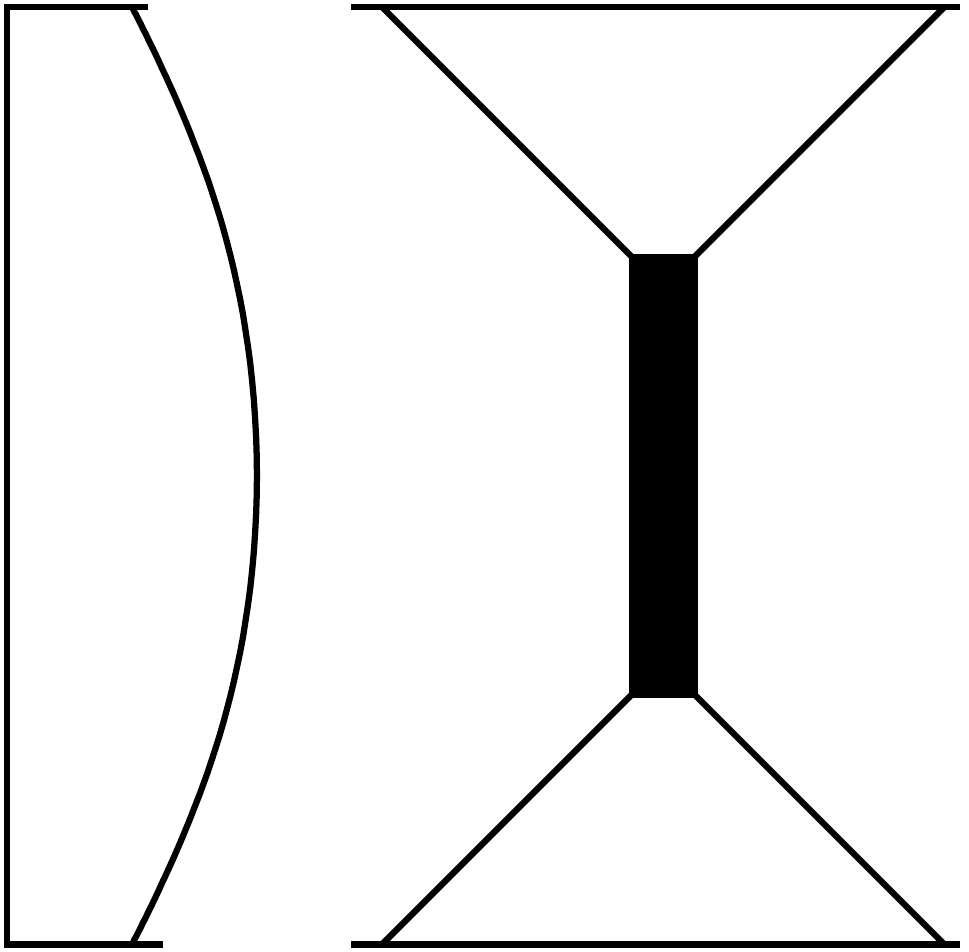}}\,\right) = 2 f\left(\,\raisebox{-7pt}{\includegraphics[height=0.25in]{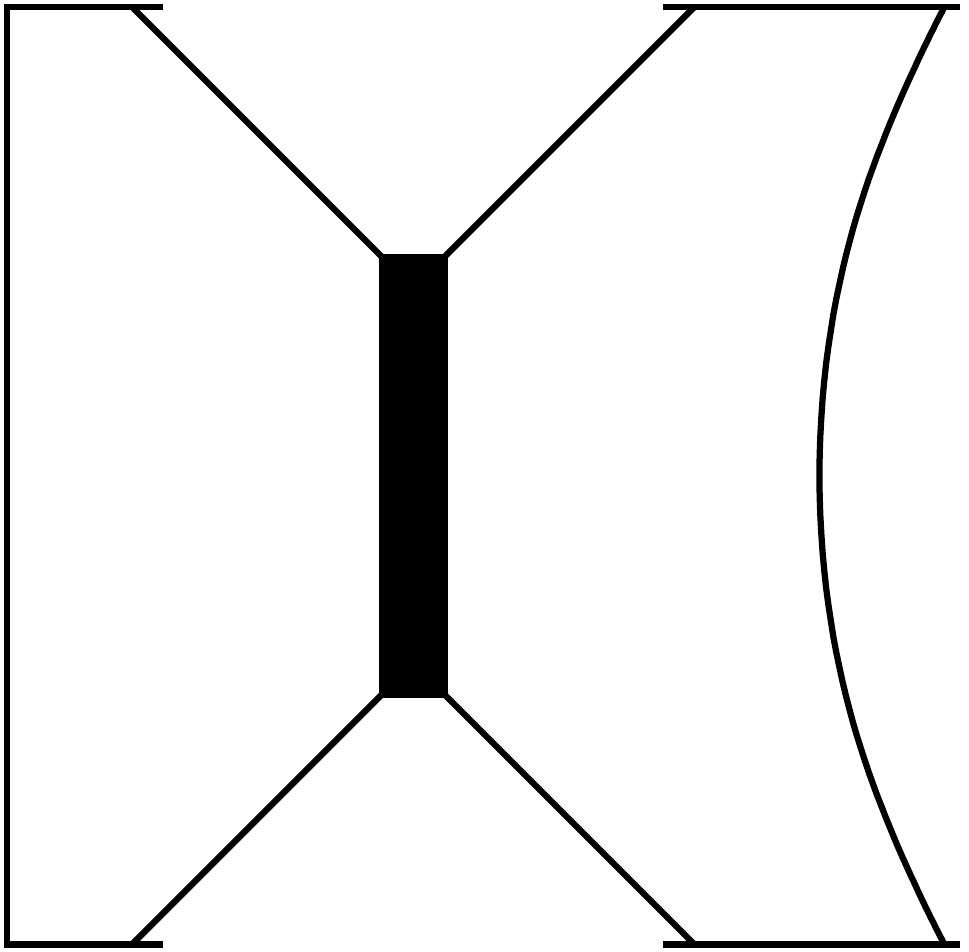}}\, \right)$, $f\left(\,\raisebox{-7pt}{\includegraphics[height=0.25in]{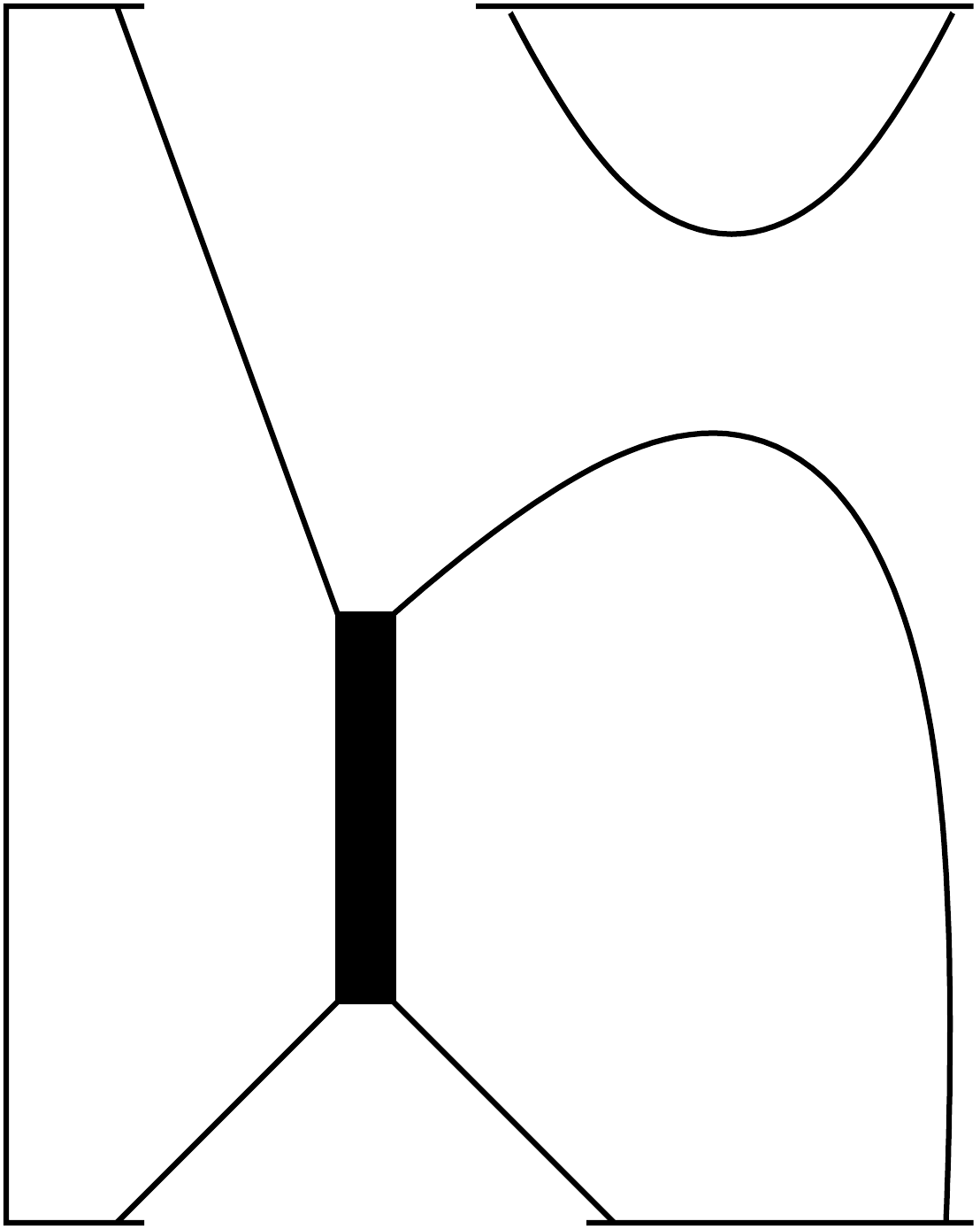}}\,\right) = 2 f\left(\,\raisebox{-7pt}{\includegraphics[height=0.25in]{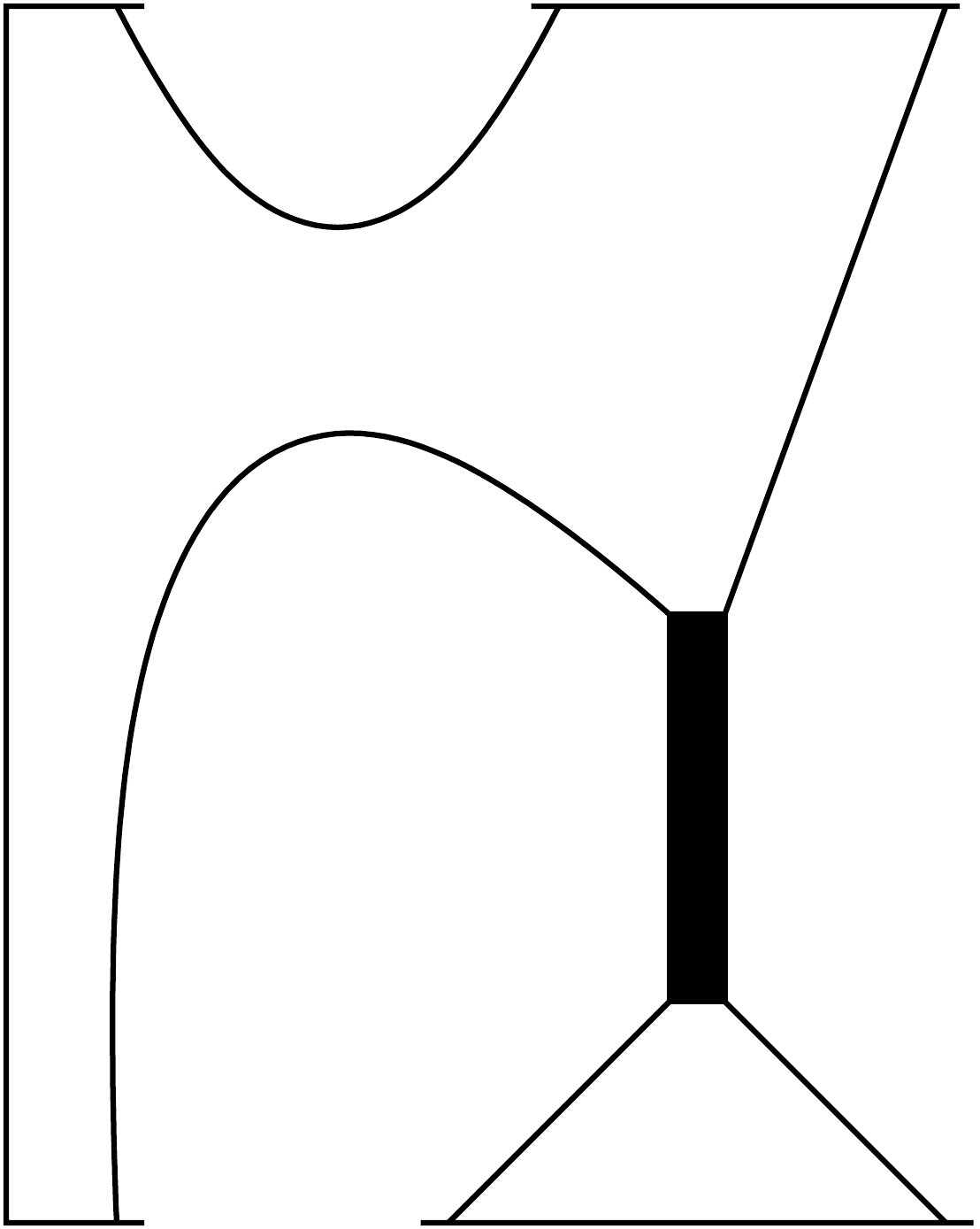}}\, \right), f\left(\,\raisebox{-7pt}{\includegraphics[height=0.25in]{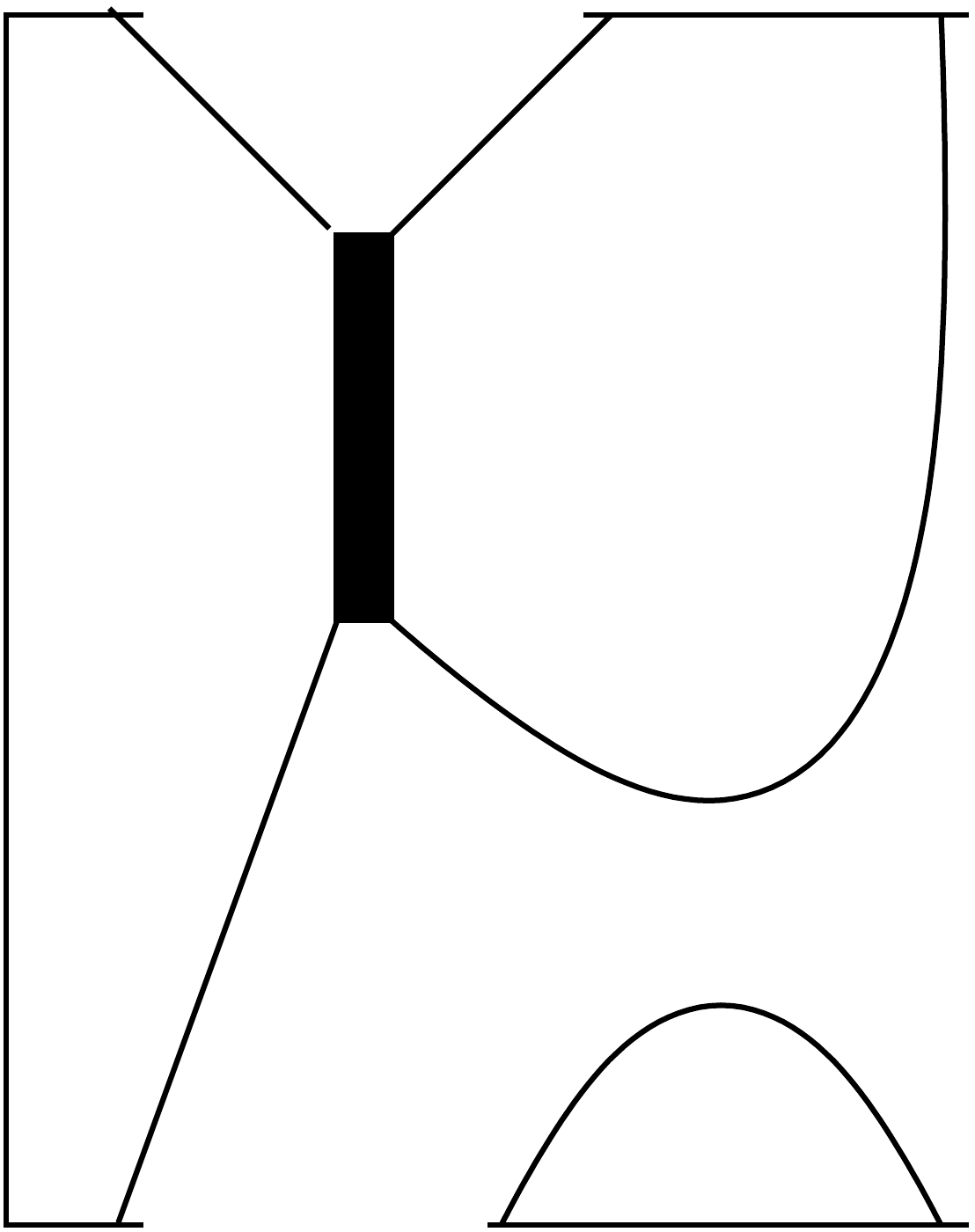}}\,\right) = 2f\left(\,\raisebox{-7pt}{\includegraphics[height=0.25in]{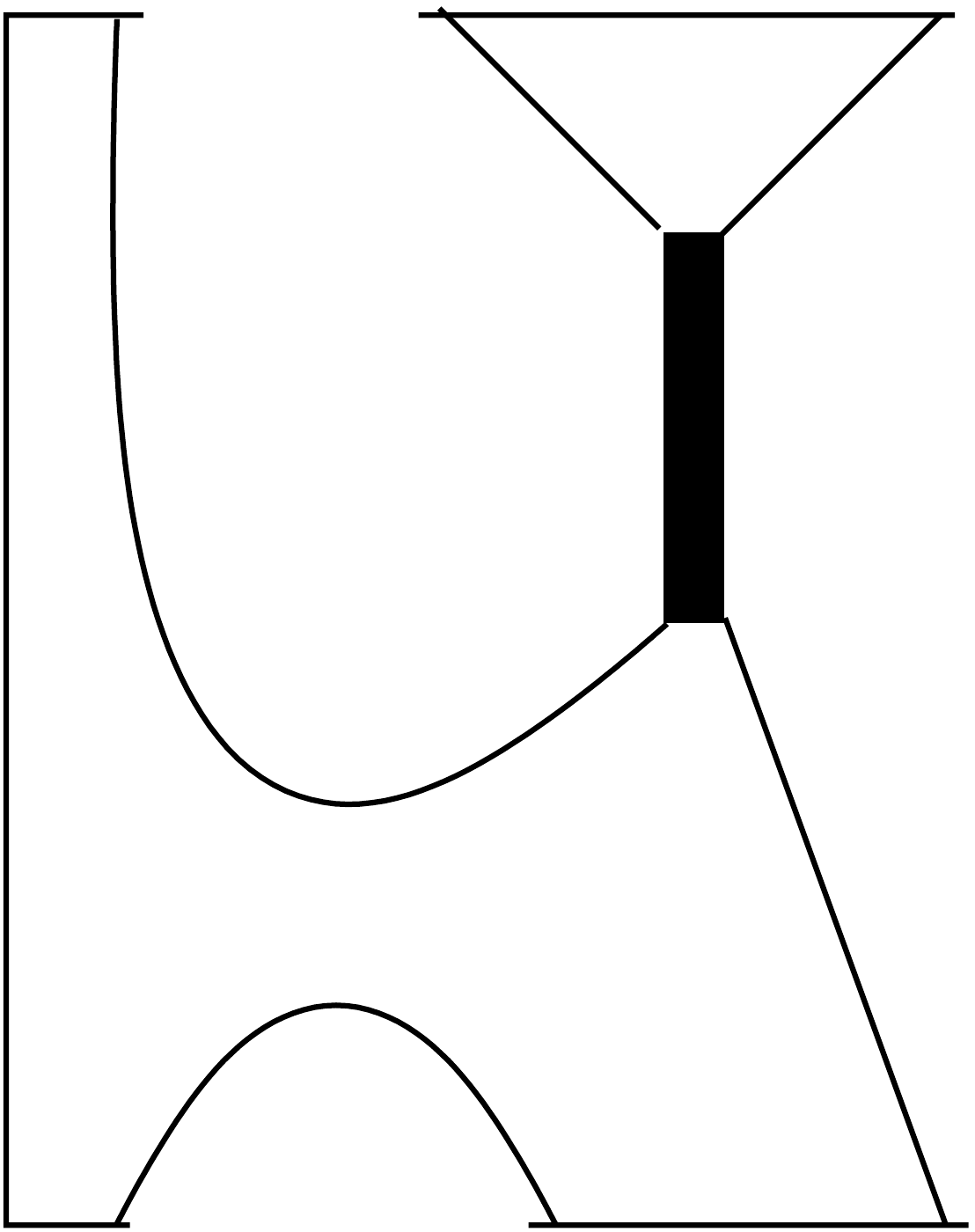}}\, \right)$ and $4f\left(\,\raisebox{-7pt}{\includegraphics[height=0.25in]{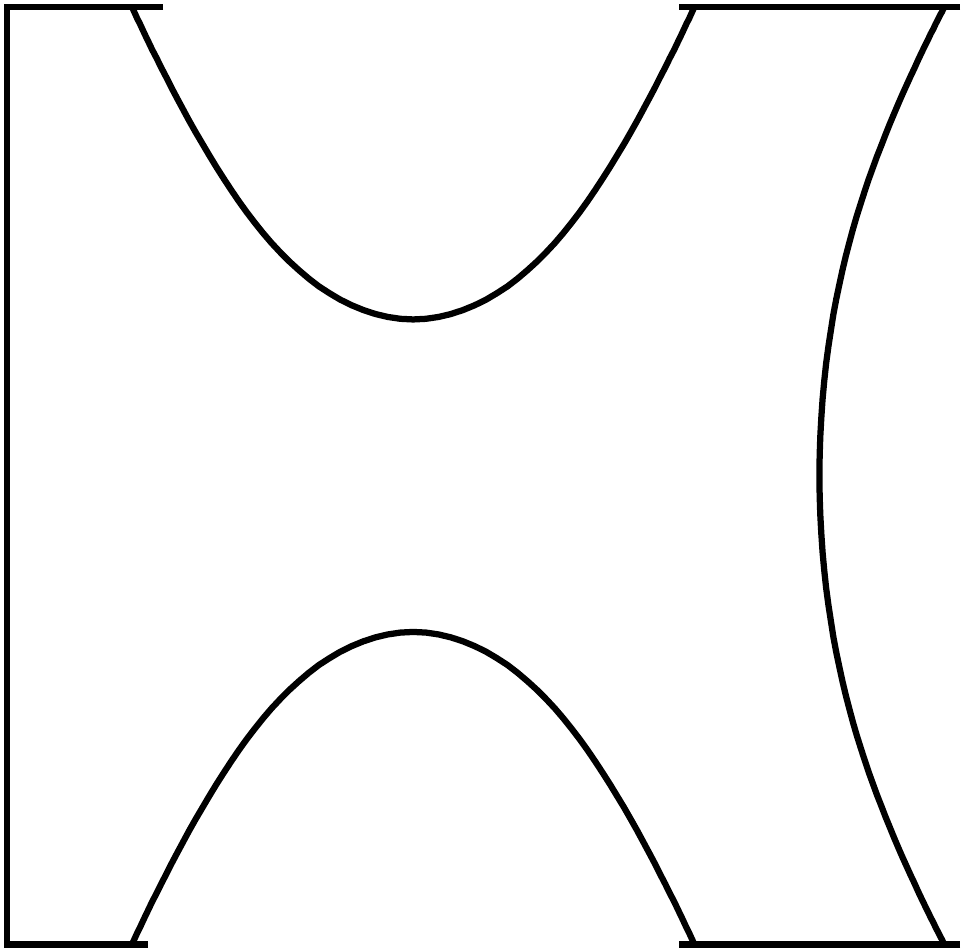}}\,\right) = f\left(\,\raisebox{-7pt}{\includegraphics[height=0.25in]{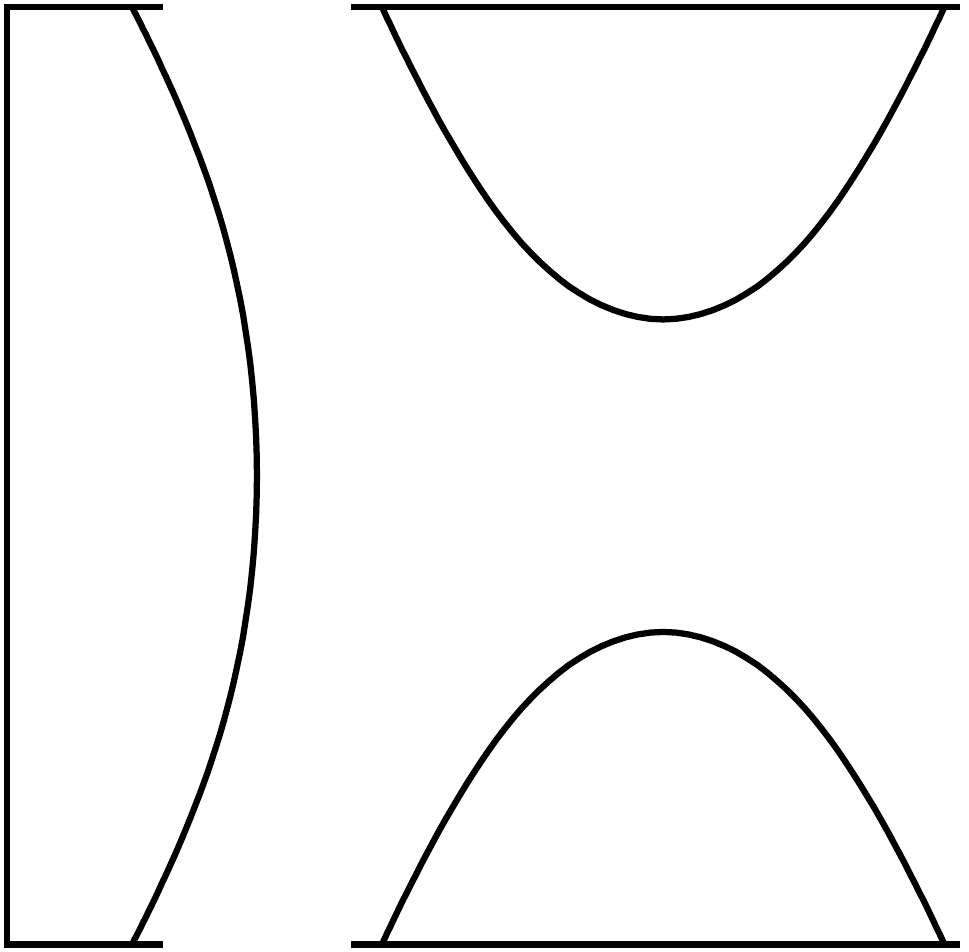}}\, \right)$. Therefore, the identity becomes
\begin{eqnarray*}
0 &=& f\left(\,\raisebox{-7pt}{\includegraphics[height=0.25in]{long-4-case3}}\, \right) + f\left(\,\raisebox{-7pt}{\includegraphics[height=0.25in]{long-6-case3}}\, \right)+ f\left(\,\raisebox{-7pt}{\includegraphics[height=0.25in]{long-8-case3}}\, \right)  + 3(q + q^{-1})f\left(\,\raisebox{-7pt}{\includegraphics[height=0.25in]{long-10-case3}}\, \right)\\
&=& 3(-q -q^{-1})f\left(\,\raisebox{-7pt}{\includegraphics[height=0.25in]{long-10-case3}}\, \right) + 3(q + q^{-1})f\left(\,\raisebox{-7pt}{\includegraphics[height=0.25in]{long-10-case3}}\, \right) = 0,
 \end{eqnarray*} 
 and therefore it holds. The other two possibilities for this case follow similarly.
\end{proof}

Theorem~\ref{thm:N=2} implies that, for the case $N = 2$, the link polynomial $\textbf{P}_L(q)$ (and thus the $SO(2)$ Kauffman polynomial) is easily computed. Obviously, the same holds for $[G]$, the evaluation of a knotted RE 3-graph. 

\begin{corollary}\label{cor:graph-one crossing}
If a trivalent graph diagram $G$ has one crossing involving standard edges and no other crossings, and if the removal of this crossing does not change the number of connected components, then in the case $N = 2$ the polynomial $[G]$ vanishes, and thus $G$ is not rigid-edge isotopic to a planar trivalent graph. 
\end{corollary}

\begin{proof}
This is easily implied by Theorem~\ref{thm:N=2} and the skein relations
\begin{eqnarray*}
 \left[ \,\raisebox{-3pt}{\includegraphics[height=0.15in]{poscrossing}}\,\right] = q \left[ \, \raisebox{-3pt}{\includegraphics[height=0.15in]{A-smoothing}}\, \right] + q^{-1} \left[ \, \raisebox{-3pt}{\includegraphics[height=0.15in]{B-smoothing}}\, \right] +  \left[ \, \raisebox{-3pt}{\includegraphics[height=0.15in]{move5-3}} \, \right],\quad
\left[ \,\raisebox{-3pt}{\includegraphics[height=0.15in]{negcrossing}}\, \right] = q \left[ \, \raisebox{-3pt}{\includegraphics[height=0.15in]{B-smoothing}}\, \right] + q^{-1} \left[\, \raisebox{-3pt}{\includegraphics[height=0.15in]{A-smoothing}}\, \right] + \left[\, \raisebox{-3pt}{\includegraphics[height=0.15in]{move5-3}}\, \right].
\end{eqnarray*}  
\end{proof}

\begin{example}
We consider the graph $G = \raisebox{-15pt}{\includegraphics[height=0.5in]{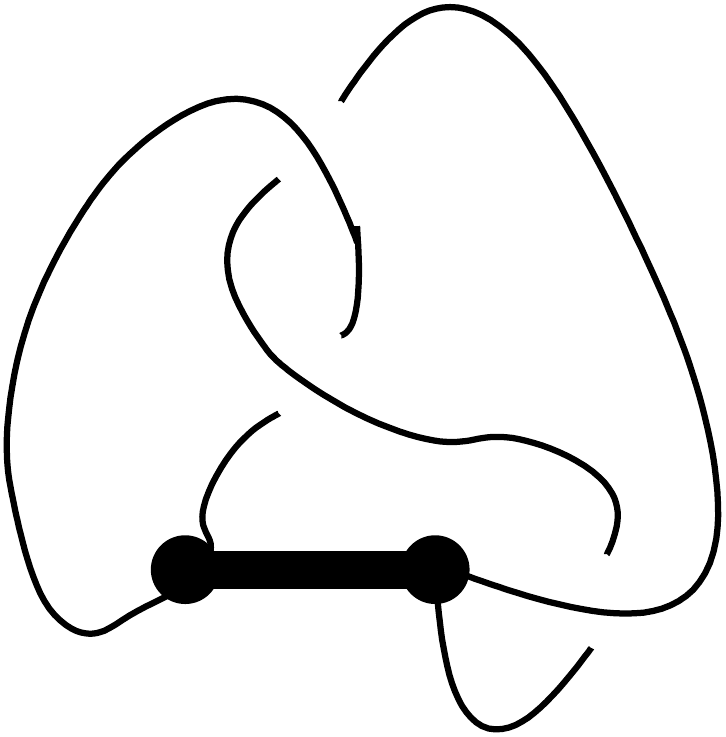}}$  and compute its polynomial $[G]$ for the case $N = 2$.

\begin{eqnarray*}
[G] &=&  \left[\raisebox{-15pt}{\includegraphics[height=0.5in]{hc-2}} \right] = q \left[\raisebox{-15pt}{\includegraphics[height=0.5in]{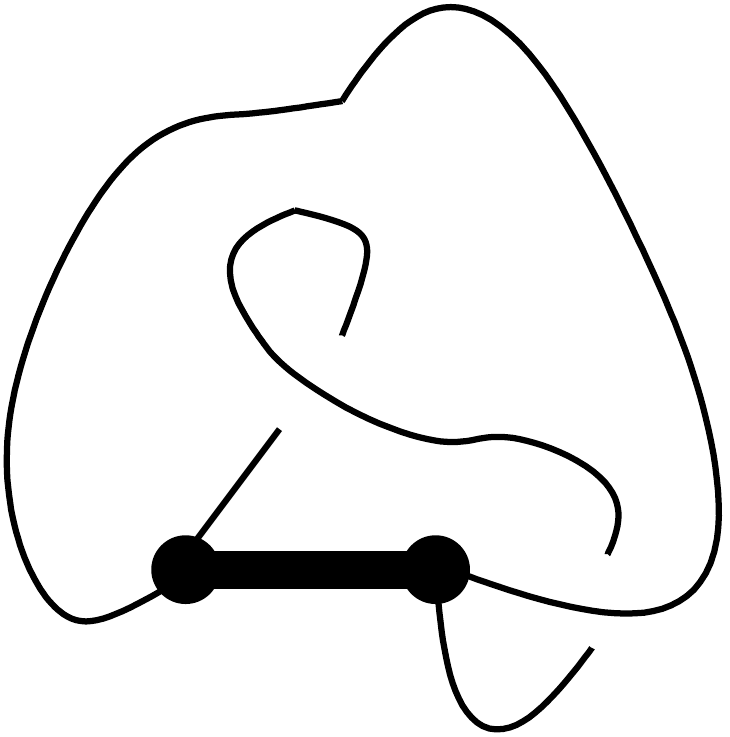}} \right] + q^{-1} \left[\raisebox{-15pt}{\includegraphics[height=0.5in]{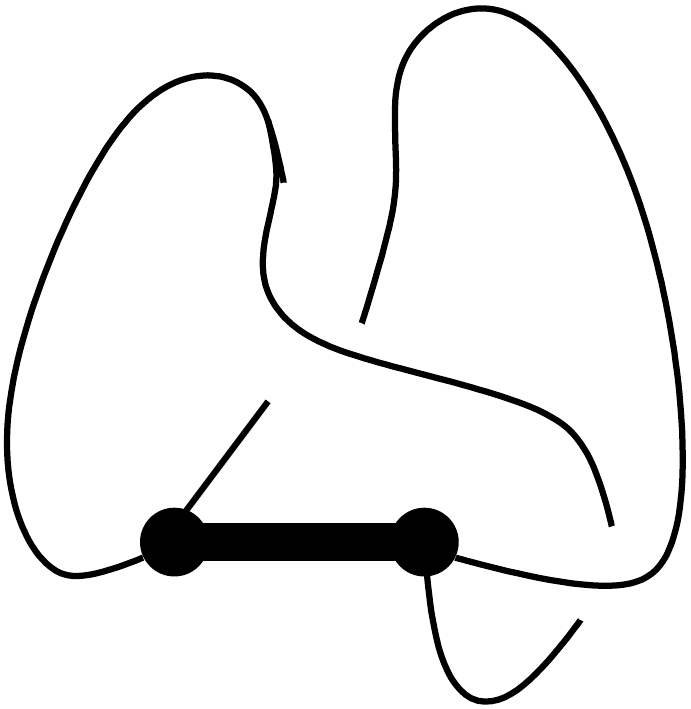}}\right] + \left[\raisebox{-15pt}{\includegraphics[height=0.5in]{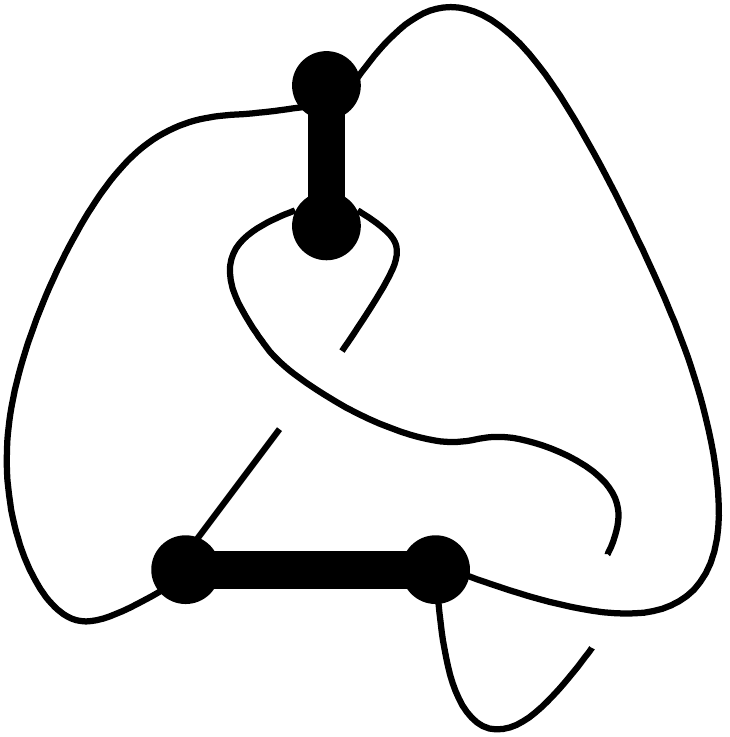}} \right] \\
& = & q^2 \left[\raisebox{-15pt}{\includegraphics[height=0.5in]{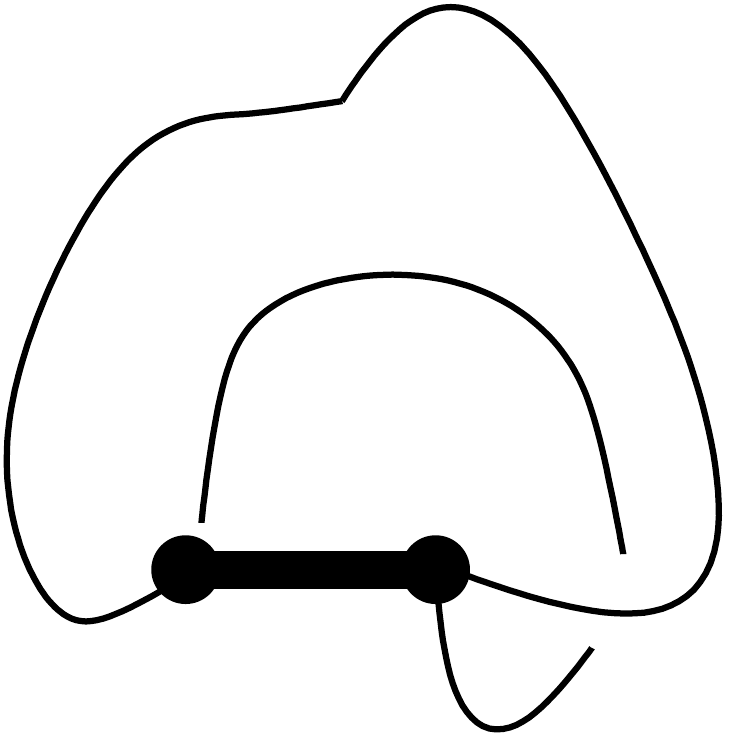}} \right] + \left[\raisebox{-15pt}{\includegraphics[height=0.5in]{ex3-4}} \right] + q^{-2}\left[\raisebox{-15pt}{\includegraphics[height=0.5in]{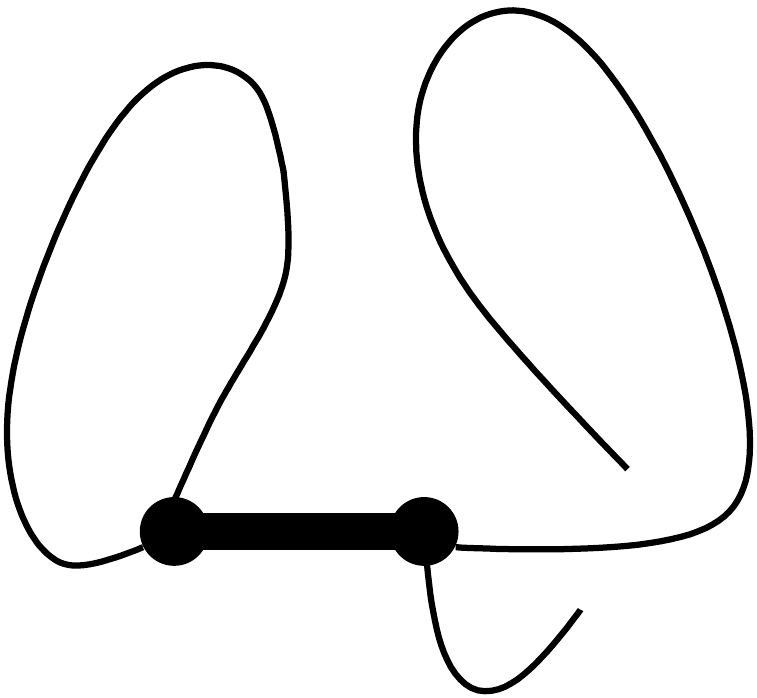}} \right] + q^{-1} \left[\raisebox{-15pt}{\includegraphics[height=0.5in]{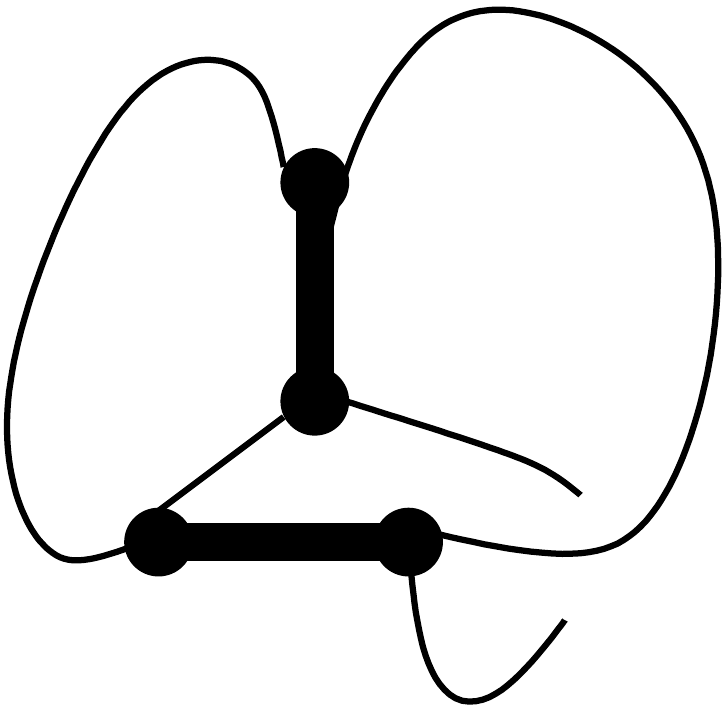}} \right]\\
&&+ q\left[\raisebox{-15pt}{\includegraphics[height=0.5in]{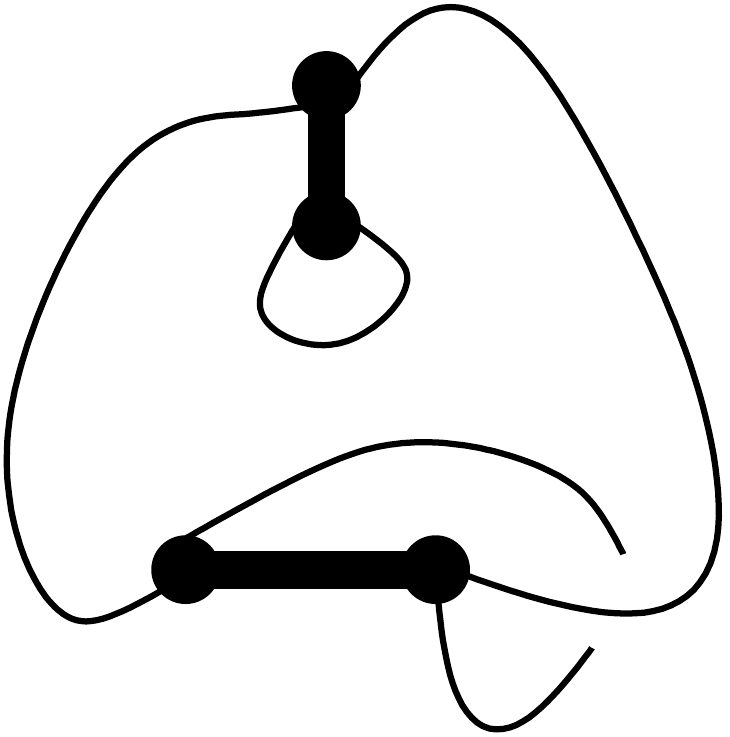}} \right] + q^{-1}\left[\raisebox{-15pt}{\includegraphics[height=0.5in]{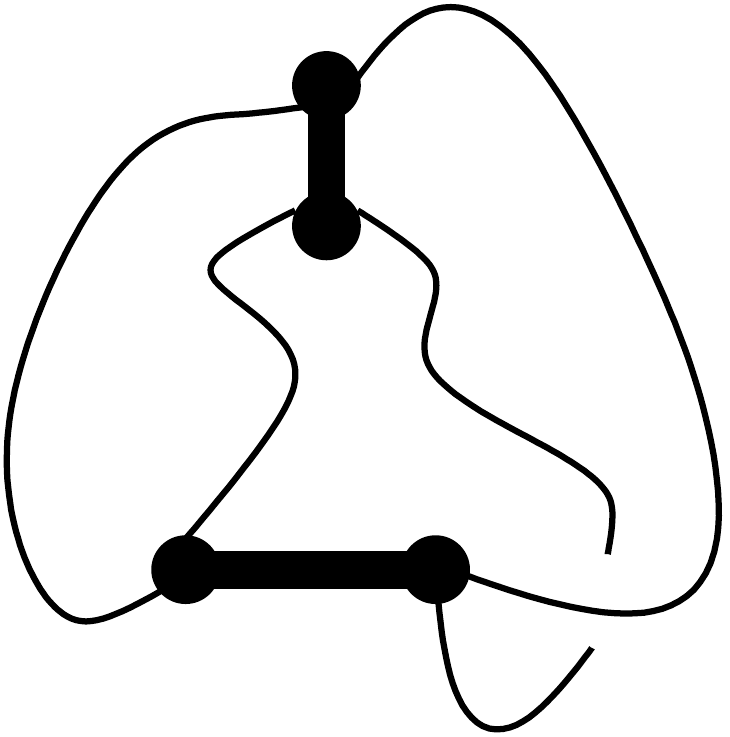}} \right] + \left[\raisebox{-15pt}{\includegraphics[height=0.5in]{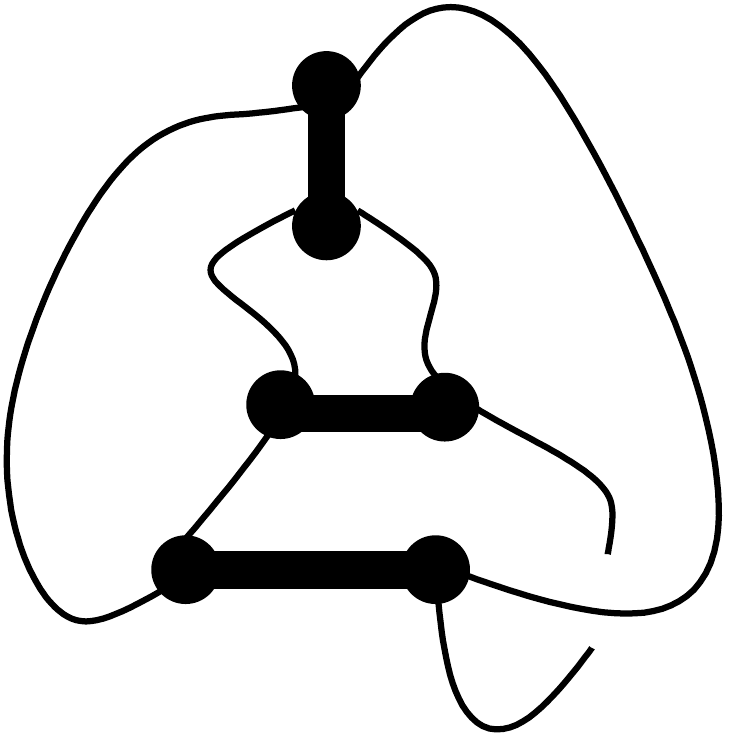}} \right].
\end{eqnarray*}
All diagrams in the last equality, with the exception of the third diagram, satisfy the hypothesis of Corollary~\ref{cor:graph-one crossing}, thus the evaluation of these diagrams is zero. We obtain
\[ [G] =  q^{-2}\left[\raisebox{-15pt}{\includegraphics[height=0.5in]{ex3-5}} \right]  = q^{-2}(-q -q^{-1}).\]

\end{example}
 
 
 \section{Concluding remarks}\label{sec:comments}
 
\subsection{}\label{ssec:4-valent} We describe here the state model for the two-variable Kauffman polynomial for links that uses planar 4-valent graphs. It is easily obtained from the model given in Section~\ref{sec:model} by collapsing the wide edges (that is, by identifying the two vertices adjacent to a wide edge): 
\[\raisebox{-6pt}{\includegraphics[height=0.25in]{resol}} \longrightarrow \raisebox{-3pt}{\includegraphics[height=0.15in]{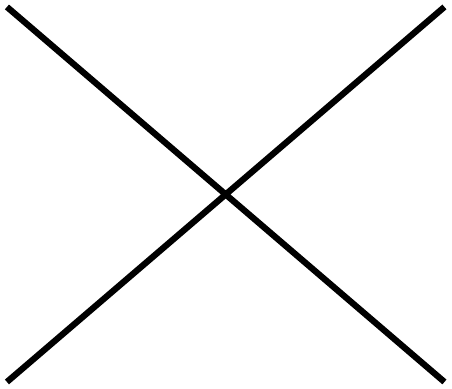}}\]

Start with a link diagram $D$ and replace each crossing in $D$ with the following formal linear combination:
\begin{eqnarray} \label{resolving crossings - version2}
\raisebox{-3pt}{\includegraphics[height=0.15in]{poscrossing}} = A\, \raisebox{-3pt}{\includegraphics[height=0.15in]{A-smoothing}} + B\, \raisebox{-3pt}{\includegraphics[height=0.15in]{B-smoothing}} +  \raisebox{-3pt}{\includegraphics[height=0.15in]{cross}}\,. 
\end{eqnarray}

Each state $\Gamma$ corresponding to the diagram $D$ is then a planar 4-valent graph, and one can uniquely associate to each such a graph a well-defined Laurent polynomial $P(\Gamma)  \in \bbZ[A^{\pm 1}, B^{\pm 1}, a^{\pm 1}, (A-B)^{\pm 1}] $ so that it satisfies the following skein relations, in which we omit the polynomial $P$ for simplicity (these skein relations are the equations of the \textit{graphical calculus} in~\cite{KV}):

\begin{eqnarray*}
\bigcirc &= &1,\quad
 \Gamma \cup \bigcirc = \alpha \Gamma, \quad 
\raisebox{-4pt}{\includegraphics[height=0.18in]{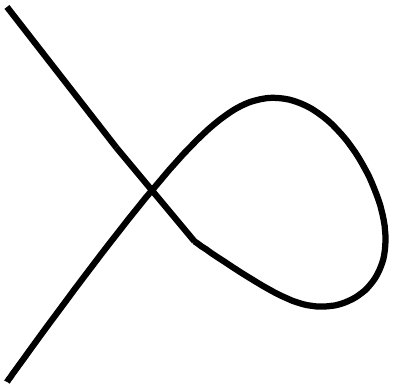}} = \beta \,\raisebox{-4pt}{\includegraphics[height=0.18in]{arc}}\\
\raisebox{-4pt}{\includegraphics[height=0.2in]{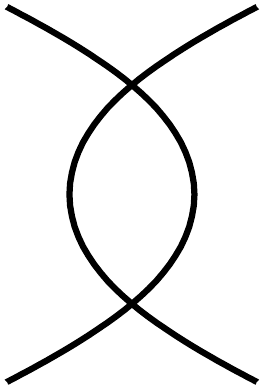}} &=& \left[1-AB \right] \, \raisebox{-4pt}{\includegraphics[height=0.18in]{A-smoothing}} + \gamma\, \raisebox{-4pt}{\includegraphics[height=0.18in]{B-smoothing}} - \left [A+B \right ] \, \raisebox{-4pt}{\includegraphics[height=0.18in]{cross}} \\
\raisebox{-5pt}{\includegraphics [height=0.2in]{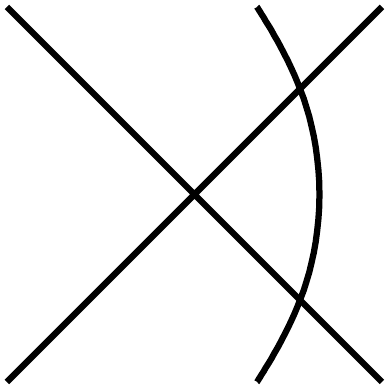}} \,-\, \raisebox{-5pt}{\includegraphics[height=0.2in]{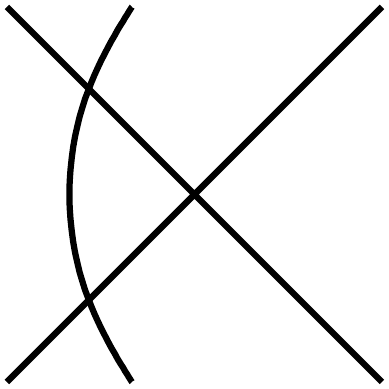}} &=& AB \left [ \,\raisebox{-5pt}{\includegraphics[height=0.2in]{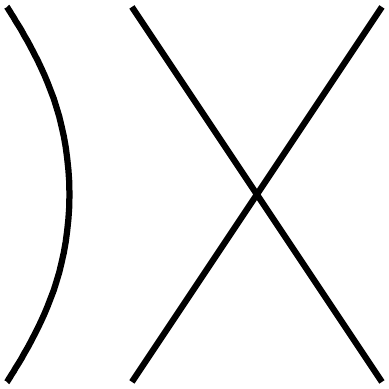}} - \raisebox{-5pt}{\includegraphics[height=0.2in]{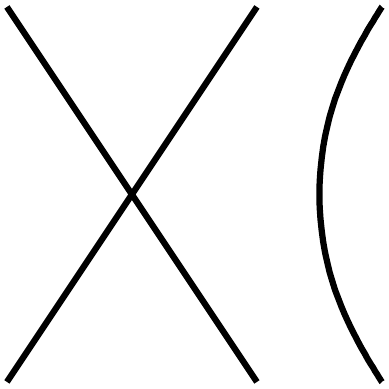}} \, + \, \raisebox{-5pt}{\includegraphics[height=0.2in]{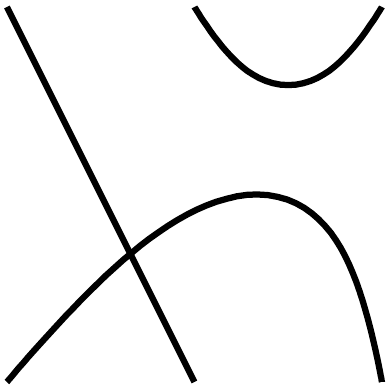}} - \raisebox{-5pt}{\includegraphics[height=0.2in]{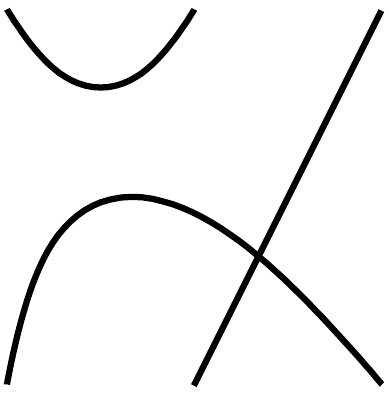}}\,  + \, \raisebox{-5pt}{\includegraphics[height=0.2in]{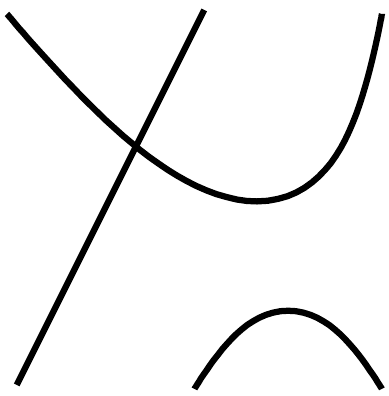}} \,-\, \raisebox{-5pt}{\includegraphics[height=0.2in]{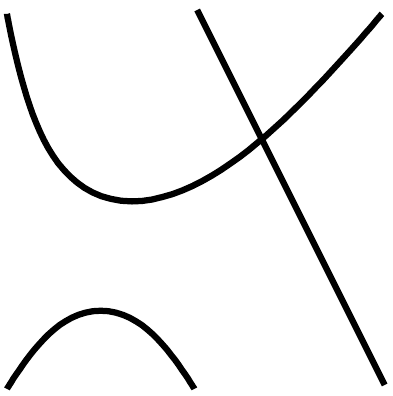}}\, \right]\, + \, \delta \, \left [\, \raisebox{-5pt}{\includegraphics[height=0.2in]{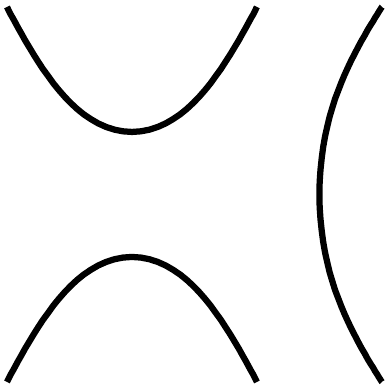}} -  \raisebox{-5pt}{\includegraphics[height=0.2in]{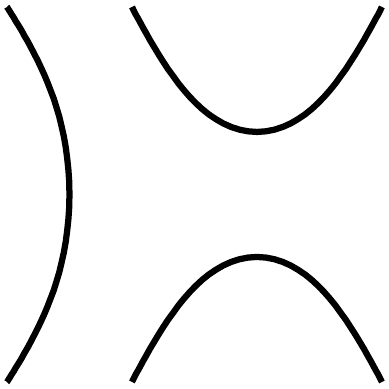}}\,\right ]
\end{eqnarray*}
where $\alpha =   \displaystyle \frac{a-a^{-1}}{A-B} + 1, \,\,\beta = \frac{Aa^{-1}-Ba}{A-B}-A-B,\,\, \gamma = \frac{B^2a-A^2a^{-1}}{A-B} + AB$ and $\delta = \displaystyle \frac{B^3a - A^3a^{-1}}{A - B}$.  A proof of the fact that these relations give a unique, well-defined polynomial associated to a planar 4-valent graph can be found in~\cite[Theorem 3]{C}).

The rational function $\textbf{P}_D = \textbf{P}_D(A, B, a)$ is computed by summing up $P(\Gamma)$ over all states $\Gamma$ of $D$, weighted by powers of $A$ and $B$ according to the rules given in~(\ref{resolving crossings - version2}). There is an analogue of Theorem~\ref{thm:invariance} for this model for the Kauffman polynomial, whose proof follows similarly as in the case for trivalent graphs.

\subsection{} In this paper we constructed a rational function $\textbf{P}_L = \textbf{P}_L(A,B,a)$ (or Laurent polynomial in variables $A,B,a$ and $A-B$), which is a regular invariant for unoriented links $L$, and which is a version of the two-variable Kauffman polynomial. The polynomial is defined via a state summation, which reduces the computational complexity of the Kauffman polynomial. The state sum  formula for $\textbf{P}_L$ involves Laurent polynomial evaluations $P(\Gamma)$, for all of the states $\Gamma$ of the link diagram $L$. A state $\Gamma$ is a trivalent planar diagram and the polynomial $P(\Gamma)$ is computed via certain graph skein relations which iteratively allow to write $P(\Gamma)$ in terms of polynomials of simpler planar trivalent diagrams. The graph skein relation $P\left( \,\raisebox{-7pt}{\includegraphics[height=0.25in]{resol}}\,\right) = P \left(\, \raisebox{-3pt}{\includegraphics[height=0.17in]{resol-ho}}\,\right)$ together with the comments in Subsection~\ref{ssec:4-valent}, raise the following natural question: Why did we choose to work with trivalent graphs instead of 4-valent graphs? As mentioned in the introduction, the main reason is that trivalent graphs are generic graphs. Another reason involves the second outcome of this paper, namely the polynomial invariant for knotted RE 3-graphs and well-known results that enable one to study knotted handlebodies through knotted trivalent graphs.

A \textit{handlebody-knot} is a handlebody embedded in the 3-sphere $S^3$. Any handlebody-knot is a regular neighborhood of some knotted trivalent graph. Hence there is a one-to-one correspondence between the set of knotted handlebodies and that of neighborhood equivalence classes of knotted trivalent graphs. Two knotted graphs are neighborhood equivalent if there is an isotopy of $S^3$ taking a regular neighborhood of one graph onto that of the other (we refer the reader to Suzuki~\cite{Su} for details about this notion). Ishii~\cite{I} proved that two knotted trivalent graphs are neighborhood equivalent if and only if they are related by a finite sequence of IH-moves and isotopies of $\bbR^3$, where an \textit{IH-move} is a local change of a knotted trivalent graph as described in Figure~\ref{fig:IH-move}. In particular, if $G_1$ and $G_2$ are knotted trivalent graphs with diagrams $D_1$ and $D_2$, then $G_1$ and $G_2$ are neighborhood equivalent if and only if $D_1$ and $D_2$ are related by a finite sequence of classical Reidemeister moves (for knots) combined with the moves depicted in~(\ref{graph moves}) (together with their mirror images), and the IH-move. This implies that if $v$  is an invariant for knotted trivalent graphs which is also invariant under the IH-move, then $v$ is an invariant for handlebody-knots. 
\begin{figure}[ht]
\[\raisebox{-8pt}{\includegraphics[height=0.3in]{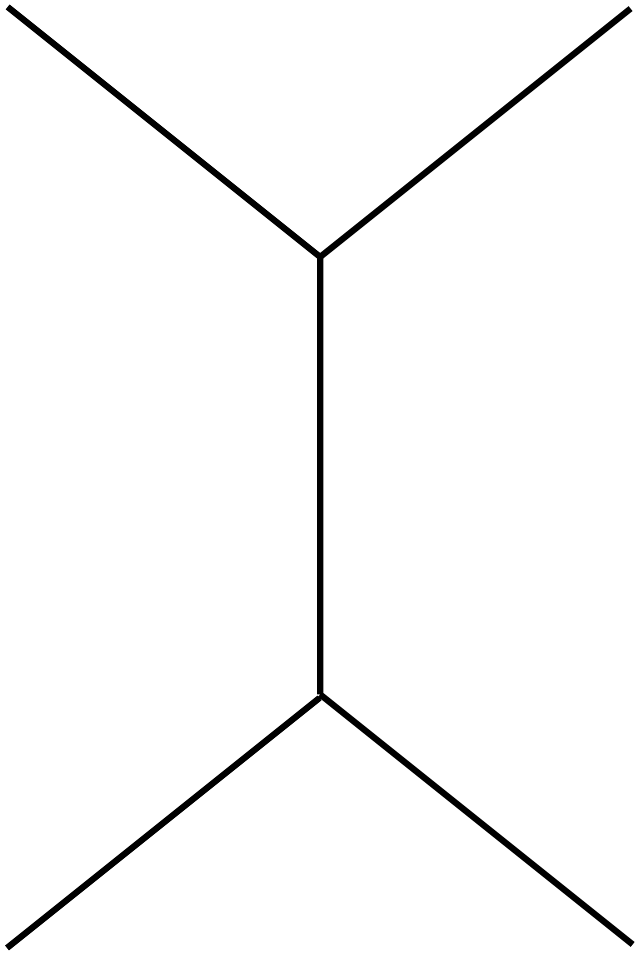}} \longleftrightarrow \raisebox{-8pt}{\includegraphics[height=0.3in]{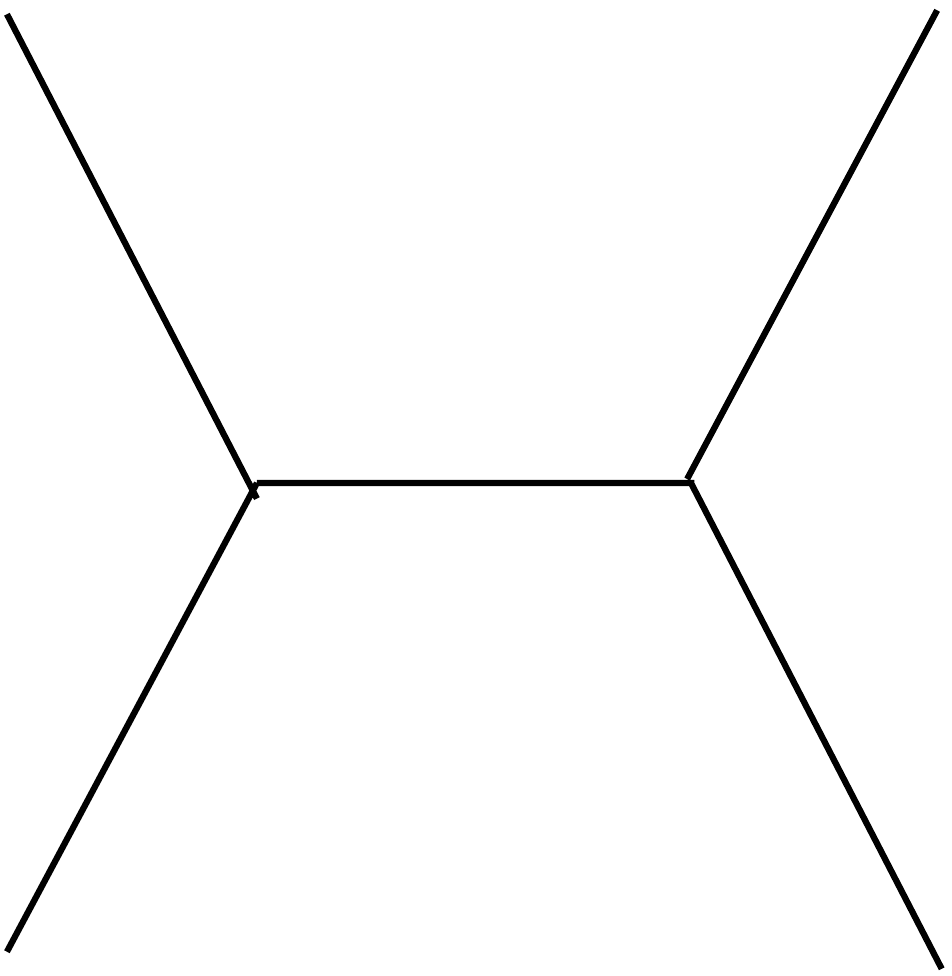}}\]
\caption{The IH-move}\label{fig:IH-move}
\end{figure}
Therefore, it is good that the graph polynomial $P$ satisfies the relation $P\left( \,\raisebox{-5pt}{\includegraphics[height=0.2in]{resol}}\,\right) = P \left(\, \raisebox{-3pt}{\includegraphics[height=0.2in, angle = 90]{resol}}\,\right)$.  The polynomial $[\,\cdot\, ]$ is an invariant of knotted RE 3-graphs, thus it falls slightly short of being an invariant for arbitrary knotted trivalent graphs, and consequently, of handlebody-knots. Can the construction presented in this paper be slightly modified so that it produces a honest invariant of topological knotted trivalent graphs? Maybe somebody reading this paper will find a solution to this question.

\subsection{}The definition of the $SO(N)$ Kauffman polynomial via the state model and trivalent graphs given in Subsection~\ref{ssec:one-variable poly} might prove useful in constructing a categorification of this polynomial. More specifically, given an unoriented link one may wish to construct a chain complex of graded vector spaces whose graded Euler characteristic is the $SO(N)$ Kauffman polynomial of the link. 

\subsection{} There is a relationship between the two-variable Kauffman polynomial and the Homflypt polynomial, due to Francois Jaeger. He obtains a Homflypt polynomial expansion of the Kauffman polynomial of an unoriented link $L$, by writing the Kauffman polynomial of $L$ as a certain weighted sum of Homflypt polynomials of oriented links associated with the given link $L$. Jaeger's model is described in~\cite{K3} (part I, Section 14). Applying the Murakami-Ohtsuki-Yamada~\cite{MOY} framework for the Homflypt polynomial to Jaeger's model, one obtains a model for the Kauffman polynomial via trivalent graphs (we thank Lorenzo Traldi for pointing this out to us). Then the following natural question arises: How is this model related to the one presented here in Section~\ref{sec:model}?


\appendix
\section{}\label{appendix:trivalent}

 We shall prove Theorem~\ref{thm:unique poly}, that is, we will show that for every planar trivalent graph diagram $\Gamma$, the graph skein relations~(\ref{gi2}) - (\ref{gi6}) determine a unique, well-defined Laurent polynomial $P(\Gamma) \in \bbZ[A^{\pm 1}, B^{\pm 1}, a^{\pm 1}, (A-B)^{\pm 1}]$. 

First we notice that the relations are consistent, in the sense that whenever there are two or more ways to calculate the polynomial of a planar trivalent diagram from the polynomials of other planar trivalent diagrams via the graph skein relations, the evaluation is the same (this can be verified by inspection). 

The graph skein relations allow us to calculate the polynomial of a planar trivalent graph diagram from the polynomials of other planar graph diagrams with less vertices, if the graph contains one of the local configurations
\begin{eqnarray}\label{good config}
\raisebox{-6pt}{\includegraphics[height=0.25in]{rem-edge}},\,\, \raisebox{-5pt}{\includegraphics[height=0.18in]{rem-loop}},\,\, \raisebox{-5pt}{\includegraphics[height=0.45in, angle = 90]{rem-digon}},\,\, \raisebox{-5pt}{\includegraphics[height=0.18in]{rem-square}}, \,\,\text{or}\,\, \raisebox{-5pt}{\includegraphics[height=0.18in]{rem-triangle}},
\end{eqnarray}
or if it contains one of these configurations after a sequence of moves of the type 
\begin{eqnarray}\label{good moves}
 \raisebox{-6pt}{\includegraphics[height=0.25in]{resol}}\, \longleftrightarrow \, \raisebox{-3pt}{\includegraphics[height=0.17in]{resol-ho}}\,,\quad \text{or} \quad \raisebox{-13pt}{\includegraphics[height=0.45in, width=0.25in]{long-1}}\, \longleftrightarrow \, \raisebox{-13pt}{\includegraphics[height=0.45in, width=0.25in]{long-2}}.
 \end{eqnarray}

\begin{proposition}\label{prop:polynomial}
If a connected planar trivalent graph diagram $\Gamma$ does not contain any of the local configurations displayed in~(\ref{good config}), then it is possible to change it, by applying a finite sequence of moves of type~(\ref{good moves}), into a  planar trivalent diagram containing one of the configurations \raisebox{-4pt}{\includegraphics[height=0.18in]{rem-square}}, \raisebox{-4pt}{\includegraphics[height=0.18in]{rem-triangle}},  or \raisebox{-4pt}{\includegraphics[height=0.45in, angle = 90]{rem-digon}}\,. Therefore, the graph skein relations allow us to calculate the polynomial $P(\Gamma)$ from the polynomials of other planar graph diagrams with less vertices.
\end{proposition}
\begin{proof} The proof is similar in spirit to that of Lemma 2 in~\cite{C}. First we notice that the two moves given in~(\ref{good moves}) yield the following moves:
\begin{eqnarray}\label{good moves-implied1}
\raisebox{-13pt}{\includegraphics[height=0.45in]{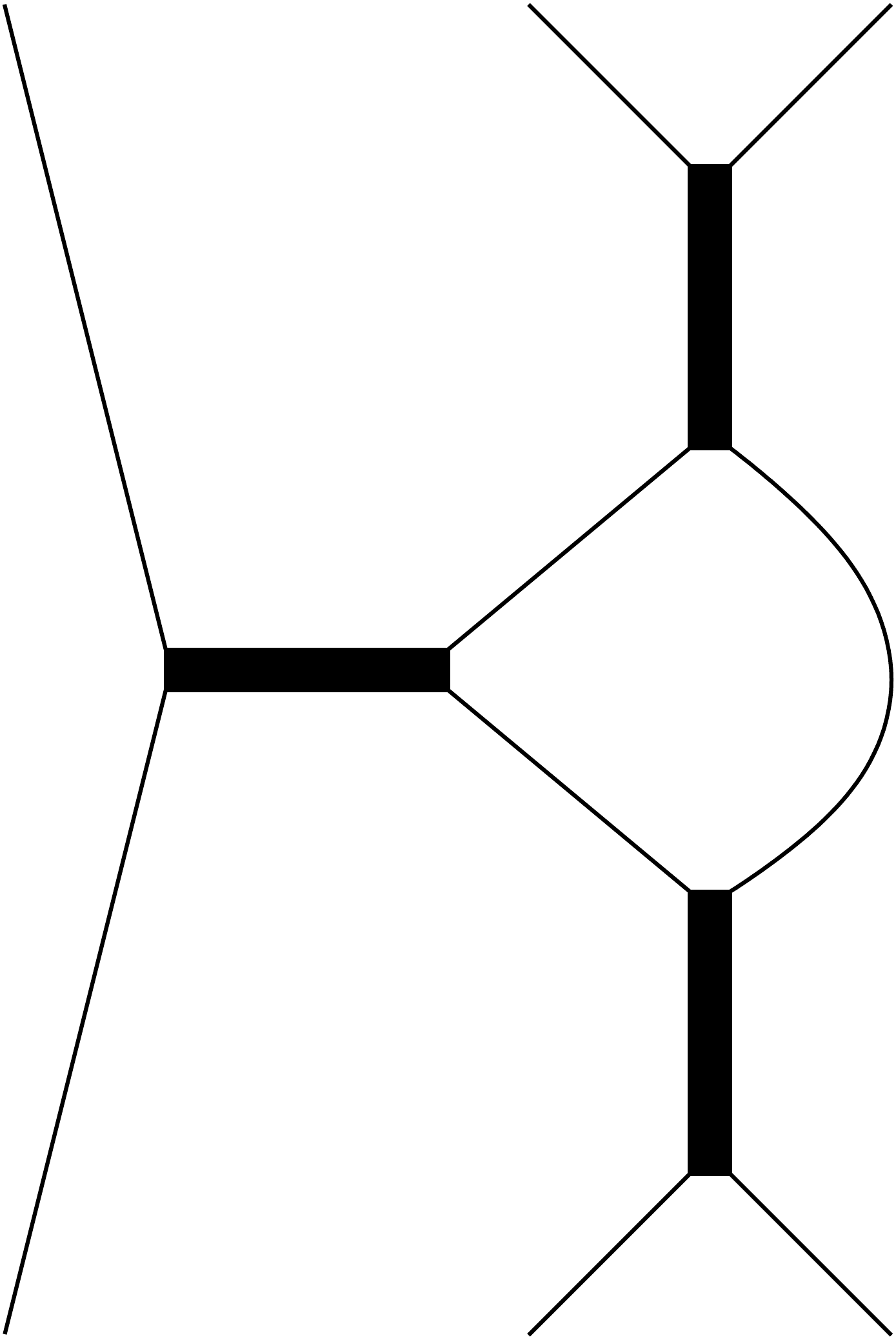}}\, \longleftrightarrow \, \raisebox{19pt}{\includegraphics[height=0.45in, angle = 180]{long-11}}, \quad \raisebox{-13pt}{\includegraphics[height=0.45in]{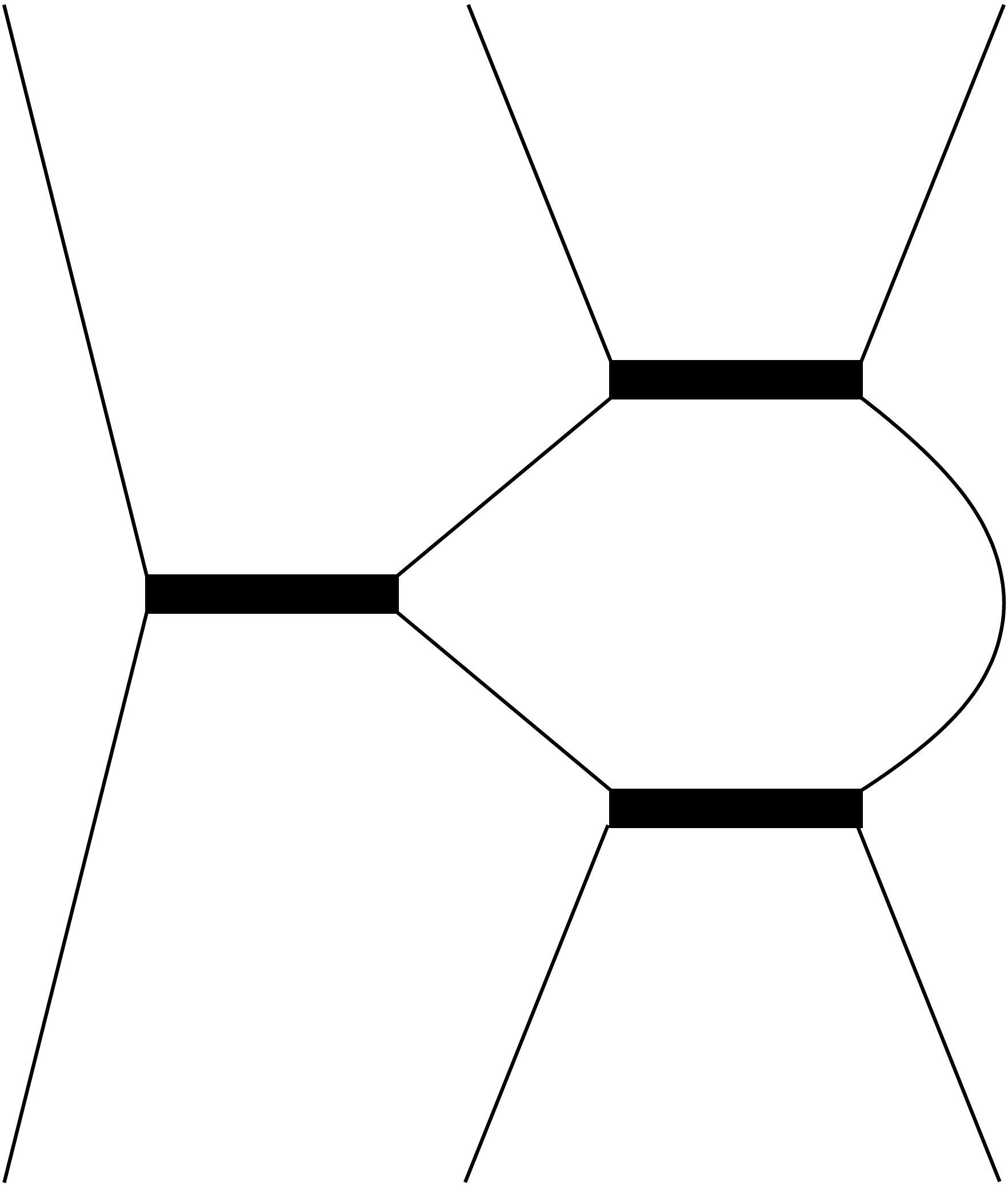}}\, \longleftrightarrow \, \raisebox{19pt}{\includegraphics[height=0.45in, angle = 180]{long-12}}, \quad \raisebox{-13pt}{\includegraphics[height=0.45in]{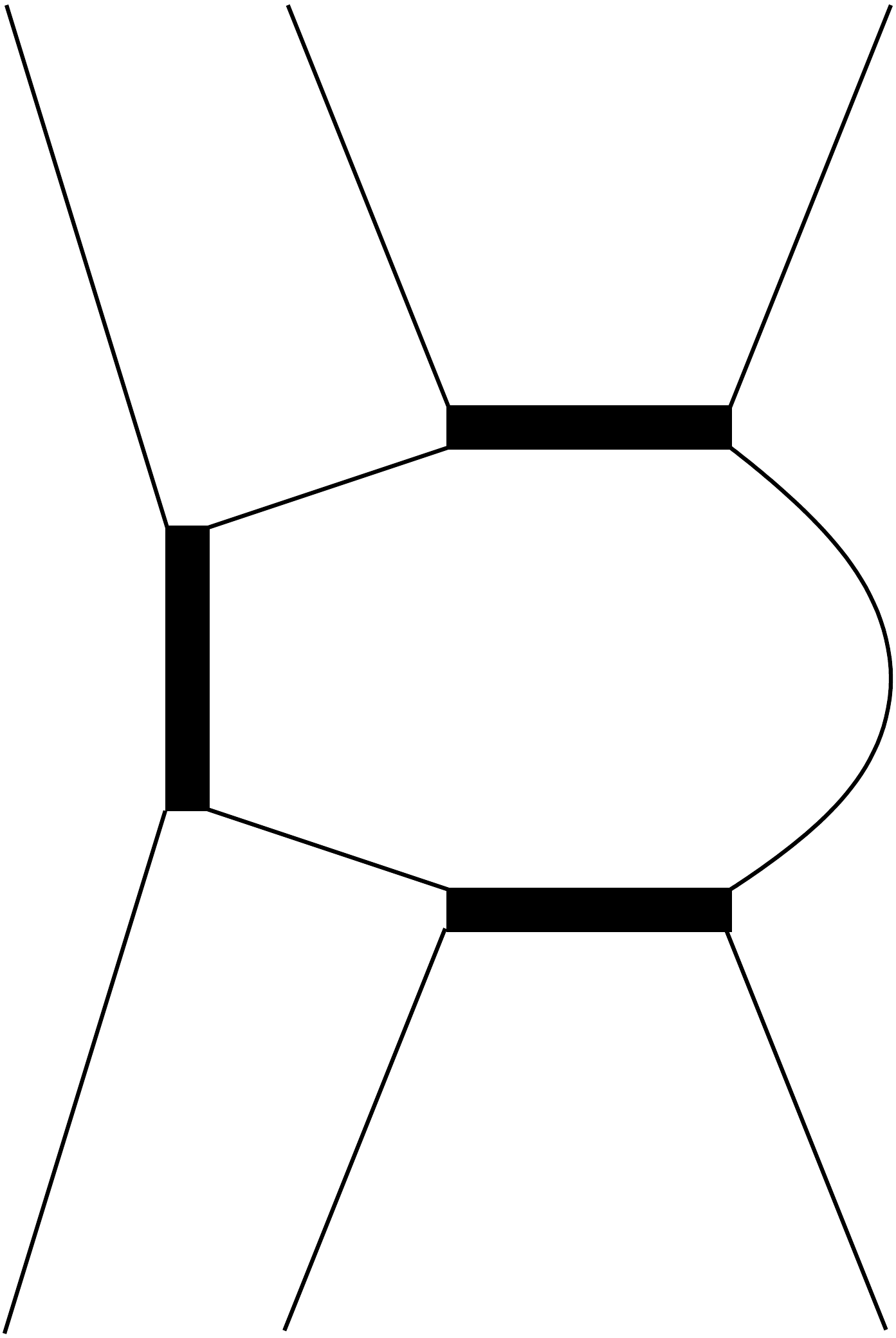}}\, \longleftrightarrow \, \raisebox{19pt}{\includegraphics[height=0.45in, angle = 180]{long-13}},
\end{eqnarray}
as well as other two variations of moves involving a 4-face
\begin{eqnarray}\label{good moves-implied2}
\raisebox{-13pt}{\includegraphics[height=0.45in]{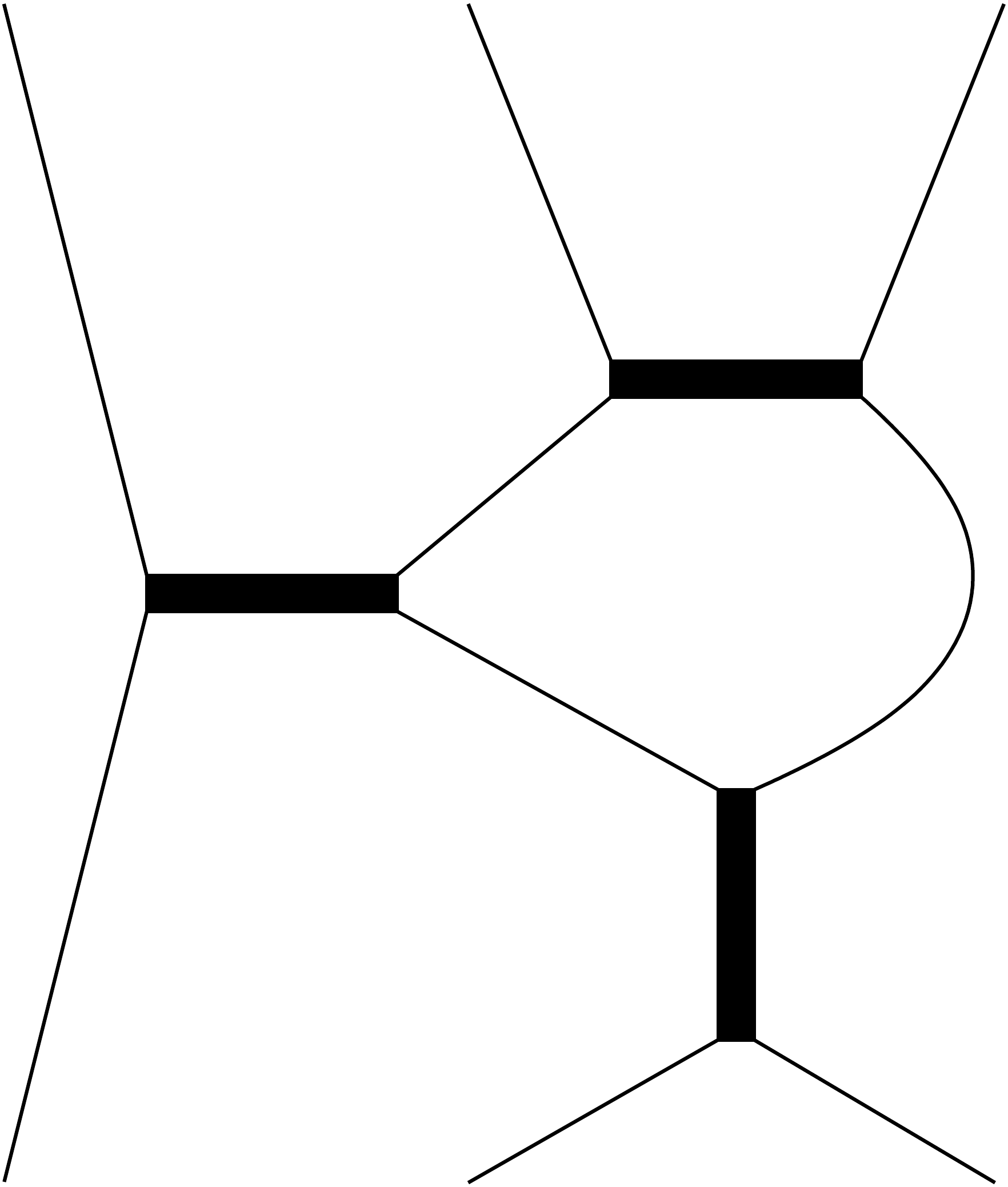}}\, \longleftrightarrow \, \reflectbox{\raisebox{-13pt}{\includegraphics[height=0.45in]{long-14}}}, \quad \reflectbox{\raisebox{19pt}{\includegraphics[height=0.45in, angle = 180]{long-14}}}\, \longleftrightarrow \, \raisebox{19pt}{\includegraphics[height=0.45in, angle = 180]{long-14}}\,.
\end{eqnarray}

We remark that the girth (the length of a shortest cycle) of the planar graph diagram $\Gamma$ is at most 5. For, if $\Gamma$ has $n$ vertices, then it has $3n/2$ edges and $2+n/2$ faces (counting the infinite face), and if the girth is 6 or larger, then by counting the edges we have:
$ 3n/2 \geq 6(2+ n/2)/2$, which is impossible.

We consider certain walks in the graph, which we call \textit{alternating walks}. Let $v$ be a vertex of the graph and choose an edge that is adjacent to $v$.  We construct a walk that starts at $v$ and proceeds along this edge, alternating between wide edges and standard edges in such a way that if the walk enters a wide edge from the right (or left) then it leaves it from the left (or right). The dotted line in the picture below represents such a walk:

\[\raisebox{-13pt}{\includegraphics[height=0.6in]{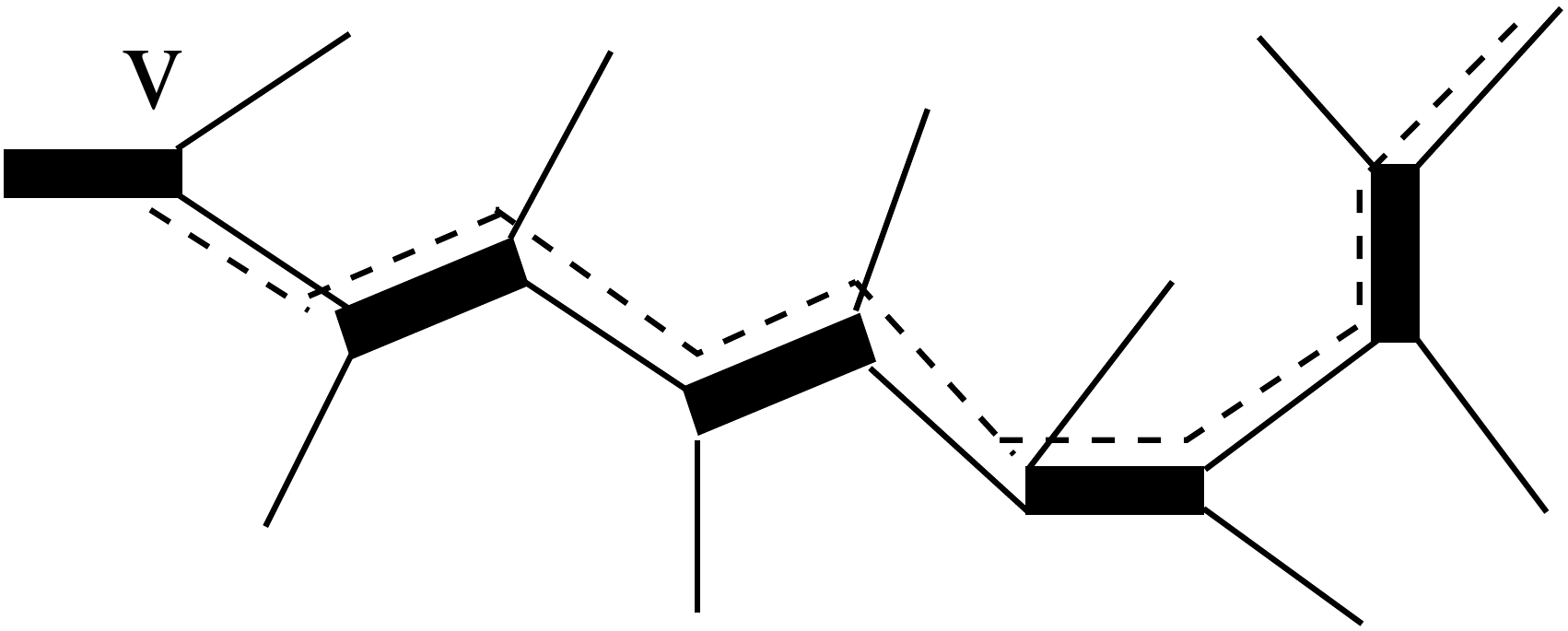}}\]

Since the graph contains a finite number of vertices, the walk returns to a vertex, $w$, it has already visited before, producing a cycle in the graph which, near $w$, looks like one of the diagrams shown in figure~\ref{fig:circuits}.

\begin{figure}[ht]
\raisebox{-13pt}{\includegraphics[height=0.6in]{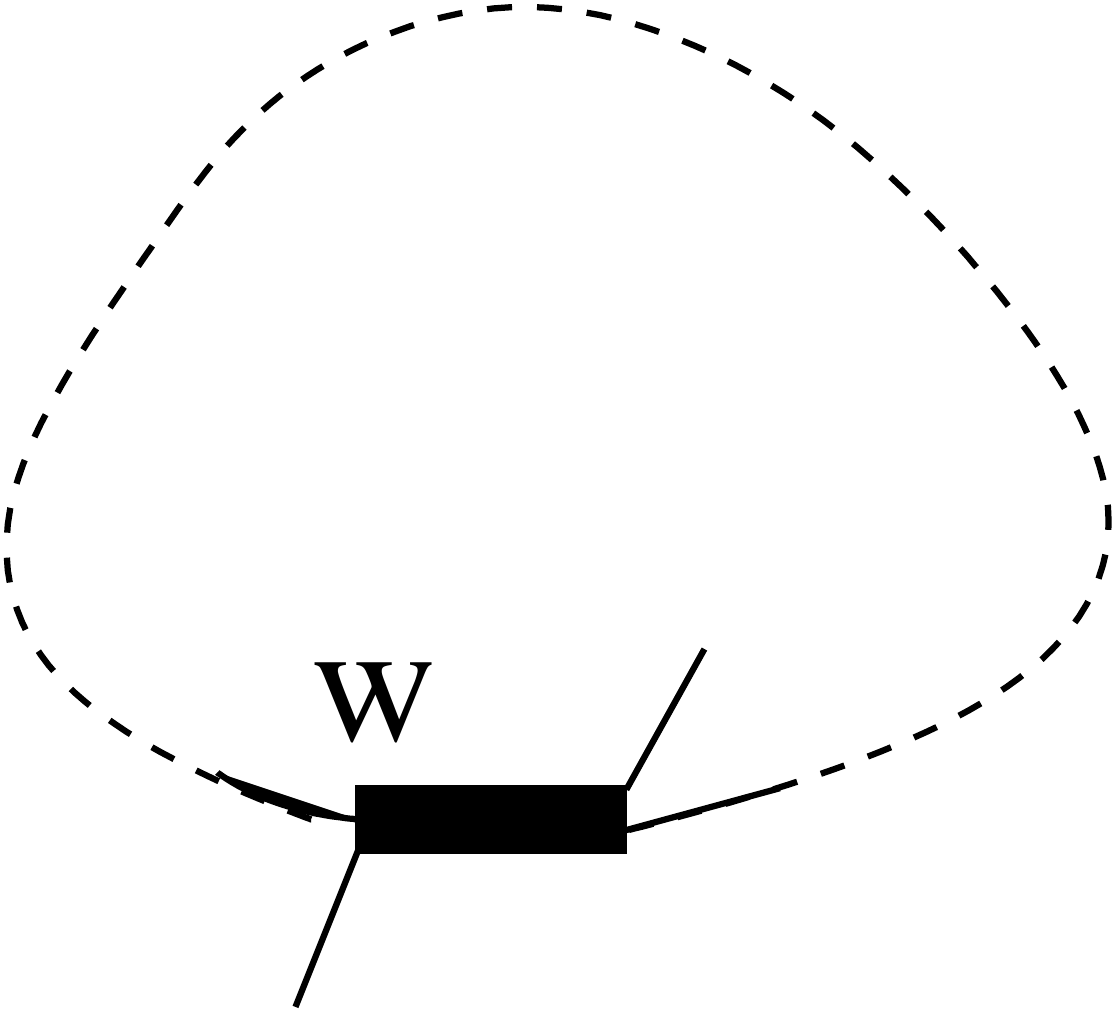}},\,\, \raisebox{-13pt}{\includegraphics[height=0.6in]{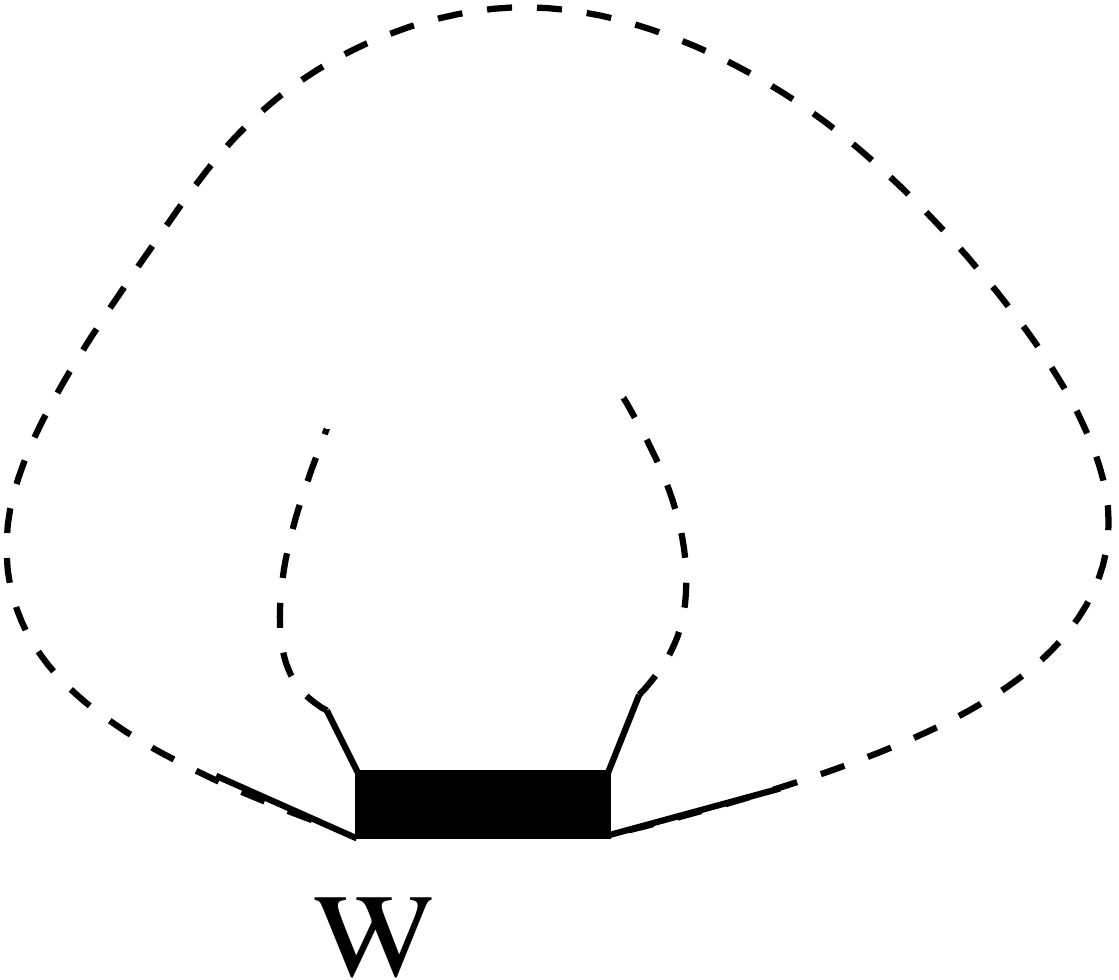}},\,\, \raisebox{-13pt}{\includegraphics[height=0.6in]{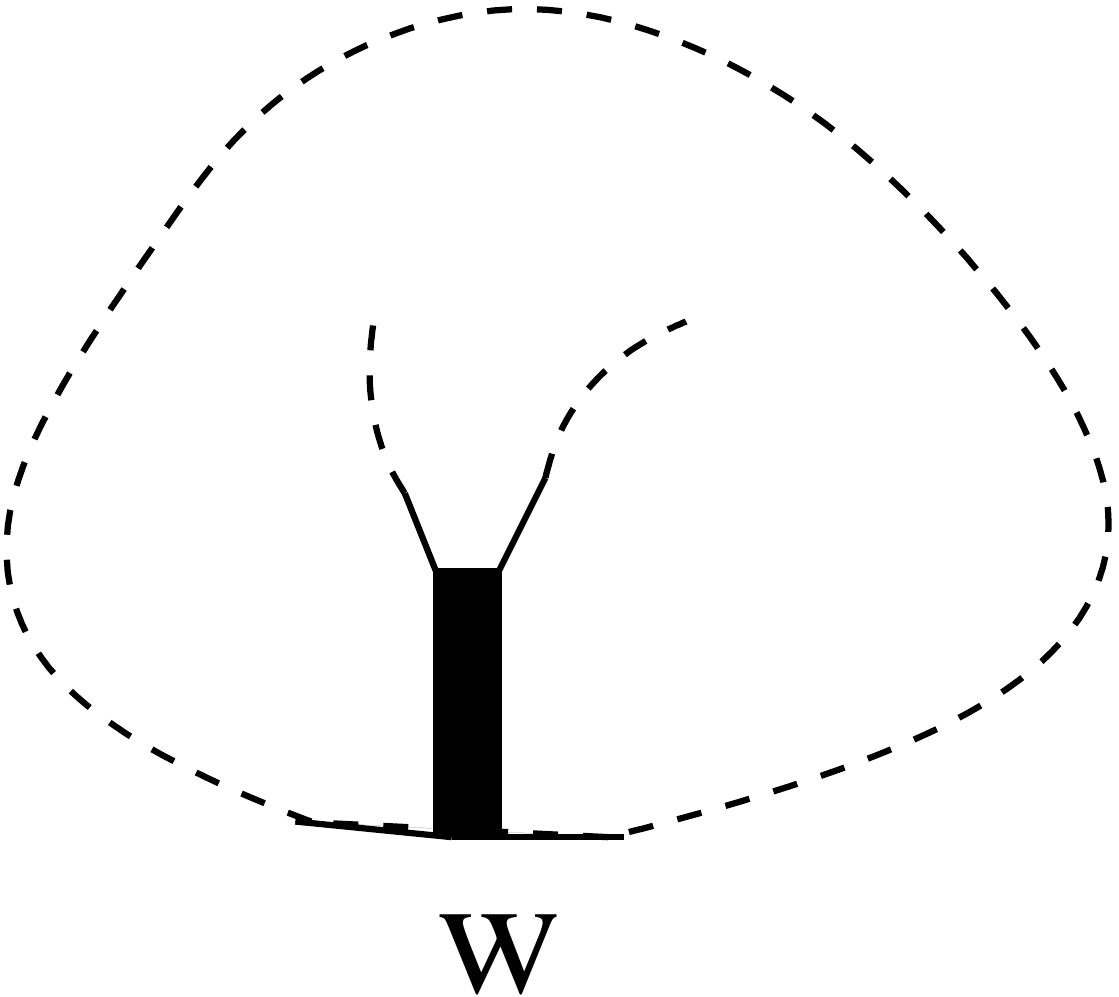}}, \,\, \raisebox{-13pt}{\includegraphics[height=0.6in]{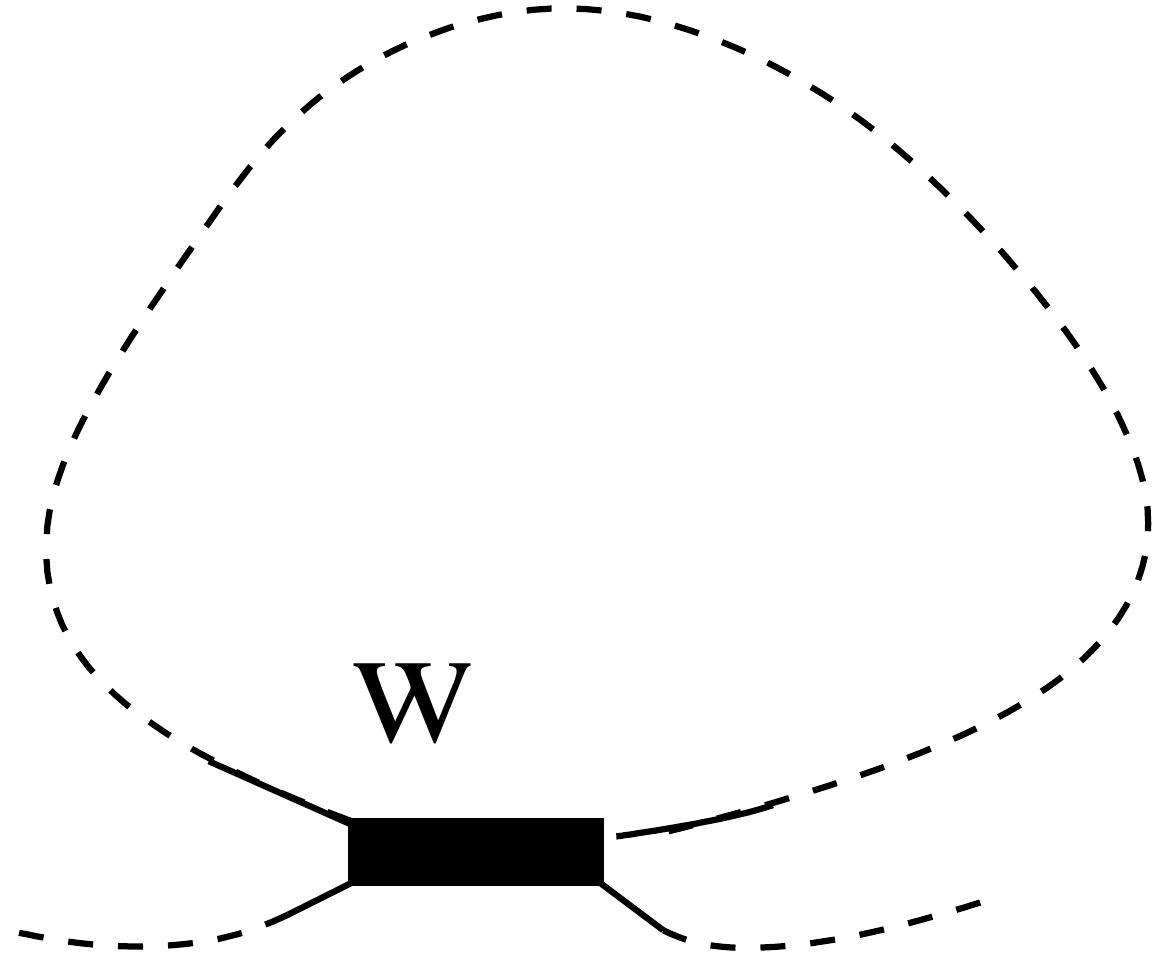}},\,\,\text{or} \,\, \raisebox{-19pt}{\includegraphics[height=0.7in]{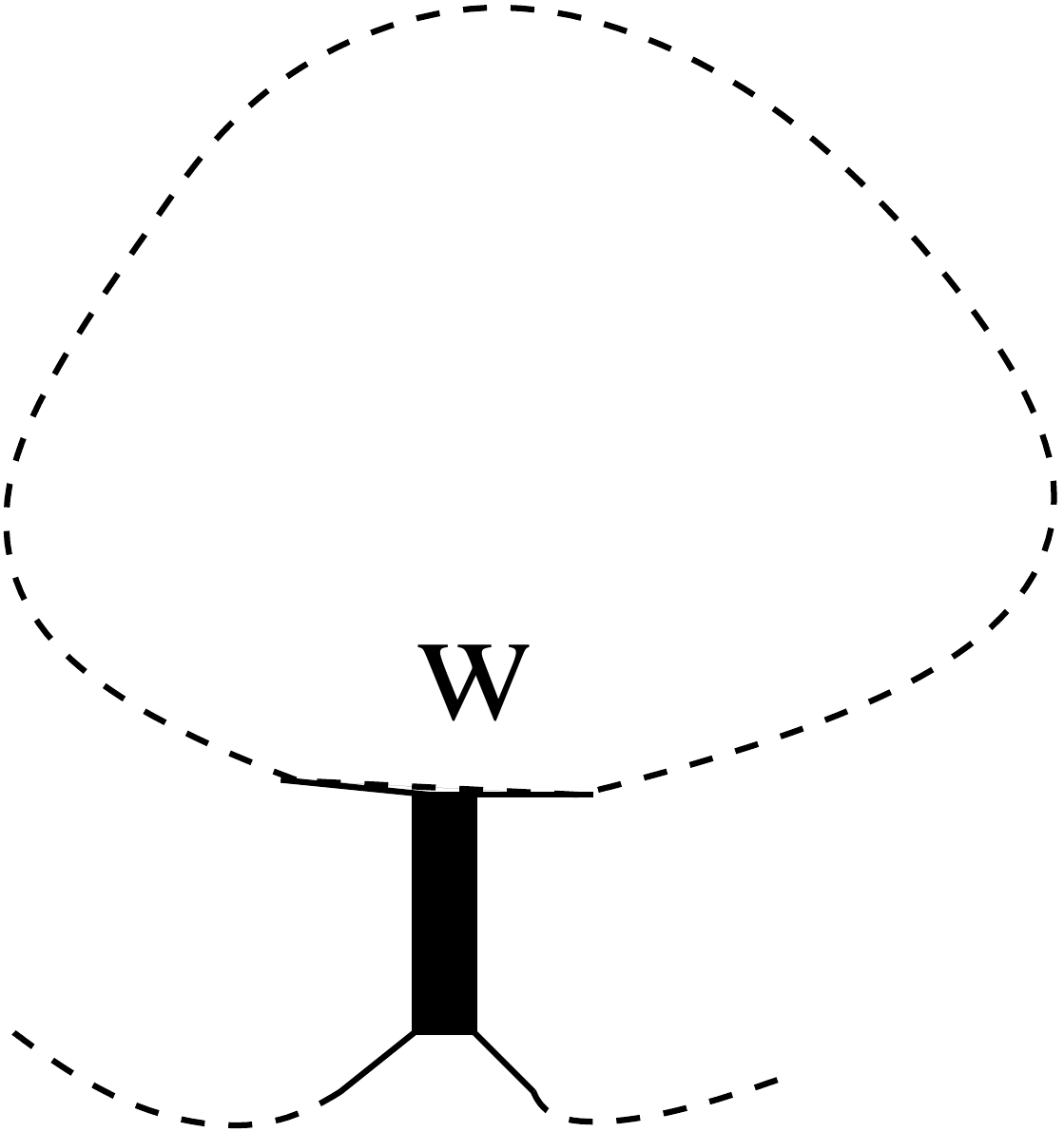}}
\caption{Cycles obtained from alternating walks}\label{fig:circuits}
\end{figure}
The dotted parts in the cycles given in figure~\ref{fig:circuits} are alternating walks, and therefore, along these walks, the two standard edges corresponding to the vertices of a wide edge and not belonging to the cycle, lie on the opposite sides of the cycle.

Using such a cycle, we show that the graph diagram contains one of the configurations depicted in Figure~\ref{fig:configurations}, where the dotted parts are alternating walks.

\begin{figure}[ht]
\raisebox{-9pt}{\includegraphics[height=0.3in]{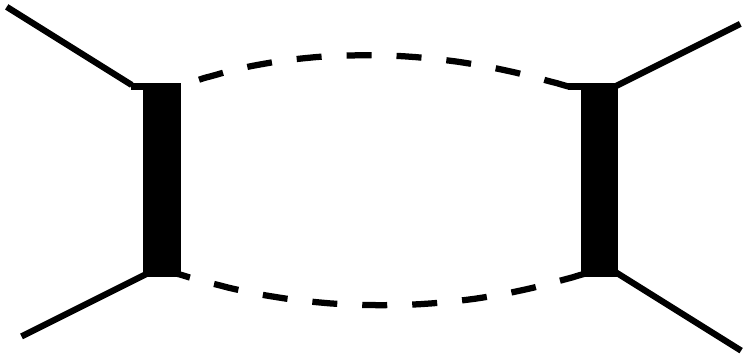}}, \quad  \quad \raisebox{-9pt}{\includegraphics[height=0.3in]{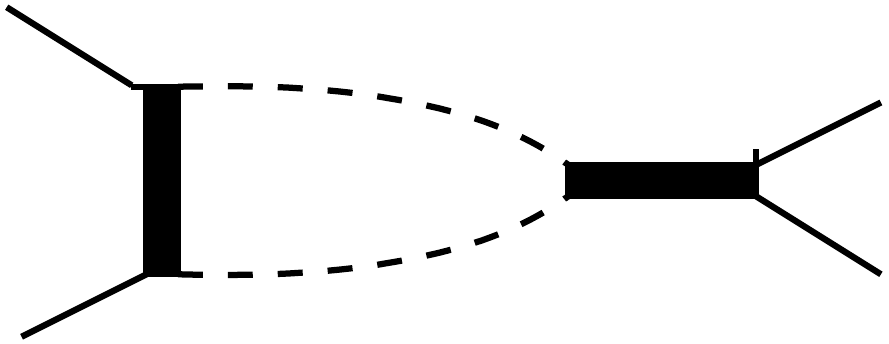}}, \quad  \quad \raisebox{-7pt}{\includegraphics[height=0.2in]{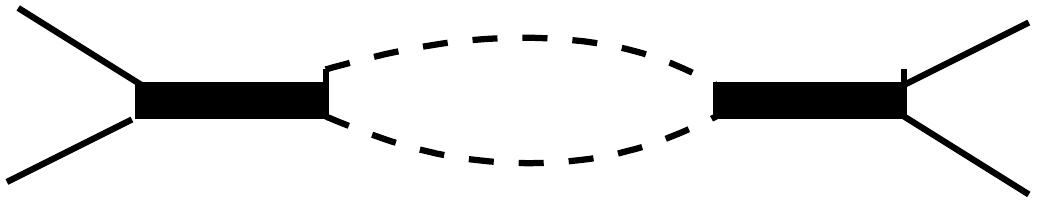}}
\caption{Desired configurations}\label{fig:configurations}
\end{figure}

Consider any of the the first three situations in Figure~\ref{fig:circuits}, and extend the walk going from the vertex $w$ inside the enclosed region. If the walk finishes at $w$ or at its neighboring vertex (shown in the diagram), or if the walk crosses itself before leaving the region, then we obtain a cycle of the fourth or fifth type as in Figure~\ref{fig:circuits}. If the walk leaves the enclosed region without crossing itself, then we arrive at one of the desired configurations in Figure~\ref{fig:configurations}.

In the fourth or fifth situation in Figure~\ref{fig:circuits}, we choose a standard edge inside the enclosed region and connected to a vertex (distinct from $w$) on the cycle. Call this vertex $x$. We follow the alternating walk which starts at the vertex $x$ and prolongs inside the enclosed region along the chosen standard edge. If this walk leaves the region without crossing itself, we arrive at a desired configuration. We give such a situation below:
\[\raisebox{-9pt}{\includegraphics[height=0.7in]{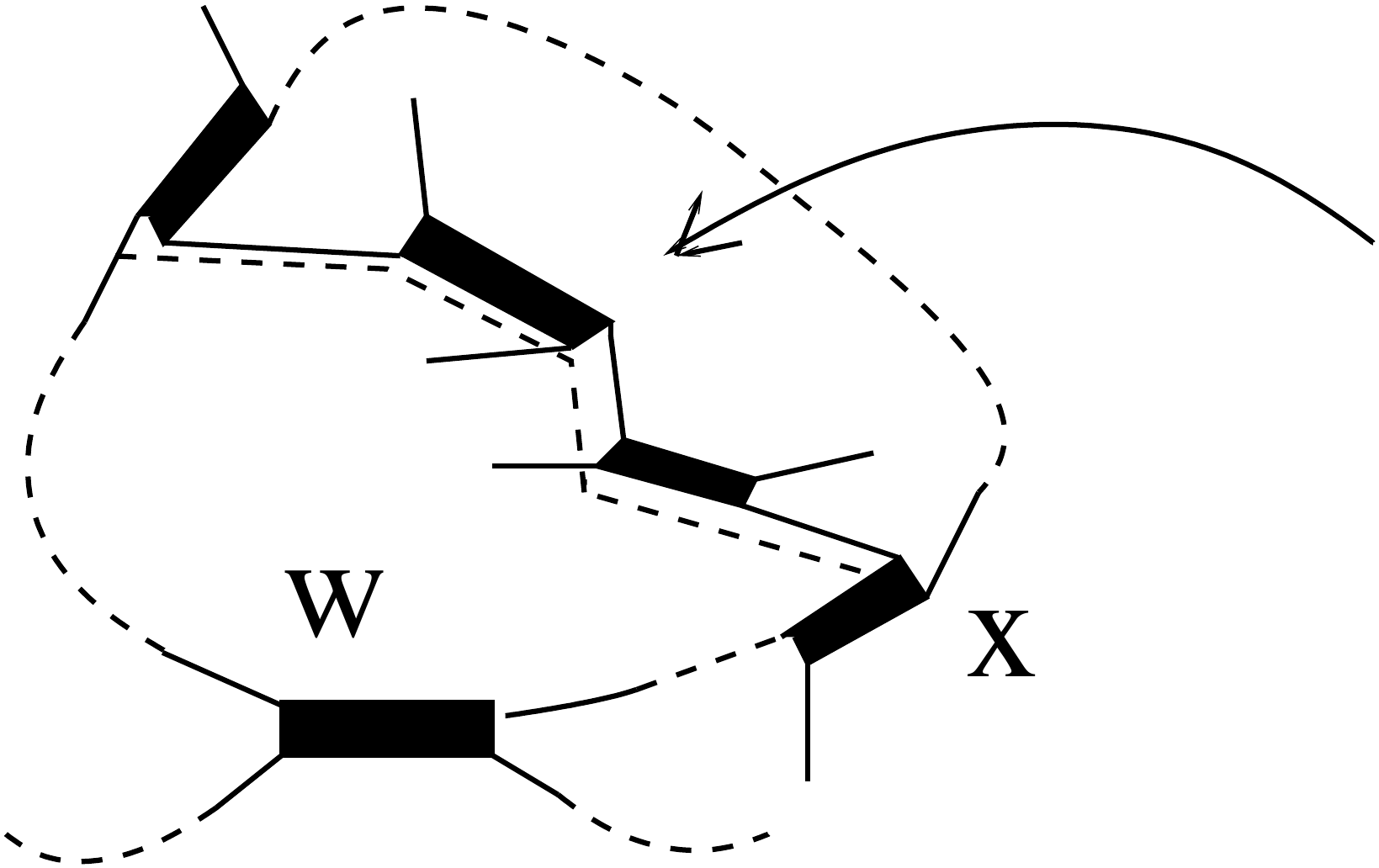}}\]
If this latter walk crosses itself before it leaves the region, we obtain another cycle of fourth or fifth type lying inside the initial one. In this situation we proceed as before, considering another walk going inside the latter cycle. If this walk crosses itself inside the latter cycle, we iterate the process until we obtain a walk that does not cross itself before it leaves the region, which produces one of the desired configurations in Figure~\ref{fig:configurations}.

Henceforth, the graph contains configurations of the desired form (as shown in Figure~\ref{fig:configurations}), and we choose one that does not contain another desired configuration inside. In other words, we choose a minimal configuration. This ensures that any walk starting inside the configuration must go outside without crossing itself. Moreover, any walk that goes in through one side of the configuration goes out through the other side without crossing itself.

We prove now that if there are vertices inside the (minimal) desired configuration, we can remove them in pairs by a sequence of the moves given in~(\ref{good moves}), or consequently, by a sequence of the moves displayed in~(\ref{good moves-implied1}) and~(\ref{good moves-implied2}). We remark that there is an even number of vertices inside the configuration, since every wide edge lying inside the configuration does not share a vertex with the cycle forming the configuration. Begin an alternating walk at one of these vertices and extend it until it crosses the configuration. At the penultimate vertex before this walk meets the configuration, proceed along the alternating walk that leaves the configuration through the same side. We obtain a `global' polygon adjacent to one side of the configuration. Such a situation is presented below, where the dotted lines are alternating walks:
\[\raisebox{-9pt}{\includegraphics[height=0.6in]{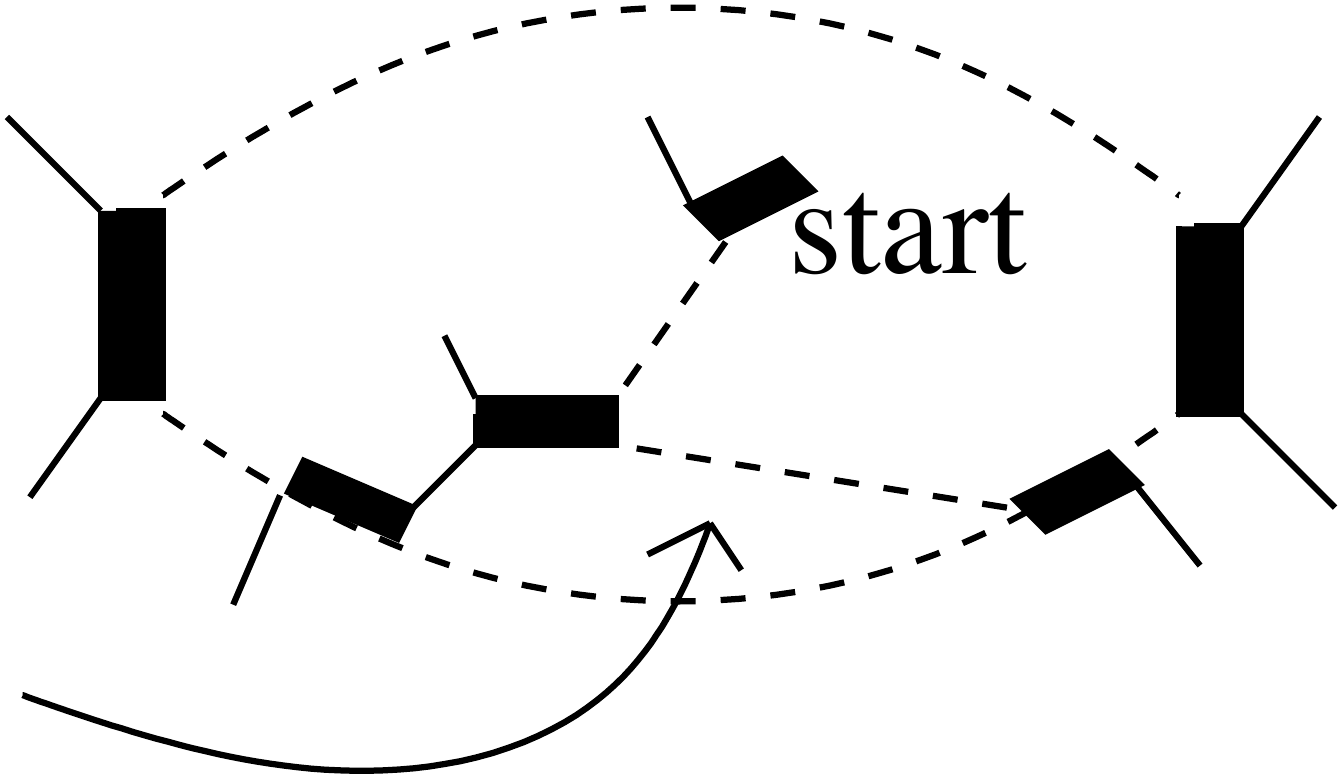}}\]

At the penultimate vertex of the latter walk before it crosses to the outside of the configuration, begin another alternating walk that goes inside the polygon and leaves it through the side of the configuration. This process produces a new global polygon inside the previous one and adjacent to the same side of the configuration:
\[\raisebox{-9pt}{\includegraphics[height=0.6in]{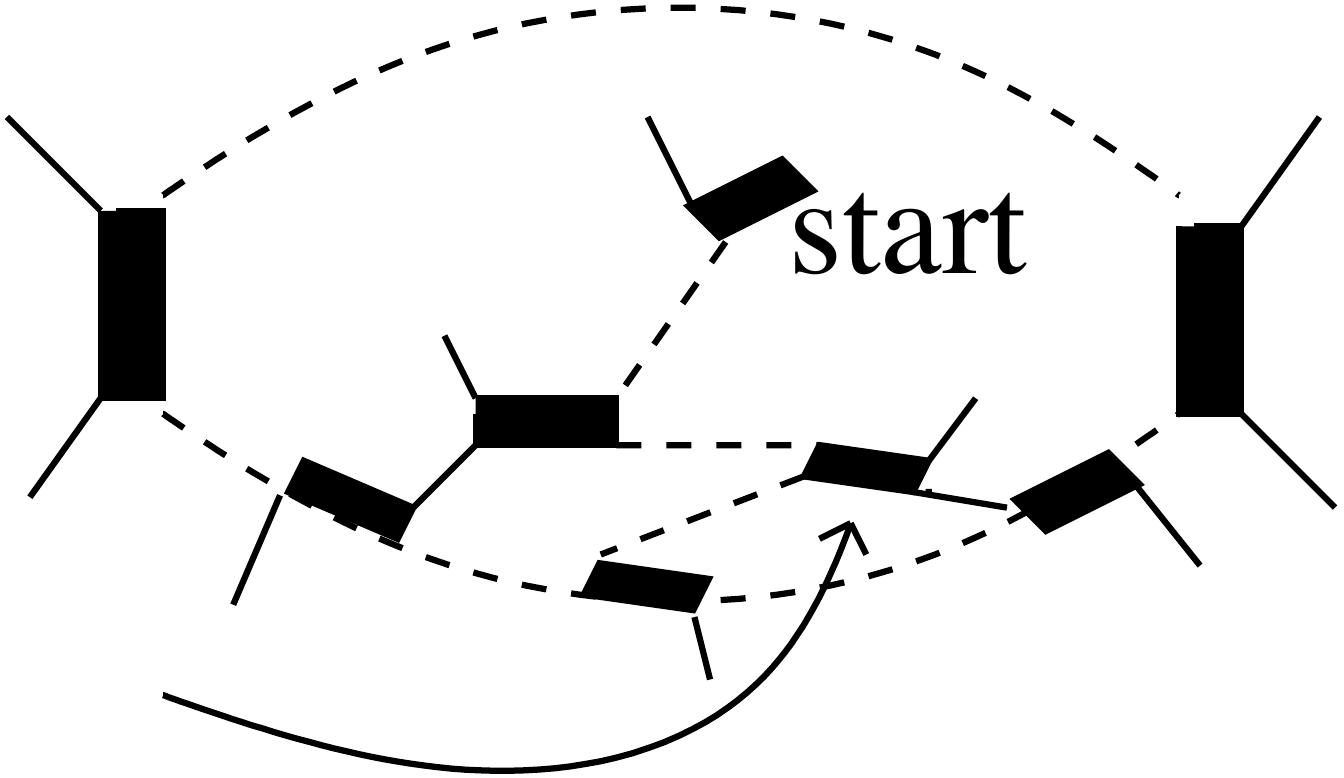}}\]

Iterating this process, we obtain a chain of polygons $p_1 \supset p _2 \supset \dots \supset p_n$ adjacent to one side of the configuration, where $p_n$ is a `simple' polygon. More precisely, $p_n$ is either a 4-face, 5-face or 6-face in the original graph diagram $\Gamma$:
\[\raisebox{-15pt}{\includegraphics[height=0.45in]{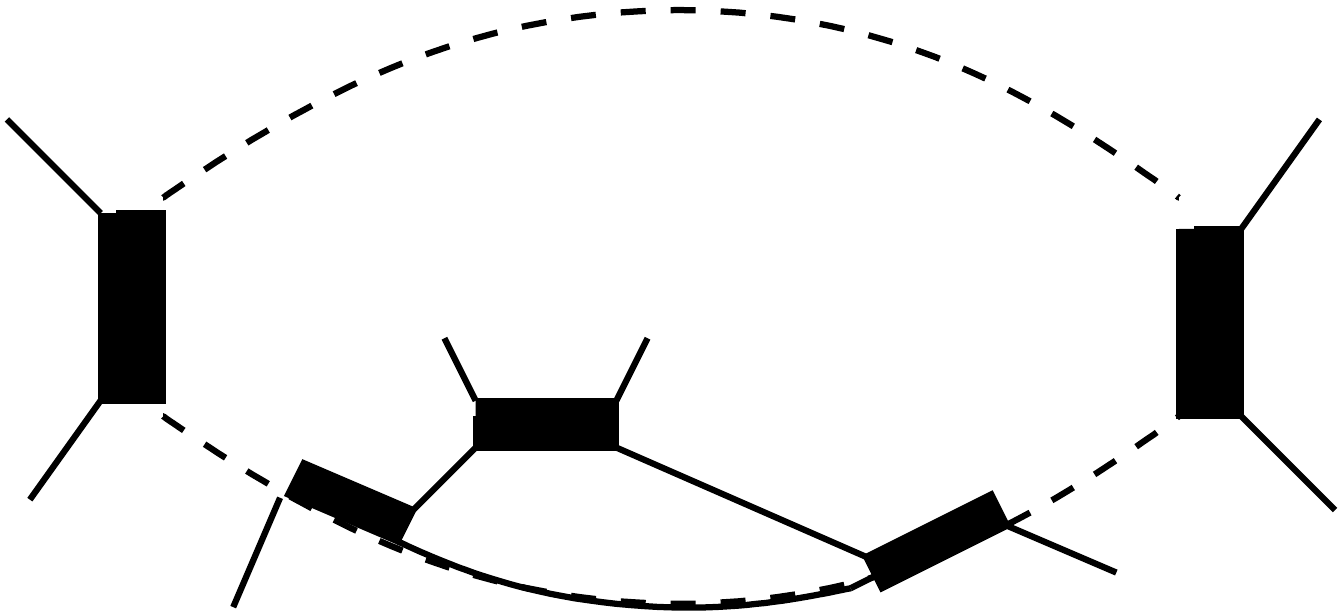}}, \quad \raisebox{-19pt}{\includegraphics[height=0.5in]{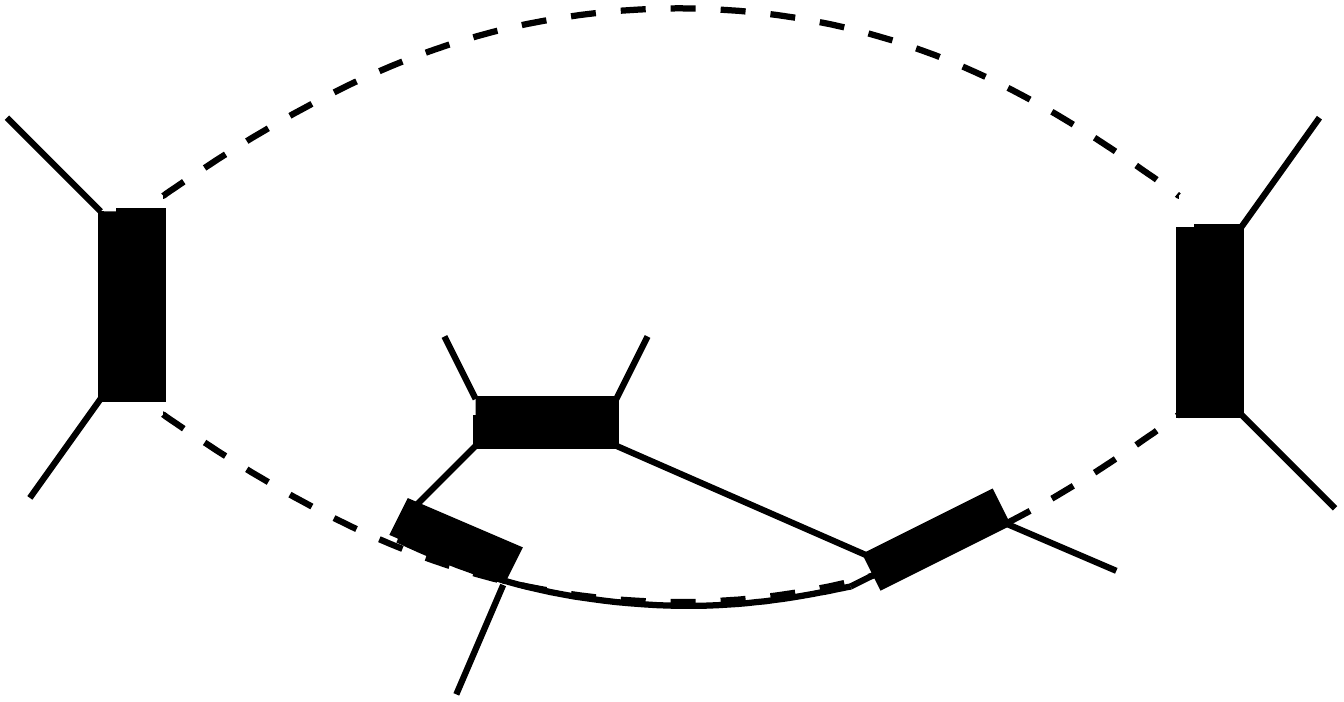}}, \quad \raisebox{-19pt}{\includegraphics[height=0.5in]{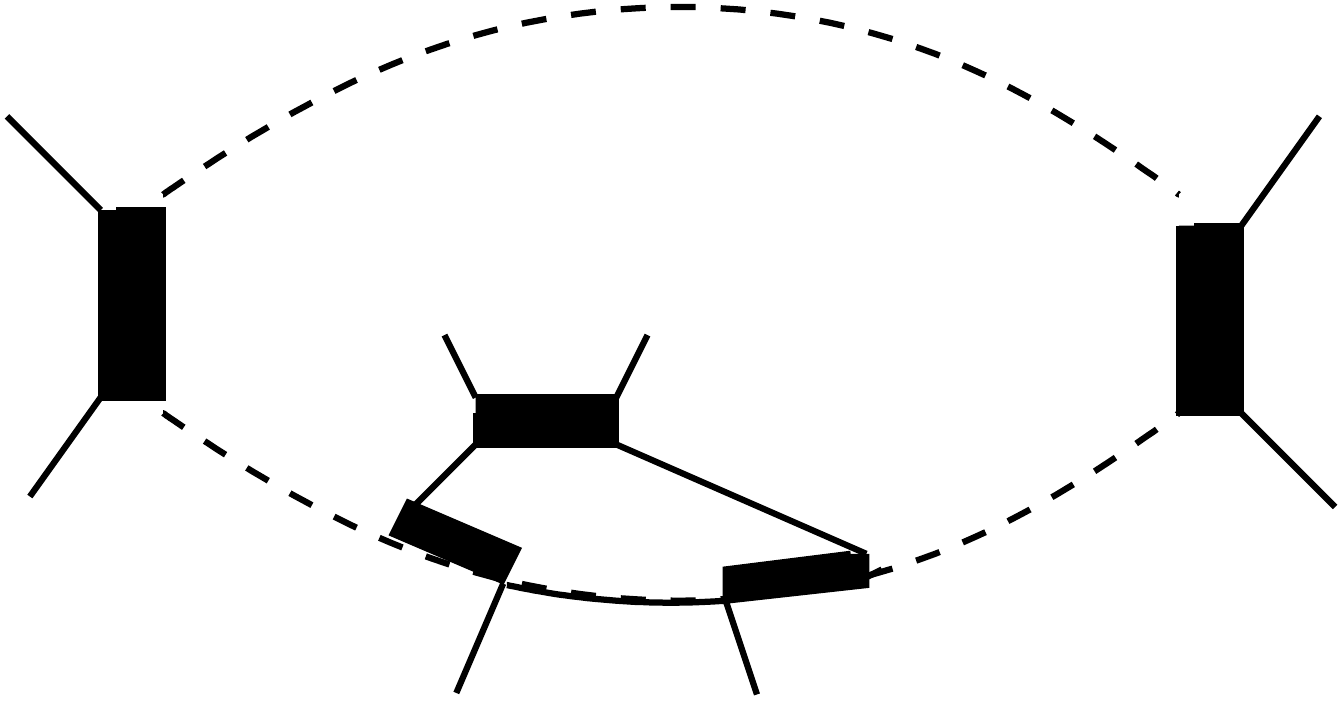}}\,\,\]
Henceforth, we can use those moves given in~(\ref{good moves}) - (\ref{good moves-implied2}) involving 4-, 5- or 6-faces, to remove a pair of vertices joined by a wide edge inside the configuration:

\[\raisebox{-19pt}{\includegraphics[height=0.5in]{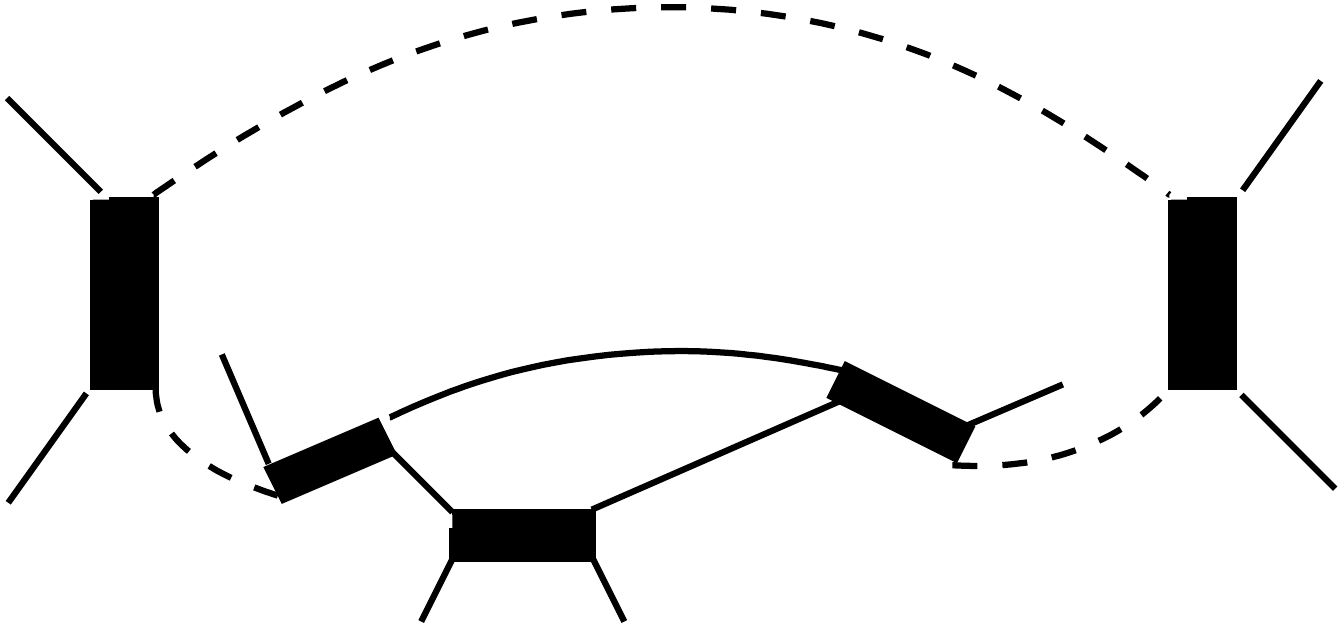}},\quad \raisebox{-19pt}{\includegraphics[height=0.5in]{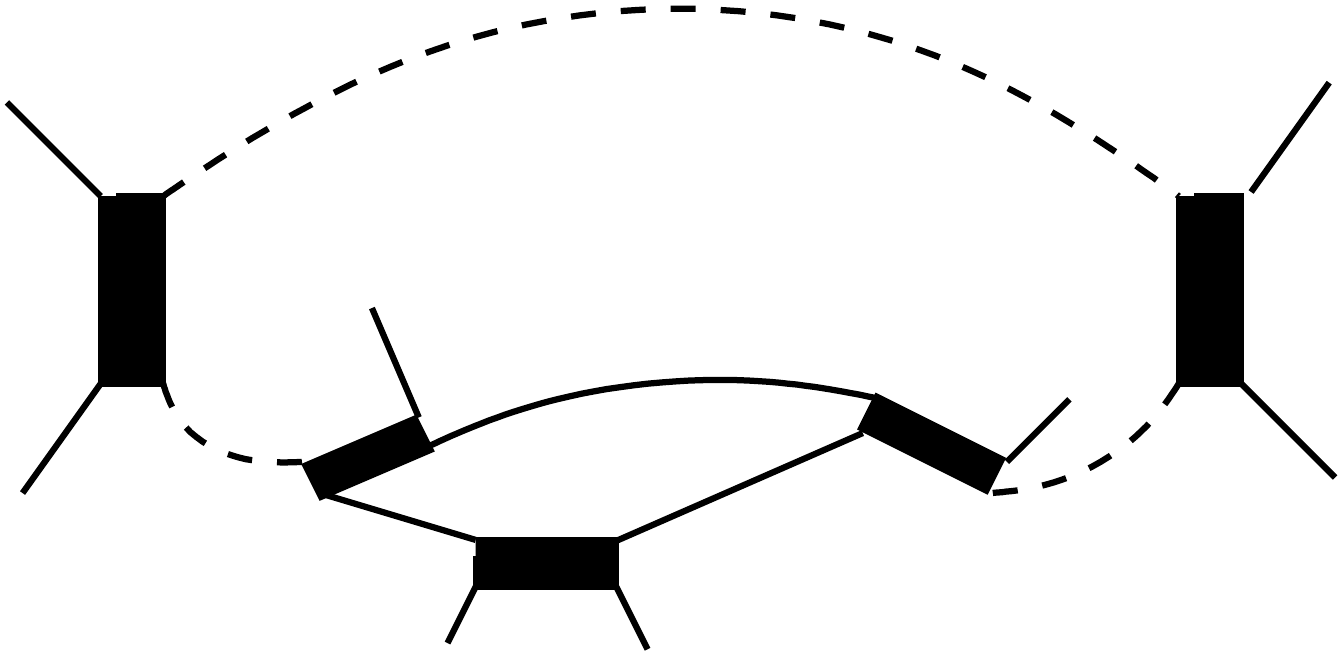}},\quad \raisebox{-19pt}{\includegraphics[height=0.5in]{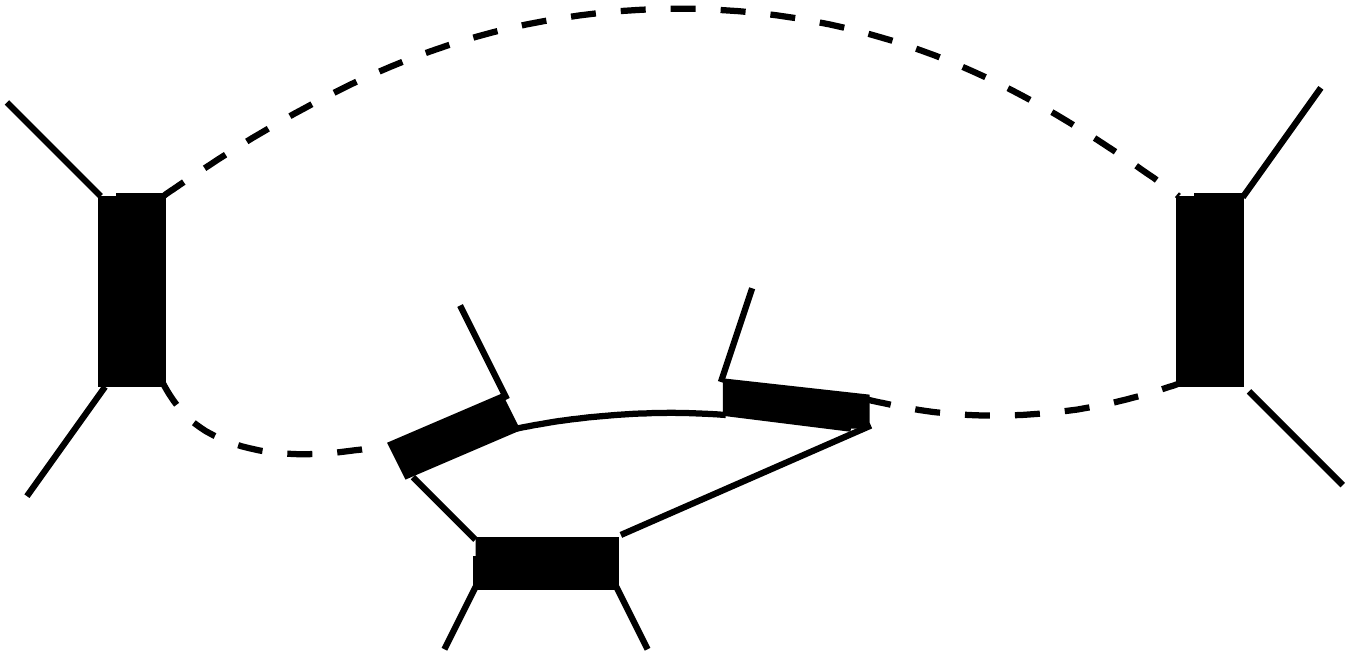}}\,\,\]

We apply this method recursively to remove all vertices within the configuration. The resulting diagram has the same number of vertices, but those vertices which were inside the configuration are now outside of it. Specifically, the resulting configuration has no vertices within it but standard edges connecting the two sides (alternating walks) of the configuration and separating $k$-faces inside of it, where $3 \leq k \leq 8$. Such a resulting configuration is shown below:
 \[\raisebox{-9pt}{\includegraphics[height=0.5in]{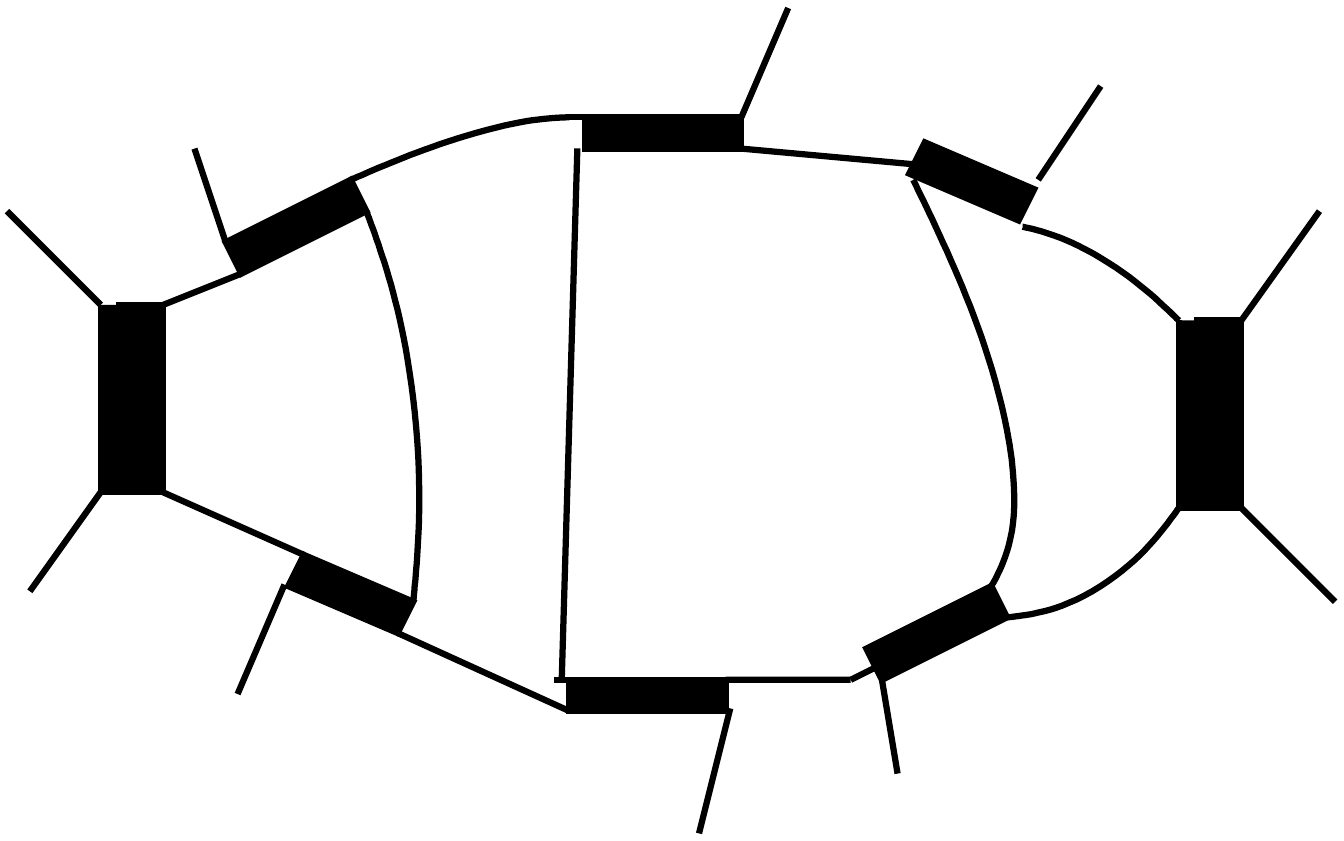}}\]
 Finally, we employ again the moves depicted in~(\ref{good moves-implied1}), (\ref{good moves-implied2}) and the second move in~(\ref{good moves}) to remove all edges inside the configuration, one at the time, starting with the leftmost (or rightmost) edge. It is important to remark that at each iteration of this process, the leftmost (or rightmost) face inside the configuration is a $k$-face, with $k\leq6$. We exemplify below such iterations, where at each stage we encircled the region where a move is applied:  
  \[\raisebox{-18pt}{\includegraphics[height=0.5in]{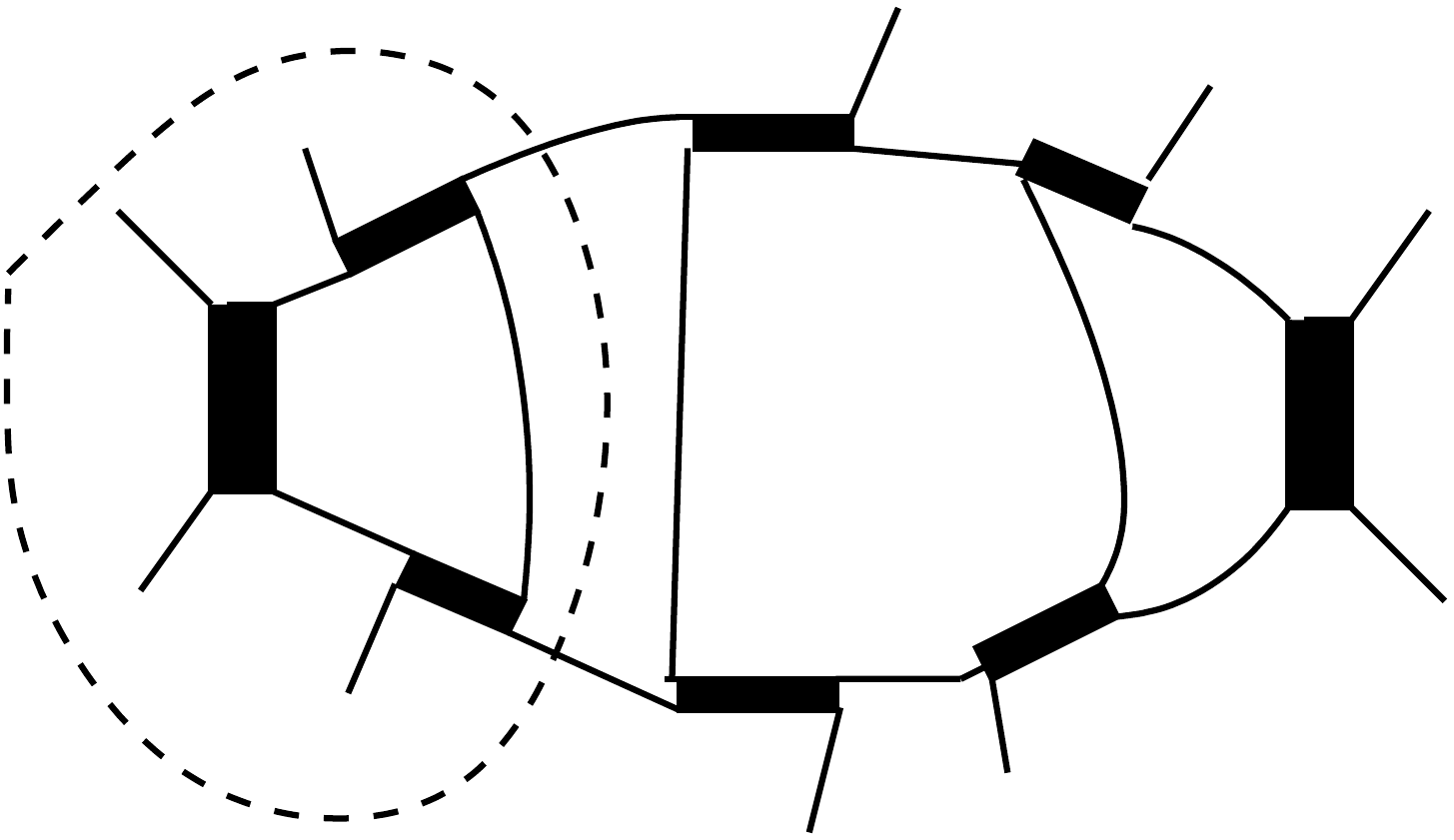}} \longrightarrow \raisebox{-18pt}{\includegraphics[height=0.5in]{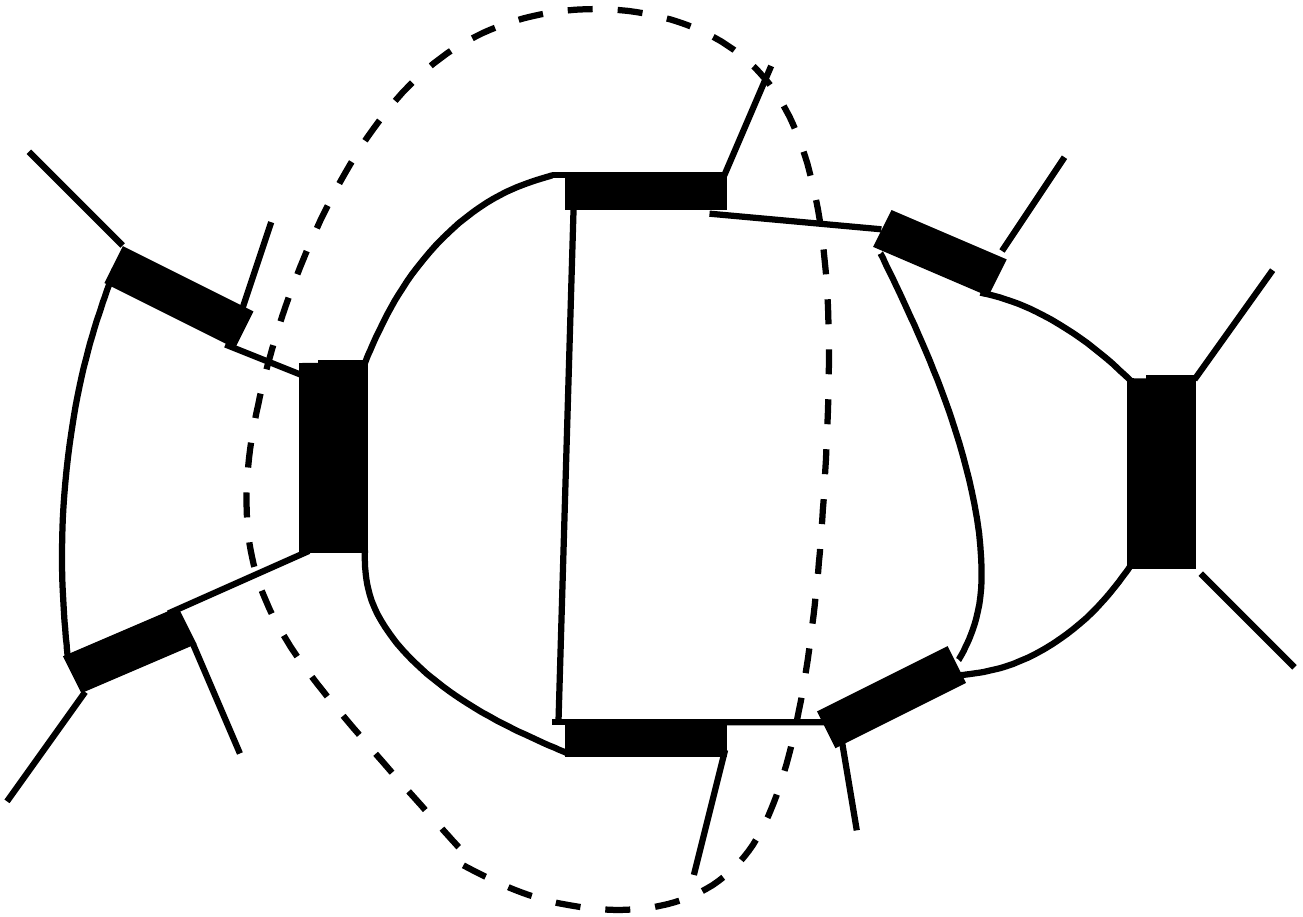}} \longrightarrow \raisebox{-18pt}{\includegraphics[height=0.5in]{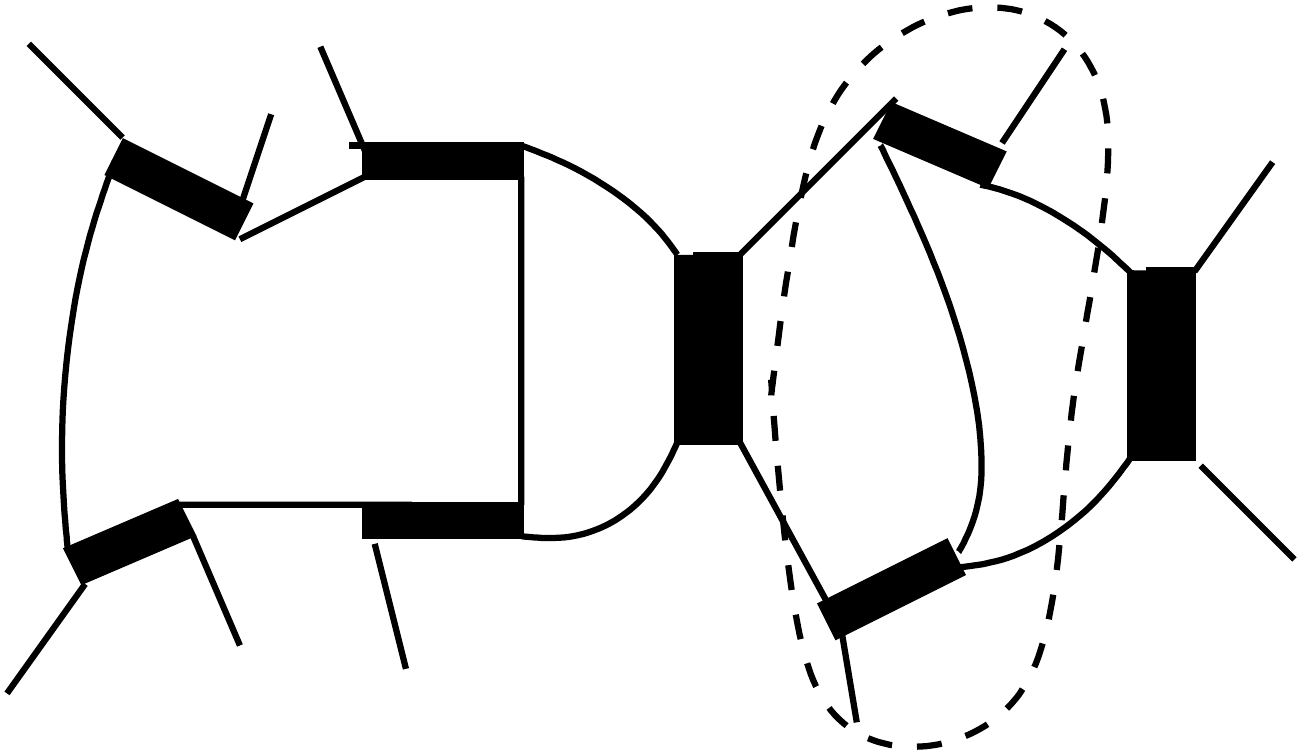}} \longrightarrow \raisebox{-18pt}{\includegraphics[height=0.5in]{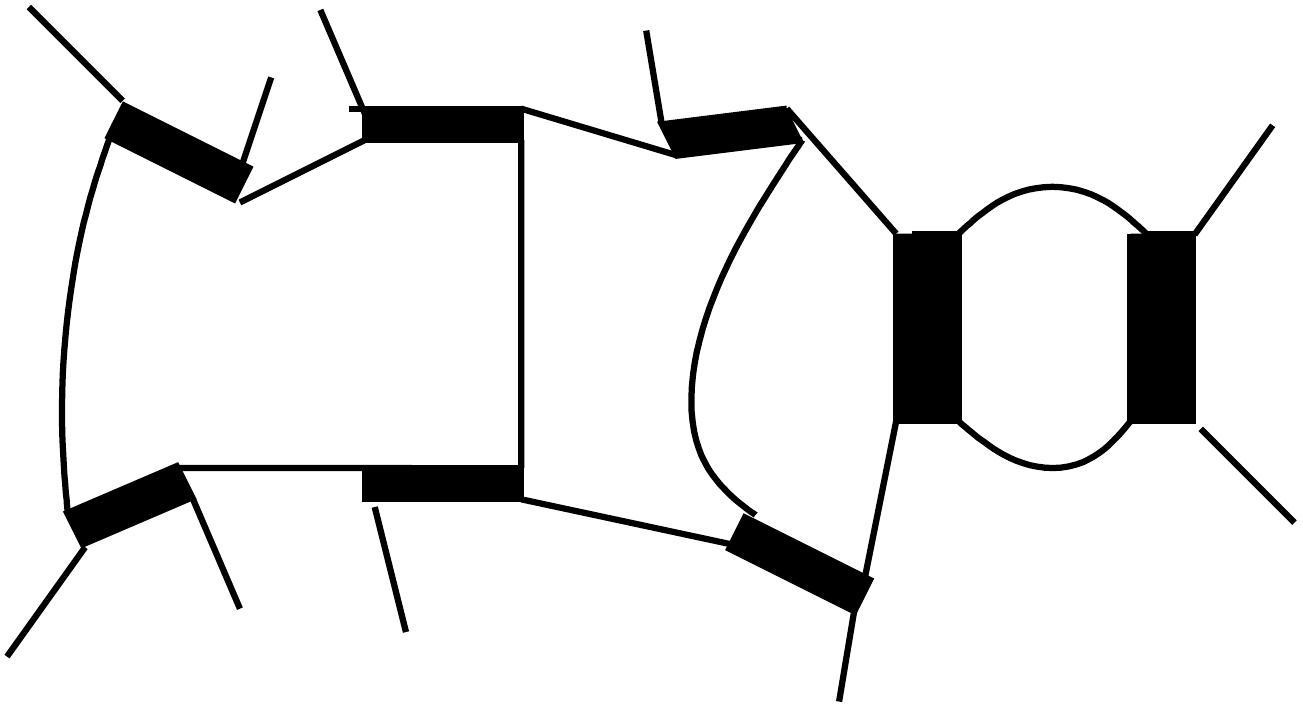}} \]
 
 Henceforth, we arrive at a graph containing one of the configurations \raisebox{-4pt}{\includegraphics[height=0.18in]{rem-square}}, \raisebox{-4pt}{\includegraphics[height=0.18in]{rem-triangle}}, or \raisebox{-4pt}{\includegraphics[height=0.45in, angle = 90]{rem-digon}}\,. This concludes the proof.
\end{proof}

\begin{proof} (of Theorem~\ref{thm:unique poly})
Let $P$ and $\tilde{P}$ be two polynomials that satisfy the skein relations (\ref{gi2}) - (\ref{gi6}). We will show by induction on the number of vertices that $P(\Gamma) = \tilde{P}(\Gamma)$ for any planar trivalent graph diagram $\Gamma$. If $\Gamma$ has no vertices then it is a disjoint union of $c$ unknots, and therefore $P(\Gamma) = \alpha^{c-1 }= \tilde{P}(\Gamma)$, by the first two graph skein relations.

Assume that $P$ and $\tilde{P}$ are equal for graphs with up to $n$ vertices, and let $\Gamma$ be a graph with $n+2$ vertices. 

If $\Gamma$ contains any of the local configuration in~(\ref{good config}), then $P(\Gamma)$ and $\tilde{P}(\Gamma)$ can be calculated in terms of the polynomials of new planar trivalent diagrams of graphs with less number of vertices, and therefore the induction hypothesis implies that $P(\Gamma) = \tilde{P}(\Gamma)$.

If $\Gamma$ does not contain any of the local configurations depicted in~(\ref{good config}), then Proposition~\ref{prop:polynomial} implies that there is a finite sequence of graphs $\Gamma = \Gamma_0\to \Gamma_1\to \dots \to \Gamma_k$, where $\Gamma_k$ contains one of the local configurations \raisebox{-4pt}{\includegraphics[height=0.18in]{rem-square}}, \raisebox{-4pt}{\includegraphics[height=0.18in]{rem-triangle}}, or \raisebox{-4pt}{\includegraphics[height=0.45in, angle = 90]{rem-digon}}\,.  Since the moves in~(\ref{good moves}) preserve the number of vertices, the graphs $\Gamma_1, \dots, \Gamma_k$ have the same number of vertices as $\Gamma$ and $P(\Gamma_k) = \tilde{P}(\Gamma_k)$. Moreover, $P(\Gamma) - P(\Gamma_1) = \tilde{P}(\Gamma) - \tilde{P}(\Gamma_1), P(\Gamma_1) - P(\Gamma_2) = \tilde{P}(\Gamma_1) - \tilde{P}(\Gamma_2), \dots, P(\Gamma_{k-1}) - P(\Gamma_k) = \tilde{P}(\Gamma_{k-1}) - \tilde{P}(\Gamma_k) $.  Hence $P(\Gamma) - P(\Gamma_k) = \tilde{P}(\Gamma) - \tilde{P}(\Gamma_k)$, implying that $P(\Gamma) = \tilde{P}(\Gamma)$.
\end{proof}

\textbf{Acknowledgements.} The first author gratefully acknowledges NSF partial support  through grant DMS - 0906401. She would also like to thank Scott Carter for helpful discussions on trivalent graphs and related topics in low-dimensional topology.

The second author would like to thank the first author for the opportunity to participate in the research and writing of this paper.  He is also very grateful for California State University, Fresno, which provided an Undergraduate Research Grant to participate in the research of this paper.


\end{document}